\documentclass{memo-l}
\usepackage{hyperref}
\usepackage{fullpage}
\usepackage{amsrefs}
\usepackage{amssymb}
\usepackage{tikz}

\usepackage[matrix,arrow,curve,cmtip]{xy}
\newdir{ >}{{}*!/-9pt/\dir{>}}

\newcommand{\Ab}	{\operatorname{Ab}}
\newcommand{\Alg}	{\operatorname{Alg}}
\newcommand{\AugAlg}	{\operatorname{AugAlg}}
\newcommand{\Aut}       {\operatorname{Aut}}
\newcommand{\CSys}	{\operatorname{CSys}}
\newcommand{\MCSys}	{\operatorname{MCSys}}
\newcommand{\Green}	{\operatorname{Green}}
\newcommand{\Groups}	{\operatorname{Groups}}
\newcommand{\Hom}	{\operatorname{Hom}}
\newcommand{\Inj}	{\operatorname{Inj}}
\newcommand{\uHom}	{\underline{\operatorname{Hom}}}
\newcommand{\MP}	{\operatorname{MP}}
\newcommand{\Mackey}    {\operatorname{Mackey}}
\newcommand{\Map}       {\operatorname{Map}}
\newcommand{\Mod}       {\operatorname{Mod}}
\newcommand{\NMod}      {\operatorname{NMod}}
\newcommand{\Orb}	{\operatorname{Orb}}
\newcommand{\Orbt}	{\operatorname{Orb}^\times}
\newcommand{\Rings}	{\operatorname{Rings}}
\newcommand{\Sec}       {\operatorname{Sec}}
\newcommand{\Semigroups}{\operatorname{Semigroups}}
\newcommand{\Semirings}	{\operatorname{Semirings}}
\newcommand{\Sets}	{\operatorname{Sets}}
\newcommand{\Spaces}	{\operatorname{Spaces}}
\newcommand{\Sub}	{\operatorname{Sub}}

\newcommand{\TP}	{\operatorname{TP}}
\newcommand{\Tambara}   {\operatorname{Tambara}}
\newcommand{\Trans}	{\operatorname{Trans}}

\newcommand{\eqker}     {\operatorname{eqker}}
\newcommand{\nrm}	{\operatorname{norm}}
\newcommand{\res}	{\operatorname{res}}
\newcommand{\stab}	{\operatorname{stab}}
\newcommand{\sub}	{\operatorname{sub}}
\newcommand{\trc}	{\operatorname{trace}}
\newcommand{\coind}	{\operatorname{coind}}
\newcommand{\coord}	{\operatorname{coord}}
\newcommand{\ind}       {\operatorname{ind}}
\newcommand{\img}       {\operatorname{img}}
\newcommand{\cok}       {\operatorname{cok}}
\newcommand{\orb}       {\operatorname{orb}}
\newcommand{\opp}	{\operatorname{op}}

\newcommand{\CA}        {{\mathcal{A}}}
\newcommand{\CB}        {{\mathcal{B}}}
\newcommand{\CC}        {{\mathcal{C}}}
\newcommand{\CD}        {{\mathcal{D}}}
\newcommand{\CE}        {{\mathcal{E}}}
\newcommand{\CF}        {{\mathcal{F}}}
\newcommand{\CJ}        {{\mathcal{J}}}
\newcommand{\CL}        {{\mathcal{L}}}
\newcommand{\CM}        {{\mathcal{M}}}
\newcommand{\CO}        {{\mathcal{O}}}
\newcommand{\CP}        {{\mathcal{P}}}

\newcommand{\CR}        {{\mathcal{R}}}
\newcommand{\CS}        {{\mathcal{S}}}
\newcommand{\CU}        {{\mathcal{U}}}

\newcommand{\CLt}       {\mathcal{L}^\tm}

\newcommand{\bCA}	{\overline{\mathcal{A}}}
\newcommand{\bCG}	{\overline{\mathcal{G}}}
\newcommand{\bCU}	{\overline{\mathcal{U}}}

\newcommand{\hA}	{\widehat{A}}

\newcommand{\hS}	{\widehat{S}}

\newcommand{\tA}        {\widetilde{A}}
\newcommand{\tB}        {\widetilde{B}}
\newcommand{\tC}        {\widetilde{C}}
\newcommand{\tT}        {\widetilde{T}}
\newcommand{\tU}        {\widetilde{U}}
\newcommand{\tW}	{\widetilde{W}}
\newcommand{\tX}	{\widetilde{X}}
\newcommand{\tY}	{\widetilde{Y}}

\newcommand{\tb}        {\tilde{b}}
\newcommand{\tf}        {\tilde{f}}
\newcommand{\tg}	{\tilde{g}}

\newcommand{\tp}	{\tilde{p}}
\newcommand{\tq}	{\tilde{q}}
\newcommand{\tr}	{\tilde{r}}
\newcommand{\ts}	{\tilde{s}}
\newcommand{\tu}	{\tilde{u}}
\newcommand{\tx}        {\tilde{x}}
\newcommand{\ty}	{\tilde{y}}

\newcommand{\al}        {\alpha}
\newcommand{\bt}        {\beta} 
\newcommand{\gm}        {\gamma}
\newcommand{\dl}        {\delta}
\newcommand{\ep}        {\epsilon}

\newcommand{\zt}        {\zeta}
\newcommand{\tht}       {\theta}
\newcommand{\lm}        {\lambda}
\newcommand{\sg}        {\sigma}
\newcommand{\om}        {\omega}

\newcommand{\Dl}        {\Delta}
\newcommand{\Lm}        {\Lambda}
\newcommand{\Sg}        {\Sigma}
\newcommand{\Tht}       {\Theta}

\newcommand{\N}         {{\mathbb{N}}}
\newcommand{\Z}         {{\mathbb{Z}}}
\newcommand{\Q}         {{\mathbb{Q}}}
\newcommand{\R}         {{\mathbb{R}}}
\newcommand{\C}         {{\mathbb{C}}}

\newcommand{\invlim} {\operatornamewithlimits{\underset{\longleftarrow}{lim}}}
\newcommand{\mxi}       {\mathfrak{m}}
\newcommand{\sse}       {\subseteq}
\newcommand{\Sgi}	{\Sigma^\infty}
\newcommand{\Smash}     {\wedge}
\newcommand{\bigSmash}  {\bigwedge}

\newcommand{\bigWedge}  {\bigvee}
\newcommand{\ot}        {\otimes}
\newcommand{\st}        {\;|\;}
\newcommand{\tm}        {\times}
\newcommand{\btm}       {\boxtimes}
\newcommand{\xla}       {\xleftarrow}
\newcommand{\xra}       {\xrightarrow}
\newcommand{\sm}        {\setminus}

\newcommand{\ov}[1]     {\overline{#1}}
\newcommand{\un}[1]     {\underline{#1}}

\newcommand{\colim}  {\operatornamewithlimits{\underset{\longrightarrow}{lim}}}
\newcommand{\ip}[1]     {\langle #1\rangle}

\renewcommand{\ss}      {\scriptstyle}
\renewcommand{\:}{\colon}

\newtheorem{theorem}{Theorem}[section]

\newtheorem{lemma}[theorem]{Lemma}
\newtheorem{proposition}[theorem]{Proposition}
\newtheorem{corollary}[theorem]{Corollary}
\theoremstyle{definition}
\newtheorem{remark}[theorem]{Remark}
\newtheorem{definition}[theorem]{Definition}
\newtheorem{example}[theorem]{Example}
\newtheorem{construction}[theorem]{Construction}

\begin{document}
\title{Tambara functors}
\author{N.~P.~Strickland}
\address{
School of Mathematics and Statistics\\
University of Sheffield\\
Sheffield S3 7RH\\
UK
}
\email{N.P.Strickland@sheffield.ac.uk}

\date{\today}
\bibliographystyle{abbrv}


\begin{abstract}
 We survey and extend the theory of Tambara functors.  These are algebraic
 structures similar to Mackey functors, but with multiplicative norm maps
 as well as additive transfer maps, and a rule governing their interaction
 that is most easily formulated in an abstract categorical framework.  
 Examples include Burnside rings, representation rings, and homotopy groups
 of equivariant $E_\infty$ ring spectra in stable homotopy theory.  Some
 other examples are related to Witt rings in the sense of Dress and
 Siebeneicher.
\end{abstract}

\maketitle 

\tableofcontents



\section{Introduction}

In this memoir we survey the theory of Tambara functors.  

We first give some motivation for these objects.

One way to think about the definition of Mackey functors (for a finite
group $G$) is as follows.  Take the Lawvere theory $\ov{\CA}$ for
(commutative) semigroups, categorify it, make it $G$-equivariant,
decategorify, and then take the category of models for the resulting
(multisorted) theory.  This category is just the category of
semigroup-valued Mackey functors, so it contains the more usual
category of group-valued Mackey functors.  All this will be explained
in more detail below.

It is well-known that for any $G$-spectrum $R$ (in the sense of
equivariant stable homotopy theory) one has a Mackey functor
$\pi^G_0R$.  If $R$ is a $G$-ring spectrum in the naive homotopical
sense, then $\pi^G_0R$ has a certain multiplicative structure; Mackey
functors with this structure are called Green rings.  If $R$ has an
equivariant $E_\infty$ structure then much more is true.  In
particular, there are multiplicative transfer maps, similar to the
Evans norm map in ordinary group cohomology or tensor induction maps
in representation theory.  These interact in a complicated way with
additive transfers and with restriction maps.  The purpose of Tambara
functors is to encapsulate these relationships.

The category of Tambara functors can be defined along the same lines
as suggested above for Mackey functors.  One simply starts with the
Lawvere theory $\bCU$ for semirings instead of the theory $\bCA$ for
semigroups. 

Another major motivation for studying Tambara functors is the light
that they shed on the theory of Witt vectors.  There is a
straightforward functor from $C_n$-Tambara functors to rings, whose
left adjoint can be expressed in terms of Witt vectors of length $d$
for all divisors $d$ of $n$.  For more general groups $G$, there is a
similar relationship with the generalised Witt vector rings of Dress
and Siebeneicher~\cite{drsi:brp}.  

Tambara functors were originally introduced by Tambara, who called them
TNR-functors~\cite{ta:mt}.  The relationship with Witt vectors was
first noticed by Brun~\cite{br:wvt}, which inspired a beautiful
reinterpretation of Witt theory by Elliott~\cite{el:bri}.  This in
turn inspired the work leading to this memoir.  During the long
gestation of this work, a number of papers and preprints by Nakaoka
have appeared~\cites{na:btt,na:fsm,na:gdc,na:itf,na:tfp,na:tmf},
covering some of the same ideas; we will give detailed references in
the main text.

\section{The theory of semigroups}
\label{sec-semigroups} 

In this section we review the definition of semigroups, in a form that
generalises easily to define Mackey functors and Tambara functors.
For us, a \emph{semigroup} will mean a set with a commutative and
associative binary operation (denoted by $+$) and an identity element
(denoted by $0$).  Mackey functors are usually defined in terms of
abelian groups, but we will find it convenient to give a slightly more
general version without additive inverses.

Any finite set $X$ gives a functor $M\mapsto M^X=\Map(X,M)$ from
semigroups to sets.  Let $\bCA(X,Y)$ be the set of natural functions 
$f\:M^X\to M^Y$.  (We do not assume \emph{a priori} that $f$ is a
homomorphism of semigroups, but it will turn out that this holds
automatically.)  This gives us a category $\bCA$, whose objects are
finite sets.  

By a \emph{span} from $X$ to $Y$ we mean a diagram
$\om=(X\xla{p}A\xra{q}Y)$, where $A$ is another finite set.  By an
\emph{isomorphism} from $\om$ to another span
$\om'=(X\xla{p'}A'\xra{q'}Y)$ we mean a bijection $f\:A\to A'$ with
$p'f=p$ and $q'f=q$.  This gives a groupoid $\CA(X,Y)$ of spans from
$X$ to $Y$.

For any span $\om$ as above, we have a map $f_\om\:M^X\to M^Y$ given
by 
\[ f_\om(m)(y) = \sum_{a\in q^{-1}\{y\}} m(p(a)). \]
It is easy to see that this is natural, and depends only on the
isomorphism class of $\om$.  

Alternatively, we can analyse $\bCA(X,Y)$ using the Yoneda lemma.
The semigroup $\N^X$ represents the functor $M\mapsto M^X$, and
it follows that
\[ \bCA(X,Y) = \Semigroups(\N^Y,\N^X) = (\N^X)^Y = \N^{X\tm Y}.
\]
Explicitly, for any map $a\:X\tm Y\to\N$ we have a natural map
$g_a\:M^X\to M^Y$ given by
\[ g_a(m)(y) = \sum_{x\in X} a(x,y) m(x). \]
For a span $\om$ as above, we find that $f_\om=g_{c(\om)}$, where 
\[ c(\om)(x,y) = |\{a\in A\st p(a)=x \text{ and } q(a)=y\}|. \]
From this it is straightforward to check that the construction 
$\om\to f_\om$ gives a bijection from the set $\pi_0\CA(X,Y)$ (of
isomorphism classes of spans) to $\bCA(X,Y)$.  We also see (as
promised) that all natural maps $M^X\to M^Y$ are automatically
homomorphisms.  This means that all the morphism sets $\bCA(X,Y)$ are
semigroups in a natural way, and that composition is bilinear.

Now suppose we have two spans
\begin{align*}
 \om_0 &= (X_0 \xla{p_0} X_{01} \xra{q_0} X_1) \\
 \om_1 &= (X_1 \xla{p_1} X_{12} \xra{q_1} X_2).
\end{align*}
The above analysis implies that the composite 
\[ M^{X_0} \xra{f_{\om_0}} M^{X_1}
           \xra{f_{\om_1}} M^{X_2}
\]
must arise from a span from $X_0$ to $X_2$, which is unique up to
isomorphism.  To find the required span, we let $X_{02}$ denote the
pullback of the maps $X_{01}\xra{q_0}X_1\xla{p_1}X_{12}$, so we have a
diagram 
\[ \xymatrix{
 & & X_{02} \ar[dl] \ar[dr] \\
 & X_{01} \ar[dl] \ar[dr] & & 
   X_{12} \ar[dl] \ar[dr] \\
 X_0 & & X_1 & & X_2
} \]
in which the middle square is cartesian.  We define $\om_1\circ\om_0$
to be the resulting span
\[ \om_1\circ\om_0 = (X_0 \xla{} X_{02} \xra{} X_2). \] It is not hard
to check that $f_{\om_1\circ\om_0}=f_{\om_{12}}f_{\om_{01}}$ as
required.

One way to encapsulate the properties of this construction is to say
that we have a bicategory (as defined at~\cite{nlab:bc} for example),
where the $0$-cells are finite sets, the $1$-cells are spans, and the
$2$-cells are isomorphisms of spans.  To produce the required
associativity isomorphisms, we need to contemplate diagrams of the
form 
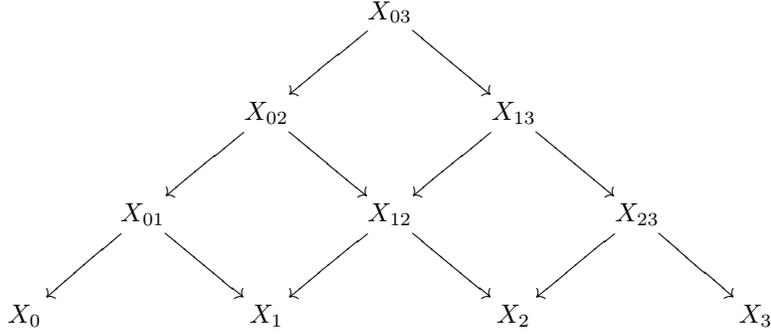
\[ \xymatrix{
 & & & X_{03} \ar[dl] \ar[dr] \\
 & & X_{02} \ar[dl] \ar[dr] & &
     X_{13} \ar[dl] \ar[dr] \\
 & X_{01} \ar[dl] \ar[dr] & &
   X_{12} \ar[dl] \ar[dr] & &
   X_{23} \ar[dl] \ar[dr] \\
 X_0 & & X_1 & & X_2 & & X_3
} \]
in which all the squares are cartesian.  To prove the pentagonal
coherence identities for these isomorphisms, we need to consider
diagrams of the form
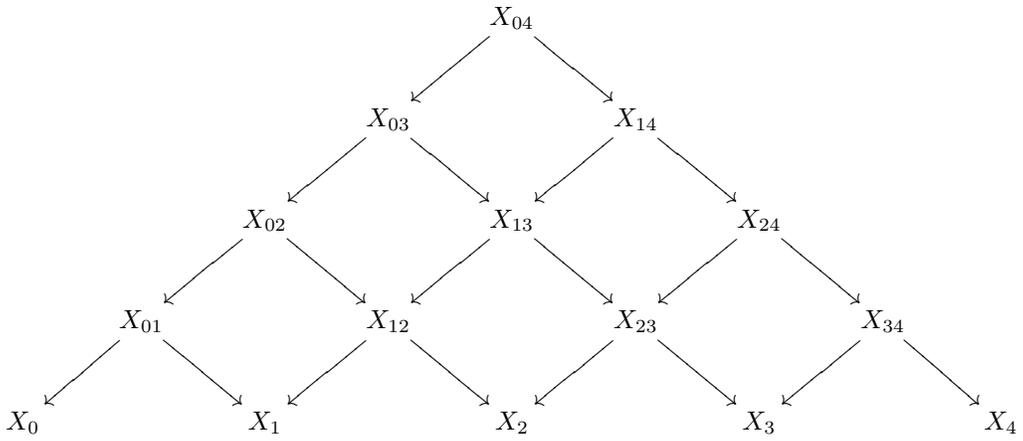
\[ \xymatrix{
 & & & & X_{04} \ar[dl] \ar[dr] \\
 & & & X_{03} \ar[dl] \ar[dr] & &
       X_{14} \ar[dl] \ar[dr] \\
 & & X_{02} \ar[dl] \ar[dr] & &
     X_{13} \ar[dl] \ar[dr] & &
     X_{24} \ar[dl] \ar[dr] \\
 & X_{01} \ar[dl] \ar[dr] & &
   X_{12} \ar[dl] \ar[dr] & &
   X_{23} \ar[dl] \ar[dr] & &
   X_{34} \ar[dl] \ar[dr] \\
 X_0 & & X_1 & & X_2 & & X_3 & & X_4
} \]
in which all the squares are again cartesian.

Another approach is to let $\CA_k$ denote the set of all diagrams of
this general type, involving sets $X_{ij}$ for $0\leq i\leq j\leq k$.
To be set-theoretically respectable we should fix an infinite set $U$,
and insist that all the sets $X_{ij}$ should be subsets of $U$.  One
can then introduce face and degeneracy operators to make
$\CA_\bullet$ into a simplicial set, and then check that this is an
$\infty$-category in the sense of Lurie \cite{lu:htt} (or in the
language of Joyal~\cite{jo:qck}, a quasicategory).  This idea has
been developed extensively in work of Cranch \cite{cr:ati}.  

It will be convenient to introduce slightly different notation as
follows.  Given a map $f\:X\to Y$, we put
\begin{align*}
       T_f &= (X\xla{1}X\xra{f}Y) \in \CA(X,Y) \\
 f^* = R_f &= (Y\xla{f}X\xra{1}X) \in \CA(Y,X).
\end{align*}
We also write $T_f$ for the corresponding isomorphism class in
$\bCA(X,Y)$, or the resulting operation $M^X\to M^Y$, and similarly
for $R_f$.  
\begin{itemize}
 \item[(a)] For all $X\xra{f}Y\xra{g}Z$ we have $T_{gf}=T_gT_f$.  More
  precisely, we see by considering the diagram 
  
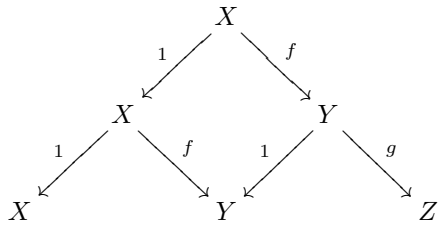
\[ \xymatrix{
   & & X \ar[dl]_1 \ar[dr]^f \\
   & X \ar[dl]_1 \ar[dr]^f & & Y \ar[dl]_1 \ar[dr]^g \\
   X & & Y & & Z
  } \] 
  that the spans $T_{gf}$ and $T_gT_f$ are canonically isomorphic in
  $\CA(X,Z)$, so $T_{gf}=T_gT_f$ in $\bCA(X,Z)$ or $\Map(M^X,M^Z)$.
 \item[(b)] In the same sense, we have $R_{gf}=R_fR_g$, or
  equivalently $(gf)^*=f^*g^*$.
 \item[(c)] Any span $(X\xla{p}A\xra{q}Y)$ can be expressed in
  $\bCA(X,Y)$ as $T_qR_p$, as we see by considering the diagram 
  
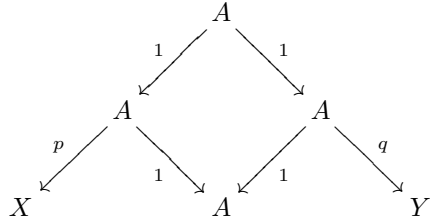
\[ \xymatrix{
   & & A \ar[dl]_1 \ar[dr]^1 \\
   & A \ar[dl]_p \ar[dr]_1 & & A \ar[dl]^1 \ar[dr]^q \\
   X & & A & & Y
  } \] 
 \item[(d)] For any cartesian square
  
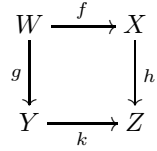
\[ \xymatrix{
   W \ar[r]^f \ar[d]_g & X \ar[d]^h \\
   Y \ar[r]_k & Z
  } \]
  we have $T_gR_f=R_kT_h\in\bCA(X,Y)$.
 \item[(e)] For any bijection $f\:X\to Y$ we have $T_f=R_f^{-1}$.
  This can be seen as a special case of~(d), for the diagram
  
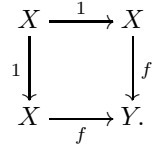
\[ \xymatrix{
   X \ar[r]^1 \ar[d]_1 & X \ar[d]^f \\ X \ar[r]_f & Y.
  } \]
\end{itemize}

Next, consider the canonical inclusions
\[ X \xra{i} X\amalg Y \xla{j} Y. \]
These give a diagram as follows in $\bCA$:
\[ X \xla{R_i} X\amalg Y \xra{R_j} Y. \]
We claim that this is a product diagram, or in other words that for
every diagram
\[ X \xla{\al} U \xra{\bt} Y \]
there is a unique morphism $\phi\in\bCA(U,X\amalg Y)$ with
$R_i\phi=\al$ and $R_j\phi=\bt$.  Indeed, we can represent $\al$ and
$\bt$ by spans as follows:
\begin{align*}
 \al &= [U\xla{p}A\xra{q}X] \\
 \bt &= [U\xla{r}B\xra{s}Y].
\end{align*}
Let $(p,r)\:A\amalg B\to U$ be the map given by $p$ on $A$ and by $r$
on $B$.  We then find that the morphism
\[ \phi = [U \xla{(p,r)} A\amalg B \xra{q\amalg s} X\amalg Y] \]
has the required property.  We also note that the maps 
\[ M^X \xla{R_i} M^{X\amalg Y} \xra{R_j} M^Y \]
are just the obvious projections, so they again give a product diagram
(in either the category of sets, or the category of semigroups).

A similar argument shows that the maps $X\xra{T_i}X\amalg Y\xla{T_j}Y$
give a coproduct diagram.  Similarly, the empty set is both initial
and terminal in $\bCA$.  Indeed, if we let $z\:\emptyset\to X$ be the
inclusion, then $\bCA(\emptyset,X)=\{T_z\}$ and
$\bCA(X,\emptyset)=\{R_z\}$.  Using this, we see that $\bCA$ is a
semiadditive category, as discussed in
Appendix~\ref{apx-semiadditive}.  (The arguments are essentially the
same as in the more familiar additive case, discussed
in~\cite{ma:cwm}*{Chapter VIII} for example).  In terms of spans, the
addition is 
\[ [X\xla{p}A\xra{q}Y] + [X\xla{p'}A'\xra{q'}Y] = 
    [X \xla{(p,p')} A\amalg A' \xra{(q,q')} X'].
\]
Alternatively, we can regard each object $X\in\bCA$ as a semigroup
object: the unit is $T_z\:\emptyset\to X$, and the addition is
$T_s\:X\amalg X\to X$, where $s$ is the function $X\amalg X\to X$
given by the identity on both copies of $X$.
 
We can recover the category of semigroups from $\bCA$, as follows:
\begin{proposition}\label{prop-semigroup-theory}
 The category of semigroups is equivalent to the category of
 product-preserving functors from $\bCA$ to the category of sets.
\end{proposition}

This should really be seen as a basic part of the setup of algebraic
theories in the sense of Lawvere~\citelist{\cite{bo:hcaii}*{Chapter 3}
 \cite{nlab:lt}}.  In Lawvere's terminology, the category $\bCA$ is
the theory of semigroups.

\begin{proof}
 If $A$ is a semigroup then for each finite set $X$ we have a set
 $A^X$, and for each element $f\in\bCA(X,Y)$ we have (tautologically)
 a map $f_A\:A^X\to A^Y$.  Moreover, these compose in the obvious way,
 so we have a functor $E_A\:\bCA\to\Sets$ sending $X$ to $A^X$.  As
 $X\amalg Y$ is the product of $X$ and $Y$ in $\bCA$ and
 $A^{X\amalg Y}=A^X\tm A^Y$ we see that $E_A$ preserves products.

 For the opposite construction, it will be convenient to identify
 numbers with sets by the usual rule $0=\emptyset$, $1=\{0\}$,
 $2=\{0,1\}$ and so on.  For any semigroup $U$ we have natural maps
 $U^0\xra{\eta_U}U^1\xla{\sg_U}U^2$ given by $\eta()=0$ and
 $\sg(u_0,u_1)=u_0+u_1$.  These can be regarded as morphisms
 $0\xra{\eta}1\xla{\sg}2$ in $\bCA$.

 Now suppose we start with a product-preserving functor
 $F\:\bCA\to\Sets$.  Put $A=F(1)$.  Any finite set $X$ can be thought
 of as $\coprod_{x\in X}1$, and disjoint unions are products in
 $\bCA$, so the natural map $F(X)\to\prod_{x\in X}F(1)=A^X$ is
 bijective.  Thus, the maps $F(0)\xra{F(\eta)}F(1)\xla{F(\sg)}F(2)$
 give maps $1\xra{0}A\xla{+}A^2$.  We claim that these make $A$ into a
 semigroup.  We will prove the associativity axiom as an example.  All
 semigroups are associative, so the two natural operations
 $(u,v,w)\mapsto\sg(\sg(u,v),w)$ and $(u,v,w)\mapsto\sg(u,\sg(v,w))$
 are the same, so the following diagram commutes in $\bCA$:
 \[ \xymatrix{
  3 \ar[r]^{\sg\tm 1} \ar[d]_{1\tm\sg} & 2 \ar[d]^\sg \\
  2 \ar[r]_\sg & 1.
 } \]
 We can apply $F$ to this and use the natural isomorphisms
 $F(k)\to F(1)^k=A^k$ to obtain a commutative square
 \[ \xymatrix{
  A^3 \ar[r]^{(+)\tm 1} \ar[d]_{1\tm(+)} & A^2 \ar[d]^+ \\
  A^2 \ar[r]_+ & A.
 } \]
 This proves that addition is associative in $A$, and the other
 semigroup axioms can be handled in the same way.  

 We leave it to the reader to check that these two constructions give
 an inverse pair of functors between the relevant categories.
\end{proof}

We next discuss quotients of semigroups.  This is a little more subtle
that the corresponding construction with groups, but the relevant
concepts are a standard part of universal algebra.
\begin{definition}\label{defn-semigroup-congruence}
 Let $M$ be a semigroup.  A \emph{congruence} on $M$ is a subset
 $E\sse M\tm M$ that is both a subsemigroup of $M\tm M$ and an
 equivalence relation on $M$.
\end{definition}
\begin{proposition}\label{prop-semigroup-cong}\ \\
 \begin{itemize}
  \item[(a)] For any morphism $\phi\:M\to M'$ of semigroups there
   is a congruence $\eqker(\phi)$ on $M$ given by 
   \[ \eqker(\phi) =
       \{(a,b)\in M\tm M \st \phi(a)=\phi(b)\in M'\}.
   \]
  \item[(b)] For any congruence $E$ on $M$ there is a unique way to
   make the quotient set $M/E$ into a semigroup
   such that the projection $\pi\:M\to M/E$ is a semigroup morphism.
  \item[(c)] If $E\leq\eqker(\phi)$ then there is a unique semigroup 
   morphism $\ov{\phi}\:M/E\to M'$ with $\phi=\ov{\phi}\pi$, but if
   $E\not\leq\eqker(\phi)$ then there is no such morphism.
  \item[(d)] The diagonal $\Dl=\{(m,m)\st m\in M\}$ is a congruence
   with $M/\Dl=M$.  The whole set $M\tm M$ is a congruence with
   $M/(M\tm M)=0$.
  \item[(e)] The intersection of any family of congruences is a
   congruence.  (In particular, the intersection of the empty family
   is $M\tm M$.)
  \item[(f)] For any subset $F\sse M\tm M$, there is a smallest
   congruence containing $F$, namely the intersection of the family of
   all congruences that contain $F$.  
 \end{itemize}
\end{proposition}
\begin{proof}
 Left to the reader.
\end{proof}

\begin{remark}\label{rem-congruences-subgroups}
 If $M$ has additive inverses and so is actually a group, then it is
 easy to check that every congruence has the form 
 \[ E_N = \{(a,b)\in M\tm M \st a-b\in N\} \]
 for some subgroup $N\leq M$, and thus that congruences biject with
 subgroups. 
\end{remark}
\begin{remark}\label{rem-smallest-congruence}
 Let $F$ be a subset of $M\tm M$.  Put 
 \begin{align*}
  E(0) &= \{(a,b)\in M\tm M\st a=b \text{ or }
               (a,b)\in F \text{ or } (b,a) \in F\} \\
  E(2k+1) &=
   \{(a+a',b+b')\st (a,b)\in E(2k) \text{ and } (a',b')\in E(2k) \}\\
  E(2k+2) &= \{(a,b)\st \text{ there exists $u\in M$ with 
     $(a,u)\in E(2k+1)$ and $(u,b)\in E(2k+1)$}\} \\
  E &= \bigcup_nE(n).
 \end{align*}
 Then one can check that $E$ is the smallest congruence containing
 $F$, as in part~(f) of the above proposition.  We suspect that there
 is no substantially simpler construction.  In particular, we suspect
 that it is necessary to alternate infinitely many times between
 operations designed to make $E$ a subsemigroup and operations
 designed to make $E$ transitive.  However, we have not constructed
 explicit examples to support this.
\end{remark}

\begin{definition}\label{defn-quotient-semigroup}
 For any subsemigroup $N\leq M$ we let $E_N$ denote the smallest
 congruence containing $N\tm 0$, and we write $M/N$ for $M/E_N$.
\end{definition}

\begin{definition}\label{defn-coinvariant-semigroup}
 Let $M$ be any semigroup with an action of a finite group $G$.  We
 put $F=\{(m,gm)\st m\in M\text{ and } g\in G\}$, and let $E$ be the
 smallest congruence containing $F$.  We then put $M_G=M/E$, and call
 this the \emph{coinvariant semigroup} for the action.
\end{definition}

\begin{remark}\label{rem-semigroup-factor}
 It is clear by construction that a homomorphism $\phi\:M\to M'$
 factors through $M/N$ iff $\eqker(\phi)\supseteq N\tm 0$ iff
 $\phi(N)=0$ (and the factorisation is unique if it exists).
 Similarly, $\phi$ factors through $M_G$ iff $\phi(gm)=\phi(m)$ for
 all $g\in G$ and $m\in M$.
\end{remark}

\begin{remark}\label{rem-funny-quotient}
 It is possible to have $M/N=0$ even when $N\neq M$.  For example, for
 any $a\in\N$ the set 
 \[ U_a = \{n\in\N\st n=0 \text{ or } n\geq a\} \]
 is a subsemigroup of $\N$.  In $\N/U_a$ we have $0\sim a$ so
 $k\sim k+a$ for all $k\in\N$, but also $k+a\in U_a$ so $k\sim 0$.
 This shows that $\N/U_a=0$.  However, the inclusion $U_a\to\N$ is not
 an epimorphism in the category of semigroups.  To see this, let $E_0$
 denote the set of pairs $((i,j),(k,l))\in\N^2\tm\N^2$ such that
 $\max(i,j)\geq a$ and $\max(k,l)\geq a$ and $i+j=k+l$.  One can check
 that $E=\Dl\cup E_0$ is a congruence on $\N^2$; let
 $\pi\:\N^2\to Q=\N^2/E$ be the resulting quotient morphism.  We have
 morphisms $\al,\bt\:N\to Q$ given by $\al(n)=\pi(n,0)$ and
 $\bt(n)=\pi(0,n)$.  These agree on $U_a$, but not on all of $\N$, as
 required. 
\end{remark}

\section{Mackey functors}
\label{sec-mackey}

Let $G$ be a finite group.  Everything in the previous section can be
done $G$-equivariantly.  Explicitly, we introduce a bicategory
$\CA_G$, where the $0$-cells are finite $G$-sets, the $1$-cells are
diagrams $(X\xla{p}A\xra{q}Y)$ of finite $G$-sets, and the $2$-cells
are the evident equivariant isomorphisms.  We can then decategorify
this: we form a category $\bCA_G$ whose objects are finite $G$-sets,
with morphisms $\bCA_G(X,Y)=\pi_0\CA_G(X,Y)$.  For us, a \emph{Mackey
 functor} will mean a product-preserving functor from $\bCA_G$ to
$\Sets$ (these are semi-Mackey functors in the terminology of
Nakaoka).  We write $\Mackey_G$ for the category of Mackey functors.

As before, the inclusions $X\xra{i}X\amalg Y\xla{j}Y$ give a product
diagram $X\xla{R_i}X\amalg Y\xra{R_j}Y$ in $\bCA_G$, and also a
coproduct diagram $X\xra{T_i}X\amalg Y\xla{T_j}Y$, and the empty set
is both initial and terminal.  This makes $\bCA_G$ a semiadditive
category.  If $M$ is a Mackey functor we thus have 
$M(X\amalg Y)\simeq M(X)\tm M(Y)$, and we can apply $M$ to the
semigroup structure maps $\emptyset\xra{T_z}X\xla{T_s}X\amalg X$ to
make $M(X)$ into a semigroup in a natural way.

\begin{example}\label{eg-MapGXA}
 Let $A$ be a semigroup with an action of $G$.  We would like to
 define a Mackey functor $cA$ by $cA(X)=\Map_G(X,A)$ (so
 $cA(G/H)=A^H$).  Given any equivariant span
 $\om=(X\xla{p}A\xra{q}Y)$, we define $f_\om\:\Map(X,A)\to\Map(Y,A)$
 by the usual rule
 \[ f_\om(u)(y) = \sum_{q(a)=y} u(p(a)). \]
 This commutes with the $G$-action and so restricts to give a map
 $f_\om\:cA(X)\to cA(Y)$.  This depends only on the equivariant
 isomorphism class of $\om$ and is compatible with composition of
 spans, so the construction $[\om]\mapsto f_\om$ makes $cA$ into a
 Mackey functor as desired.

 If $G$ acts trivially on $A$ we just have $cA(X)=\Map(X/G,A)$.  In
 particular, we can identify $cA(G/H)$ with $A$, and for any
 $f\:G/K\to G/H$ the resulting map $R_f\:cA(G/H)\to cA(G/K)$ is just the
 identity $A\to A$.  For this reason, Mackey functors of this type are
 called \emph{constant} Mackey functors.  However, this name is
 somewhat misleading because the map $T_f\:cA(G/K)\to cA(G/H)$ is
 $|H|/|K|$ times the identity on $A$, rather than the identity
 itself.  (Note here that the existence of $f$ means that $K$ is
 conjugate to a subgroup of $H$, so $|H|/|K|$ is an integer.) 
\end{example}
\begin{example}\label{eg-MapXGB}
 Now let $B$ be a semigroup without $G$-action, and define
 $N(X)=\Map(X^G,B)$ (where $X^G$ is the subset of $G$-fixed points in
 $X$).  The construction $X\mapsto X^G$ preserves the pullbacks used
 to define composition of spans, and it also preserves disjoint
 unions, so it gives a product-preserving functor $\bCA_G\to\bCA$.
 Given this, there is an evident way to regard $N$ as a Mackey
 functor.  Explicitly, for any equivariant span
 $\om=(X\xla{p}A\xra{q}Y)$, we define $f_\om\:N(X)\to N(Y)$ by 
 \[ f_\om(u)(y) = \sum_{a\in A^G,\;q(a)=y} u(p(a)). \]
\end{example}
\begin{example}\label{eg-HU}
 For any finite $G$-set $U$ we have a representable Mackey functor
 $H_U\:\bCA_G\to\Sets$ given by $H_U(X)=\bCA_G(U,X)$.  We also write
 $A(X)$ for $H_1(X)$, which is the set of isomorphism classes of spans
 $(1\xla{p}T\xra{q}X)$.  It is clear here that $p$ gives no
 information, so $A(X)$ is the set of isomorphism classes of finite
 $G$-sets equipped with a map to $X$, which is known as the
 \emph{Burnside semigroup} of $X$.
\end{example}
\begin{example}\label{eg-pi-zero}
 Let $E$ be a $G$-equivariant spectrum in the sense of stable homotopy
 theory.  For any finite $G$-set $X$ we have another $G$-spectrum
 $\Sgi_G(X_+)$, and we write $\pi^G_0(E)(X)$ for the set of homotopy
 classes of maps $\Sgi_G(X_+)\to E$.  It is well known that
 $\pi^G_0(E)$ is a Mackey functor in a natural way.  We will give a
 proof in Section~\ref{sec-spectrum}, in a form that is convenient for
 generalisation to the Tambara framework.
\end{example}

\begin{proposition}\label{prop-mackey-orbits}
 Let $\ov{\CA\CO}_G$ denote the full subcategory of $\bCA_G$ whose
 objects are the orbits $G/H$ for all subgroups $H\leq G$.  Then
 $\Mackey_G$ is equivalent to the category of preadditive functors
 from $\ov{\CA\CO}_G$ to the category of semigroups.
\end{proposition}
\begin{proof}
 As every finite $G$-set is a disjoint union of orbits, we see that
 every object in $\bCA_G$ can be expressed as a product (or
 equivalently, a coproduct) of objects in $\ov{\CA\CO}_G$.  The claim
 therefore follows from Proposition~\ref{prop-res-equiv} (in
 Appendix~\ref{apx-semiadditive}). 
\end{proof}
\begin{remark}
 Using this idea, we can also regard $\Mackey_G$ as the category of
 algebras for a coloured or multisorted Lawvere theory, with one
 colour for each subgroup of $G$.  (Such theories are discussed
 in~\cite{bovo:hia}*{Chapter II}, for example).
\end{remark}

From either of the above points of view, it is important to understand
$\bCA_G(G/H,G/K)$.  There are three basic kinds of elements:
\begin{itemize}
 \item[(a)] Suppose that $K\leq H\leq G$, so there is an evident
  projection $p\:G/K\to G/H$.  We put 
  \begin{align*}
   T_K^H &= T_p \in \bCA_G(G/K,G/H) \\
   R_K^H &= R_p \in \bCA_G(G/H,G/K).
  \end{align*}
 \item[(b)] For any $g\in G$ and $H\leq G$ we define
  $c^H_g\:G/H\to G/gHg^{-1}$ by 
  \[ c^H_g(xH)=xHg^{-1}=xg^{-1}(gHg^{-1}). \]
  We also define
  $C^H_g=T_{c^H_g}=R_{c^H_g}^{-1}\in\bCA_G(G/H,G/gHg^{-1})$.  
\end{itemize}

Any element of $\bCA_G(G/H,G/K)$ can be represented by a span
$\om=(G/H\xla{p}A\xra{q}G/K)$.  We can decompose $A$ as a disjoint
union of orbits, each orbit gives a span from $G/H$ to $G/K$, and
$\om$ is the sum of these terms.  We can therefore focus on the case
where $A$ itself is an orbit.  As $G$ acts transitively on $G/H$, the
map $p$ is necessarily surjective, so we can choose $a\in A$ with
$p(a)=H$.  This identifies $A$ with $G/L$ for some subgroup $L\leq H$,
and $p$ with the canonical projection $G/L\to G/H$.  As $q$ is also
surjective we will have $q(xL)=K$ for some $x\in G$.  As the coset
$xL\in G/L$ is fixed by the subgroup $xLx^{-1}$ and $q$ is
equivariant, we must have $xLx^{-1}\leq K$.  We therefore have maps
\[ G/H \xra{R_L^H} G/L \xra{C^L_x} G/xLx^{-1}
    \xra{T^K_{xLx^{-1}}} G/K 
\]
in $\bCA_G$, and it is straightforward to check that the composite is
$\om$.  This allows us to express an arbitrary element of
$\bCA_G(G/H,G/K)$ as a sum of composites of our basic operators.

These satisfy relations as follows:
\begin{itemize}
 \item[(a)] For $L\leq K\leq H\leq G$ it is clear that
  $T_K^HT_L^K=T_L^H\:G/L\to G/H$ and
  $R^K_LR^H_K=R^H_L\:G/H\to G/L$.
 \item[(b)] It is also clear that
  $C^{gHg^{-1}}_fC^H_g=C^H_{fg}\:G/H\to G/fgH(fg)^{-1}$.
 \item[(c)] By regarding $C^H_g$ as $T_{c^H_g}$ we see that the left
  hand square below commutes.  By regarding it as $R_{c^H_g}^{-1}$, we
  see that the right hand square commutes.
  \[ \xymatrix{
   G/K \ar[r]^{C^K_g} \ar[d]_{T^H_K} & 
   G/gKg^{-1} \ar[d]^{T^{gHg^{-1}}_{gKg^{-1}}} \\
   G/H \ar[r]_{C^H_g} &
   G/gHg^{-1}
   } \hspace{5em}
   \xymatrix{
   G/K \ar[r]^{C^K_g} & 
   G/gKg^{-1}  \\
   G/H \ar[r]_{C^H_g} \ar[u]^{R^H_K} &
   G/gHg^{-1} \ar[u]_{R^{gHg^{-1}}_{gKg^{-1}}}.
   } 
  \]
 \item[(d)] Now consider a composite $G/H\xra{T^H_L}G/L\xra{R^K_L}G/K$
  (where $H,K\leq L$).  Let $A$ denote the pullback
  $(G/H)\tm_{G/L}(G/K)$, so $R^K_LT^H_L$ is represented by an evident
  span $(G/H\xla{p}A\xra{q}G/K)$, so it can be written as a sum of
  terms indexed by the $G$-orbits in $A$.  We can let $K\tm H$ act on
  $L$ by $(k,h).l=klh^{-1}$, and thus decompose $L$ as
  $\coprod_{t\in T}KtH$ for some subset $T\sse L$.  For each $t\in T$
  we have a point $\phi(t)=(H,t^{-1}K)\in A$, whose isotropy group is
  $M_t=H\cap tKt^{-1}$.  We find that each orbit in $A$ contains
  precisely one of the points $\phi(t)$, and that $p(\phi(t))=H$ and
  $q(t\,\phi(t))=K$.  Note that the conjugate $M'_t=tM_tt^{-1}$ is
  $tHt^{-1}\cap K$, and in particular is contained in $K$.  From this
  we deduce the double coset formula: $R^K_LT^H_L$ is the sum over
  $t\in T$ of the composites
  \[ G/H \xra{R^H_{M_t}} G/M_t \xra{C^{M_t}_t} G/M'_t
      \xra{T^K_{M'_T}} G/K.
  \]
\end{itemize}

\begin{construction}\label{cons-CA-times}
 Given spans
 \begin{align*}
  \al &= (X \xla{p} A \xra{q} Y) \in \CA_G(X,Y) \\
  \al' &= (X' \xla{p'} A' \xra{q'} Y') \in \CA_G(X',Y')
 \end{align*}
 we write $\al\tm\al'$ for the product span
 \[ \al\tm\al' = (
     X\tm X' \xla{p\tm p'} A\tm A' \xra{q\tm q'} Y\tm Y'
    ) \in \CA_G(X\tm X',\;Y\tm Y').
 \]
 As composition of spans is defined using pullbacks, and pullbacks
 commute with products, we see that there are natural isomorphisms
 \[ (\bt\tm\bt')\circ(\al\tm\al') \simeq
      (\bt\circ\al)\tm(\bt'\circ\al').
 \]
 We can thus define a functor $\tm\:\bCA_G\tm\bCA_G\to\bCA_G$ by
 $(X,X')\mapsto X\tm X'$ on objects, and
 $([\al],[\al'])\mapsto[\al\tm\al']$ on morphisms.  This gives a
 symmetric monoidal structure on the category $\bCA_G$.
\end{construction}

\begin{remark}\label{rem-CA-times}
 It is clear that for any maps $f\:X\to Y$ and $f'\:X'\to Y'$ of
 finite $G$-sets, we have $R_f\tm R_{f'}=R_{f\tm f'}$ and
 $T_f\tm T_{f'}=T_{f\tm f'}$.  Using the natural isomorphisms
 \[ (A\amalg B)\tm(C\amalg D)\simeq
     (A\tm C)\amalg(A\tm D)\amalg(B\tm C)\amalg(B\tm D)
 \]
 we also see that for $u,v\in\bCA_G(X,Y)$ and $u',v'\in\bCA_G(X',Y')$
 we have
 \[ (u+v)\tm(u'+v') = (u\tm u')+(u\tm v')+(v\tm u')+(v\tm v'), \]
 so the product functor is bilinear.
\end{remark}

\begin{definition}\label{defn-box-mackey}
 Given Mackey functors $M$ and $N$, we write $M\btm N$ for the Day
 tensor product of $M$ and $N$.  This was originally defined
 in~\cite{da:ccf}, and the construction is reviewed in
 Appendix~\ref{apx-semiadditive}.  It is characterised by the fact
 that morphisms $M\btm N\to P$ biject with natural maps
 $M(X)\tm N(Y)\to P(X\tm Y)$ for $(X,Y)\in\bCA_G^2$.  This gives a
 biadditive symmetric monoidal structure on $\Mackey_G$.

 There is a tautological map $M(X)\tm N(Y)\to(M\btm N)(X\tm Y)$.  We
 write $m\btm n\in (M\btm N)(X\tm Y)$ for the image of a pair
 $(m,n)\in M(X)\tm N(Y)$ under this map.
\end{definition}

\begin{remark}\label{rem-box-kan}
 The functor $M\btm N\:\bCA_G\to\Sets$ is defined as a left Kan
 extension, as explained in Appendix~\ref{apx-semiadditive}.  There is
 a subtlety that becomes important in the Tambara context.  The box
 product would usually be defined as a Kan extension of a certain
 functor taking values in semigroups, but it is shown in the appendix
 that we can form the Kan extension in the category of sets instead
 and it automatically has the required semigroup structure (which is
 unusual for colimit constructions).

 By unwrapping the usual construction of the Kan extension as a
 colimit over a comma category, we obtain the following description.
 Given a span $\al\in\CA(X,Y)$, we will write $f_\al$ for the
 resulting map $M(X)\to M(Y)$.
 \begin{itemize}
  \item[(a)] Every element of $(M\btm N)(X)$ has the form
   $f_\al(m\btm n)$ for some span $\al\in\CA_G(U\tm V,X)$
   and some elements $m\in M(U)$ and $n\in N(V)$.
  \item[(b)] Suppose we have spans $\lm\in\CA_G(U',U)$ and
   $\mu\in\CA_G(V',V)$.  Then for any $m'\in M(U')$
   and $n'\in N(V')$ we have 
   \[ f_{\al\circ(\lm\tm\mu)}(m'\btm n') =
       f_\al(f_\lm(m')\btm f_\mu(n')).
   \]
  \item[(c)] All identities between elements of type~(a) can be
   deduced by chaining together identities of type~(b).
 \end{itemize}
\end{remark}

The above description is somewhat cumbersome; we now outline a
slightly different description that is easier to use.

\begin{definition}
 Given $m\in M(X)$ and $n\in N(X)$ we put $m\ot n=R_\dl(m\btm n)$,
 where $\dl\:X\to X\tm X$ is the diagonal map.
\end{definition}

The operation $\btm$ on element can be expressed in terms of the
operation $\ot$ as follows: 
\begin{lemma}
 Let $X\xla{p}X\tm Y\xra{q}Y$ be the projections.  Then for
 $m\in M(X)$ and $n\in N(Y)$ we have $m\btm n=(R_p(m))\ot(R_q(n))$.
\end{lemma}
\begin{proof}
 By definition we have 
 \[ (R_p(m))\ot(R_q(n)) = R_\dl(R_p(m)\btm R_q(n)). \]
 By the naturality properties of the $\btm$ pairing we also have 
 \[ R_p(m)\btm R_q(n) = R_{p\tm q}(m\btm n). \]
 Note also that $(p\tm q)\circ\dl\:X\tm Y\to X\tm Y$ is just the
 identity, so $R_\dl R_{p\tm q}$ is the identity on
 $(M\btm N)(X\tm Y)$.  The claim follows by combining these facts.
\end{proof}

We also have a Frobenius reciprocity formula as follows:
\begin{lemma}\label{lem-frobenius}
 Suppose we have a map $f\:W\to X$ of finite $G$-sets.
 \begin{itemize}
  \item[(a)] For all $m'\in M(W)$ and $n\in N(X)$ we have
   $T_f(m'\ot R_f(n))=T_f(m')\ot n\in (M\btm N)(X)$.
  \item[(b)] For all $m\in M(X)$ and $n'\in N(W)$ we have
   $T_f(R_f(m)\ot n')=m\ot T_f(n')\in (M\btm N)(X)$.
 \end{itemize}
\end{lemma}
\begin{proof}
 For part~(a), it is straightforward to check that the square
 \[ \xymatrix{
  W \ar[r]^\dl \ar[d]_f & 
  W\tm W \ar[r]^{1\tm f} &
  W\tm X \ar[d]^{f\tm 1} \\
  X \ar[rr]_\dl & & X\tm X
 } \]
 is cartesian, so we have 
 \[ T_f R_\dl R_{1\tm f} = R_\dl T_{f\tm 1} \: 
     (M\btm N)(W\tm X) \to (M\btm N)(X).
 \]
 We now apply this to the element $m'\btm n\in(M\btm N)(W\tm X)$.  By
 the naturality properties of the $\btm$ pairing we have
 $R_{1\tm f}(m'\btm n)=m'\btm R_f(n)$ and
 $T_{f\tm 1}(m'\btm n)=T_f(m')\btm n$.  We therefore get
 $T_fR_\dl(m'\btm R_f(n))=R_\dl(T_f(m')\btm n)$, or in other words
 $T_f(m'\ot R_f(n))=T_f(m')\ot n$.  The proof for~(b) is similar.
\end{proof}

This can be sharpened as follows.

\begin{proposition}\label{prop-box-equiv}
 Fix Mackey functors $M$ and $N$ and a finite $G$-set $X$.  Let $\CE$
 be the set of quadruples $(U,p,m,n)$ where $p\:U\to X$ is an
 equivariant map, and $(m,n)\in M(U)\tm N(U)$.  Let $E$ be the
 smallest equivalence relation on $\CE$ such that 
 \begin{itemize}
  \item[(a)] For all $U'\xra{r}U\xra{q}X$ and $(m',n)\in M(U')\tm N(U)$
   we have 
   \[ (U',qr,m',R_r(n)) E (V,q,T_r(m'),n) \]
  \item[(b)] For all $U'\xra{r}U\xra{q}X$ and $(m,n')\in M(U)\tm N(U')$
   we have 
   \[ (U,qr,R_r(m),n') E (V,q,m,T_r(n')). \]
 \end{itemize}
 Define $\ep\:\CE\to(M\btm N)(X)$ by $\ep(U,p,m,n)=T_p(m\ot n)$.  Then
 $\ep$ induces a bijection $\CE/E\to(M\btm N)(X)$.
\end{proposition}
\begin{proof}
 First, it is clear from Lemma~\ref{lem-frobenius} that $\ep$ respects
 the equivalence relation $E$ and so induces a map
 $\CE/E\to(M\btm N)(X)$.  Next, let $\CF$ be the set of systems
 \[ y = (U,V,W,i,j,p,m,n) \]
 where the first six entries give an equivariant span diagram
 $\al=(V\tm W\xla{(i,j)}U\xra{p}X)$ and $(m,n)\in M(V)\tm N(W)$.
 Let $F$ be the equivalence relation on $\CF$ that is implicit in
 Remark~\ref{rem-box-kan}.  In more detail, suppose we have spans
 $\lm=(V'\xla{q}A\xra{r}V)$ and $\mu=(W'\xla{s}B\xra{t}W)$ and that
 \[ \al\circ(\lm\tm\mu)=(V'\tm W' \xla{(i',j')} U' \xra{p'} X). \]
 Suppose we also have an element $(m',n')\in M(V')\tm N(W')$, and thus
 elements 
 \begin{align*}
  z  &= (U,V,W,i,j,p,T_rR_q(m'),T_tR_s(n')) \\
  z' &= (U',V',W',i',j',p',m',n')
 \end{align*}
 in $\CF$; then $zFz'$, and $F$ is the smallest equivalence relation
 with this property.

 We have a map $\xi\:\CF\to (M\btm N)(X)$ given by
 \[ \xi(U,V,W,i,j,p,m,n) = T_pR_{(i,j)}(m\btm n), \]
 and Remark~\ref{rem-box-kan} tells us that this induces a bijection
 $\ov{\xi}\:\CF/F\to(M\btm N)(X)$.

 Now define $\phi\:\CF\to\CE$ by 
 \[ \phi(U,V,W,i,j,p,m,n) = (U,p,R_i(m),R_j(n)). \]
 We can write $(i,j)$ as $(i\tm j)\circ\dl$ so
 \[ R_{(i,j)}(m\btm n) = R_\dl R_{i\tm j}(m\btm n) = 
     R_\dl(R_i(m)\btm R_j(n)) = R_i(m)\ot R_j(n).
 \]
 Using this we see that $\ep\phi=\xi$ (and it follows that $\ep$ is
 surjective).  

 We next claim that $\phi$ induces a map $\CF/F\to\CE/E$.  It will be
 enough to show that in the notation used to introduce the relation
 $F$, we have $\phi(z)E\phi(z')$.  Here the definition of $z'$
 involves the composite $\al\circ(\lm\tm\mu)$.  To analyse this, we
 first construct pullback squares as follows:
 \[ \xymatrix{
  U' \ar[r]^{t'} \ar[d]_{r'} & 
  \tA \ar[r]^{i^*} \ar[d]^{r^*} &
  A \ar[d]^r \\
  \tB \ar[r]_{t^*} \ar[d]_{j^*} &
  U \ar[d]^j \ar[r]_i & 
  V \\
  B \ar[r]_t &
  W
 } \]
 We then put 
 \begin{align*}
  k  &= r^*t'=t^*r' \: U' \to U \\
  p' &= pk=pr^*t'=pt^*r' \: U' \to X \\
  i' &= qi^*t' \: U' \to V'\\
  j' &= sj^*r' \: U' \to W'.
 \end{align*}
 This gives a diagram
 \[ \xymatrix{
  & &
  U'
   \ar[dl]^{(i^*t',j^*r')}
   \ar@/_4ex/[ddll]_{(i',j')}
   \ar[dr]_k
   \ar@/^4ex/[ddrr]^{p'} \\
  & 
  A\tm B
   \ar[dl]^{q\tm s} 
   \ar[dr]_{r\tm t} & &
  U 
   \ar[dl]^{(i,j)} 
   \ar[dr]_p \\
  V'\tm W' & &
  V\tm W & & 
  X
 } \]
 One can check that the middle square is a pullback, so the diagram
 exhibits the span 
 \[ V'\tm W'\xla{(i',j')}U'\xra{p'}X \]
 as the composite $\al\circ(\lm\tm\mu)$, so our notation is consistent
 with that used previously.  We now have
 \[ \phi(z') = (U',p',R_{i'}(m'),R_{j'}(n')) = 
     (U',pr^*t',R_{qi^*t'}(m'),R_{sj^*r'}(n')) =
     (U',pr^*t',R_{t'}R_{qi^*}(m'),R_{sj^*r'}(n')).
 \]
 Now put $y_1=(\tA,pr^*,R_{qi^*}(m'),T_{t'}R_{sj^*r'}(n'))$.  By
 clause~(b) in the definition of the relation $E$, we have
 $\phi(z')Ey_1$.   On the other hand, because the square defining $U'$
 is a pullback, we have $T_{t'}R_{r'}=R_{r^*}T_{t^*}$, so 
 $y_1=(\tA,pr^*,R_{qi^*}(m'),R_{r^*}T_{t^*}R_{sj^*}(n'))$.  Using this
 description together with clause~(a), we get $y_1Ey_2$, where 
 $y_2=(U,p,T_{r^*}R_{qi^*}(m'),T_{t^*}R_{sj^*}(n'))$.  Next, the
 square defining $\tA$ is also a pullback, so $T_{r^*}R_{i^*}=R_iT_r$.
 Using this and the corresponding fact for $\tB$ we obtain
 $y_2=(U,p,R_iT_rR_q(m'),R_jT_tR_s(n'))$.  Now inspection of the
 definitions shows that $y_2=\phi(z)$, so $\phi(z)=\phi(z')$ as
 required.  There is thus an induced operation
 $\ov{\phi}\:\CF/F\to\CE/E$ with $\ov{\ep}\ov{\phi}=\ov{\xi}$ as
 claimed.  As $\ov{\xi}$ is bijective we see that $\ov{\phi}$ is
 injective and $\ov{\ep}$ is surjective.

 In the opposite direction, we define $\psi\:\CE\to\CF$ by 
 \[ \psi(U,p,m,n) = (U,U,U,1,1,p,m,n). \]
 As 
 \[ \ep(U,p,m,n)=T_pR_\dl(m\btm n) = T_pR_{(1,1)}(m\btm n) \]
 we see that $\xi\psi=\ep$.  It is also clear that
 $\phi\psi=1\:\CE\to\CE$.  Now put
 $\ov{\psi}=\ov{\xi}^{-1}\ov{\ep}\:\CE/E\to\CF/F$.  Using the
 bijectivity of $\ov{\xi}$, we find that the following diagram
 commutes:
 \[ \xymatrix{
  \CE \ar@{->>}[r] \ar@{ >->}[d]_{\psi} &
  \CE/E \ar@{->>}[r]^(0.35){\ov{\ep}} \ar@{ >->}[d]_{\ov{\psi}} &
  (M\btm N)(X) \ar@{=}[d] \\
  \CF \ar@{->>}[r] \ar@{->>}[d]_{\phi} &
  \CF/F \ar@{->}[r]^(0.35){\ov{\xi}}_(0.35)\simeq \ar@{->>}[d]_{\ov{\phi}} &
  (M\btm N)(X) \ar@{=}[d] \\
  \CE \ar@{->>}[r]  &
  \CE/E \ar@{->>}[r]_(0.35){\ov{\ep}}  &
  (M\btm N)(X)
 } \]

 We now claim that for any element $y=(U,V,W,i,j,p,m,n)$ we have
 $yF\psi(\phi(y))$.  Indeed, $\psi(\phi(y))$ is the system
 $(U,U,U,1,1,p,R_i(m),R_j(n))$.  To show that this is $F$-related to
 $y$ it will suffice to exhibit spans $\lm\in\CA_G(V,U)$ and
 $\mu\in\CA_G(W,U)$ with 
 \begin{align*}
  T_p\circ R_{(1,1)}\circ(\lm\tm\mu) &= T_pR_{(i,j)} \\
  f_\lm(m) &= R_i(m) \\
  f_\mu(n) &= R_j(n).
 \end{align*}
 Clearly we can just take $\lm=R_i$ and $\mu=R_j$.  It follows that
 $\ov{\phi}$ and $\ov{\psi}$ are inverse to each other, as required.
\end{proof}

We next investigate the relationship between $\Mackey_G$ and the more
obvious category of semigroups with a $G$-action.

\begin{definition}\label{defn-omega-mackey}
 We write $\Semigroups_G$ for the category of semigroups with an
 action of $G$.  We note that for each $g\in G$ we have a $G$-set
 automorphism $\rho(g)\:G\to G$ given by $\rho(g)(x)=xg^{-1}$.  Thus,
 for any Mackey functor $M$, we have semigroup maps
 $T_{\rho(g)}=R_{\rho(g)}^{-1}\:M(G)\to M(G)$, which we can use to
 give a $G$-action on $M(G)$.  We write $\om$ for the resulting
 functor $\Mackey_G\to\Semigroups_G$.
\end{definition}

For later use, it will be helpful to understand the left and right
adjoints of $\om$.  One possible approach is as follows: we let
$\ov{\CF\CA}_G$ denote the full subcategory of $\bCA_G$ whose objects
are the free finite $G$-sets.  Note that any $G$-set that admits an
equivariant map to a free $G$-set is itself automatically free.  Thus,
if $X$ and $Y$ are free and $(X\xla{}A\xra{}Y)$ is a span diagram then
$A$ is free as well.  Moreover, we have 
$\Map_G(G,G)=\Aut_G(G,G)\simeq G$.  Given these facts, one can check
that $\Semigroups_G$ is equivalent to the category of
product-preserving functors from $\ov{\CF\CA}_G$ to sets.  From this
point of view (which is used in~\cite{br:wvt}), the functor $\om$
becomes the restriction functor associated to the inclusion
$\ov{\CF\CA}_G\to\bCA_G$ of coloured theories, and there is a general
theory giving left adjoints for such functors.  However, we prefer to
give more direct and explicit constructions.

\begin{proposition}\label{prop-om-c}
 If we define $cA(X)=\Map_G(X,A)$ as in Example~\ref{eg-MapGXA}, we
 get a functor $c\:\Semigroups_G\to\Mackey_G$ that is right adjoint to
 $\om$.  Moreover, the counit $\ep\:\om cA\to A$ is an isomorphism, so
 $c$ is a full and faithful embedding.
\end{proposition}
\begin{proof}
 We define a semigroup isomorphism
 \[ \ep\:\om cA=cN(G)=\Map_G(G,A) \to A \]
 by $\ep(u)=u(1)$.  In the opposite direction, we define
 $\eta\:M\to c\om(M)$ as follows.  Consider a finite $G$-set $X$ and a
 point $x\in X$.  This gives an equivariant map $\hat{x}\:G\to X$ by
 $\hat{x}(g)=gx$, and this in turn gives a map
 $R_{\hat{x}}\:M(X)\to M(G)=\om M$.  We can thus define
 \[ \eta\:M(X) \to \Map(X,M(G)) \]
 by $\eta(m)(x)=R_{\hat{x}}(m)$.  We have
 $\widehat{gx}=\hat{x}\circ\rho(g)^{-1}$, and using this we see that
 $\eta$ lands in $\Map_G(X,M(G))=c\om(M)(X)$.  Now consider a map
 $f\:X\to Y$ of finite $G$-sets.  We claim that the following diagram
 commutes: 
 \[ \xymatrix{
  M(X) \ar[r]^{T_f} \ar[d]_{\eta} &
  M(Y) \ar[r]^{R_f} \ar[d]_{\eta} &
  M(X)              \ar[d]^{\eta} \\
  c\om M(X) \ar[r]_{T_f} &
  c\om M(Y) \ar[r]_{R_f} &
  c\om M(X).
 } \]
 For the right-hand square, we have
 \[ (\eta R_f(n))(x) = R_{\hat{x}}R_f(n) = R_{f\hat{x}}(n) =
    R_{\widehat{f(x)}}(n) = \eta(n)(f(x)) = (R_f\eta(n))(x).
 \]
 Now consider the left hand square.  For any point $y\in Y$, we have a
 pullback square 
 \[ \xymatrix{
  G\tm f^{-1}\{y\} \ar[r]^(0.64){\text{proj}} \ar[d]_p &
  G \ar[d]^{\hat{y}} \\
  X \ar[r]_f & Y,
 } \]
 where $p(g,x)=gx=\hat{x}(g)$.  It follows that for $m\in M(X)$ we
 have 
 \[ R_{\hat{y}}T_f(m)=T_{\text{proj}}R_p(m) = 
     \sum_{x\in f^{-1}\{y\}} R_{\hat{x}}(m),
 \]
 and it follows that the left square commutes.  This means that $\eta$
 is a morphism of Mackey functors.  Next, we claim that the standard
 triangular diagrams 
 \[ \xymatrix{
  \om M \ar[r]^{\om \eta_M} \ar@{=}[dr] &
  \om c\om M \ar[d]^{\ep_{\om M}} &&
  cN \ar[r]^{\eta_{cN}} \ar@{=}[dr] &
  c\om c N \ar[d]^{c\ep_N} \\
  & \om M &&& cN 
 } \]
 commute.  This can be proved by unwinding the definitions, and is
 left to the reader.  Thus, $\eta$ and $\ep$ are the unit and counit
 of an adjunction, as claimed.

 We have seen already that $\ep\:\om cA\to A$ is an isomorphism.
 Together with the adjunction this gives 
 \[ \Mackey_G(cA',cA) \simeq \Semigroups_G(A',\om cA) \simeq 
     \Semigroups_G(A',A).
 \]
 It is standard and straightforward that the isomorphism arising here
 is inverse to the map induced by $c$, and we conclude that $c$ is
 full and faithful.
\end{proof}

\begin{definition}\label{defn-coconstant}
 Let $A$ be a semigroup with an action of $G$.  For any finite $G$-set
 $X$ we let $G$ act on $\Map(X,A)$ by $(gu)(x)=g.u(g^{-1}x)$ as usual,
 and using this we can construct a coinvariant quotient $\Map(X,A)_G$
 as in Definition~\ref{defn-coinvariant-semigroup}.  We define
 $dA(X)=\Map(X,A)_G$.  Given any equivariant span
 $\om=(X\xla{p}A\xra{q}Y)$, we define 
 $f_\om\:\Map(X,A)\to\Map(Y,A)$ by the usual rule
 \[ f_\om(u)(y) = \sum_{q(a)=y} u(p(a)). \]
 This commutes with the $G$-action and so induces a map 
 $f_\om\:dA(X)\to dA(Y)$ of coinvariant quotients.  It is clear that
 this makes $dA$ into a Mackey functor.
\end{definition}

\begin{remark}\label{rem-dA-fixed}
 Let $\pi\:X\to X/G$ denote the obvious quotient map, which gives a
 map $T_\pi\:\Map(X,A)\to\Map(X/G,A)$.  If $G$ acts trivially on $A$
 then it is not hard to check that $T_\pi$ induces an isomorphism
 $dA(X)=\Map(X,A)_G\to\Map(X/G,A)$.  However, if $G$ acts nontrivially
 on $A$ then $T_\pi$ does not interact well with the actions.
\end{remark}
\begin{remark}\label{rem-dA-orbit}
 One can also check that $dA(G/H)\simeq A_H$.  More generally, any
 finite $G$-set $X$ can be written in the form
 $X\simeq\coprod_{i=1}^rG/H_i$, and then we have
 $dA(X)\simeq\bigoplus_iA_{H_i}$.  
\end{remark}

\begin{proposition}\label{prop-d-om}
 The functor $d\:\Semigroups_G\to\Mackey_G$ is left adjoint to $\om$.
 Moreover, the unit map $\eta\:A\to\om dA$ is an isomorphism, so $d$
 is a full and faithful embedding.
\end{proposition}
\begin{proof}
 First, we define maps 
 \[ A \xra{\zt} \Map(G,A) \xra{\sg} A \]
 by $\sg(u)=\sum_{x\in G} x.u(x^{-1})$ and 
 \[ \zt(a)(x) = \begin{cases} 
      a & \text{ if } x=1 \\
      0 & \text{ otherwise. }
    \end{cases}
 \]
 It is clear that $\sg\zt=1$.  Next, recall that the standard action
 of $G$ on $\Map(X,A)$ is $(gu)(x)=g.u(g^{-1}x)$.  Applying this in
 the case $X=G$ we get
 \[ \sg(gu) = \sum_{x\in G} xg.u(g^{-1}x^{-1}) =
     \sum_{y\in G} y.u(y^{-1}) = \sg(u).
 \]
 It follows that there is a unique map 
 \[ \ov{\sg}\:\Map(G,A)_G = dA(G) = \om d(A)\to A \]
 satisfying $\ov{\sg}\pi=\sg$ (where $\pi\:\Map(G,A)\to\Map(G,A)_G$ is
 the usual quotient map).  We write 
 \[ \eta=\pi\zt\:A\to \Map(G,A)_G = dA(G)=\om dA, \]
 so $\ov{\sg}\eta=1$.  Next, note that $g.\zt(a)$ is the map $G\to A$
 sending $g$ to $ga$ and all other points to zero.  It follows that an
 arbitrary element $u\in\Map(G,A)$ can be written as
 $u=\sum_g g.\zt(g^{-1}u(g))$.  As $\pi(g.v)=\pi(v)$ it follows that
 \[ \pi(u)=\sum_g\pi(\zt(g^{-1}u(g)))=
     \eta\left(\sum_g g^{-1}u(g)\right)=\eta\sg(u)=
     \eta\ov{\sg}\pi(u).
 \]
 This implies that $\eta\ov{\sg}=1$, so $\ov{\sg}$ is an inverse for
 $\eta$ and $\eta$ is an isomorphism.

 Next, for any $x\in X$ we have a map $\hat{x}\:G\to X$ sending $g$ to
 $gx$, so we can define 
 \[ \ep_0 \: \Map(X,M(G)) \to M(X) \]
 by $\ep_0(u)=\sum_{x\in X}T_{\hat{x}}(u(x))$.  After recalling that
 $G$ acts on $M(G)$ via the maps $T_{\rho(g)}$ and that
 $\hat{x}\circ\rho(g)=\widehat{g^{-1}x}$ we find that
 $\ep_0(g.u)=\ep_0(g)$, so there is a unique homomorphism
 $\ep\:d\om M(X)=\Map(X,M(G))_G\to M(X)$ with $\ep\pi=\ep_0$.  

 Now consider a map $f\:X\to Y$ of finite $G$-sets.  We claim that the
 following diagram commutes:
 \[ \xymatrix{
  d\om M(X) \ar[r]^{T_f} \ar[d]_\ep &
  d\om M(Y) \ar[r]^{R_f} \ar[d]_\ep &
  d\om M(X)              \ar[d]^\ep \\
  M(X) \ar[r]_{T_f} &
  M(Y) \ar[r]_{R_f} &
  M(X). 
 } \]
 It will clearly suffice to show that the related diagram
 \[ \xymatrix{
  \Map(X,M(G)) \ar[r]^{T_f} \ar[d]_{\ep_0} &
  \Map(Y,M(G)) \ar[r]^{R_f} \ar[d]_{\ep_0} &
  \Map(X,M(G))              \ar[d]^{\ep_0} \\
  M(X) \ar[r]_{T_f} &
  M(Y) \ar[r]_{R_f} &
  M(X)
 } \]
 commutes.  For the left square we note that
 $(T_fu)(y)=\sum_{f(x)=y}u(x)$ and $\widehat{f(x)}=f\circ\hat{x}$ so 
 \begin{align*}
  \ep_0T_f(u) &= \sum_{y\in Y} T_{\hat{y}}((T_fu)(y)) 
    = \sum_{y\in Y} T_{\hat{y}}\left(\sum_{f(x)=y} u(x)\right) \\
   &= \sum_{x\in X} T_{\widehat{f(x)}}(u(x)) 
    = T_f\left(\sum_{x\in X} T_{\hat{x}}(u(x))\right) 
    = T_f\ep_0(u).
 \end{align*}
 For the right square, we recall that there is a cartesian square
 \[ \xymatrix{
  G\tm f^{-1}\{y\} \ar[r]^(0.64){\text{proj}} \ar[d]_p &
  G \ar[d]^{\hat{y}} \\
  X \ar[r]_f & Y,
 } \]
 giving $R_fT_{\hat{y}}=\sum_{f(x)=y}T_{\hat{x}}\:M(G)\to M(X)$.
 Using this we get 
 \[ R_f\ep_0(u) = \sum_{y\in Y}R_fT_{\hat{y}}u(y) = 
     \sum_{y\in Y}\sum_{x\in f^{-1}\{y\}} T_{\hat{x}}u(y) =
     \sum_{x\in X}T_{\hat{x}}u(f(x)) = \ep_0R_f(u).
 \]
 This proves that $\ep$ is a morphism of Mackey functors.  We again
 leave the reader to check the triangular diagrams 
 \[ \xymatrix{
  dA \ar[r]^{d\eta_A} \ar@{=}[dr] &
  d\om dA \ar[d]^{\ep_{dA}} & &
  \om M \ar[r]^{\eta_{\om M}} \ar@{=}[dr] &
  \om d\om M \ar[d]^{\om\ep_M} \\
  & dA &&& \om M  
 }\]
 showing that we have an adjunction.  We saw earlier that
 $\eta\:A\to\om dA$ is an isomorphism, so we have 
 \[ \Mackey_G(dA',dA) \simeq \Semigroups_A(A',\om dA)
     \simeq \Semigroups_G(A',A),
 \]
 showing that $d$ is full and faithful.
\end{proof}

\section{Mackey functors for the group of order two}
\label{sec-two-mackey}

By way of example, we will study the case where $G=\{1,\chi\}$ with
$\chi^2=1$.  The simplest (and most standard) approach would be to use
Proposition~\ref{prop-mackey-orbits}.  Although we will mention this
in passing, we will mostly focus on a different construction which
generalises more easily to the nonadditive context of Tambara
functors. 

When $G$ acts on a set $X$, we will usually write $\ov{x}$ for
$\chi.x$.

\begin{definition}\label{defn-MP}
 A \emph{Mackey pair} consists of semigroups $A$ and $B$, together
 with an action of $G$ on $A$ by semigroup maps, and semigroup maps
 $A\xra{\trc}B\xra{\res}A$ satisfying 
 \begin{align*}
  \trc(0)          &= 0  \\
  \trc(a_0+a_1)    &= \trc(a_0)+\trc(a_1) \\
  \trc(\ov{a})     &= \trc(a) \\
  \ov{\res(b)}     &= \res(b) \\
  \res(\trc(a))    &= a+\ov{a}.
 \end{align*}
 We write $\MP$ for the category of Mackey pairs.
\end{definition}

\begin{construction}\label{cons-mackey-MP}
 Given a $G$-Mackey functor $M$, put $A=M(G/1)=M(G)$ and
 $B=M(G/G)=M(1)$ (so both of these are semigroups).  Next, as $G$ is
 commutative we see that $\chi\:G\to G$ is a $G$-map, with
 $\chi=\chi^{-1}$, so $T_\chi=R_\chi\:A\to A$.  We use this map
 to define an action of $G$ on $A$ by semigroup maps.

 The projection $\ep\:G\to 1$ gives maps
 \begin{align*}
  \res=\ep^* & \: B \to A \\
  \trc=T_\ep & \: A \to B.
 \end{align*} 
\end{construction}

We will prove that this gives an equivalence between Mackey functors
and Mackey pairs.  The first thing to check is that we at least have
a functor.

\begin{proposition}\label{prop-MP-functor}
 The above construction gives a faithful functor $F\:\Mackey_G\to\MP$.
\end{proposition}
\begin{proof}
 We first need to check that the construction gives a Mackey pair.  We
 have already seen that $A$ and $B$ have natural semigroup structures
 such that $\res$ and $\trc$ are semigroup homomorphisms.  By applying
 $M$ to the identity $\ep\chi=\ep$ we see that $\ov{\res(b)}=\res(b)$,
 so the image of $\res$ lies in the subsemigroup
 $A^G=\{a\in A\st\ov{a}=a\}$.  Similarly, we can apply $T$ to the
 identity $\ep=\ep\chi$ to see that $\trc(a)=\trc(\ov{a})$.  Next, it
 is straightforward to check that the diagram
 \[ \xymatrix{
  G\amalg G \ar[d]_s \ar[r]^{1\amalg\chi} &
  G\amalg G \ar[r]^s &
  G \ar[d]^\ep \\
  G \ar[rr]_\ep && 1
 } \]
 is a pullback.  By the Mackey property, it follows that
 \[ R_\ep T_\ep = T_s T_{1\amalg\chi} R_s \: M(G)\to M(G), \]
 or in other words that $\res(\trc(a))=a+\ov{a}$ for all $a\in A$.
 There is thus an evident way to make $F$ into a 
 functor $\Mackey_G\to\MP$.

 To check that $F$ is faithful, suppose we have morphisms
 $\phi,\psi\:M\to M'$ with $M(\phi)=M(\psi)$, or equivalently
 $\phi_1=\psi_1$ and $\phi_G=\psi_G$.  Let $X$ be an arbitrary finite
 $G$-set.  Let $n$ be the number of free orbits, and let $m$ be the
 number of fixed points, so $X$ is isomorphic to the disjoint union of
 $n$ copies of $G$ and $m$ copies of $1$.  This gives commutative
 diagrams 
 \[ \xymatrix{
  M(X) \ar[d]_{\phi_X} \ar[r]^\simeq &
  A^n\tm B^m \ar[d]^{\phi_G^n\tm\phi_1^m} & & 
  M(X) \ar[d]_{\psi_X} \ar[r]^\simeq &
  A^n\tm B^m \ar[d]^{\psi_G^n\tm\psi_1^m} \\
  M'(X) \ar[r]_\simeq & 
  (A')^n\tm (B')^m & & 
  M'(X) \ar[r]_\simeq & 
  (A')^n\tm (B')^m.  
 } \]
 As $\phi_G=\psi_G$ and $\phi_1=\psi_1$ we deduce that
 $\phi_X=\psi_X$, as required.
\end{proof}

It is now not hard to describe the structure of the semigroups
$\bCA_G(G/H,G/K)$ for $H,K\in\{1,G\}$ and then use
Proposition~\ref{prop-res-equiv} to see that $F$ is an equivalence.
However, we will construct the inverse functor in a more explicit
way. 

\begin{construction}\label{cons-EP}
 Let $P=(A,B)$ be a Mackey pair.  For any finite $G$-set $X$, we
 put 
 \[ EP(X) =
     \{(u,v)\in\Map_G(X,A)\tm\Map(X^G,B)\st 
         u(x)=\res(v(x)) \text{ for all } x\in X^G\},
 \]
 so we have a cartesian square
 \[ \xymatrix{
  EP(X) \ar[d] \ar[r] & \Map(X^G,B) \ar[d]^{\res_*} \\
  \Map_G(X,A) \ar[r]_\rho & \Map(X^G,A^G).
 } \]
 Now suppose we have a $G$-equivariant map $f\:X\to Y$.  We define
 $R_f=f^*\:EP(Y)\to EP(X)$ by $R_f(m,n)=(m\circ f,n\circ f^G)$ (where
 $f^G\:X^G\to Y^G$ is just the restriction of $f$).
\end{construction}

After some further discussion and examples we will define maps 
$T_f\:EP(X)\to EP(Y)$ which will make $EP$ into a Mackey functor.  
First, however, we mention an approach that does \emph{not} work.  
The constructions $X\mapsto\Map_G(X,A)$ and $X\mapsto\Map(X^G,A^G)$
and $X\mapsto\Map(X^G,B)$ define Mackey functors as in
Examples~\ref{eg-MapGXA} and~\ref{eg-MapXGB}.  However, the
restriction map $\rho\:\Map_G(X,A)\to\Map(X^G,A^G)$ is not a morphism
of Mackey functors, so we do not have a pullback square in $\Mackey_G$
as one might naively expect.

\begin{example}\label{eg-EP-free}
 Suppose that $G$ acts freely on $X$, so $X$ is a disjoint union of
 $n$ copies of $G$ say.  Then $X^G=\emptyset$, so $\Map(X^G,A^G)$ and
 $\Map(X^G,B)$ are singletons, so $EP(X)=\Map_G(X,A)\simeq A^n$.  More
 precisely, the defining pullback square for $EP(X)$ has the form
 \[ \xymatrix{
  EP(X) \ar[d]_\simeq \ar[r] & 1 \ar[d]^{\res_*} \\
  \Map_G(X,A) \ar[r]_\rho & 1.
 } \]
 In particular, for the case $X=G$ we have $EP(G)=A$.
\end{example}

\begin{example}\label{eg-EP-fixed}
 Now suppose instead that $G$ acts trivially on $X$ and $|X|=k$.  Then
 $X=X^G$ and $\Map_G(X,A)=\Map(X,A^G)=\Map(X^G,A^G)$.  It follows that
 $EP(X)=\Map(X,B)=B^k$.  More precisely, the defining pullback square
 for $EP(X)$ has the form 
 \[ \xymatrix{
  EP(X) \ar[d] \ar[r]^\simeq & \Map(X,B) \ar[d]^{\res_*} \\
  \Map(X,A^G) \ar[r]_= & \Map(X,A^G).
 } \]
 In particular, for the case $X=1$ we have $EP(1)=B$.
\end{example}

\begin{construction}\label{cons-EP-transfer}
 Suppose we have a $G$-equivariant map $f\:X\to Y$.  Note that
 $f^{-1}(Y^G)$ will contain $X^G$, and possibly some free orbits as
 well.  If we choose a point in each such free orbit, we get a
 decomposition $f^{-1}(Y^G)=X^G\amalg X_1\amalg\ov{X_1}$ say.

 We define $T_f\:EP(X)\to EP(Y)$ by $T_f(u,v)=(m,n)$, where 
 \begin{align*}
  m(y) &= \sum_{x\in f^{-1}\{y\}} u(x) \\
  n(y) &= \sum_{x_0\in X^G\cap f^{-1}\{y\}} v(x_0) + 
           \sum_{x_1\in X_1\cap f^{-1}\{y\}} \trc(u(x_1)).
 \end{align*}
 To see that this does indeed define an element of $EP(Y)$, observe
 that when $x_1\in X_1$ we have $\ov{u(x_1)}=u(\ov{x_1})$, so when
 $y\in Y^G$ we have 
 \begin{align*}
  \res(n(y))
    &= \sum_{x_0\in X^G\cap f^{-1}\{y\}} \res(v(x_0)) + 
       \sum_{x_1\in X_1\cap f^{-1}\{y\}} \res(\trc(u(x_1))) \\
    &= \sum_{x_0\in X^G\cap f^{-1}\{y\}} u(x_0) + 
       \sum_{x_1\in X_1\cap f^{-1}\{y\}} (u(x_1)+\ov{u(x_1)}) \\
    &= \sum_{x_0\in X^G\cap f^{-1}\{y\}} u(x_0) + 
       \sum_{x_1\in X_1\cap f^{-1}\{y\}} u(x_1) + 
       \sum_{\ov{x_1}\in \ov{X_1}\cap f^{-1}\{y\}} u(\ov{x_1}) \\
    &= m(y).
 \end{align*}
 Using the identity
 $\trc(u(\ov{x_1}))=\trc(\ov{u(x_1)})=\trc(u(x_1))$ we also see that
 the construction is independent of the choice of $X_1$.
\end{construction}

\begin{proposition}\label{prop-EP-mackey}
 The above definitions make $EP$ into a Mackey functor.
\end{proposition}
\begin{proof}
 Suppose we have maps $X\xra{f}Y\xra{g}Z$.  It is clear that
 $R_{gf}=R_fR_g$.  We must check that we also have $T_{gf}=T_gT_f$.
 Consider a pair $(u,v)\in EP(X)$, so $T_f(u,v)=(m,n)$ and
 $T_g(m,n)=(p,q)$ say.  For the first component, we have 
 \begin{align*}
  m(y) &= \sum_{x\in f^{-1}\{y\}} u(x) \\
  p(z) &= \sum_{y\in g^{-1}\{z\}} m(y)
        = \sum_{y\in g^{-1}\{z\}} \sum_{x\in f^{-1}\{y\}} u(x) \\
       &= \sum_{x\in (gf)^{-1}\{z\}} u(x),
 \end{align*}
 which is the same as the first component in $T_{gf}(u,v)$.  The
 second component requires more work.  Fix a point $z\in Z^G$. Put
 $Y_0=Y^G\cap g^{-1}\{z\}$ and choose $Y_2\sse Y$ such that
 $g^{-1}\{z\}=Y_0\amalg Y_2\amalg\ov{Y_2}$.  Then put
 $Z_0=X^G\cap f^{-1}(Y_0)$, and choose $X_1\sse X$ such that
 $f^{-1}(Y_0)=X_0\amalg X_1\amalg\ov{X_1}$.  Put $X_2=f^{-1}(Y_2)$ and
 observe that $\ov{X_2}=f^{-1}(\ov{Y_2})$ and so 
 \[ (gf)^{-1}\{z\} = X_0\amalg X_1 \amalg X_2
                         \amalg\ov{X_1}\amalg\ov{X_2}.
 \]
 \[ \xymatrix{
  X_0 \ar[drr]^f \\
  X_1 \amalg \ov{X}_1 \ar[rr]_f && Y_0 \ar[drr]^g \\
  X_2 \amalg \ov{X}_2 \ar[rr]_f && Y_2 \amalg \ov{Y}_2 \ar[rr]_g && \{z\}
 } \]

 Now, for $y_0\in Y_0$ we have 
 \[ n(y_0) = \sum_{x_0\in X_0\cap f^{-1}\{y_0\}} v(x_0) +
             \sum_{x_1\in X_1\cap f^{-1}\{y_0\}} \trc(u(x_1)).
 \]
 It follows that
 \begin{align*}
  q(z) &= \sum_{y_0\in Y_0} n(y_0) + \sum_{y_2\in Y_2}\trc(m(y_2)) \\
       &= \sum_{x_0\in X_0} v(x_0) + 
          \sum_{x_1\in X_1} \trc(u(x_1)) + 
          \sum_{x_2\in X_2} \trc(u(x_2)),
 \end{align*}
 which is the same as the second component in $T_{gf}(u,v)(z)$, as
 required.

 Now suppose we have a cartesian square
 \[ \xymatrix{
   W \ar[r]^f \ar[d]_g & X \ar[d]^h \\
   Y \ar[r]_k & Z
 } \]
 We claim that $T_gR_f=R_kT_h\:EP(X)\to EP(Y)$.  To see this, consider
 a pair $(u,v)\in EP(X)$, so $T_h(u,v)=(p,q)$ say, and
 $R_kT_h(u,v)=(p\circ k,q\circ k^G)$.  The cartesian property means
 that $f$ induces a bijection $g^{-1}\{y\}\to h^{-1}\{k(y)\}$ for
 all $y\in Y$.  This means that 
 \[ p(k(y)) = \sum_{x\in h^{-1}\{k(y)\}} u(x) 
            = \sum_{w\in g^{-1}\{y\}} u(f(w)),
 \]
 so $p\circ k$ is also the first component in $T_gR_f(u,v)$.  Next,
 consider a point $y\in Y^G$.  Put $W_0=W^G\cap g^{-1}\{y\}$ and
 choose $W_1\sse W$ such that
 $g^{-1}\{y\}=W_0\amalg W_1\amalg\ov{W_1}$.  Put $X_i=f(W_i)$; the
 cartesian property means that $f$ induces bijections $W_i\to X_i$,
 and that $h^{-1}\{k(y)\}=X_0\amalg X_1\amalg\ov{X_1}$.  We thus have 
 \begin{align*}
  q(k(y)) &= \sum_{x_0\in X_0} v(x_0) +
              \sum_{x_1\in X_1} \trc(u(x_1)) \\
          &= \sum_{w_0\in W_0} v(f(w_0)) + 
              \sum_{w_1\in W_1} \trc(u(f(w_1))), 
 \end{align*}
 so $q\circ k$ is also the second component in $T_gR_f(u,v)$.  This
 proves that $T_gR_f=R_kT_h$, so we have a Mackey functor as claimed.
\end{proof}

\begin{theorem}\label{thm-MP}
 The functor $F\:\Mackey_G\to\MP$ is an equivalence, with inverse
 given by $E$.
\end{theorem}
\begin{proof}
 There is an evident way to define $E$ on morphisms, giving a
 functor $E\:\MP\to\Mackey_G$.  It is clear from
 Examples~\ref{eg-EP-free} and~\ref{eg-EP-fixed} that
 $FEP=(EP(G),EP(1))=(A,B)=P$, so $FE=1$.  

 In the opposite direction, suppose we start with a Mackey functor $M$
 and put $A=M(G)$ and $B=M(1)$, so $FM=(A,B)$.  Consider an arbitrary
 finite $G$-set $X$.  For $x\in X$ we have a map $\hat{x}\:G\to X$ given
 by $\hat{x}(1)=x$ and $\hat{x}(\chi)=\ov{x}$.  This gives a map
 $R_{\hat{x}}\:M(X)\to M(G)=A$, and by combining these we get a map
 $\al\:M(X)\to\Map(X,A)$.  It is straightforward to check that this
 actually lands in $\Map_G(X,A)$.  On the other hand, for $x\in X^G$
 we have a map $\breve{x}\:1\to X$ sending $0$ to $x$, and thus
 $R_{\breve{x}}\:M(X)\to M(1)=B$.  By combining these we get a map
 $\bt\:M(X)\to\Map(X^G,B)$.  Using $\hat{x}=\breve{x}\ep$ we see that the
 square
 \[ \xymatrix{
  M(X) \ar[d]_\al \ar[r]^\bt & \Map(X^G,B) \ar[d] \\
  \Map_G(X,A) \ar[r] & \Map(X^G,A^G) 
 } \]  
 commutes, so we have a natural map from $M(X)$ to the pullback
 $EP(X)=EFM(X)$.  This is a morphism of Mackey functors, which is an
 isomorphism for $X=G$ or $X=1$.  As Mackey functors convert disjoint
 unions to products, we see that $M\simeq EFM$ as required.  
\end{proof}

\section{The theory of semirings}
\label{sec-semirings}

We now repeat much of section~\ref{sec-semigroups}, for semirings
rather than semigroups.  For us, a \emph{semiring} will mean a
semigroup with a second commutative, associative and unital binary
operation (written as multiplication) that distributes over addition.
In particular, we assume that $0a=0$ for all $a$ (which is the
distributivity rule for the sum of no terms).

\begin{example}\label{eg-semirings}
 \begin{itemize}
  \item[(a)] Any commutative ring is of course also a semiring.
  \item[(b)] $\N$ is a semiring under the usual operations.  We can
   also define polynomial semirings $\N[t]$ or $\N[t_1,\dotsc,t_r]$ in
   an evident way.
  \item[(c)] For any semigroup $A$ we have a semigroup semiring
   $\N[A]$.  This is freely generated as a semigroup by elements $[a]$
   for all $a\in A$, with the multiplication rule $[a][b]=[a+b]$.  The
   polynomial semiring $\N[t_1,\dotsc,t_r]$ can be identified with
   $\N[\N^r]$.
  \item[(d)] Let $G$ be a finite group, and let $A(G)$ denote the
   set of isomorphism classes of finite $G$-sets.  This is a semiring
   under the operations $[X]+[Y]=[X\amalg Y]$ and $[X][Y]=[X\tm Y]$.
   We call this the \emph{Burnside semiring} of $G$.
  \item[(e)] Let $G$ be a finite group, and let $R(G)$ denote the
   set of isomorphism classes of complex representations of $G$.  This
   is a semiring under the operations $[V]+[W]=[V\oplus W]$ and
   $[V][W]=[V\ot W]$.  We call this the \emph{representation semiring}
   of $G$.
  \item[(f)] Consider a set $E$ consisting of elements $\ep^n$ for all
   $n\in\Z$, together with two more elements $0$ and $\al$.  We write
   $1$ for $\ep^0$.  We can make this a semiring by the following
   rules 
   \[ \begin{array}{|c|c|c|c|c|} \hline
       +     & \al & \ep^n           & 0     \\ \hline
       \al   & \al & \al             & \al   \\ \hline
       \ep^m & \al & \ep^{\min(n,m)} & \ep^m \\ \hline
       0     & \al & \ep^n           & 0     \\ \hline
      \end{array}
      \hspace{6em}
      \begin{array}{|c|c|c|c|c|} \hline
       \cdot & \al & \ep^n           & 0     \\ \hline
       \al   & \al & \al             & 0     \\ \hline
       \ep^m & \al & \ep^{n+m}       & 0     \\ \hline
       0     & 0   & 0               & 0     \\ \hline
      \end{array}
   \]
   Alternatively, we can regard $\al$ as $\ep^{-\infty}$ and $0$ as
   $\ep^\infty$ and then we have $\ep^n+\ep^m=\ep^{\min(n,m)}$ for
   all $n$ and $m$, and also $\ep^n\ep^m=\ep^{n+m}$ provided that we
   interpret $\infty+(-\infty)$ as $\infty$.  The real point about
   this example is as follows.  Let $C_*$ be a graded vector space
   over $\Q$, and put $B(C_*)=\{n\st C_n\neq 0\}$ and
   $\bt(C_*)=\ep^{\inf(B(C_*))}\in E$ (with the convention
   $\inf(\emptyset)=\infty$).  We then have
   $\bt(C_*\oplus D_*)=\bt(C_*)+\bt(D_*)$ and
   $\bt(C_*\ot D_*)=\bt(C_*)\bt(D_*)$.  For more general graded
   abelian groups we still have $\bt(C_*\oplus D_*)=\bt(C_*)+\bt(D_*)$
   and $\bt(C_*\ot D_*)\leq\bt(C_*)\bt(D_*)$, but the inequality can
   be strict if $C_*$ and $D_*$ have torsion at different primes.
 \end{itemize}
\end{example}

Any finite set $X$ gives a functor $R\mapsto R^X=\Map(X,R)$ from
semirings to sets.  Let $\bCU(X,Y)$ be the set of natural maps 
$R^X\to R^Y$.  

\begin{remark}
 In contrast with the case of semigroups, there are many natural
 maps $R^X\to R^Y$ that are not semiring homomorphisms, for example
 the map $f\:R^2\to R^2$ given by $f(x,y)=(y^2+1,x^2+1)$.
\end{remark}

\begin{definition}\label{defn-TNR}
 For any function $f\:X\to Y$ we define three different operations as
 follows: 
 \begin{itemize}
  \item[(a)] We have a map $T_f\:R^X\to R^Y$ given by
   $T_f(r)(y)=\sum_{f(x)=y}r(x)$.
  \item[(b)] We have another map $N_f\:R^X\to R^Y$ given by
   $N_f(r)(y)=\prod_{f(x)=y}r(x)$.
  \item[(c)] We have a map $R_f=f^*\:R^Y\to R^X$ given by
   $R_f(r)(x)=r(f(x))$, or equivalently $f^*(r)=r\circ f$.
 \end{itemize}
\end{definition}

Note that $R_f$ and $T_f$ only use the additive semigroup structure,
so their properties and interactions are the same as in
Section~\ref{sec-semigroups}.  Similarly, $R_f$ and $N_f$ only use the
multiplicative semigroup structure, so their properties and
interactions are the same as in Section~\ref{sec-semigroups} up to a
slight change of notation.  In more detail:
\begin{itemize}
 \item[(a)] For all $X\xra{f}Y\xra{g}Z$ we have $T_{gf}=T_gT_f$ and
  $N_{gf}=N_gN_f$ and $R_{gf}=R_fR_g$.
 \item[(b)] For any cartesian square
  \[ \xymatrix{
   W \ar[r]^f \ar[d]_g & X \ar[d]^h \\
   Y \ar[r]_k & Z
  } \]
  we have $T_fR_g=R_kT_h$ and $N_fR_g=R_kN_h$.
 \item[(c)] As a special case of~(b), for any bijection $f\:X\to Y$ we
  have $N_f=T_f=R_f^{-1}$.
\end{itemize}
This just leaves the problem of understanding how maps of the form
$N_f$ interact with maps of the form $T_g$, which comes down to
exploring the combinatorics of expanding products of sums as sums of
products. 

First, however, we will study the set $\bCU(X,Y)$ of all natural
operations.  By the Yoneda Lemma, we have
\[ \bCU(X,Y) = \Map(Y,\N[t_x\st x\in X]) = 
  \text{Semirings}(\N[t_y\st y\in Y],\N[t_x\st x\in X]).
\] 
Now suppose we have a decomposition $Y=\coprod_{a\in A}Y_a$.  The
inclusions $i_a\:Y_a\to Y$ give morphisms $R_{i_a}\:Y\to Y_a$ in
$\bCU$, which induce maps $(R_{i_a})_*\:\bCU(X,Y)\to\bCU(X,Y_a)$,
which we can combine to give a single natural map
\[ \bCU(X,Y) \to \prod_{a\in A} \bCU(X,Y_a). \]
Using the description $\bCU(X,Y)=\Map(Y,\N[t_x\st x\in X])$ we see
that this map is a natural isomorphism.  This means that $Y$ is the
categorical product in $\bCU$ of the objects $Y_a$.  

\begin{proposition}\label{prop-semiring-theory}
 The category of semirings is equivalent to the category of
 product-preserving functors from $\bCU$ to the category of sets.
\end{proposition}
\begin{proof}
 Essentially the same as for Proposition~\ref{prop-semigroup-theory}.
\end{proof}

The description $\bCU(X,Y)=\Map(Y,\N[t_x\st x\in X])$ gives a
canonical semiring structure on the set $\bCU(X,Y)$.  It is generated
by elements $t_x$ and $e_y$ (corresponding to the natural maps
$t_x(f)(z)=f(x)$ and $e_y(f)(z)=\dl_{yz}$) subject only to the
relations $e_ye_z=\dl_{yz}e_z$ and $\sum_ye_y=1$.  This semiring
structure behaves well with respect to functions $X\to X'$, but not
with respect to arbitrary morphisms $X\to X'$ in $\bCU$.  Thus, we
cannot say that $Y$ is a semiring object in $\bCU$.

The idea of Tambara theory is to introduce a bicategory $\CU$ whose
$0$-cells are finite sets, such that $\CU(X,Y)$ is a groupoid with
$\pi_0\CU(X,Y)=\bCU(X,Y)$.

Specifically, we take $\CU(X,Y)$ to be the category of diagrams of
the form
\[ X \xla{p} A \xra{q} B \xra{r} Y, \]
where $A$ and $B$ are finite sets.  We call these diagrams
\emph{bispans}.  The morphisms are commutative diagrams 
\[ \xymatrix{
     X \ar@{=}[d] & 
     A \ar[l] \ar[r] \ar[d]_{\al} & 
     B \ar[r] \ar[d]^{\bt} & 
     Y \ar@{=}[d] \\
     X & 
     A' \ar[l] \ar[r] &
     B' \ar[r] &
     Y
   }
\]
in which $\al$ and $\bt$ are bijections.  Given a bispan
\[ \om = (X\xla{p}A\xra{q}B\xra{r}Y) \in \CU(X,Y) \]
as above, we define $f_\om=T_rN_qR_p\:R^X\to R^Y$, or more explicitly 
\[ f_\om(s)(y) =
    \sum_{b\in r^{-1}\{y\}} \prod_{a\in q^{-1}\{b\}} s(p(a)).
\]
Alternatively, in terms of our generators $t_x$ and $e_y$ for the
semiring $\bCU(X,Y)$, we have
\[ f_\om = \sum_y e_y 
    \sum_{b\in r^{-1}\{y\}} \prod_{a\in q^{-1}\{b\}} t_{p(a)} =
     \sum_{b\in B} e_{r(b)} \prod_{a\in q^{-1}\{b\}} t_{p(a)}.
\]

One can check that this construction gives a bijection
$\pi_0\CU(X,Y)\to\bCU(X,Y)$.  Indeed, a typical monomial in
$\bCU(X,Y)$ has the form $e_y\prod_xt_x^{n(x)}$ for some $y\in Y$ and
$n\in\N^X$.  Thus a typical element of $\bCU(X,Y)$ has the form
\[ g = \sum_{(y,n)} m(y,n) e_y \prod_x t_x^{n(x)} \]
for some map $m\:Y\tm\N^X\to\N$ of finite support.  Now put 
\begin{align*}
 A &= \{(y,n,i,x,j) \in Y\tm\N^X\tm\N\tm X\tm\N \st 
          i < m(y,n),\; j < n(x) \} \\
 B &= \{(y,n,i)\in Y\tm\N^X\tm\N \st i<m(y,n)\} \\
 p(y,n,i,x,j) &= x \\
 q(y,n,i,x,j) &= (y,n,i) \\
 r(y,n,i) &= y.
\end{align*}
This gives a bispan $\om=(X\xla{p}A\xra{q}B\xra{r}Y)$ with $f_\om=g$,
and it is not hard to check that any other bispan with this property
is isomorphic to $\om$.

In particular, if we have maps $X\xra{f}Y\xra{g}Z$, the operation
$N_gT_f\:R^X\to R^Z$ must come from a bispan.  We can construct one as
follows.
\begin{definition}\label{defn-distributor}
 For any maps $X\xra{g}Y\xra{h}Z$ we define a bispan 
 \[ \Dl(g,h) = (X\xla{p}A\xra{q}B\xra{r}Z) \]
 by 
 \begin{align*}
  B &=
   \{(z,s)\st z\in Z,\;s\:h^{-1}\{z\}\to X,\;gs=1_{h^{-1}\{z\}}\} \\
  A &= Y\tm_ZB = \{(y,s)\st y\in Y,\;s\:h^{-1}\{h(y)\}\to X,\;gs=1_{h^{-1}\{h(y)\}}\} \\
  p(y,s) &= s(y) \\
  q(y,s) &= (h(y),s) \\
  r(z,s) &= z.
 \end{align*}
 We call this the \emph{distributor} for $(g,h)$.
\end{definition}

\begin{proposition}\label{prop-NT}
 In the above context, we have $N_hT_g=f_{\Dl(g,h)}=T_rN_qR_p$.
\end{proposition}
\begin{proof}
 First note that all the sets $X$, $Y$, $A$ and $B$ have compatible
 maps to $Z$, and everything happens independently over the different
 points of $Z$.  Because of this, we can easily reduce to the special
 case where $Z$ is a single point.  In that context we just have
 $B=\{s\:Y\to X\st gs=1_Y\}$ and $A=Y\tm B$.  Now put
 $X_y=g^{-1}\{y\}\sse X$, so the requirement $gs=1_Y$ just means that
 $s(y)\in X_y$ for all $y$.  We thus have $B=\prod_yX_y$ and
 $A=Y\tm\prod_yX_y$.  Now consider an element $u\in R^X$.  The image
 $N_hT_g(u)\in R$ is given by $\prod_y\sum_{x\in X_y}u(x)$.  If we
 expand this out in the obvious way we get
 $\sum_{s\in B}\prod_{y\in Y}u(s(y))$, and by unwinding the
 definitions we see that this is $T_rN_qR_p(u)$, as required.
\end{proof}

For another perspective on the above construction, define $m\:A\to Y$
by $m(y,s)=y$.  We then have a diagram $D$ as follows, in which the
right hand square is cartesian.
\[ \xymatrix{
 & A \ar[r]^q \ar[d]_m \ar[dl]_p & B \ar[d]^r \\
 X \ar[r]_g & Y \ar[r]_h & Z.
} \]
Let $\CD(g,h)$ denote the category of all diagrams $D'$ like
\[ \xymatrix{
 & A' \ar[r]^{q'} \ar[d]_{m'} \ar[dl]_{p'} & B' \ar[d]^{r'} \\
 X \ar[r]_g & Y \ar[r]_h & Z,
} \]
where the bottom row is the same as before, and the right hand square
is cartesian.  A morphism from $D'$ to $D''$ consists of maps 
$A'\to A''$ and $B'\to B''$ with the evident commutativity
properties.  The following result gives a more abstract
characterisation of $D$.

\begin{proposition}\label{prop-cromulent}
 The object $D$ is terminal in $\CD(g,h)$.
\end{proposition}
\begin{proof}
 Consider another object $D'$ as above.  Let $b'$ be an element of
 $B'$.  Put $z=r'(b')\in Z$ and $Y_1=h^{-1}\{z\}\sse Y$ and
 $A_1=(q')^{-1}\{b'\}\sse A$.  As the square in $D'$ is assumed to be
 cartesian, the map $m'\:A'\to Y$ must restrict to give a bijection
 $m_1\:A_1\to Y_1$.  We thus have a map $s=p'm_1^{-1}\:Y_1\to X$.  As
 $D'$ commutes we have $gp'=m'$ and so $gs=1|_{Y_1}$, which means that
 $(z,s)\in B$.  We can thus define $\bt\:B'\to B$ by $\bt(b')=(z,s)$.
 Next, given $a'\in A'$ we can put $b'=q'(a')\in B'$ and then define
 $z$, $Y_1$, $A_1$, $m_1$ and $s$ as before.  We then have $a'\in A_1$
 so $m_1(a')$ is defined and is an element of $Y_1$.  We thus have a
 point $\al(a')=(m_1(a'),z,s)\in A$.  This construction gives a map
 $\al\:A'\to A$, and it is straightforward to check the equations
 \[ m\al=m' \hspace{5em}
    p\al=p' \hspace{5em}
    q\al=\bt q' \hspace{5em}
    r\bt=r'.
 \]
 This means that the pair $(\al,\bt)$ gives a morphism $D'\to D$.
 Suppose we have another morphism, say $(\al^*,\bt^*)$.  Consider a
 point $b'\in B'$, and put $(z,s)=\bt^*(b')\in B$.  As $r\bt^*=r'$ we
 see that $z=r'(b')$.  Now consider a point $y\in h^{-1}\{z\}$, so
 $s(y)\in X$.  By the pullback property for $D'$, there is a unique
 element $a'\in(q')^{-1}\{b'\}$ with $m'(a')=y$.  The point
 $\al^*(a')\in A$ has $q(\al^*(a'))=\bt^*(q'(a'))=\bt^*(b')=(z,s)$ and
 $m(\al^*(a'))=m'(a')=y$ so we must have $\al^*(a')=(y,z,s)$.  This
 means that $s(y)=p(y,z,s)=p(\al^*(a'))=p'(a')=p'm_1^{-1}(y)$.  We now
 see that $s=p'm_1^{-1}$ as before, so $\bt^*=\bt$.  Given this, it is
 also clear from the above equations that $\al^*=\al$.  
\end{proof}

It is instructive to recover the distributivity axiom
$a_0(a_1+a_2)=a_0a_1+a_0a_2$ directly from the above construction.
Let $S\:\bCU\to\Sets$ be a product-preserving functor, so we can
identify $S(n)=S(\{0,1,\dotsc,n-1\})$ with $S(1)^n$.  The map
$z\:0\to 1$ gives elements $0=T_z()\in S(1)$ and $1=N_z()\in S(1)$,
and the map $s\:2\to 1$ gives binary operations $a+b=T_s(a,b)$ and
$ab=N_s(a,b)$.  Proposition~\ref{prop-semiring-theory} shows
abstractly that this gives a semiring structure.  For a direct check
of distributivity, we define maps 
\[ \xymatrix{
  3 \ar[r]^f & 2 \ar[r]^g & 1 \\
  4 \ar[u]^{p'} \ar[rr]_{q'} && 2 \ar[u]_{r'}
} \]
as follows:
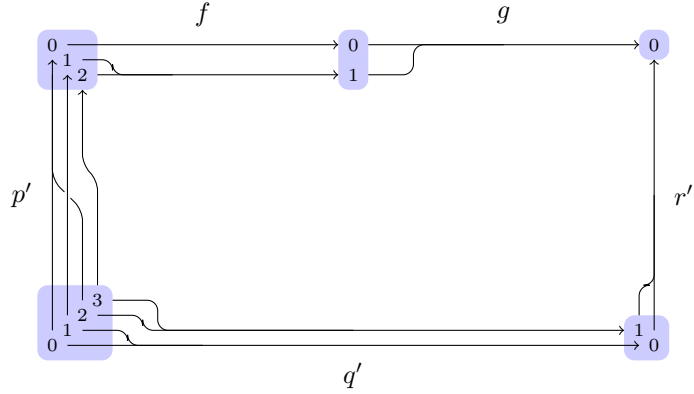
\begin{center}
 \begin{tikzpicture}[scale=2,rounded corners]
  \fill[blue!20] (-0.1,-0.1) rectangle (0.4,0.4);
  \fill[blue!20] (-0.1,1.7)  rectangle (0.3,2.1);
  \fill[blue!20] (1.9,1.7)   rectangle (2.1,2.1);
  \fill[blue!20] (3.9,1.9)   rectangle (4.1,2.1);
  \fill[blue!20] (3.8,-0.1)  rectangle (4.1,0.2);

  \draw (0.0,0.0) node{$\ss 0$};
  \draw (0.1,0.1) node{$\ss 1$};
  \draw (0.2,0.2) node{$\ss 2$};
  \draw (0.3,0.3) node{$\ss 3$};

  \draw (4.0,0.0) node{$\ss 0$};
  \draw (3.9,0.1) node{$\ss 1$};

  \draw (0.0,2.0) node{$\ss 0$};
  \draw (0.1,1.9) node{$\ss 1$};
  \draw (0.2,1.8) node{$\ss 2$};

  \draw (2.0,2.0) node{$\ss 0$};
  \draw (2.0,1.8) node{$\ss 1$};

  \draw (4.0,2.0) node{$\ss 0$};

  \draw[->] (0.0,0.1) -- (0.0,1.9);
  \draw[->] (0.1,0.2) -- (0.1,1.8);
  \draw (0.2,0.3) -- (0.2,0.9) -- (0.12,0.98);
  \draw (0.08,1.02) -- (0.0,1.1) -- (0.0,1.8);
  \draw[->] (0.3,0.4) -- (0.3,1.1) -- (0.2,1.2) -- (0.2,1.7);

  \draw[->] (0.1,0.0) -- (3.9,0.0);
  \draw (0.2,0.1) -- (0.5,0.1) -- (0.5,0.0) -- (1.0,0.0);
  \draw[->] (0.3,0.2) -- (0.6,0.2) -- (0.6,0.1) -- (3.8,0.1);
  \draw (0.4,0.3) -- (0.7,0.3) -- (0.7,0.1) -- (2.0,0.1);

  \draw[->] (0.1,2.0) -- (1.9,2.0);
  \draw (0.2,1.9) -- (0.4,1.9) -- (0.4,1.8) -- (0.8,1.8);
  \draw[->] (0.3,1.8) -- (1.9,1.8);

  \draw[->] (2.1,2.0) -- (3.9,2.0);
  \draw (2.1,1.8) -- (2.4,1.8) -- (2.4,2.0) -- (3.0,2.0);

  \draw[->] (4.0,0.1) -- (4.0,1.9);
  \draw (3.9,0.2) -- (3.9,0.4) -- (4.0,0.4) -- (4.0,1.0);

  \draw (1.0,2.2) node{$f$}; 
  \draw (3.0,2.2) node{$g$};

  \draw (-0.2,1.0) node{$p'$}; 
  \draw (2.0,-0.2) node{$q'$}; 
  \draw (4.2,1.0)  node{$r'$}; 

 \end{tikzpicture}
\end{center}
The operation $N_gT_f$ is $(a_0,a_1,a_2)\mapsto a_0(a_1+a_2)$, whereas
the operation $T_{r'}N_{q'}R_{p'}$ is 
$(a_0,a_1,a_2)\mapsto a_0a_1+a_0a_2$.  Suppose instead that we define
a bispan $\om=\Dl(f,g)=(3\xla{p}A\xra{q}B\xra{r}1)$ as in
Definition~\ref{defn-distributor}.  From the definitions we have 
\[ B = \{(0,s)\st s\:2\to 3,\quad fs=1_2\}. \]
If we define maps $s_0,s_1\:2\to 3$ by 
\begin{align*}
 s_0(0) &= 0 & s_0(1) &= 1 \\
 s_1(0) &= 0 & s_1(1) &= 2 
\end{align*}
then $B=\{(0,s_0),\;(0,s_1)\}$.  We then have 
\[ A = 2\tm_1B = \{(0,0,s_0),\;(1,0,s_0),\;(0,0,s_1),\;(1,0,s_1)\}. \]
It follows that we have a commutative diagram
\[ \xymatrix{
 3 \ar@{=}[d] &
 4 \ar[l]_{p'} \ar[r]^{q'} \ar[d]_\al^\simeq &
 2 \ar[r]^{q'} \ar[d]^\bt_\simeq &
 1 \ar@{=}[d] \\
 3 &
 A \ar[l]^p \ar[r]_q &
 B \ar[r]_r & 
 1
} \]
where 
\begin{align*}
 \al(0) &= (0,0,s_0) & 
 \al(1) &= (1,0,s_0) &
 \al(2) &= (0,0,s_1) &
 \al(3) &= (1,0,s_1) \\
 \bt(0) &= (0,s_0) & & 
 \bt(1) &= (0,s_1).
\end{align*}
This gives $N_{r'}T_{q'}R_{p'}=N_rT_qR_p=N_gT_f$ as expected.

We now discuss composition of more general bispans.

\begin{definition}\label{defn-bispan-circ}
 Consider a pair of bispans 
 \begin{align*}
  \om_0 &= (X_0\xla{p_0}A_0\xra{q_0}B_0\xra{r_0}X_1) \in \CU(X_0,X_1) \\
  \om_1 &= (X_1\xla{p_1}A_1\xra{q_1}B_1\xra{r_1}X_2) \in \CU(X_1,X_2).
 \end{align*}
 We define a bispan 
 \[ \om_1\circ\om_0 = (X_0\xla{p}A\xra{q}B\xra{r}X_2) \in \CU(X_0,X_2)
 \] 
 as follows:
 \begin{align*}
  A &= \{(a_0,a_1,s) \st 
         s\:q_1^{-1}\{q_1(a_1)\}\to B_0,\;r_0s=p_1,\;
         a_0\in q_0^{-1}\{s(a_1)\}\} \\
  B &= \{(b_1,s) \st 
          s\:q_1^{-1}\{b_1\}\to B_0,\;\;r_0s=p_1\} \\
  p(a_0,a_1,s) &= p_0(a_0) \\
  q(a_0,a_1,s) &= (q_1(a_1),s) \\
  r(b_1,s)     &= r_1(b_1).
 \end{align*}
 (Here it is implicit that $a_0\in A_0$, $a_1\in A_1$ and $b_1\in B_1$.)
\end{definition}

\begin{remark}\label{rem-circ-diagram}
 It is sometimes useful to consider the enlarged diagram
 
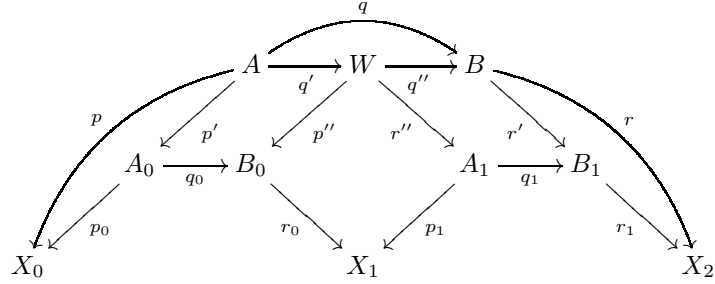
\[ \xymatrix{ & &
  A \ar[r]_{q'} \ar[dl]^{p'} \ar@/^4ex/[rr]^q \ar@/_4ex/[ddll]_p &
  W \ar[dl]^{p''} \ar[dr]_{r''} \ar[r]_{q''} &
  B \ar[dr]_{r'} \ar@/^4ex/[ddrr]^r \\ &
  A_0 \ar[dl]^{p_0} \ar[r]_{q_0} & 
  B_0 \ar[dr]_{r_0} & & 
  A_1 \ar[dl]^{p_1} \ar[r]_{q_1} &
  B_1 \ar[dr]_{r_1} \\
  X_0 & & & 
  X_1 & & &
  X_2  
 } \]
 where 
 \begin{align*}
  A &= \{(a_0,a_1,s) \st 
         s\:q_1^{-1}\{q_1(a_1)\}\to B_0,\;r_0s=p_1,\;
         a_0\in q_0^{-1}\{s(a_1)\}\} \\
  W &= \{(a_1,s) \st 
          s\:q_1^{-1}\{q(a_1)\}\to B_0,\;\;r_0s=p_1\} \\
  B &= \{(b_1,s) \st 
          s\:q_1^{-1}\{b_1\}\to B_0,\;\;r_0s=p_1\} 
 \end{align*}
 \begin{align*}
  p(a_0,a_1,s)  &= p_0(a_0) &
  p'(a_0,a_1,s) &= a_0 &
  p''(a_1,s)    &= s(a_1) \\
  q(a_0,a_1,s)  &= (q_1(a_1),s) &
  q'(a_0,a_1,s) &= (a_1,s) &
  q''(a_1,s)    &= (q_1(a_1),s) \\
  r(b_1,s)      &= r_1(b_1) &
  r'(b_1,s)     &= b_1 &
  r''(a_1,s)    &= a_1
 \end{align*}
 One can check that everything commutes and that the two rhombi are
 cartesian (but the middle square is not, in general).
\end{remark}

\begin{proposition}
 For any semiring $R$ we have
 $f_{\om_1\circ\om_0}=f_{\om_1}f_{\om_0}\:R^{X_0}\to R^{X_2}$.
\end{proposition}
\begin{proof}
 We first claim that in the expanded diagram we have 
 \[ N_{q_1}R_{p_1}T_{r_0}=T_{r'}N_{q''}R_{p''} \:
     R^{B_0} \to R^{B_1}.
 \]
 Indeed, for any map $m\:B_0\to R$ and and $b_1\in B_1$, the left hand
 side gives 
 \[ \prod_{a_1\in q_1^{-1}\{b_1\}}
     \sum_{b_0\in r_0^{-1}\{p_1(a_1)\}}  m(b_0), 
 \]
 whereas the right hand side gives
 \[ \sum_{\begin{array}{c}
       \scriptstyle s\:q_1^{-1}\{b_1\}\to B_0 \\
       \scriptstyle r_0s=p_1
     \end{array}}
     \quad \prod_{a_1\in q_1^{-1}\{b_1\}} m(s(a_1)).
 \]
 To give a map $s\:q_1^{-1}\{b_1\}\to B_0$ with $r_0s=p_1$ is the same
 as to pick out one summand in each of the factors on the left hand
 side.  This means that the right hand side is just what we get by
 expanding out the left hand side, as required.  We can now compose on
 the left with $T_{r_1}$ and on the right with $N_{q_0}R_{p_0}$ to get 
 \[ f_{\om_1}f_{\om_0} =
     T_{r_1}N_{q_1}R_{p_1}T_{r_0}N_{q_0}R_{p_0}=
     T_{r_1}T_{r'}N_{q''}R_{p''}N_{q_0}R_{p_0} \:
     R^{X_0} \to R^{X_2}.
 \]
 As the square $(q',p',p'',q_0)$ is cartesian, we also have
 $R_{p''}N_{q_0}=N_{q'}R_{p'}$.  Using this together with the
 functoriality of $T$, $N$ and $R$ we obtain
 \[ f_{\om_1}f_{\om_0} =
     T_{r_1}T_{r'}N_{q''}N_{q'}R_{p'}R_{p_0} =
     T_r N_q R_p = f_{\om_1\circ\om_0}
 \]
 as claimed.
\end{proof}

Now suppose we have a third bispan $\om_2$ from $X_2$ to $X_3$.  It
follows from the proposition that the bispans
$\om_2\circ(\om_1\circ\om_0)$ and $(\om_2\circ\om_1)\circ\om_0$ give
the same natural semiring operation, so they are at least unnaturally
isomorphic.  It is reasonable to expect that there should be a natural
isomorphism satisfying a pentagonal coherence identity.  In order to
prove this, we will introduce a single-step construction of multiple
compositions. In more detail, consider a family of bispans 
\[ \om_i = (X_i \xla{p_i} A_i \xra{q_i} B_i \xra{r_i} X_{i+1}) \]
for $n\leq i<m$.  We would like to build a composite bispan from $X_n$
to $X_m$.   

\begin{definition}
 First, for $n\leq i\leq j\leq m$ we put $A_i^j=\prod_{k=i}^{j-1}A_i$.
 In the case $i=j$ this should be interpreted as the one-point set:
 $A_i^i=\{1\}$.  There are evident projections $\pi_i\:A_i^j\to A_i$
 and $\pi_{>i}\:A_i^j\to A_{i+1}^j$.
\end{definition}

\begin{construction}\label{cons-bispan-circs}
 Now consider a system of subsets $S_i\sse A_i^m$ (for $n\leq i\leq m$)
 together with maps $s_i\:S_i\to B_{i-1}$ (for $n<i\leq m$).  We say
 that the above data form a \emph{term system} if
 \begin{itemize}
  \item[(a)] $S_m=\{1\}$
  \item[(b)] When $n\leq i<m$, the projection $\pi\:A_i^m\to A_{i+1}^m$ has
   $\pi(S_i)\sse S_{i+1}$.  Moreover, the square
   \[ \xymatrix{
    S_i \ar[d]_{\pi_i} \ar[r]^{\pi_{>i}} & S_{i+1} \ar[d]^{s_{i+1}}  \\
    A_i \ar[r]_{q_i} & B_i
   } \]
   is a pullback.
  \item[(c)] When $n<i<m$, the square
   \[ \xymatrix{
    B_{i-1} \ar[d]_{r_{i-1}} & 
    S_i \ar[l]_{s_i} \ar[d]^{\pi_i} \\
    X_i & 
    A_i \ar[l]^{p_i}
   } \]
   commutes.
 \end{itemize}
 A \emph{thread} for a term system as above is just an element of the
 set $S_n$.  We write $B_{nm}$ for the set of all term systems, and
 $A_{nm}$ for the set of pairs consisting of a term system and a
 choice of thread.  In other words, we have
 \[ A_{nm} = \{(\un{a},\un{S},\un{s})\st
    (\un{S},\un{s})\in B_{nm},\quad \un{a}\in S_n\}.
 \]
 We define maps
 \[ X_n \xla{p_{nm}} A_{nm} \xra{q_{nm}} B_{nm} \xra{r_{nm}} X_m \]
 as follows:
 \begin{align*}
  p_{nm}(\un{a},\un{S},\un{s}) &= p_n(\pi_n(\un{a})) \\
  q_{nm}(\un{a},\un{S},\un{s}) &= (\un{S},\un{s}) \\
  r_{nm}(\un{S},\un{s})        &= r_{m-1}(s_m(1)).
 \end{align*}
 This gives a bispan $\om_{nm}$ from $X_n$ to $X_m$.
\end{construction}

We now unpack this definition in the simplest cases. 

\begin{proposition}
 Suppose that $n=0$ and $m=1$, so our input data consists of a
 single bispan
 \[ \om_0 = (X_0 \xla{p_0} A_0 \xra{q_0} B_0 \xra{r_0} X_1). \]
 Then the resulting bispan $\om_{01}$ is naturally isomorphic to
 $\om_0$.
\end{proposition}
\begin{proof}
 In this context a term system consists of sets $S_0$ and $S_1$,
 together with a map $s_1\:S_1\to B_0$.  Axiom~(a) says that
 $S_1=\{1\}$, so $s_1$ just gives a point $b_0=s_1(1)\in B_0$.
 Axiom~(b) therefore says that $S_0=q_0^{-1}\{b_0\}\sse A_0$, and~(c)
 is vacuously satisfied.  We can therefore identify $B_{01}$ with
 $B_0$.  More precisely, we have a bijection $\bt\:B_0\to B_{01}$
 given by 
 \[ \bt(b_0) = (q_0^{-1}\{b_0\},\{1\};1\mapsto b_0). \]
 Similarly, there is a bijection $\al\:A_0\to A_{01}$ given by 
 \[ \al(a_0) = (a_0;q_0^{-1}\{q_0(a_0)\},\{1\};1\mapsto q_0(a_0)). \]
 It is clear from the definitions that the diagram
 \[ \xymatrix{
  X_0 \ar@{=}[d] & 
  A_0 \ar[l]_{p_0} \ar[r]^{q_0} \ar[d]_\al &
  B_0 \ar[r]^{r_0} \ar[d]^\bt &
  X_1 \ar@{=}[d] \\
  X_0 &
  A_{01} \ar[l]^{p_{01}} \ar[r]_{q_{01}} &
  B_{01} \ar[r]_{r_{01}} &
  X_1
 } \]
 commutes, so we have the claimed isomorphism of bispans.
\end{proof}

\begin{proposition}
 Suppose that $n=0$ and $m=2$, so our input data consists of bispans 
 \begin{align*}
  \om_0 &= (X_0\xla{p_0}A_0\xra{q_0}B_0\xra{r_0}X_1) \in \CU(X_0,X_1) \\
  \om_1 &= (X_1\xla{p_1}A_1\xra{q_1}B_1\xra{r_1}X_2) \in \CU(X_1,X_2).
 \end{align*}
 Then $\om_{02}$ is naturally isomorphic to $\om_1\circ\om_0$.
\end{proposition}
\begin{proof}
 We write $\om_1\circ\om_0$ as 
 \[ X_0 \xla{p} A \xra{q} B \xra{r} X_2 \]
 as in Definition~\ref{defn-bispan-circ}.

 In the construction of $\om_{02}$, a term system consists of sets
 $S_0\sse A_0\tm A_1$ and $S_1\sse A_1$ and $S_2=\{1\}$, together with
 maps $s_1\:S_1\to B_0$ and $s_2\:S_2\to B_1$.  These fit into a
 diagram
 \[ \xymatrix{
      S_0 \ar[dr] \ar[rrr] & & & 
      S_1 \ar[dl] \ar[dr] \ar[rrr] & & &
      S_2=\{1\} \ar[dl] \\ & 
      A_0 \ar[dl] \ar[r] & 
      B_0 \ar[dr] & & 
      A_1 \ar[dl] \ar[r] &
      B_1 \ar[dr] \\
      X_0 & & & X_1 & & & X_2
 } \]
 in which the middle square is commutative and the other two
 quadrilaterals are cartesian.  The map $s_2$ just gives a point
 $b_1\in B_2$ and the cartesian property forces
 $S_1=q_1^{-1}\{b_1\}\sse A_1$.  Now $s_1$ is a map 
 $q_1^{-1}\{b_1\}\to B_0$, and the map 
 $r_0s_1\:q_1^{-1}\{b_1\}\to X_1$ must be the same as the restriction
 of $p_1$ to make the middle square commute.  Once $s_1$ has been
 chosen, the cartesian property forces us to take
 \[ S_0 = \{(a_0,a_1)\in A_0\tm A_1\st 
     a_1\in q_1^{-1}\{b_1\},\; a_0\in q_0^{-1}\{s_1(a_1)\}\},
 \]
 so there is no additional information in $S_0$.  This identifies the
 set $B_{02}$ (of all term systems) with the set 
 \[ B = \{(b_1,s) \st s\:q_1^{-1}\{b_1\}\to B_0,\;\;r_0s=p_1\}. \]
 Next, a point of $A_{02}$ consists of a term system together with a
 thread $(a_0,a_1)$.  Here $b_1$ must be equal to $q_1(a_1)$, so we
 need not record it as a separate piece of data.  For $(a_0,a_1)$ to
 be a thread we must have $a_0\in q_0^{-1}\{s_1(a_1)\}$.  This
 identifies $A_{02}$ with the set
 \[ A = \{(a_0,a_1,s) \st 
         s\:q_1^{-1}\{q_1(a_1)\}\to B_0,\;r_0s=p_1,\;
         a_0\in q_0^{-1}\{s(a_1)\}\}.
 \]
 We leave it to the reader to check that these identifications are
 compatible with the maps $p$, $q$ and $r$, so they give an
 isomorphism of bispans.
\end{proof}

Now suppose we have bispans $\om_n,\dotsc,\om_{m-1}$ as before, and an
index $k$ with $n<k<m$.  By the procedure described above, we can
construct a bispan $\om_{nm}$ from $X_n$ to $X_m$.  Alternatively, we
can construct $\om_{nk}$ (from $X_n$ to $X_k$) and $\om_{km}$ (from
$X_k$ to $X_m$), and then construct $\om_{km}\circ\om_{nk}$.  

\begin{proposition}
 The bispan $\om_{km}\circ\om_{nk}$ is naturally isomorphic to
 $\om_{nm}$. 
\end{proposition}
\begin{proof}
 Let $A'_{nm}$ and $B'_{nm}$ denote the sets occurring in 
 $\om_{km}\circ\om_{nk}$.  A point of $B'_{nm}$ consists of a term
 system $\bt=(S_k,\dotsc,S_m;s_{k+1},\dotsc,s_m)$ together with a map
 $\sg\:q_{km}^{-1}\{\bt\}\to B_{nk}$ such that $r_{nk}\sg=p_{kn}$.
 Here $q_{km}^{-1}\{\bt\}$ is easily identified with the set $S_k$, and
 after this identification $p_{kn}$ becomes the map 
 $p_k\pi_k\:S_k\to X_k$.  Now for $a\in S_k$ we have a point
 $\sg(a)\in B_{nk}$.  This will have the form
 \[ \sg(a) = 
     (\Sg_n(a),\dotsc,\Sg_k(a);
       \sg_{n+1}(a),\dotsc,\sg_k(a))
 \]
 for some subsets $\Sg_i(a)\sse A_i^k$ and some maps
 $\sg_i(a)\:\Sg_i(a)\to B_{i-1}$.  Here $\Sg_k(a)=\{1\}$ so we can
 define $s_k\:S_k\to B_{k-1}$ by $s_k(a)=\sg_k(a)(1)$.  From the
 definitions we have
 $r_{nk}(\sg(a))=r_{k-1}(\sg_k(a)(1))=r_{k-1}(s_k(a))$, so the axiom
 $r_{nk}\sg=p_{kn}$ becomes $r_{k-1}s_k=p_k\pi_k\:S_k\to X_k$.  Next, for
 $n\leq i<k$ we can split $A_i^m$ as $A_i^k\tm A_k^m$ and put
 \[ S_i = \{(a',a'')\in A_i^k\tm A_k^m \st 
     a'' \in S_k \text{ and } a'\in \Sg_i(a'')\}.
 \]
 When $i>n$ we can also define $s_i\:S_i\to B_{i-1}$ by 
 \[ s_i(a',a'') = \sg_i(a'')(a'). \]
 We claim that the list $\al=(S_n,\dotsc,S_m;s_{n+1},\dotsc,s_m)$ is a
 term system for $(\om_n,\dotsc,\om_{m-1})$.  Indeed, axiom~(a) for
 $\al$ follows from axiom~(a) for $\bt$.  Similarly, axiom~(b) for
 $\al$ follows from axiom~(b) for $\bt$ when $i\geq k$, and from
 axiom~(b) for $\sg(a'')$ when $i<k$.  Axiom~(c) works the same way
 except that the case where $i=k$ must be treated separately.  That
 case says that the diagram 
 \[ \xymatrix{
    B_{k-1} \ar[d]_{r_{k-1}} & 
    S_k \ar[l]_{s_k} \ar[d]^{\pi_k} \\
    X_k & 
    A_k \ar[l]^{p_k}
 } \]
 must commute, and we observed above that this follows from our
 assumption that $r_{nk}\sg=p_{kn}$.  We thus have a point 
 $\al\in B_{nm}$ as claimed.  This was constructed in a natural way
 from $\bt$, so we can define a map $\phi\:B'_{nm}\to B_{nm}$ by 
 $\phi(\bt)=\al$.  We leave it to the reader to check that this is a
 bijection. 

 Now consider instead the set $A'_{nm}$.  The elements have the form
 $\al=(\al',\al'',\sg)$, where 
 \begin{itemize}
  \item $\al''$ is an element of $A_{km}$, so it consists of a term
   system $\bt=(S_k,\dotsc,S_m;s_{k+1},\dotsc,s_m)$ together with a
   thread $a''\in S_k$;
  \item $(\bt,\sg)$ is a point in $B'_{nm}$, corresponding to a point 
   \[ \phi(\bt,\sg)=(S_n,\dotsc,S_m;s_{n+1},\dotsc,s_m) \in B_{nm}; \]
  \item $\al'$ is an element of $A_{nk}$ with
   $q_{nk}(\al')=\sg(\al'')\in B_{nk}$.  This means that $\al'$
   consists of a term system $(T_n,\dotsc,T_k;t_{n+1},\dotsc,t_k)$
   together with a thread $a'\in T_n$.  The fact that
   $q_{nk}(\al')=\sg(\al'')$ means that 
   \[ T_i = \{u\in A_i^k\st (u,a'')\in S_i\} \]
   and $t_i(u)=s_i(u,a'')$.  In particular, the sets $T_i$ and maps
   $t_j$ are determined by the other data that we have mentioned
   already.  The thread $a'\in T_n$ can be combined with $a''$ to get
   a thread $a=(a',a'')\in S_n$ for $\phi(\bt,\sg)$.
 \end{itemize}
 We now see that the construction 
 \[ (\al',\al'',\sg)\mapsto(\phi(\bt,\sg),(a',a'')) \]
 gives a natural map $\psi\:A'_{nm}\to A_{nm}$.  We leave it to the reader
 to check that this is bijective and that the diagram 
 \[ \xymatrix{
  X_n \ar@{=}[d] &
  A'_{nm} \ar[l] \ar[r] \ar[d]_\psi &
  B'_{nm} \ar[r] \ar[d]^\phi &
  X_m \ar@{=}[d] \\
  X_n &
  A_{nm} \ar[l] \ar[r] &
  B_{nm} \ar[r] &
  X_m
 } \]
 commutes, so we have the required isomorphism of bispans.
\end{proof}

Now suppose we have composable bispans $\om_0$, $\om_1$ and $\om_2$.
We can use the proposition repeatedly to construct natural
isomorphisms 
\[ \om_2\circ(\om_1\circ\om_0) \simeq
    \om_2\circ\om_{02} \simeq \om_{03} \simeq
     \om_{13}\circ\om_0 \simeq (\om_2\circ\om_1)\circ\om_0.
\]
By combining these, we obtain an associativity isomorphism 
\[ \al \:
     \om_2\circ(\om_1\circ\om_0) \to
       (\om_2\circ\om_1)\circ\om_0.
\]
If we have another bispan $\om_3$ that can be composed with $\om_2$,
we get a pentagonal coherence diagram
\[ \xymatrix{
 \om_3\circ(\om_2\circ(\om_1\circ\om_0))
  \ar[r]^\al \ar[d]_{1\circ\al} & 
 (\om_3\circ\om_2)\circ(\om_1\circ\om_0)) \ar[r]^\al & 
 ((\om_3\circ\om_2)\circ\om_1)\circ\om_0)) \\
 \om_3\circ((\om_2\circ\om_1)\circ\om_0) \ar[rr]_\al & & 
 (\om_3\circ(\om_2\circ\om_1))\circ\om_0 \ar[u]_{\al\circ 1}
} \]
By comparing everything with $\om_{04}$, one can check that this
commutes.  Similarly, one can check that bispans of the form
\[ \eta_X = (X\xla{1}X\xra{1}X\xra{1}X) \]
act as identities for composition up to coherent natural isomorphism.  

All this can probably be repackaged to give a quasicategory of
bispans, which should be a quasicategorical Lawvere theory in the
sense of Cranch~\cite{cr:ati}.  However, we have not checked the
details.  

We have previously remarked that the morphism set $\bCU(X,Y)$ has a
canonical semiring structure.  It is natural to expect that this
should arise from a symmetric bimonoidal stucture on the category
$\CU(X,Y)$.  Such a structure can be given as follows: for bispans 
\begin{align*}
 \om      &= (X\xla{u}A\xra{v}B\xra{w}Y) \\
 \upsilon &= (X\xla{p}C\xra{q}D\xra{r}Y)
\end{align*}
we put 
\begin{align*}
 \om\oplus\upsilon &= 
  (X\xla{(u,p)}A\amalg C \xra{v\amalg q} B\amalg D \xra{(w,r)} Y) \\
 \om\ot\upsilon &= 
  (X\xla{}((A\tm_YD)\amalg(B\tm_YC))\xra{} B\tm_YD \xra{} Y).
\end{align*}
We will not make serious use of this so we leave further details to
the reader.

\begin{definition}\label{defn-TNR-bispans}
 For any map $f\:X\to Y$ of finite sets, we will write $T_f$, $N_f$
 and $R_f$ for the evident bispans representing the corresponding
 semiring operations, namely 
 \begin{align*}
  R_f &= (Y \xla{f} X \xra{1} X \xra{1} X) \in \CU(Y,X) \\
  N_f &= (X \xla{1} X \xra{f} Y \xra{1} Y) \in \CU(X,Y) \\
  T_f &= (X \xla{1} X \xra{1} X \xra{f} Y) \in \CU(X,Y).
 \end{align*}
\end{definition}

\begin{proposition}\label{prop-bispan-rels}
 \begin{itemize}
  \item[(a)] Any bispan $\om=(X\xla{p}A\xra{q}B\xra{r}Y)$ is
   naturally isomorphic to $T_r\circ N_q\circ R_p$.
  \item[(b)] For all $X\xla{f}Y\xra{g}Z$ there are natural isomorphisms
   $T_{gf}\simeq T_g\circ T_f$ and $N_{gf}\simeq N_g\circ N_f$ and
   $R_{gf}\simeq R_f\circ R_g$.  Moreover, $T_1$, $N_1$ and $R_1$ are
   all equal to the identity bispan
   $\eta_X=(X\xla{1}X\xra{1}X\xra{1}X)$. 
  \item[(c)] There is also a natural isomorphism from $N_g\circ T_f$
   to the distributor $\Dl(f,g)$.
  \item[(d)] For any cartesian square
   \[ \xymatrix{
    W \ar[r]^f \ar[d]_g & X \ar[d]^h \\
    Y \ar[r]_k & Z
   } \]
   there are natural isomorphisms $T_g\circ R_f\simeq R_k\circ T_h$
   and $N_g\circ R_f\simeq R_k\circ N_h$.
  \item[(e)] For any bijection $f\:X\to Y$ there are natural
   isomorphisms $T_f\simeq N_f\simeq R_{f^{-1}}$.
 \end{itemize}
\end{proposition}
\begin{proof}
 In each case we see by considering the induced semiring operations
 that there is at least an unnatural isomorphism between the relevant
 bispans, and we can produce a natural isomorphism by simply unwinding
 the definition of bispan composition.
\end{proof}

\begin{proposition}\label{prop-bispan-shift}
 For any finite set $U$, there is a product-preserving functor
 $P_U\:\bCU\to\bCU$ given on objects by $P_U(X)=U\tm X$, and on
 morphisms by  
 \[ P_U[X\xla{p}A\xra{q}B\xra{r}Y] = 
     [U\tm X\xla{1\tm p}U\tm A\xra{1\tm q}
      U\tm B\xra{1\tm r}U\tm Y].
 \]
\end{proposition}
\begin{proof}
 It is clear that this construction gives well-defined maps
 $\bCU(X,Y)\to\bCU(U\tm X,U\tm Y)$, and we just need to check that
 these are compatible with bispan composition.  This follows directly
 from the definitions, as $U$ is just carried through the various
 constructions and does not interact with the other data in any
 interesting way.  As products in $\bCU$ are given by disjoint unions,
 it is clear that they are preserved by $P_X$.
\end{proof}

We now examine two special cases of the associativity isomorphism for
bispan composition.  These will turn out to be useful later.

\begin{proposition}\label{prop-NNT}
 For any maps $W\xra{f}X\xra{g}Y\xra{h}Z$ there is a functorially
 associated diagram 
 \[ \xymatrix{
  \tA \ar[r]^\al \ar[d]_i &
  A \ar[r]^p \ar[d]^q &
  W \ar[r]^f &
  X \ar[d]^g \\
  \tB \ar[r]_\bt \ar[d]_j &
  B \ar[rr]_r & &
  Y \ar[d]^h \\
  \tC \ar[rrr]_k &&& 
  Z
 } \]
 such that the square and two rectangles are cartesian, and
 \begin{align*}
  \Dl(f,g)  &= (W \xla{p} A \xra{q} B \xra{r} Y) \\
  \Dl(f,hg) &= (W \xla{p\al} \tA \xra{ji} \tC \xra{k} Z) \\
  \Dl(r,h)  &= (B \xla{\bt} \tB \xra{j} \tC \xra{k} Z). 
 \end{align*}
\end{proposition}
This will help us to analyse the operation $N_hN_gT_f$.
\begin{proof}
 We put 
 \begin{align*}
  A &= \{(x,s)\st x\in X,\; s\:g^{-1}\{g(x)\}\to W,\; fs=1\} \\
  B &= \{(y,s)\st y\in X,\; s\:g^{-1}\{y\}   \to W,\; fs=1\} \\
  \tA &= \{(x,t)\st x\in X,\; t\:(hg)^{-1}\{hg(x)\}\to W,\; ft=1\} \\
  \tB &= \{(y,t)\st y\in Y,\; t\:(hg)^{-1}\{h(y)\} \to W,\; ft=1\} \\
  \tC &= \{(z,t)\st z\in Z,\; t\:(hg)^{-1}\{z\}    \to W,\; ft=1\} \\
  \al(x,t) &= (x,t|_{g^{-1}\{g(x)\}}) \\
  \bt(y,t) &= (y,t|_{g^{-1}\{y\}}) \\
  i(x,t) &= (g(x),t) \\
  j(y,t) &= (h(y),t) \\
  k(z,t) &= z \\
  p(x,s) &= s(x) \\
  q(x,s) &= (g(x),s) \\
  r(y,s) &= y.
 \end{align*}
 It is straightforward to check that the diagram commutes, that the
 three regions are cartesian, and that $\Dl(f,g)$ and $\Dl(f,hg)$ are
 as described.  The real point is to understand $\Dl(r,h)$.  By
 definition, this has the form 
 \[ B \xla{\bt^*} B^* \xra{j^*} C^* \xra{k^*} Z, \]
 where 
 \begin{align*}
  C^* &= \{(z,u)\st z\in Z,\; u\:h^{-1}\{z\}\to B,\; ru=1\} \\
  B^* &= \{(y,u)\st y\in Y,\; u\:h^{-1}\{h(y)\}\to B,\; ru=1\} \\
  \bt^*(y,u) &= u(y) \\
  j^*(y,u) &= (h(y),u) \\
  k^*(z,u) &= z.
 \end{align*}
 Consider a point $(z,u)\in C^*$.  For $y\in h^{-1}\{z\}$ we then have
 $u(y)\in B$ with $r(u(y))=y$.  This means that $u(y)$ has the form
 $(y,s_y)$, where $s_y\:g^{-1}\{y\}\to W$ with $fs_y=1$.  Now for any
 point $x\in (hg)^{-1}\{z\}$ we have $g(x)\in h^{-1}\{z\}$ and we can
 define $t(x)=s_{g(x)}(x)\in W$.  This defines a map
 $t\:(hg)^{-1}\{z\}\to W$ with $ft=1$, so $(z,t)\in\tC$.  This
 construction gives a map $\lm\:C^*\to\tC$, and it is not hard to
 check that this is bijective.  From this we also get a canonical
 bijection $\mu\:B^*=Y\tm_ZC^*\to Y\tm_Z\tC=\tB$.  One can now check
 that the diagram 
 \[ \xymatrix{
  B \ar@{=}[d] &
  B^* \ar[l]_{\bt^*} \ar[d]_\mu^\simeq \ar[r]^{j^*} &
  C^* \ar[d]^\lm_\simeq \ar[r]^{k^*} &
  Z \ar@{=}[d] \\
  B &
  \tB \ar[l]^\bt \ar[r]_j &
  \tC \ar[r]_k &
  Z
 } \]
 commutes, giving the claimed identification of $\Dl(r,h)$.
\end{proof}

\begin{proposition}\label{prop-NTT}
 For any maps $W\xra{f}X\xra{g}Y\xra{h}Z$ there is a functorially
 associated diagram 
 \[ \xymatrix{
  W \ar@{=}[d] &
  \tA \ar[l]_{\tp} \ar[d]^i \ar[r]^{\tq} &
  \tB \ar[r]^{\tr} \ar[dd]^k &
  Z \ar@{=}[dd] \\
  W \ar[d]_f &
  A^* \ar[l]_{p^*} \ar[d]^j \\
  X &
  A \ar[l]^p \ar[r]_q &
  B \ar[r]_r &
  Z
 } \]
 such that the bottom left square and the middle rectangle are cartesian, and 
 \begin{align*}
  \Dl(g,h)  &= (X \xla{p} A \xra{q} B \xra{r} Z) \\
  \Dl(gf,h) &= (W \xla{\tp} \tA \xra{\tq} \tB \xra{\tr} Z) \\
  \Dl(j,q)  &= (A^* \xla{i} \tA \xra{\tq} \tB \xra{k} B).
 \end{align*}
\end{proposition}
This will help us to analyse the operation $N_hT_gT_f$.
\begin{proof}
 We put 
 \begin{align*}
  A &= \{(y,s)\st y\in Y,\; s\:h^{-1}\{h(y)\}\to X,\; gs=1\} \\
  B &= \{(z,s)\st z\in Z,\; s\:h^{-1}\{z\}   \to X,\; gs=1\} \\
  \tA &= \{(y,t)\st y\in Y,\; t\:h^{-1}\{h(y)\}\to W,\; fgt=1\} \\
  \tB &= \{(z,t)\st z\in Z,\; t\:h^{-1}\{z\}   \to W,\; fgt=1\} \\
  A^* &= \{(w,y,s)\st w\in W,\;y\in Y,\;
              s\:h^{-1}\{h(y)\}\to X,\; gs=1,\; f(w)=s(y)\} \\
  p(y,s) &= s(y) \\
  q(y,s) &= (h(y),s) \\
  r(y,s) &= y \\
  \tp(y,t) &= t(y) \\
  \tq(y,t) &= (h(y),t) \\
  \tr(z,t) &= z \\
  i(y,t) &= (t(y),y,gt) \\
  j(w,y,s) &= (y,s) \\
  k(z,t) &= (z,gt) \\
  p^*(w,y,s) &= w.
 \end{align*}
 It is straightforward to check that the diagram commutes, the bottom
 left square is cartesian, and that $\Dl(g,h)$ and $\Dl(gf,h)$ are as
 described.  The real point is to understand $\Dl(j,q)$.  By
 definition this has the form 
 \[ A^* \xla{i'} A' \xra{q'} B' \xra{k'} B, \]
 where 
 \begin{align*}
  B' &= \{ (b,u) \st b\in B,\; u\:q^{-1}\{b\}\to A^*,\; ju=1\} \\
  A' &= A\tm_BB'.
 \end{align*}
 Consider a point $(b,u)\in B'$.  The point $b\in B$ has the form
 $(z,s)$ for some $z\in Z$ and $s\:h^{-1}\{z\}\to X$ with $gs=1$.  The
 domain of $u$ consists of the points $(y,s)$ for $y\in h^{-1}\{z\}$,
 and the constraint $ju=1$ means that $u(y,s)=(t(y),y,s)$ for some
 $t(y)\in W$.  Moreover, we have $u(y,s)\in A^*$, and by inspecting
 the definition of $A^*$ this gives $ft(y)=s(y)$.  We thus have a map
 $t\:h^{-1}\{z\}\to W$ with $ft=s$ and so $gft=gs=1$.  The
 construction $(b,u)\mapsto (z,t)$ now gives a map $B'\to\tB$, which
 is easily seen to be bijective.  From this we also get a canonical
 bijection $A'=A\tm_BB'\to A\tm_B\tB=\tA$.  We leave it to the reader
 to check that these bijections are compatible with the maps in the
 distributor diagram, giving the claimed description of $\Dl(j,g)$.
\end{proof}

\begin{proposition}\label{prop-NRT}
 Suppose we have a diagram as follows in which both squares are pullbacks:
 \[ \xymatrix{
  \tW \ar[d]_i \ar[r]^{\tf} & X \ar[d]^j \ar[r]^{\tg} & Y \ar[d]^k \\
  W \ar[r]_f & X \ar[r]_g & Y.
 } \]
 Then there is a commutative diagram 
 \[ \xymatrix{
  \tW \ar[d]_i &
  \tA \ar[l]_{\tp} \ar[r]^{\tq} \ar[d]_\al &
  \tB \ar[r]^{\tr} \ar[d]^\bt &
  \tY \ar[d]^k \\
  W &
  A \ar[l]^p \ar[r]_q &
  B \ar[r]_r &
  Y
 } \]
 in which the top row is $\Dl(\tf,\tg)$, the bottom row is $\Dl(f,g)$,
 and the middle and right squares are cartesian.
\end{proposition}
\begin{proof}
 We define the two rows by the standard distributor construction as in
 Definition~\ref{defn-distributor}.  Consider a point
 $\tb=(\ty,\ts)\in\tB$.  This means that $\ts\:\tg^{-1}\{\ty\}\to\tW$
 with $\tf\ts=1$.  As the square $(\tX,\tY,X,Y)$ is cartesian, we see
 that $j$ restricts to give a bijection
 $j_{\ty}\:\tg^{-1}\{\ty\}\to g^{-1}\{k(\ty)\}$.  We can thus define
 $s\:g^{-1}\{k(y)\}\to W$ by $s=i\ts j_{\ty}^{-1}$.  This satisfies
 $fs=fi\ts j_{\ty}^{-1}=j\tf\ts j_{\ty}^{-1}=jj_{\ty}^{-1}=1$, so
 $(k(y),s)\in B$.  We can thus define $\bt\:\tB\to B$ by
 $\bt(\ty,\ts)=(k(y),s)=(k(y),isj_{\ty}^{-1})$.  Similarly, we define
 $\al\:\tA\to A$ by $\al(\tx,\ts)=(j(\tx),isj_{\tg(\tx)}^{-1})$.  It
 is easy to see that this gives a commutative diagram as claimed.

 We now show that the right hand square is cartesian.  Suppose we have
 a point $\ty\in\tY$ and a point $b=(y,s)\in B$ with $r(b)=k(\ty)$.
 By the definition of $r$ we have $r(b)=y$, so the condition is that
 $k(\ty)=y$, so $j_{\ty}$ gives a bijection
 $\tg^{-1}\{\ty\}\to g^{-1}\{y\}$.  Similarly, $i$ restricts to give a
 bijection $i_{\ty}\:(\tg\tf)^{-1}\{\ty\}\to(gf)^{-1}\{y\}$.  Next, as
 $(y,s)\in B$ we must have $fs=1\:g^{-1}\{y\}\to X$, which implies
 that $s$ lands in $(gf)^{-1}\{y\}$.  We can thus define
 $\ts=i_{\ty}^{-1}sj_{\ty}\:\tg^{-1}\{\ty\}\to\tW$.  We find that
 $\tf\ts=1$, so we have an element $\tb=(\ty,\ts)\in\tB$.  This is
 easily seen to be the unique element with $\tr(\tb)=\ty$ and
 $\bt(\tb)=b$, which gives the required pullback property.  A similar
 argument shows that the middle square is also cartesian.
\end{proof}

\section{Tambara functors}

Let $\CU_G$ be the bicategory that is the natural $G$-equivariant
analogue of $\CU$.  Explicitly:
\begin{itemize}
 \item The $0$-cells are finite $G$-sets.
 \item The $1$-cells from $X$ to $Y$ are diagrams
  $(X\xla{}A\xra{}B\xra{}Y)$ of finite $G$-sets.
 \item The $2$-cells from $(X\xla{}A\xra{}B\xra{}Y)$ to
  $(X\xla{}A'\xra{}B'\xra{}Y)$ consist of pairs $(\al,\bt)$, where
  $\al\:A\to A'$ and $\bt\:B\to B'$ are isomorphisms of finite
  $G$-sets making the evident diagram commute.
\end{itemize}
Composition of $2$-cells is the obvious thing, and composition of
$1$-cells is performed by calculating the composite in $\CU$ and
giving it the evident $G$-action.  In more detail, suppose we have
equivariant bispans
\begin{align*}
 \om_0 &= (X_0\xla{p_0}A_0\xra{q_0}B_0\xra{r_0}X_1) \in \CU(X_0,X_1) \\
 \om_1 &= (X_1\xla{p_1}A_1\xra{q_1}B_1\xra{r_1}X_2) \in \CU(X_1,X_2).
\end{align*}
As before, we define a bispan 
\[ \om_1\circ\om_0 = (X_0\xla{p}A\xra{q}B\xra{r}X_2) \in \CU(X_0,X_2)
\] 
as follows:
\begin{align*}
  A &= \{(a_0,a_1,s) \st 
         s\:q_1^{-1}\{q_1(a_1)\}\to B_0,\;r_0s=p_1,\;
         a_0\in q_0^{-1}\{s(a_1)\}\} \\
  B &= \{(b_1,s) \st 
          s\:q_1^{-1}\{b_1\}\to B_0,\;\;r_0s=p_1\} \\
  p(a_0,a_1,s) &= p_0(a_0) \\
  q(a_0,a_1,s) &= (q_1(a_1),s) \\
  r(b_1,s)     &= r_1(b_1).
\end{align*}
Note that in the definition of $A$, the map $s$ is not required to
have any kind of equivariance.  We let $G$ act on $A$ by the rule
\[ g.(a_0,a_1,s) = (ga_0,ga_1,gsg^{-1}), \]
where $gs$ denotes the composite 
\[ q_1^{-1}\{q_1(ga_1)\} \xra{g^{-1}} 
    q_1^{-1}\{q_1(a_1)\} \xra{s} B_0 \xra{g} B_0. 
\]
We also let $G$ act on $B$ by a similar rule, and it is easy to check
that the maps $p$, $q$ and $r$ respect these actions, so
$\om_1\circ\om_0$ is an equivariant bispan.

If $f\:X\to Y$ is a map of finite $G$-sets, then the bispans $T_f$,
$N_f$ and $R_f$ all have natural $G$-actions, so they can be regarded
as $1$-cells in $\CU_G$.  As the isomorphisms in
Proposition~\ref{prop-bispan-rels} are natural, they are automatically
equivariant.  For ease of reference elsewhere, we record this
formally:
\begin{proposition}\label{prop-G-bispan-rels}
 \begin{itemize}
  \item[(a)] Any $G$-equivariant bispan
   $\om=(X\xla{p}A\xra{q}B\xra{r}Y)$ is 
   naturally isomorphic to $T_r\circ N_q\circ R_p$.
  \item[(b)] For all finite $G$-sets $X$, $Y$ and $Z$, and all
   $G$-maps $X\xla{f}Y\xra{g}Z$, there are natural equivariant isomorphisms
   $T_{gf}\simeq T_g\circ T_f$ and $N_{gf}\simeq N_g\circ N_f$ and
   $R_{gf}\simeq R_f\circ R_g$.  Moreover, $T_1$, $N_1$ and $R_1$ are
   all equal to the identity bispan
   $\eta_X=(X\xla{1}X\xra{1}X\xra{1}X)$. 
  \item[(c)] There is also a natural equivariant isomorphism from
   $N_g\circ T_f$ to the distributor $\Dl(f,g)$.
  \item[(d)] For any cartesian square
   \[ \xymatrix{
    W \ar[r]^f \ar[d]_g & X \ar[d]^h \\
    Y \ar[r]_k & Z
   } \]
   (of finite $G$-sets and equivariant maps) there are natural
   equivariant isomorphisms $T_g\circ R_f\simeq R_k\circ T_h$
   and $N_g\circ R_f=R_k\circ N_h$.
  \item[(e)] For any equivariant bijection $f\:X\to Y$ there are natural
   equivariant isomorphisms $T_f\simeq N_f\simeq R_{f^{-1}}$.
 \end{itemize}
\end{proposition}

Now let $\bCU_G$ be the category whose objects are finite $G$-sets,
and whose morphisms from $X$ to $Y$ are isomorphism classes of
$1$-cells in $\CU_G(X,Y)$.  One checks that $X\amalg Y$ is a
categorical product of $X$ and $Y$ in $\bCU_G$.

The key definition, taken from~\cite{ta:mt}, is as follows.
\begin{definition}\label{defn-tambara-functor}
 A \emph{Tambara functor} for $G$ is a product-preserving functor from
 $\bCU_G$ to the category of sets.  We write $\Tambara_G$ for the
 category of Tambara functors.
\end{definition}

\begin{example}\label{eg-MapGXA-tambara}
 Let $R$ be a semiring with an action of $G$, and put
 $cR(X)=\Map_G(X,R)$ as before.  For any bispan
 $\om=(X\xla{p}A\xra{q}B\xra{r}Y)$ we define $f_\om\:cR(X)\to cR(Y)$ by 
 \[ f_\om(s)(y) = \sum_{r(b)=y}\prod_{q(a)=b}s(p(a)). \]
 This makes $cR$ into a Tambara functor.  Using the same unit and
 counit maps as in Proposition~\ref{prop-om-c}, we see that
 $c\:\Semirings_G\to\Tambara_G$ is right adjoint to the functor
 $\om\:\Tambara_G\to\Semirings_G$ given by $\om(S)=S(G)$.
\end{example}
\begin{example}\label{eg-burnside-tambara}
 Let $A(X)$ denote the set of isomorphism classes of finite $G$-sets
 over $X$, as in Example~\ref{eg-HU}.  We claim that this is the same
 as $\bCU_G(\emptyset,X)$.  Indeed, in any bispan
 $(\emptyset\xla{i}P\xra{j}Q\xra{k}X)$ the set $P$ admits a map to
 $\emptyset$, which is only possible if $P=\emptyset$, in which case
 there is only one possibility for $j$, so everything is determined by
 $(Q\xra{k}X)$.  The claim is clear from this.  This means that $A$
 can be regarded as a functor, represented by $\emptyset$.  One can
 check that the $T$, $N$ and $R$ operations can be described as
 follows.
 \begin{itemize}
  \item[(a)] For $f\:X\to Y$ and $[Q\xra{k}X]\in A(X)$ we have
   \[ T_f[Q\xra{k}X] = [Q\xra{fk}Y]. \]
   Next, if we put 
   \[ S=\{(y,s)\st y\in Y,\;s\:f^{-1}\{y\}\to Q,\;ks=1\} \]
   and define $m\:S\to Y$ by $m(y,s)=y$, then
   \[ N_f[Q\xra{k}X] = [S\xra{m}Y]. \]
  \item[(b)] For $e\:W\to X$ we have
   \[ R_e[Q\xra{k}X]=[e^*Q\xra{\ov{k}}W], \]
   where $e^*Q=\{(w,q)\in W\tm Q\st e(w)=k(q)\}$ and $\ov{k}$ is the
   obvious projection.
 \end{itemize}
 In slightly different notation, if we write $Q_x$ for the fibre
 $k^{-1}\{x\}$ and so on, we get
 \begin{align*}
  T_f[Q\xra{k}X]_y &= \coprod_{f(x)=y} Q_x \\
  N_f[Q\xra{k}X]_y &= \prod_{f(x)=y} Q_x \\
  R_e[Q\xra{k}X]_w &= Q_{e(w)}.
 \end{align*}
 This is clearly analogous to Definition~\ref{defn-TNR}.  We also
 observe that $R_e$ comes from an evident functor
 \[ e^* \: \{ \text{ finite $G$-sets over $X$} \} \to 
             \{ \text{ finite $G$-sets over $W$} \},
 \]
 and that $T_e$ and $N_e$ come from the left and right adjoints to
 this functor.
\end{example}
\begin{example}\label{eg-representation-tambara}
 Let $\CR(X)$ denote the category of $G$-equivariant complex
 vector bundles over $X$, and let $R(X)$ denote the set of
 isomorphism classes in $\CR(X)$.  (As $X$ is just a finite
 $G$-set, the study of these vector bundles involves no topology or
 analysis, just combinatorics and representation theory.)  For any
 $G$-map $f\:X\to Y$ we have functors
 $T_f,N_f\:\CR(X)\to\CR(Y)$ and $R_f\:\CR(Y)\to\CR(X)$
 given by
 \begin{align*}
  T_f(V)_y &= \bigoplus_{f(x)=y} V_x \\
  N_f(V)_y &= \bigotimes_{f(x)=y} V_x \\
  R_f(W)_x &= W_{f(x)}.
 \end{align*}
 One can check that this gives a Tambara functor.

 We now want to make some remarks about comparison between different
 groups, so we will write $\CR_G(X)$ rather than $\CR(X)$.  
 It is useful to note that $\CR_G(G/H)$ is equivalent to
 the category $\CR_H$ of representations of $H$.  Indeed, if $V$ is
 an equivariant vector bundle over $G/H$ then the fibre $V_H$ at the
 identity coset has an action of $H$.  On the other hand, if $W$ is a
 representation of $H$ then the set $G\tm_HW$ is the total space of an
 equivariant vector bundle over $G/H$.  These constructions are easily
 seen to be inverse to each other.  We deduce that $R_G(G/H)$ can be
 identified with the representation semiring $R_H$.

 If $K\leq H\leq G$ then we have an obvious map $f\:G/K\to G/H$, which
 gives maps
 \begin{align*}
  T_f &\: R_K \to R_H \\ 
  N_f &\: R_K \to R_H \\ 
  R_f &\: R_H \to R_K.
 \end{align*}
 These are usually called \emph{induction}, \emph{tensor induction}
 and \emph{restriction}, respectively.
\end{example}
\begin{example}\label{eg-completion}
 In Section~\ref{sec-completion} we will show
 (following~\cite{ta:mt}*{Section 6}) that we can adjoin
 additive inverses to any semiring-valued Tambara functor, giving a
 ring-valued Tambara functor.  We call this process \emph{additive
  completion}. 
\end{example}
\begin{example}\label{eg-pi-zero-tambara}
 Let $E$ be a $G$-equivariant spectrum in the sense of stable homotopy
 theory.  We remarked previously that the assignment
 $\pi^G_0(E)(X)=[\Sgi_G(X_+),E]^G$ gives a Mackey functor.  We will
 show in Section~\ref{sec-spectrum} that when $E$ has a strictly
 commutative product structure, then $\pi^G_0(E)$ is actually a
 Tambara functor.  (This was also proved in~\cite{br:wve}, but we will
 explain an alternative approach that avoids many homotopical
 technicalities.)  In this context we always have additive inverses,
 so $\pi^G_0(E)(X)$ is a ring rather than just a semiring.  For the
 sphere spectrum it works out that $\pi^G_0(S)$ is the additive
 completion of the Burnside semiring Tambara functor as in
 Example~\ref{eg-burnside-tambara}.  Similarly, for the complex
 $K$-theory spectrum $KU$ it can be shown that $\pi^G_0(KU)$ is the
 additive completion of the representation semiring Tambara functor as
 in Example~\ref{eg-representation-tambara}.  Moreover, for any
 commutative ring $A$ with an action of $G$ one can define an
 equivariant Eilenberg-MacLane spectrum $A$ with a strictly
 commutative product such that $\pi^G_0(HA)(X)=\Map_G(X,A)$ as in
 Example~\ref{eg-MapGXA-tambara}.  However, we will prove any of these
 identifications in the present memoir.
\end{example}

We now prove a small lemma which will be useful later, and which is a
good exercise in handling the definitions.
\begin{lemma}\label{lem-norm-zero}
 Consider a map $g\:X\to Y$ of finite $G$-sets, and note that this
 gives an equivariant splitting $Y=g(X)\amalg g(X)^c$ and thus
 $S(Y)=S(g(X))\tm S(g(X)^c)$ for any Tambara functor $S$.  With
 respect to this decomposition we have $N_g(0)=(0,1)$.
\end{lemma}
\begin{proof}
 Let $f$ be the map $\emptyset\to X$.  Recall that $S(\emptyset)$
 consists of a single element which we can call $a$, and by definition
 the element $0\in S(X)$ is $T_f(a)$, so $N_g(0)=N_gT_f(a)$.  The
 distributor $\Dl(f,g)$ is easily identified as the diagram
 $(\emptyset\xla{1}\emptyset\xra{q}f(X)^c\xra{r}Y)$, where $r$ is just
 the inclusion.  This gives $N_g(0)=T_rN_qR_1(a)=T_rN_q(a)=T_r(1)$,
 and $T_r$ is just the inclusion of the summand $S(f(X)^c)$ in $S(Y)$
 so the claim follows.
\end{proof}

\begin{proposition}\label{prop-UG-pres}
 The category $\bCU_G$ is generated by the morphisms $T_f$, $N_f$
 and $R_f$, subject only to the following relations:
 \begin{itemize}
  \item[(a)] For any maps $X\xla{f}Y\xra{g}Z$ we have $T_{gf}=T_gT_f$
   and $N_{gf}=N_gN_f$ and $R_{gf}=R_fR_g$.  Moreover, $T_1$, $N_1$ and
   $R_1$ are all equal to the identity.
  \item[(b)] If $\Dl(f,g)=(X\xla{p}A\xra{q}B\xra{r}Z)$ then
   $N_gT_f=T_rN_qR_p$. 
  \item[(c)] For any cartesian square
   \[ \xymatrix{
    W \ar[r]^f \ar[d]_g & X \ar[d]^h \\
    Y \ar[r]_k & Z
   } \]
   we have $T_gR_f=R_kT_h$ and $N_gR_f=R_kN_h$.
 \end{itemize}
\end{proposition}
\begin{proof}
 Let $\CU'_G$ be the category with the indicated generators and
 relations.  In more detail, the objects are the finite $G$-sets, and
 the morphisms are the composable strings of $T$'s, $N$'s and $R$'s,
 modulo the smallest equivalence relation necessary to ensure that the
 given relations are satisfied and composition remains associative and
 unital.  Proposition~\ref{prop-G-bispan-rels} tells us that the
 stated relations do in fact hold in $\bCU_G$, so we have a functor
 $\pi\:\CU'_G\to\bCU_G$ that is the identity on objects.  Every
 morphism in $\bCU_G(X,Y)$ can be written as $T_rN_qR_p$ for some $p$,
 $q$ and $r$, so $\pi$ is full.  The real issue is to prove that $\pi$
 is faithful.

 We first claim that every morphism in $\CU'_G$ can be expressed as
 $T_rN_qR_p$ for some $p$, $q$ and $r$.  It will clearly suffice to
 show that the class of morphisms of this form is closed under
 composition.  The argument can be presented schematically as follows:
 \begin{align*}
  (TNR)(TNR) &= TN(RT)NR \overset{(c)}{=} TN(TR)NR \\
             &= TNT(RN)R \overset{(c)}{=} TNT(NR)R \\
             &= T(NT)NRR \overset{(b)}{=} T(TNR)NRR \\
             &= TTN(RN)RR \overset{(c)}{=} TTN(NR)RR \\
             &= (TT)(NN)(RRR) \overset{(a)}{=} TNR.
 \end{align*}
 On the first line, we have a subword of the form $R_kT_h$, and
 relation~(c) allows us to rewrite this in the form $T_gR_f$ for some
 other maps $f$ and $g$.  The other lines can be interpreted in a
 similar way.

 Now suppose we have two morphisms $\xi,\xi'\in\CU'_G(X,Y)$ with
 $\pi(\xi)=\pi(\xi')$.  By the previous paragraph, we can choose
 bispans
 \begin{align*}
  \om  &= (X\xla{p}A\xra{q}B\xra{r}Y) \\
  \om' &= (X\xla{p'}A'\xra{q'}B'\xra{r'}Y)
 \end{align*}
 such that $\xi=T_rN_qR_p$ and $\xi'=T_{r'}N_{q'}R_{p'}$.  Now
 $\pi(\xi)=[\om]$ and $\pi(\xi')=[\om']$ so the assumption
 $\pi(\xi)=\pi(\xi')$ means that $\om$ and $\om'$ are isomorphic.
 Thus, there are equivariant bijections $\al\:A\to A'$ and
 $\bt\:B\to B'$ such that the diagram
 \[ \xymatrix{
      X \ar@{=}[d] & 
      A \ar[l] \ar[r] \ar[d]_{\al} & 
      B \ar[r] \ar[d]^{\bt} & 
      Y \ar@{=}[d] \\
      X & 
      A' \ar[l] \ar[r] &
      B' \ar[r] &
      Y
    }
 \]
 commutes.  We can now apply~(c) to the cartesian squares 
 \[ \xymatrix{
  A \ar[r]^\al \ar[d]_\al & 
  A' \ar[d]^1 &
  B \ar[r]^\bt \ar[d]_\bt &
  B' \ar[d]^1 \\
  A' \ar[r]_1 &
  A' &
  B' \ar[r]_1 &
  B'
 } \]
 to see that $T_\al=N_\al=R_\al^{-1}$ and $T_\bt=N_\bt=R_\bt^{-1}$.  
 It follows easily from this that
 $T_rN_qR_p=T_{r'}N_{q'}R_{p'}$ in $\CU'_G$, so $\xi=\xi'$ as
 required. 
\end{proof}

\begin{proposition}\label{prop-G-bispan-shift}
 For any finite set $G$-set $U$, there is a functor
 $P_U\:\bCU_G\to\bCU_G$ given on objects by $P_U(X)=U\tm X$, and on
 morphisms by
 \[ P_U[X\xla{p}A\xra{q}B\xra{r}Y] = 
     [U\tm X\xla{1\tm p}U\tm A\xra{1\tm q}
      U\tm B\xra{1\tm r}U\tm Y].
 \]
\end{proposition}
\begin{proof}
 Let $G$ act in the obvious way on all ingredients in
 Proposition~\ref{prop-bispan-shift}. 
\end{proof}

\begin{proposition}\label{prop-tambara-semiring}
 Let $S$ be a Tambara functor.  Then each set $S(X)$ has a canonical
 structure as a semiring.  Moreover, for every map $f\:X\to Y$:
 \begin{itemize}
  \item[(a)] The map $R_f\:S(Y)\to S(X)$ is a semiring homomorphism
   (which we use to regard $S(X)$ as an $S(Y)$-module).
  \item[(b)] The map $T_f\:S(X)\to S(Y)$ is a homomorphism of
   $S(Y)$-modules (and in particular, respects addition).
  \item[(c)] The map $N_f\:S(X)\to S(Y)$ sends $1$ to $1$ and respects
   multiplication. 
 \end{itemize}
\end{proposition}
\begin{proof}
 We can identify $\bCU$ with the (non-full) subcategory of $\bCU_G$
 where $G$ acts trivially on everything, and this subcategory is
 closed under products.  We thus have a product-preserving functor
 \[ \bCU \xra{\text{inc}} \bCU_G \xra{P_X} \bCU_G \xra{S} \text{Sets}.
 \]
 As in Proposition~\ref{prop-semiring-theory}, the image of $1$ under
 this functor is a semiring; but that image is just $S(X)$.  By
 unwinding the definitions a little, we see that addition and
 multiplication are given by $a+b=T_s(a,b)$, and $ab=N_s(a,b)$, where
 $s\:X\amalg X\to X$ is given by the identity on both copies of $X$.
 Similarly, the identity elements are $0=T_z()$ and $1=N_z()$, where
 $z\:\emptyset\to X$ is the inclusion.  

 Now suppose we have a map $f\:X\to Y$.  This gives a commutative
 diagram 
 \[ \xymatrix{
  \emptyset \ar[r]^z \ar[d]_1 & 
  X \ar[d]^f &
  X \amalg X \ar[l]_s \ar[d]^{f\amalg f} \\
  \emptyset \ar[r]_z & 
  Y &
  Y\amalg Y \ar[l]^s
 } \]
 in which both squares are cartesian.  By applying $T$ to the diagram
 we see that $T_f$ preserves addition (and zero).  By applying $N$ to
 the diagram we see that $N_f$ preserves multiplication (and one).
 Using the cartesian property together with part~(d) of
 Proposition~\ref{prop-G-bispan-rels} we see that $R_f$ is a semiring
 homomorphism.  All that is left is to check that $T_f$ is a morphism
 of $S(Y)$-modules, or more explicitly that $T_f(u\,R_f(v))=T_f(u)v$
 for all $u\in S(X)$ and $v\in S(Y)$.  Consider the maps
 \[ \xymatrix{
  X\amalg Y \ar[r]^{f\amalg 1} &
  Y\amalg Y \ar[r]^s & 
  Y \\
  X\amalg X \ar[u]_{1\amalg f} \ar[rr]_s &&
  X \ar[u]_f.
 } \] 
 The operation $(u,v)\mapsto T_f(u)v$ is $N_sT_{f\amalg 1}$, whereas
 the operation $(u,v)\mapsto T_f(u\,R_f(v))$ is
 $T_fN_sR_{1\amalg f}$.  It will thus suffice to check that the bispan 
 \[ \om'=(X\amalg Y \xla{1\amalg f} X\amalg X \xra{s} X \xra{f} Y) \]
 is isomorphic to the distributor
 \[ \Dl(f\amalg 1,s)=(X\amalg Y \xla{p} A \xra{q} B \xra{r} Y) \]
 constructed in Definition~\ref{defn-distributor}.
 To analyse this, we write $y_L$ for the copy of $y$ in the left
 factor of $Y\amalg Y$, and $y_R$ for the copy in the right factor, so
 $s^{-1}\{y\}=\{y_L,y_R\}$.  This means that $B$ is the set of pairs
 $(y,t)$, where $y\in Y$ and $t\:\{y_L,y_R\}\to X\amalg Y$ with 
 $(f\amalg 1)\circ t=1$.  This means that $t(y_R)=y_R$ and
 $t(y_L)=x_L$ for some $x\in X$ with $f(x)=y$.  We see that everything
 is determined by $x$, so $B$ can be identified with $X$.  Next, $A$
 is the set of triples $(w,y,t)$, where $(y,t)$ is as before and
 $w\in\{y_L,y_R\}$.  The triples with $w=y_L$ form one copy of $X$,
 and the triples with $w=y_R$ form another copy of $X$, so we can
 identify $A$ with $X\amalg X$.  This means that $\om$ is isomorphic
 to some bispan of the form 
 \[ (X\amalg Y \xla{} X\amalg X \xra{} X \xra{} Y); \]
 we leave it to the reader to check that the maps are $1\amalg f$, $s$
 and $f$.
\end{proof}

\section{The group of order two}
\label{sec-two}

We now analyse how the theory works out for a group $G=\{1,\chi\}$ of
order two, building on the corresponding result for Mackey functors in
Theorem~\ref{thm-MP}.

\begin{definition}
 A \emph{Tambara pair} consists of rings $A$ and $B$, together
 with a right action of $G$ on $A$ by semiring maps, a semiring map
 $\res\:B\to A^G$, and functions $\trc,\nrm\:A\to B$ satisfying 
 \begin{align*}
  \trc(0)          &= 0  & 
  \nrm(1)          &= 1 \\
  \trc(a_0+a_1)    &= \trc(a_0)+\trc(a_1) &
  \nrm(a_0a_1)     &= \nrm(a_0)\nrm(a_1) \\
  \trc(\ov{a})     &= \trc(a) &
  \nrm(\ov{a})     &= \nrm(a) \\
  \res(\trc(a))    &= a+\ov{a} &
  \res(\nrm(a))    &= a\,\ov{a} \\
  \nrm(0)          &= 0 &
  \nrm(a_0+a_1)    &= \nrm(a_0)+\nrm(a_1)+\trc(a_0\ov{a_1}) \\
  \trc(a\res(b))   &= \trc(a)b.
 \end{align*}
 We write $\TP$ for the category of Tambara pairs.
\end{definition}

\begin{remark}\label{rem-TP-MP}
 Note that a Tambara pair can be regarded as a Mackey pair using
 addition and the trace map, and it can also be regarded as a Mackey
 pair in a different way using multiplication and the norm map.  These
 structures encode all but the last three axioms for a Tambara pair.
\end{remark}

\begin{construction}
 Given a $G$-Tambara functor $S$, put $A=S(G/1)=S(G)$ and
 $B=S(G/G)=S(1)$ (so both of these are semirings).  Next, as $G$ is
 commutative we see that $\chi\:G\to G$ is a $G$-map, with
 $\chi=\chi^{-1}$, so $T_\chi=N_\chi=R_\chi\:A\to A$.  We use this map
 to define an action of $G$ on $A$ by semiring maps.

 The projection $\ep\:G\to 1$ gives maps
 \begin{align*}
  \res=\ep^* & \: B \to A \\
  \trc=T_\ep & \: A \to B \\
  \nrm=N_\ep & \: A \to B.
 \end{align*} 
\end{construction}

We will prove that this gives an equivalence between Tambara functors
and Tambara pairs.  The first thing to check is that we at least have
a functor.

\begin{proposition}
 The above construction gives a faithful functor $F\:\Tambara_G\to\TP$.
\end{proposition}
\begin{proof}
 We first need to check that the construction gives a Tambara pair.
 In view of Remark~\ref{rem-TP-MP} and Theorem~\ref{thm-MP}, we need
 only consider the last three axioms.  The identities $\nrm(0)=0$ and
 $\trc(a\res(b))=\trc(a)b$ follows immediately from
 Lemma~\ref{lem-norm-zero} and Proposition~\ref{prop-G-bispan-rels}.
 Finally, we need to understand $\nrm(a_0+a_1)=N_\ep T_s(a_0,a_1)$.
 By Proposition~\ref{prop-NT}, the composite $N_\ep T_s$ can be
 written as $T_rN_qR_p$ for certain maps
 \[ G\amalg G \xla{p} E \xra{q} F \xra{r} 1. \]
 As the target set is just a single point, the definitions can be
 simplified as follows:
 \begin{align*}
  F &= \{t\:G\to G\amalg G\st st=1\} \\
  E &= G\tm F \\
  p(g,t) &= t(g) \\
  q(t,g) &= t \\
  r(t) &= 0.
 \end{align*}
 We can name the elements of $G\amalg G$ in an obvious way as
 $\{1_L,\chi_L,1_R,\chi_R\}$.  The elements of $F$ are
 $\{t_0,t_1,t_2,t_3\}$, where 
 \begin{align*}
  t_0(1)    &= 1_L &
  t_1(1)    &= 1_R &
  t_2(1)    &= 1_L &
  t_3(1) &= 1_R \\
  t_0(\chi) &= \chi_L & 
  t_1(\chi) &= \chi_R & 
  t_2(\chi) &= \chi_R & 
  t_3(\chi) &= \chi_L. 
 \end{align*}
 Note that $t_0$ and $t_1$ are equivariant, or in other words, fixed
 under the standard action of $G$ by conjugation.  On the other hand,
 we have $\ov{t_2}=t_3$ and $\ov{t_3}=t_2$.  We now put
 $E_0=\{t_0\}\simeq 1$ and $E_1=\{t_1\}\simeq 1$ and
 $E_2=\{t_2,t_3\}\simeq G$ and $F_i=G\tm E_i$, so we have equivariant
 decompositions $E=E_0\amalg E_1\amalg E_2$ and
 $F=F_0\amalg F_1\amalg F_2$.  It follows that $N_\ep T_s$ can be
 written as a sum of three terms, one for each bispan 
 \[ \om_i=(G\amalg G\xla{p_i}E_i\xra{q_i}F_i\xra{r_i}1). \]
 Now $E_0$ can be identified with $G$, and $p_0$ with the inclusion of
 the left factor in $G\amalg G$, and $q_0$ with $\ep$.  It follows
 that $\om_0$ gives the operation $(a_0,a_1)\mapsto\nrm(a_0)$.
 Similarly, $\om_1$ gives the operation $(a_0,a_1)\mapsto\nrm(a_1)$.
 For the last term, we define an equivariant bijection $m\:E_2\to G$
 by $m(t_2)=1$ and $m(t_3)=\chi$.  One can check that
 $p_2\:E_2\to G\amalg G$ is also an equivariant bijection, and that
 the composite $mq_2p_2^{-1}\:G\amalg G\to G$ is $(1,\chi)$, so the
 associated operation $N_{mq_2p_2^{-1}}\:S(G)\tm S(G)\to S(G)$ is
 $(a_0,a_1)\mapsto a_0\ov{a_1}$.  As $m$ and $p_2$ are bijective we
 have $N_m=T_m$ and $N_{p_2^{-1}}=R_{p_2}$, and it is also clear that
 $\ep m=r_2\:E_2\to 1$.  Putting this together we see that
 $T_{r_2}N_{q_2}R_{p_2}(a_0,a_1)=\trc(a_0\ov{a_1})$ as required.  

 It now follows that the pair $F(S)=(A,B)$ (with structure maps as
 above) is a Tambara pair.  Now suppose we have another Tambara
 functor $S'$, with $F(S')=(A',B')$ say, and a morphism
 $\phi\:S\to S'$ of Tambara functors.  This has components
 $\phi_G\:A=S(G)\to S'(G)=A'$ and $\phi_1\:B=S(1)\to S'(1)=B'$, and it
 is tautological that these commute with the all the Tambara pair
 structure maps.  There is thus an evident way to make $F$ into a
 functor $\Tambara_G\to\TP$.  This is faithful by the same argument as
 for Mackey functors.
\end{proof}

\begin{construction}\label{cons-EP-tambara}
 Let $P=(A,B)$ be a Tambara pair.  For any finite $G$-set $X$, we
 put 
 \[ EP(X) =
     \{(u,v)\in\Map_G(X,A)\tm\Map(X^G,B)\st 
         u(x)=\res(v(x)) \text{ for all } x\in X^G\},
 \]
 as before.   Now suppose we have a $G$-equivariant map $f\:X\to Y$.
 Note that $f^{-1}(Y^G)$ will contain $X^G$, and possibly some free
 orbits as well.  If we choose a point in each such free orbit, we get
 a decomposition $f^{-1}(Y^G)=X^G\amalg X_1\amalg\ov{X_1}$ say.
 \begin{itemize}
  \item[(a)] We define $R_f\:EP(Y)\to EP(X)$ by
   $R_f(m,n)=(m\circ f,n\circ f^G)$ (where $f^G\:X^G\to Y^G$ is just
   the restriction of $f$).
  \item[(b)] We define $T_f\:EP(X)\to EP(Y)$ by $T_f(u,v)=(m,n)$,
   where 
   \begin{align*}
    m(y) &= \sum_{x\in f^{-1}\{y\}} u(x) \\
    n(y) &= \sum_{x_0\in X^G\cap f^{-1}\{y\}} v(x_0) + 
             \sum_{x_1\in X_1\cap f^{-1}\{y\}} \trc(u(x_1)).
   \end{align*}
  \item[(c)] We define $N_f\:EP(X)\to EP(Y)$ by $N_f(u,v)=(p,q)$,
   where 
   \begin{align*}
    p(y) &= \prod_{x\in f^{-1}\{y\}} u(x) \\
    q(y) &= \left(\prod_{x_0\in X^G\cap f^{-1}\{y\}} v(x_0)\right) 
            \left(\sum_{x_1\in X_1\cap f^{-1}\{y\}} \nrm(u(x_1))\right).
   \end{align*}
 \end{itemize}
 By reducing to the Mackey functor case, we see that these operators
 are well-defined.
\end{construction}

\begin{proposition}\label{prop-EP-tambara}
 The above construction makes $EP$ into a Tambara functor.
\end{proposition}
\begin{proof}
 By applying Proposition~\ref{prop-EP-mackey} to the underlying
 additive and multiplicative Mackey pairs, we see that all composites
 of the operators $T$, $N$ and $R$ work correctly except possibly for
 composites of the form $NT$.  Suppose we have maps
 $X\xra{f}Y\xra{g}Z$ of finite $G$-sets, and we define 
 \[ \Dl(f,g) = (X \xla{p} M \xra{q} N \xra{r} Z) \]
 by the usual distribution rule:
 \begin{align*}
  N &= \{(z,s)\st z\in Z,\; s\:g^{-1}\{z\}\to X,\; fs=1\} \\
  M &= \{(y,s)\st y\in Y,\; s\:g^{-1}\{g(y)\}\to X,\; fs=1\} \\
  p(y,s) &= s(y) \\
  q(y,s) &= (g(y),s) \\
  r(z,s) &= z.
 \end{align*}
 We say that the chain $(X\xra{f}Y\xra{g}Z)$ is \emph{distributable}
 if $N_gT_f=T_rN_qR_p\:EP(X)\to EP(Z)$.  We must show that this holds
 for all $f$ and $g$.

 Consider a pair $(u,v)\in EP(X)$, so $N_gT_f(u,v)=(d,e)$ and
 $T_rN_qR_p(u,v)=(d^*,e^*)$ say.  Here $d\:Z\to A$ is given by 
 \[ d(z) = \prod_{y\in g^{-1}\{z\}}\sum_{x\in f^{-1}\{y\}} u(x), \]
 and it is clear from the nonequivariant theory that this is the same
 as $d^*(z)$.  

 Now suppose that $Z$ is a free $G$-set, so $Z^G=\emptyset$, so the
 set $\Map(Z^G,B)$ in which $e$ and $e^*$ live has only one element,
 so $e=e^*$.  It follows that in this case we have $N_gT_f=T_rN_qR_p$.

 For a more general finite $G$-set $Z$, we note that everything
 happens independently over the different orbits in $Z$, so 
 we can reduce to the cases $Z=G$ and $Z=1$.  If $Z=G$ then it is
 free, and this case has already been covered.  We may thus assume
 that $Z=1$, so the definitions simplify to 
 \begin{align*}
  N &= \{s\:Y\to X,\; fs=1\} \\
  M &= Y\tm N \\
  p(y,s) &= s(y) \\
  q(y,s) &= s \\
  r(s) &= 0.
 \end{align*}
 We now argue by induction on the number of orbits in $Y$.  If $Y$ is
 empty then the claim is easy.  For the induction step, we may assume
 that $X\xra{f}Y\xra{g}1$ is distributable, and we must show that the
 same holds for slightly larger chains.  The first case to consider is
 a chain $X_1\xra{f_1}Y_1\xra{g_1}Z$, where $Y_1=Y\amalg 1$, and
 $X_1=X\amalg X'$ for some $X'$, and $f_1=f\amalg f'$ for the unique
 map $f'\:X'\to 1$.  In this case the distributor is 
 \[ X\amalg X' \xla{p_1} ((Y\tm N\tm X')\amalg (N\tm X'))
     \xra{q_1}  (N\tm X') \xra{r_1} 1, 
 \]
 where 
 \begin{align*}
  p_1(y,s,x') &= s(y) \in X \sse X\amalg X' \\
  p_1(s,x')   &= x' \in X' \sse X\amalg X' \\
  q_1(y,s,x') &= (s,x') \\
  q_1(s,x')   &= (s,x').
 \end{align*}
 Consider an element $(u_1,v_1)\in EP(X_1)=EP(X)\tm EP(X')$,
 with components $(u,v)\in EP(X)$ and $(u',v')\in EP(X')$.  By our
 inductive hypothesis we have 
 \[ N_gT_f(u,v)=T_rN_qR_p(u,v) = e \in EP(1)\simeq B \]
 say.  Now split $X'$ as $X'_0\amalg(G\tm X'_1)$, with $G$ acting
 trivially on $X'_0$ and $X'_1$.  From the definitions, we have 
 \[ N_{g_1}T_{f_1}(u_1,v_1) = 
     \left(\sum_{x'\in X'_0}v'(x') + \sum_{x'\in X'_1}\trc(u'(x'))
      \right) e.
 \]
 We need to compare this with the element
 $e_1^*=T_{r_1}N_{q_1}R_{p_1}(u_1,v_1)$.  To do this, we split $N$
 equivariantly as $N_0\amalg(G\tm N_1)$ with $G$ acting trivially on
 $N_0$ and $N_1$.  This gives a splitting
 \[ N\tm X' =
      N_0\tm X'_0 \amalg
      G\tm((N_1\tm X'_0) \amalg
           (N_0\tm X'_1) \amalg
           (G\tm N_1\tm X'_1)).
 \]
 Each orbit in $N\tm X'$ contributes a term to $e^*$.  These terms can
 be described as follows.
 \begin{itemize}
  \item[(a)] Consider an element $(s,x')\in N_0\tm X'_0$.  As
   $s\in N_0=N^G$, we see that it is an equivariant map $Y\to X$, so
   it gives an operator $R_s\:EP(X)\to EP(Y)$.  The element
   \[ N_gR_s(u,v)\in EP(1)\sse A\tm B \]
   will have the form $(\prod_yu(s(y)),w_s)$ for some $w_s\in B$ with
   $\res(w_s)=\prod_yu(s(y))$.  By unwinding the definitions, we see
   that the term in $e^*$ corresponding to $(s,x')$ is $w_s\,v'(x')$.
  \item[(b)] Each element $(s,x')\in N_1\tm X'_0$ contributes a
   term $\trc\left(\left(\prod_{y\in Y}u(s(y))\right)u'(x')\right)$ to
   $e^*$.  As $x'\in X'_0$ we have $u'(x')=\res(v'(x'))$, and using
   the rule $\trc(a\,\res(b))=\trc(a)b$ we can rewrite this term as 
   $\trc\left(\prod_{y\in Y}u(s(y))\right)v'(x')$.
  \item[(c)] Each element $(s,x')\in N_0\tm X'_1$ contributes a term 
   $\trc\left(\left(\prod_{y\in Y}u(s(y))\right)u'(x')\right)$ to
   $e^*$.  As in~(a) we have $\prod_yu(s(y))=\res(w_s)$, so this term
   can be rewritten as $w_s\trc(u'(x'))$.
  \item[(d)] Each element $(s,x')\in N_1\tm X'_1$ contributes two
   orbits in $N\tm X'$, and so contributes two terms in $e^*$, namely
   $\trc\left(\left(\prod_{y\in Y}u(s(y))\right)u'(x')\right)$ and 
   $\trc\left(\left(\prod_{y\in Y}u(s(y))\right)\ov{u'(x')}\right)$.
   Here $u'(x')+\ov{u'(x')}=\res(\trc(u'(x')))$ so the sum of these
   two terms is $\trc\left(\prod_{y\in Y}u(s(y))\right)\trc(u'(x'))$.
 \end{itemize}
 We also have
 \[ e = T_rN_qR_p(u,v) =
     \sum_{s\in N_0} w_s +
      \sum_{s\in N_1}\trc\left(\prod_{y\in Y}u(s(y))\right),
 \]
 and it follows that 
 \[ e^* = \left(\sum_{x'\in X'_0}v'(x') + \sum_{x'\in X'_1}\trc(u'(x'))
      \right) e = N_{g_1}T_{f_1}(u_1,v_1) 
 \]
 as required.

 The other case that we need to consider is a chain
 $X_1\xra{f_1}Y_1\xra{g_1}1$, where $Y_1=Y\amalg G$, and 
 $X_1=X\amalg(G\tm X')$ for some $X'$, and $f_1=f\amalg\pi$, where
 $\pi\:G\tm X'\to G$ is the projection.  In this case we have 
 \[ EP(X_1) = EP(X)\tm \Map(X',A) \sse
     \Map_G(X,A)\tm\Map((X')^G,B)\tm\Map(X',A),
 \]
 and 
 \[ N_{g_1}T_{f_1}(u,v,u') =
     \nrm\left(\sum_{x'\in X'}u'(x')\right) e
 \]
 (where $e=N_gT_f(u,v)=T_rN_qR_p(u,v)$ as before).  We need to compare
 this with the element $e^*=T_{r^*}N_{q^*}R_{p^*}(u,v)$ defined using
 the distributor
 \[ \Dl(f^*,g^*) = 
     (X_1 \xla{p^*} Y_1\tm N^* \xra{q^*} N^* \xra{r^*} 1).
 \] 
 Here $N^*=\{s_1\:Y_1\to X_1\st f_1s_1=1\}$.  Given
 $s\in N$ and $x',x''\in X'$ we define 
 \[ s_1\:Y_1=Y\amalg\{1,\chi\} \to X_1=X\amalg (G\tm X') \]
 by 
 \begin{align*}
  s_1(y) &= s(y) \in X & \text{ for } y\in Y \\
  s_1(1) &= (1,x') \in \{1\}\tm X' & \\
  s_1(\chi) &= (\chi,x'') \in \{\chi\}\tm X' &
 \end{align*}
 This construction identifies $N^*$ with $N\tm X'\tm X'$.  The
 resulting $G$-action is 
 \[ \chi(s,x',x'') = (\chi s\chi,x'',x'). \]
 To classify the orbits, we first decompose $N$ as
 $N_0\amalg (G\tm N_1)$ as before, and we choose a total order on $X'$.
 \begin{itemize}
  \item[(a)] For each $s\in N_0$ and $x'\in X'$ we have a fixed point
   $(s,x',x')$.  If we put $w_s=N_gR_s(u,v)\in B$ as before, then the
   corresponding term in $e^*$ is $\nrm(u(x'))w_s$.
  \item[(b)] For each $s\in N_1$ and $x',x''\in X'$ we have a free
   orbit containing $(s,x',x'')$.  The corresponding term in $e^*$ is
   $\trc(u'(x')\ov{u'(x'')}\prod_yu(s(y)))$.
  \item[(c)] For each $s\in N_0$ and $x',x''\in X'$ with $x'<x''$ we
   have a free orbit containing $(s,x',x'')$.  The corresponding term
   in $e^*$ is
   $\trc(u'(x')\ov{u'(x'')}\prod_yu(s(y)))=\trc(u'(x')\ov{u'(x'')})w_s$. 
 \end{itemize}
 Recall that $\nrm(a_0+a_1)=\nrm(a_0)+\nrm(a_1)+\trc(a_0\ov{a_1})$.
 It follows inductively from this that 
 \[ \nrm(\sum_{i=0}^{n-1}a_i)=
     \sum_i\nrm(a_i)+\sum_{i<j}\trc(a_i\ov{a_j})
 \]
 We can use this to combine the terms of types~(a) and~(c); together
 they contribute $\nrm(z)\sum_{s\in N_0}w_s$, where
 $z=\sum_{x'\in X'}u'(x')$.   Similarly, the sum of all terms of
 type~(b) is 
 \[ \trc\left(z\,\ov{z}\,\sum_{s\in N_1}\prod_{y\in Y}u(s(y))\right).
 \]
 We can use the relation $z\,\ov{z}=\res(\nrm(z))$ to rewrite this,
 and then combine with the~(a) and~(c) terms to get 
 \[ e^* = \nrm(z)
     \left(\sum_{s\in N_0}w_s +
           \sum_{s\in N_1}\trc\left(\prod_{y\in
             Y}u(s(y))\right)\right) = 
      \nrm(z)e = N_{g_1}T_{f_1}(u,v,u')
 \]
 as required.
\end{proof}

\begin{theorem}\label{thm-TP}
 The functor $F\:\Tambara_G\to\TP$ is an equivalence, with inverse
 given by $E$.
\end{theorem}
\begin{proof}
 There is an evident way to define $E$ on morphisms, giving a
 functor $E\:\TP\to\Tambara_G$.  As in the Mackey functor case, there
 is a natural isomorphism $FE(A,B)\simeq(A,B)$ of Tambara pairs.
 Similarly, for any Tambara functor $S$ there is a canonical system of
 bijections $\al_{S,X}\:S(X)\to EFS(X)$ for all finite $G$-sets $X$,
 and these are evidently natural with respect to morphisms $S\to S'$
 of Tambara functors.  We still need to check, however, that for fixed
 $S$ and varying $X$, the maps $\al_{S,X}$ give a morphism of Tambara
 functors.  In other words, for any morphism $f\:X\to Y$ of finite
 $G$-sets, we need to show that the diagrams
 \[ \xymatrix{
  S(X) \ar[d]_\al \ar[r]^{T_f} & 
  S(Y) \ar[d]^\al & & 
  S(X) \ar[d]_\al \ar[r]^{N_f} & 
  S(Y) \ar[d]^\al & & 
  S(X) \ar[d]_\al & 
  S(Y) \ar[d]^\al \ar[l]_{R_f} \\
  EFS(X) \ar[r]_{T_f} &
  EFS(Y) & & 
  EFS(X) \ar[r]_{N_f} &
  EFS(Y) & & 
  EFS(X) &
  EFS(Y) \ar[l]^{R_f}
 } \]
 commute.  This is left as an exercise for the reader.
\end{proof}

\begin{example}\label{eg-TP}
 We will describe a Tambara pair $T$ that does not arise from any of
 the standard constructions of Tambara functors that we have discussed
 previously. (It will be used as an interesting counterexample in
 Section~\ref{sec-modules}.)

 Let $S$ denote the Tambara pair given by $A=B=\Z$ with $\ov{k}=k$ and
 $\trc(k)=2k$ and $\nrm(k)=k^2$ and $\res(k)=k$.  This corresponds to
 the Tambara functor $X\mapsto\Map_G(X,\Z)$ (with $G$ acting trivially
 on $\Z$).

 Next, let $T$ denote the Tambara pair given as follows:
 \begin{align*}
  A &= \Z[\al]/\al^2 = \Z\oplus \Z\al \\
  B &= \Z[\bt,\gm]/(\bt^2,\bt\gm,\gm^2,2\gm) 
     = \Z\oplus \Z\bt \oplus (\Z/2)\gm \\
  \ov{i+j\al} &= i+j\al \\
  \res(i+j\bt+k\gm) &= i+2j\al \\
  \trc(i+j\al) &= 2i+j\bt \\
  \nrm(i+j\al) &= i^2 + ij\bt + j^2\gm.
 \end{align*}
 Note here that $j^2\gm=j\gm$ because $2\gm=0$ and $j^2-j$ is always
 even.  Using this it is straightforward, if somewhat lengthy, to
 check all the axioms for a Tambara pair.  There is an evident
 inclusion $\eta\:S\to T$ and an augmentation $\ep\:T\to S$ given by
 $\ep(i+j\al)=i$ and $\ep(i+j\bt+k\gm)=i$; these are morphisms of
 Tambara pairs.
\end{example}

\section{Tambara functors in stable homotopy theory}
\label{sec-spectrum}

We next outline a construction showing that equivariant $E_\infty$
ring spectra give rise to Tambara functors.  Some examples were
discussed in Example~\ref{eg-pi-zero-tambara}.  Our construction is
simpler than that of Brun~\cite{br:wve} and covers the existing
applications, although it is easy to imagine more advanced
applications where something closer to Brun's approach might be
needed. 

A \emph{Euclidean space} is a finite-dimensional vector space over
$\R$ equipped with an inner product.  An \emph{isometry} between
Euclidean spaces is a linear map that preserves inner products (and so
is injective, but not necessarily surjective).  We write $\CL(V,W)$
for the space of isometries from $V$ to $W$; this gives us a category
$\CL$.  We also have a quotient category $\CL_0$, where the morphism
set $\CL_0(V,W)$ is the set of path-components in $\CL(V,W)$.  Note
that 
\begin{itemize}
 \item[(a)] If $\dim(V)<\dim(W)$ then $\CL(V,W)$ admits a transitive
  action of the connected group $SO(W)$, so $|\CL_0(V,W)|=1$.
 \item[(b)] If $\dim(V)=\dim(W)=n$ then $V\simeq W\simeq\R^n$ so
  $\CL(V,W)\simeq O(n)$ and $|\CL_0(V,W)|=2$.
 \item[(c)] If $\dim(V)>\dim(W)$ then
  $\CL_0(V,W)=\CL(V,W)=\emptyset$. 
\end{itemize}
This implies that the category $\CL_0$ is filtered, and that the
subcategory consisting of the objects $\R^n$ and the standard
inclusions between them is cofinal.

Next, there is a vector bundle over $\CL(V,W)$ whose fibre at a map
$\al\:V\to W$ is the cokernel of $\al$, or equivalently the orthogonal
complement $W\ominus\al(V)$.  We write $\cok$ for this bundle, and
$\Tht(V,W)=\CL(V,W)^{\cok}$ for the associated Thom space.  There are
canonical isomorphisms $\cok(\bt\al)\simeq\cok(\bt)\oplus\cok(\al)$,
and using these we can define composition maps
\[ \Tht(V,W) \Smash \Tht(U,V) \to \Tht(U,W), \]
giving a category $\Tht$ enriched in based spaces.  There are also
evident pairings 
\[ \Tht(V,W) \Smash \Tht(V',W') \to \Tht(V\oplus V',W\oplus W'), \] 
using which we see that the direct sum gives a symmetric monoidal
structure on $\Tht$.

A \emph{spectrum} is a functor $E$ from $\Tht$ to the category
$\Spaces_*$ of compactly generated, weak Hausdorff based spaces with
the property that the maps
\[ E \: \Tht(V,W) \to \Spaces_*(E(V),E(W)) \]
are continuous and preserve basepoints.  (It would be more standard to
call $E$ an \emph{orthogonal prespectrum}, but we will just use the
word spectrum for brevity.)  A \emph{ring spectrum} is a spectrum
equipped with maps $\eta\:S^0\to E(0)$ and 
\[ \mu \: E(V) \Smash E(W) \to E(V\oplus W) \]
that are commutative, associative, unital and compatible with the
action of $\Tht$.  (Again, it would be more standard to say
\emph{strictly commutative orthogonal ring prespectrum}.)

Consider a spectrum $E$.  For any Euclidean space $V$ we have a
one-point compactification $S^V$ (which is homeomorphic to a sphere).
We write $\Pi(V;E)=\Spaces_*(S^V,E(V))$, and we let $\pi(V;E)$
be the set of path-components in $\Pi(V;E)$.  This can be described as
$\pi_d(E(V))$, where $d=\dim(V)$, so it will have a natural structure
as an abelian group provided that $d\geq 2$.

We claim that the construction $V\mapsto\Pi(V;E)$ gives a functor
$\CL\to\Spaces_*$.  To see this, consider an isometry $\al\:V\to W$,
and a based map $f\:S^V\to E(V)$.  Any point $w\in W$ can be written
uniquely as $\al(v)+u$ with $v\in V$ and
$u\in\cok_\al=W\ominus\al(V)$.  This gives a pair
$(\al,u)\in\Tht(V,W)$ and thus a map $(\al,\tht)_*\:X(V)\to X(W)$.  We
define a map $\al_*(f)\:S^W\to E(W)$ by
\[ \al_*(f)(\al(v)+u) = (\al,\tht)_*(f(v)). \]
We leave it the reader to check that this has the right continuity and
functoriality properties.  We now have a map
\[ \CL(V,W) \tm \Pi(V;E) \to \Pi(W;E), \]
and we can pass to path components to get a map
\[ \CL_0(V,W) \tm \pi(V;E) \to \pi(W;E), \]
giving a functor $\CL_0\to\Sets$.  We define
$\pi_0(E)$ to be the colimit of this functor.  By the cofinality
statement that we mentioned before, $\pi_0(E)$ is also the colimit of
the smaller diagram 
\[ \pi_{\R^0}(E(\R^0)) \to 
   \pi_{\R^1}(E(\R^1)) \to 
   \pi_{\R^2}(E(\R^2)) \to 
   \pi_{\R^3}(E(\R^3)) \to \dotsb
\]
However, our original description is better for analysing naturality
questions.  

It is known that our definition of $\pi_0$ is only the first glimpse
of an elaborate homotopy theory of spectra involving spectrification
functors, several different Quillen model structures and so
on~\cite{mama:eos}.  The associated homotopy category is equivalent to
the stable category of spectra as originally introduced by Boardman.
However, none of this is required for our purposes here; we can just
work directly with $\pi_0$.

Now let $G$ be a finite group.  A \emph{$G$-Euclidean space} is a
Euclidean space with isometric action of $G$.  A \emph{$G$-spectrum}
is just a spectrum with an action of $G$.

Let $E$ be a $G$-spectrum, and let $V$ be a $G$-Euclidean space.  Note
that $\Tht(V,V)$ is just $\CL(V,V)$ with an added basepoint.  The
composite
\[ G \to \CL(V,V) \xra{\text{inc}} \Tht(V,V) \xra{E}
    \Spaces_*(E(V),E(V)) 
\]
gives an action of $G$ on $E(V)$.  On the other hand, $G$ is assumed
to act on $E$, so it acts on $E(U)$ for every Euclidean space $U$, and
this gives another action on $E(V)$.  These actions commute, and so
give an action of $G\tm G$ on $E(V)$.  Using the diagonal embedding
$G\to G\tm G$ we get a third action of $G$ on $E(V)$.  It is this
third action that we will use unless otherwise specified.  In other
words, we let $G$ act on $E(V)$ by combining the actions on all
ingredients.  This in turn gives an action of $G$ on the space
$\Pi(V;E)$, and we write $\Pi^G(V;E)$ for the subspace of $G$-fixed
points.  Equivalently, $\Pi^G(V;E)$ is the space of $G$-equivariant
(continuous, based) maps from $S^V$ to $E(V)$.  We then define
$\pi^G(V;E)=\pi_0(\Pi^G(V;E))$.  As before, this gives a functor from
the category $G\CL_0$ to sets, where $G\CL_0(V,W)$ is the set of
path-components in the space $G\CL(V,W)$ of equivariant isometries
from $V$ to $W$.  We write $\pi^G_0(E)$ for the colimit of this
functor.

Next, let $S_1,\dotsc,S_r$ be a list containing precisely one
representative of every isomorphism class of irreducible
$\R[G]$-modules.  We can give $S_i$ a $G$-invariant inner product
(which will be unique up to a multiplicative constant) and thus regard
it as a $G$-euclidean space.  For an arbitrary $G$-Euclidean space $V$
we put $Q_i(V)=\Hom_G(S_i,V)$.  Standard representation theory shows
that $V\simeq\bigoplus_iQ_i(V)\ot S_i$ and that the evident map
\[ \Hom_G(V,W) \to \prod_i\Hom(Q_i(V),Q_i(W)) \]
is an isomorphism.  One can check that the spaces $Q_i(V)$ can be
given inner products in a natural way, and that
\[ G\CL(V,W) \simeq \prod_i\CL(Q_i(V),Q_i(W)). \]
Thus, the category $G\CL$ is equivalent to the product of $r$ copies
of $\CL$, and similarly for $G\CL_0$.  In particular, if
$\dim(Q_i(W))>\dim(Q_i(V))$ for all $i$ (which will certainly hold if
$W$ contains a copy of $V\oplus\R[G]$) then $|G\CL_0(V,W)|=1$.  It
follows that $\pi^G_0(E)$ is also the colimit of the sequence 
\[ \pi^G(\R[G]^0;E) \to 
   \pi^G(\R[G]^1;E) \to 
   \pi^G(\R[G]^2;E) \to 
   \pi^G(\R[G]^3;E) \to \dotsb
\]

Now suppose we have a finite $G$-set $X$.  We write $G\CL[X]$ for the
category of $G$-equivariant Euclidean vector bundles over $X$.  Thus,
for $V\in G\CL[X]$ we have a Euclidean vector space $V_x$ for each
$x\in X$, together with isometric isomorphisms $g\:V_x\to V_{gx}$ that
compose in the obvious way.  We now define
$\Pi(V;E)=\prod_x\Pi(V_x;E)$, and note that this also has a natural
$G$-action.  We write $\Pi^G(V;E)$ for the subspace of fixed points,
and $\pi^G(V;E)$ for the set of path-components in $\Pi^G(V;E)$.  We
then write $\pi^G_0(E)(X)$ for the colimit of the sets $\pi^G(V;E)$
over the category $G\CL[X]_0$.

\begin{theorem}\label{thm-spectrum}
 For any $G$-spectrum $E$, the construction $X\mapsto\pi^G_0(E)(X)$
 gives a Mackey functor in a natural way.  Moreover, if $E$ is a ring
 spectrum (in the strict sense that we mentioned previously) then this
 construction gives a Tambara functor.
\end{theorem}

The rest of this section will form the proof.

\begin{construction}\label{cons-spectrum-R}
 Consider a morphism $f\:X\to Y$ of finite $G$-sets and a $G$-spectrum
 $E$.  We will define a map
 \[ R_f\:\pi^G_0(E)(Y) \to \pi^G_0(E)(X). \]
 Consider a $G$-Euclidean bundle $W$ over $Y$, and a point
 $s\in\Pi^G(W;E)$, or in other words a system of based maps
 $s_y\:S^{W_y}\to E(W_y)$ that are compatible with the $G$-action.
 For each $x\in X$ we let $t_x$ denote the map
 \[ S^{(f^*W)_x} = 
    S^{W_{f(x)}} \xra{s_{f(x)}}
    E(W_{f(x)}) = 
    E((f^*W)_x).
 \]
 This defines a point $t\in\Pi^G(f^*W;E)$, which depends naturally and
 continuously on all the ingredients, so we can define a map 
 \[ R_f \: \Pi^G(W;E) \to \Pi^G(f^*W;E) \]
 by $s\mapsto t$.  This in turn gives a map
 $R_f\:\pi^G(W;E)\to\pi^G(f^*W;E)$.  Next, by
 considering bundles with total space $\R[G]^N\tm Y$ for large $N$ we
 can see that the functor $f^*\:G\CL[Y]_0\to G\CL[X]_0$ is
 cofinal.  Thus, after passing to the colimit over $W$ we get a map 
 \[ R_f\:\pi^G_0(E)(Y) \to \pi^G_0(E)(X). \]
 It is formal to check that this is contravariantly functorial.
\end{construction}

\begin{construction}\label{cons-spectrum-N}
 Consider again a morphism $f\:X\to Y$ of finite $G$-sets, and a
 ring $G$-spectrum $E$.  We will define a map
 \[ N_f \: \pi^G_0(E)(X) \to \pi^G_0(E)(Y). \]
 
 Consider a $G$-Euclidean bundle $V$ over $X$, and put
 $(f_*V)_y=\bigoplus_{f(x)=y}V_x$.  Now suppose we have a point
 $t\in\Pi^G(V;E)$, or in other words a system of based maps
 $t_x\:S^{V_x}\to E(V_x)$ that are compatible with the $G$-action.
 For each $y\in Y$ we let $s_y$ denote the composite
 \[ S^{(f_*V)_y} = 
    \bigSmash_{x\in f^{-1}\{y\}} S^{V_x} \xra{\bigSmash_xt_x}
    \bigSmash_{x\in f^{-1}\{y\}} E(V_x) \xra{\mu}
    E\left(\bigoplus_{x\in f^{-1}\{y\}}V_x\right) =
    E((f_*V)_y).
 \]
 This defines a point $s\in\Pi^G(f_*V;E)$, which depends naturally and
 continuously on all the ingredients, so we can define a map 
 \[ N_f \: \Pi^G(V;E) \to \Pi^G(f_*V;E) \]
 by $t\mapsto s$.  This in turn gives a map
 $N_f\:\pi^G(V;E)\to\pi^G(f_*V;E)$.  

 We can now pass to the colimit over $V$ to get a map
 $N_f\:\pi^G_0(E)(X)\to\pi^G_0(E)(Y)$.  (In this case, the variances
 are the right way around so we get an induced map of colimits without
 needing to assume that $f_*$ is cofinal.  It will in fact be cofinal
 if $f$ is surjective, but not otherwise.)  It is formal to check that
 this construction is covariantly functorial. 
\end{construction}

\begin{proposition}\label{prop-spectrum-RN}
 Suppose we have a cartesian square of finite $G$-sets as shown:
 \[ \xymatrix{
  X' \ar[d]_g \ar[r]^{f'} & Y' \ar[d]^h \\
  X \ar[r]_f & Y
 } \]
 Then $R_hN_f=N_{f'}R_g\:\pi^G_0(E)(X)\to\pi^G_0(E)(Y')$.
\end{proposition}
\begin{proof}
 Let $V$ be a $G$-Euclidean bundle over $X$.  As the square is
 cartesian, we see that $g$ induces a bijection
 $(f')^{-1}\{y'\}\to f^{-1}\{h(y')\}$ for all $y'\in Y'$.  Using this
 we see that the bundle $W=h^*f_*V$ over $Y'$ can also be identified
 with $(f')_*g^*V$.  We claim that the following diagram commutes:
 \[ \xymatrix{
  \Pi^G(g^*V;E) \ar[r]^{N_{f'}} &
  \Pi^G(W;E) \\
  \Pi^G(V;E) \ar[u]^{R_g} \ar[r]_{N_f} &
  \Pi^G(f_*V;E) \ar[u]_{R_h}.
 } \]
 This is proved by simply working through the definitions; details are
 left to the reader.  We can now pass to homotopy and take colimits to
 recover the original claim.
\end{proof}

We now need to deal with transfers.  The key ideas for our approach
are due to Steiner~\cite{st:cop}, but we interpret them in a slightly
different context.

\begin{definition}\label{defn-CLt}
 We write $\CLt$ for the category of Euclidean spaces and isometric
 isomorphisms, so $\CLt(V,W)=\CL(V,W)$ if $\dim(V)=\dim(W)$, and
 $\CLt(V,W)=\emptyset$ otherwise.
\end{definition}

\begin{definition}\label{defn-shrinking}
 Let $V$ be a Euclidean space.  A \emph{shrinking} of $V$ is a family
 of maps $\al(t)\:V\to V$ (for $0\leq t\leq 1$) such that
 \begin{itemize}
  \item[(a)] The combined map $[0,1]\tm V\to V$ (given by
   $(t,v)\mapsto\al(t)(v)$) is continuous.
  \item[(b)] Each map $\al(t)\:V\to V$ is an open embedding, with 
   \[ \|\al(t)(v)-\al(t)(v')\| \leq \|v-v'\| \]
   for all $v,v'\in V$. 
  \item[(c)] The map $\al(0)\:V\to V$ is the identity.
 \end{itemize}
 We call~(b) the \emph{shrinking condition}.  

 Given shrinkings $\al$ and $\bt$, the maps $\bt(t)\circ\al(t)$ form
 another shrinking.  The space $\CS(V)$ of shrinkings is thus a
 topological monoid in a natural way.  There is also an evident
 equivariant analogue $G\CS(V)$.
\end{definition}

\begin{remark}\label{rem-shrink-diff}
 The shrinking condition can often be checked differentially, as
 follows.  Suppose we have a continuously differentiable map
 $\phi\:V\to V$, so for each $v\in V$ we have a derivative
 $\phi'(v)\in\Hom(V,V)$.  The \emph{operator norm} of a linear map
 $\lm\in\Hom(V,V)$ is the number 
 \[ \|\lm\|_{\opp} =
      \sup\{\|\lm(w)\|\st w\in W,\;\|w\|\leq 1\},
 \]
 or equivalently, the largest eigenvalue of the self-adjoint operator
 $\lm^\dag\lm$.  Standard arguments show that
 $\|\phi(v)-\phi(w)\|\leq\|v-w\|$ for all $v$ and $w$ iff
 $\|\phi'(v)\|_{\opp}\leq 1$ for all $v$.
\end{remark}

\begin{example}\label{eg-shrunk}
 We can define $\tht\in\CS(V)$ by $\tht(t)(v)=v/(1+t\|v\|)$.
 This has the property that $\tht(1)(V)$ is the open unit ball centred
 at the origin.  More generally, given $a\in V$ and $r>0$ we can
 define $\tht'\in\CS(V)$ by 
 \[ \tht'(t)(v) = tv + (1-t)\left(a + \frac{v}{1+\|v\|/r}\right), \]
 and then $\tht'(1)(V)$ is the open ball of radius $r$ centred at
 $a$.  In both cases one can check the shrinking condition by a
 differential calculation as in Remark~\ref{rem-shrink-diff};
 details are given in~\cite{st:cop}. 
\end{example}

\begin{remark}\label{rem-CS-contractible}
 Given any shrinking $\al$, we can define shrinkings $H_u\al$ for
 $0\leq u\leq 1$ by $(H_u\al)(t)=\al(1-u+ut)$.  This construction
 gives a contraction of the space $\CS(V)$.
\end{remark}

\begin{definition}\label{defn-shriek}
 Given an open embedding $f\:V\to W$, we write $f^!$ for the map
 $S^W\to S^V$ obtained by the Pontrjagin-Thom construction, so
 $f^!(f(v))=v$ and $f^!(w)=\infty$ if $w\not\in f(V)$.  Given a
 shrinking $\al\in\CS(V)$ we write $\al^!$ for $\al(0)^!$ (which is
 homotopic to $\al(1)^!=1^!=1$).  This gives a morphism 
 $\CS(V)^{\opp}\to\Spaces_*(S^V,S^V)$ of topological monoids.  
\end{definition}

\begin{definition}\label{defn-T-lifting}
 Consider again a morphism $f\:X\to Y$ of finite $G$-sets, equipped
 with a $G$-Euclidean bundle $W$ over $Y$.  A
 \emph{$T$-lifting} of $f$ is a system of shrinkings
 $\al_x\in\CS(W_{f(x)})$ such that for each $y\in Y$, the resulting
 map 
 \[ \al_{[y]} = \left(\al_x(0)\right)_{x\in f^{-1}\{y\}}\:
      \coprod_{x\in f^{-1}\{y\}} W_y \to W_y
 \]
 is injective (and thus an open embedding).
\end{definition}

\begin{lemma}
 Let $L$ be the space of $T$-liftings of $f$, and
 let $L'$ be the space of $G$-equivariant injective maps
 $f'\:X\to\coprod_yW_y$ such that $f'(x)\in W_{f(x)}$ for all $x$.
 Then $L$ is homotopy equivalent to $L'$.  Moreover, if $W$ contains a
 subbundle isomorphic to $\C[G]\tm Y$, then $L$ and $L'$ are
 connected. 
\end{lemma}
\begin{proof}
 First, we can define $\pi\:L\to L'$ by $\pi(\al)(x)=\al_x(0)(0)$.  In
 the opposite direction, suppose we have a point $f'\in L'$.  We
 define $r\:Y\to(0,\infty]$ by 
 \[ r(y) = \frac{1}{2}
       \inf\{\|f'(x)-f'(x')\|\st x,x'\in f^{-1}\{y\},\;x\neq x'\},
 \]
 then we define 
 \[ \al_x(t)(v) = tv +
      (1-t)\left(f'(x)+\frac{v}{1+\|v\|/r(f(x))}\right).
 \]
 This gives a $T$-lifting $\al\in L$.  It depends continuously on
 $f'$, so we can define a map $\sg\:L'\to L$ by $\sg(f')=\al$.  By
 construction we have $\al_x(0)(0)=f'(x)$, so $\pi\sg=1\:L'\to L'$.
 We claim that $\sg\pi\:L\to L$ is homotopic to the identity.  The
 proof is essentially the same as that of the main theorem
 in~\cite{st:cop}, using the formulae displayed in Figure~1 of that
 paper.  The resulting homotopy involves no arbitrary choices, so it
 is not hard to check that it respects equivariance.

 We next claim that $L'$ (and therefore $L$) is path connected,
 provided that $W$ contains a subbundle isomorphic to $\C[G]\tm Y$.
 It is clear that everything works independently over the different
 orbits in $Y$, so we can reduce to the case $Y=G/H$.  Next, we recall
 that the category of $G$-sets over $G/H$ is equivalent to the
 category of $H$-sets, by the functor sending $(U\xra{p}G/H)$ to
 $p^{-1}\{H/H\}$.  Similarly, the category of $G$-equivariant vector
 bundles over $G/H$ is equivalent to the category of representations
 of $H$.  Under this equivalence, the constant bundle $\C[G]\tm G/H$
 becomes a sum of $|G/H|$ copies of $\C[H]$, which certainly contains
 $\C[H]$.  Using these observations we can reduce further to the case
 where $Y$ is a single point, and $W$ is a representation of $G$
 containing $\C[G]$.

 Now choose a list $H_1,\dotsc,H_r$ of subgroups of $G$ containing
 precisely one subgroup in each conjugacy class.  For each orbit in
 $X$ we can choose a representative, whose stabiliser will be
 conjugate to one of the groups $H_i$.  After changing the
 representative we can assume thet the stabiliser is actually equal to
 $H_i$.  Using this we obtain a splitting $X=\coprod_iG/H_i\tm A_i$,
 where $G$ acts trivially on $A_i$.  This in turn gives 
 \[ \Map_G(X,W) = \prod_i \Map(A_i,W^{H_i}). \]
 Now put 
 \[   U_i = \{w\in W\st\stab_G(w)=H_i\}
       = W^{H_i} \setminus \bigcup_{K>H_i} W^K.
 \]
 We find that 
 \[ L' = \Inj_G(X,W) = \prod_i\Inj(A_i,U_i). \]
 The assumption $W\geq\C[G]$ implies that $W^{H_i}$ is a vector space
 of real dimension at least two, and that $U_i$ is obtained from
 $W^{H_i}$ by removing finitely many subspaces of real codimension at
 least two.  Similarly, we find that $\Map(A_i,W^{H_i})$ has dimension
 at least two, and $\Inj(A_i,U_i)$ is obtained by removing finitely
 many subspaces of real codimension at least two.  This implies that
 $\Inj(A_i,U_i)$ is path connected, and thus that $L'$ is path
 connected.  
\end{proof}

\begin{construction}\label{cons-spectrum-T}
 Consider a morphism $f\:X\to Y$ of finite $G$-sets, and an
 equivariant vector bundle $W$ over $Y$.  Let $\al$ be a
 $T$-lifting of $f$.  We can apply the Pontrjagin-Thom
 construction to $\al_{[y]}$ to get a map
 \[ \al_{[y]}^! \: S^{W_y} \to
     \bigWedge_{x\in f^{-1}\{y\}} S^{W_y}.
 \]

 Now suppose we have a point $t\in\Pi^G(f^*W;E)$, or in other
 words a system of based maps $t_x\:S^{W_{f(x)}}\to E(W_{f(x)})$ that
 are compatible with the $G$-action.  For each $y\in Y$ we let $s_y$
 denote the composite
 \[ S^{W_y} \xra{\al_{[y]}^!}
    \bigWedge_{x\in f^{-1}\{y\}} S^{W_y} \xra{t_{[y]}}
    E(W_y),
 \]
 where $t_{[y]}$ is given by $t_x$ on the wedge summand indexed by
 $x$.  This defines a point $s\in\Pi^G(W;E)$, which depends
 naturally and continuously on all the ingredients, so we can define a
 map 
 \[ T_{f,\al} \: \Pi^G(f^*W;E) \to \Pi^G(W;E) \]
 by $t\mapsto s$.  Note also that if we embed $W$ in a larger bundle
 $W'=W\oplus A$ then the maps 
 \[ \al'_x(t) = \al_x(t)\tm 1 \:
      W'_{f(x)} = W_{f(x)}\tm A_{f(x)} \to 
      W_{f(x)}\tm A_{f(x)} = W'_{f(x)}
 \]
 also form a $T$-lifting of $f$, and the diagram 
 \[ \xymatrix{
  \Pi^G(f^*W;E) \ar[r]^{T_{f,\al}} \ar[d] &
  \Pi^G(W;E) \ar[d] \\
  \Pi^G(f^*W';E) \ar[r]_{T_{f,\al'}} &
  \Pi^G(W';E)
 } \]
 commutes.

 Next, we can apply $\pi_0$ to the above map $T_{f,\al}$ to get a map
 $T_{f,\al}\:\pi^G(f^*W;E)\to\pi^G(W;E)$.  Provided that
 $W$ contains $\C[G]\tm Y$, the space of possible lifts $\al$ will be
 path-connected, so $T_{f,\al}\:\pi^G(f^*W;E)\to\pi^G(W;E)$
 will be independent of $\al$, so we can denote it by $T_f$.  Now note
 that the category of bundles $W$ containing $\C[G]\tm Y$ is cofinal
 in $G\CL[Y]$, and the functor $f^*$ from this subcategory to
 $G\CL[X]$ is also cofinal.  We can thus pass to the colimit to get a
 map $T_f\:\pi_0^G(E)(X)\to\pi_0^G(E)(Y)$.
\end{construction}

\begin{proposition}\label{prop-spectrum-TT}
 For any maps $X\xra{f}Y\xra{g}Z$ we have 
 \[ T_{gf}=T_gT_f\:\pi^G_0(E)(X)\to\pi^G_0(E)(Z). \]
\end{proposition}
\begin{proof}
 Let $U$ be a $G$-euclidean bundle over $Z$, giving a bundle $g^*(U)$
 over $Y$ and a bundle $(gf)^*(U)$ over $X$.  Choose 
 $T$-liftings $\al$ and $\bt$ for $f$ and $g$ with respect to these
 bundles.  Next, for $x\in X$ and $t\in [0,1]$ let $\gm_x(t)$ denote
 the composite
 \[ U_{gf(x)} \xra{\al_{x}(t)} U_{gf(x)}
              \xra{\bt_{f(x)}(t)} U_{gf(x)}.
 \]
 These maps give a $T$-lifting for $gf$, and by inspection of
 the definitions we have 
 \[ T_{gf,\gm} = T_{g,\bt}T_{f,\al} \:
     \Pi^G((gf)^*(U);E) \to \Pi^G(U;E).
 \]
 The claim follows by taking $\pi_0$ and passing to colimits.
\end{proof}

\begin{proposition}\label{prop-spectrum-RT}
 Suppose we have a cartesian square of finite $G$-sets as shown:
 \[ \xymatrix{
  X' \ar[d]_g \ar[r]^{f'} & Y' \ar[d]^h \\
  X \ar[r]_f & Y
 } \]
 Then $R_hT_f=T_{f'}R_g\:\pi^G_0(E)(X)\to\pi^G_0(E)(Y')$.
\end{proposition}
\begin{proof}
 Let $W$ be a $G$-euclidean bundle over $Y$, and put $W'=h^*W$.
 Suppose we have a $T$-lifting $\al$ for $f$ with respect to $W$.  For
 $x'\in X'$ we put 
 \[ \al'_{x'}(t) = \al_{g(x')}(t) \: 
     ((f')^*W')_{x'} = W_{hf'(x')} = W_{fg(x')} \to ((f')^*W')_{x'}.
 \]
 By inspection of the definitions we find that these maps give a
 $T$-lifting of $f'$, and that the diagram 
 \[ \xymatrix{
     \Pi^G((f')^*W';E) 
      \ar[r]^{T_{f',\al'}} &
     \Pi^G(W';E) \\
     \Pi^G(f^*W;E) 
      \ar[u]_{R_g}
      \ar[r]_{T_{f,\al}} &
     \Pi^G(W;E)
      \ar[u]^{R_h}
    }
 \]
 commutes.  The claim follows by taking $\pi_0$ and passing to
 colimits. 
\end{proof}

\begin{proposition}\label{prop-spectrum-NT}
 Suppose we have maps $X\xra{f}Y\xra{g}Z$ of finite $G$-sets, with
 distributor 
 \[ \Dl(f,g) = (X \xla{p} A \xra{q} B \xra{r} Y). \]
 Then $N_gT_f=T_rN_qR_p\:\pi^G_0(E)(X)\to\pi^G_0(E)(Z)$.
\end{proposition}
\begin{proof}
 Consider a $G$-euclidean bundle $W$ over $Y$.  Note that for
 $(y,s)\in A$ we have $fp(y,s)=fs(y)=y$, so for $(z,s)\in B$ we have 
 \[ (q_*p^*f^*W)_{(z,s)} = 
     \bigoplus_{y\in g^{-1}\{z\}} (p^*f^*W)_{(y,s)} = 
     \bigoplus_{y\in g^{-1}\{z\}} W_y = 
     (r^*g_*W)_{(z,s)},
 \]
 so we can identify $q_*p^*f^*W$ with $r^*g_*W$.  

 Now choose a $T$-lifting for $f$, consisting of maps
 $\al_x(t)\:W_{f(x)}\to W_{f(x)}$.   Consider a point $(z,s)\in B$.
 For each $y\in g^{-1}\{z\}$ we have a point $s(y)\in X$ with
 $fs(y)=y$, and thus a shrunk isometry $\al_{s(y)}(t)\:W_y\to W_y$.
 We can take the product over all $y\in g^{-1}\{z\}$ to get a map 
 $\bt_{(z,s)}(t)\:(r^*g_*W)_y\to(r^*g_*W)_y$.  These maps form a
 $T$-lifting of $r$.  We now have maps
 \[ \Pi^G(f^*W;E) \xra{T_{f,\al}} \Pi^G(W;E)
      \xra{N_f} \Pi^G(g_*W;E)
 \]
 and
 \[ \Pi^G(f^*W;E) \xra{R_p}
    \Pi^G(p^*f^*W;E) \xra{N_q}
    \Pi^G(q_*p^*f^*W;E) = \Pi^G(r^*g_*W;E) \xra{T_{r,\bt}}
    \Pi^G(g_*W;E).
 \]
 We claim that the two composites are the same.  To see this, consider
 an element $t\in\Pi^G(f^*W;E)$, consisting of maps
 $t_x\:S^{W_{f(x)}}\to E(W_{f(x)})$ for all $x\in X$.  The point
 $t'=T_{f,\al}(t)\in\Pi^G(W;E)$ consists of maps
 $t'_y\:S^{W_y}\to E(W_y)$, where $t'_y$ is the composite
 \[ S^{W_y} \xra{\al_{[y]}^!}
     \bigWedge_{x\in f^{-1}\{y\}}S^{W_y} \xra{t_{[y]}} E(W_y).
 \]
 Now the point $t''=N_ft'\in \Pi^G(g_*W;E)$ consists of maps 
 \[ t''_z\: \bigSmash_{y\in g^{-1}\{z\}} S^{W_y}
     = S^{(g_*W)_z} \to E((g_*W)_z).
 \]
 The map $t''_z$ is obtained by smashing together the maps $t'_y$ for
 all $y\in g^{-1}\{z\}$, then composing with the map 
 \[ \bigSmash_{y\in g^{-1}\{z\}}E(W_y) \to 
     E\left(\bigoplus_{y\in g^{-1}\{z\}}W_y\right)
 \]
 given by the ring structure of the spectrum $E$.  The smash product
 of the maps $t'_y$ factors through the space
 \[ Q = \bigSmash_{y\in g^{-1}\{z\}}
     \bigWedge_{x\in f^{-1}\{y\}}
      S^{W_y}.
 \]
 As the smash product distributes over the wedge product, this can be
 rewritten as 
 \[ Q = \bigWedge_s \bigSmash_{y\in g^{-1}\{z\}} S^{W_y}
      = \bigWedge_s \bigSmash_{y\in g^{-1}\{z\}} S^{W_{fs(y)}},
 \]
 where $s$ runs over the set of maps $f^{-1}\{y\}\to X$ with $fs=1$.
 Using the fact that the Pontrjagin-Thom construction sends direct
 sums to smash products, we see that the map
 $S^{(r_*W)_z}=\bigSmash_{y\in g^{-1}\{z\}} S^{W_y}\to Q$ occurring in
 $t''_z$ is just $\bt_{[z]}^!$.  It is now straightforward to check
 that 
 \[ N_gT_{f,\al}=T_{r,\bt}N_qR_p \:
     \Pi^G(f^*W;E) \to \Pi^G(f_*W;E).
 \]
 We now apply $\pi_0$ and pass to the colimit over $W$.  The functor
 $f^*\:G\CL[Y]\to G\CL[X]$ is cofinal, so
 $\colim_W\pi^G(f^*W;E)=\pi^G_0(E)(X)$.  The functor
 $f_*\:G\CL[Y]\to G\CL[Z]$ need not be cofinal, but we still have a
 comparison map $\colim_W\pi^G(f_*W;E)\to\pi^G_0(E)(Z)$.  We deduce
 that $N_gT_f=T_rN_qR_p\:\pi^G_0(E)(X)\to\pi^G_0(E)(Z)$ as claimed.
\end{proof}

\section{Tensor products of Tambara functors}
\label{sec-tambara-tensor}

Our main result is as follows:
\begin{proposition}\label{prop-tambara-coproduct}
 Let $M$ and $N$ be Tambara functors.  Then there is a unique way to
 make the Mackey functor $M\btm N$ into a Tambara functor such that
 $N_f(m\ot n)=N_f(m)\ot N_f(n)$ for all $f\:X\to Y$ and all
 $(m,n)\in M(X)\tm N(X)$.  Moreover, with this structure, $M\btm N$ is
 the coproduct of $M$ and $N$ in the category of Tambara functors.
\end{proposition}
According to~\cite{na:gdc}*{Remark 1.9}, this result also appears in
an unpublished manuscript of Tambara.

The rest of this section constitutes the proof; the threads are
gathered together in Corollary~\ref{cor-box-tambara} and
Lemma~\ref{lem-tambara-coproduct}. 

\begin{definition}\label{defn-NT}
 Let $M$ and $N$ be Tambara functors, and let $X\xra{g}Y\xra{h}Z$ be
 equivariant maps of finite $G$-sets.  We define 
 \[ NT_{h,g} \: M(X)\tm N(X) \to (M\btm N)(Z) \]
 as follows.  We first construct the distributor 
 \[ \Dl(g,h) = (X\xla{p}A\xra{q}B\xra{r}Z) \]
 as in Definition~\ref{defn-distributor}, and then put
 \[ NT_{h,g}(m,n) = T_r(N_qR_p(m)\ot N_qR_p(n)). \]
\end{definition}

\begin{lemma}\label{lem-NT}
 Suppose we have maps $W\xra{f}X\xra{g}Y\xra{h}Z$ and elements
 $m'\in M(W)$ and $n\in N(X)$.  Then 
 \[ NT_{h,g}(T_f(m'),n) = NT_{h,gf}(m',R_f(n)). \]
\end{lemma}
\begin{proof}
 We will use the diagram 
 \[ \xymatrix{
  W \ar@{=}[d] &
  \tA \ar[l]_{\tp} \ar[d]^i \ar[r]^{\tq} &
  \tB \ar[r]^{\tr} \ar[dd]^k &
  Z \ar@{=}[dd] \\
  W \ar[d]_f &
  A^* \ar[l]_{p^*} \ar[d]^j \\
  X &
  A \ar[l]^p \ar[r]_q &
  B \ar[r]_r &
  Z
 } \]
 constructed in Proposition~\ref{prop-NTT}.  As
 $\Dl(g,h)=(X\xla{p}A\xra{q}B\xra{r}Z)$ and the middle square is
 cartesian we have
 \[ NT_{h,g}(T_f(m'),n)
     = T_r(N_qR_pT_f(m')\ot N_qR_p(n))
     = T_r(N_qT_jR_{p^*}(m')\ot N_qR_p(n)).
 \]
 As $\Dl(j,q)=(A^*\xla{i}\tA\xra{\tq}\tB\xra{k}B)$, this can be
 rewritten as
 \[ NT_{h,g}(T_f(m'),n)
     = T_r(T_kN_{\tq}R_i R_{p^*}(m')\ot N_qR_p(n))
     = T_r(T_kN_{\tq} R_{\tp}(m')\ot N_qR_p(n)).
 \]
 Using Frobenius reciprocity with respect to $k$ this becomes
 \[ NT_{h,g}(T_f(m'),n)
     = T_{rk}(N_{\tq} R_{\tp}(m')\ot R_kN_qR_p(n)).
 \]
 On the outside we have $rk=\tr$.  For the second factor inside, we
 note that the middle rectangle is cartesian, so
 $R_kN_q=N_{\tq}R_{ji}$.  Moreover, we have $pji=f\tp$, so
 $R_{ji}R_p=R_{\tp}R_f$.  Putting this together we get 
 \[ NT_{h,g}(T_f(m'),n)
     = T_{\tr}(N_{\tq} R_{\tp}(m')\ot N_{\tq}R_{\tp}R_f(n)).
 \]
 As 
 \begin{align*} 
  \Dl(gf,h) &= (W\xla{\tp}\tA\xra{\tq}\tB\xra{\tr}Z) \\
   &= (W \xla{p^*i} \tA \xra{\tq} \tB \xra{rk} Z),
 \end{align*}
 this is the same as $NT_{h,gf}(m',R_f(n))$.
\end{proof}

\begin{corollary}\label{cor-NT}
 For any $h\:Y\to Z$ there is a unique map
 $N_h\:(M\btm N)(Y)\to(M\btm N)(Z)$ such that
 $N_hT_g(m\ot n)=NT_{h,g}(m,n)$ for all $g\:X\to Y$ and all
 $(m,n)\in M(X)\tm N(X)$.
\end{corollary}
\begin{proof}
 Lemma~\ref{lem-NT} gives an identity 
 \[ NT_{h,g}(T_f(m'),n) = NT_{h,gf}(m',R_f(n)). \]
 There is a similar identity
 \[ NT_{h,g}(m,T_f(n')) = NT_{h,gf}(R_f(m),n') \]
 which can be proved in the same way or deduced using the twist map.
 The claim follows directly from these together with
 Proposition~\ref{prop-box-equiv}. 
\end{proof}

\begin{lemma}\label{lem-box-N-composite}
 For any maps $X\xra{g}Y\xra{h}Z$ we have 
 \[ N_hN_g = N_{hg} \: (M\btm N)(X) \to (M\btm N)(Z). \]
\end{lemma}
\begin{proof}
 Consider another map $f\:W\to X$ and elements
 $(m,n)\in M(W)\tm N(W)$.  We must show that
 $N_hN_gT_f(m\ot n)=N_{hg}T_f(m\ot n)$.  We use the diagram 
 \[ \xymatrix{
  \tA \ar[r]^\al \ar[d]_i &
  A \ar[r]^p \ar[d]^q &
  W \ar[r]^f &
  X \ar[d]^g \\
  \tB \ar[r]_\bt \ar[d]_j &
  B \ar[rr]_r & &
  Y \ar[d]^h \\
  \tC \ar[rrr]_k &&& 
  Z
 } \]
 constructed in Proposition~\ref{prop-NNT}.  As
 $\Dl(f,g)=(W\xla{p}A\xra{q}B\xra{r}Y)$, we have 
 $N_gT_f(m\ot n)=T_r(N_qR_p(m)\ot N_qR_p(n))$.  As
 $\Dl(r,h)=(B\xla{\bt}\tB\xra{j}\tC\xra{k}Z)$, we get
 \[ N_hN_gT_f(m\ot n) = T_k(N_jR_\bt N_qR_p(m)\ot N_jR_\bt N_qR_p(n)). \]  
 As the top left square is cartesian, we have $R_\bt N_q=N_iR_\al$, so
 we get
 \[ N_hN_gT_f(m\ot n) = T_k(N_{ji}R_{p\al}(m)\ot N_{ji}R_{p\al}(n)). \]  
 As $\Dl(f,hg)=(W \xla{p\al} \tA \xra{ji} \tC \xra{k} Z)$, this is the
 same as $N_{hg}T_f(m\ot n)$, as required. 
\end{proof}

\begin{lemma}\label{lem-box-N-mackey}
 For any pullback square
 \[ \xymatrix{
  \tX \ar[r]^{\tg} \ar[d]_j & \tY \ar[d]^k \\
  X \ar[r]_g & Y
 } \]
 we have $R_kN_g=N_{\tg}R_j\:(M\btm N)(X)\to(M\btm N)(\tY)$.
\end{lemma}
\begin{proof}
 Consider a map $f\:W\to X$ and an element
 $u=T_f(m\ot n)\in(M\btm N)(X)$.  We let $\tW$ be the pullback of $W$,
 so we have a diagram
 \[ \xymatrix{
  \tW \ar[d]_i \ar[r]^{\tf} & X \ar[d]^j \ar[r]^{\tg} & Y \ar[d]^k \\
  W \ar[r]_f & X \ar[r]_g & Y.
 } \]
 in which both squares (and thus the full rectangle) are cartesian.
 As in Proposition~\ref{prop-NRT} there is a commutative diagram 
 \[ \xymatrix{
  \tW \ar[d]_i &
  \tA \ar[l]_{\tp} \ar[r]^{\tq} \ar[d]_\al &
  \tB \ar[r]^{\tr} \ar[d]^\bt &
  \tY \ar[d]^k \\
  W &
  A \ar[l]^p \ar[r]_q &
  B \ar[r]_r &
  Y
 } \]
 in which the middle and right squares are cartesian, the top row is
 $\Dl(\tf,\tg)$, and the bottom row is $\Dl(f,g)$.  We now have
 \begin{align*}
  N_g(u) &= T_r(N_qR_p(m)\ot N_qR_p(n)) \\
  R_kN_g(u) &= T_{\tr}R_\bt(N_qR_p(m)\ot N_qR_p(n))
             = T_{\tr}(R_\bt N_qR_p(m)\ot R_\bt N_qR_p(n)).
 \end{align*}
 As the middle square is cartesian we have $R_\bt N_q=N_{\tq}R_\al$.
 As $p\al=i\tp$ we have $R_\al R_p=R_{\tp}R_i$.  Putting this
 together, we get
 \[ R_kN_g(u) =
     T_{\tr}(N_{\tq}R_{\tp}R_i(m)\ot N_{\tq}R_{\tp}R_i(n)).
 \]
 On the other hand, we have 
 \begin{align*}
  R_j(u) &= T_{\tf}R_i(m\ot n) = T_{\tf}(R_i(m)\ot R_i(n)) \\
  N_{\tg}R_j(u)
   &= T_{\tr}(N_{\tq}R_{\tp}R_i(m)\ot N_{\tq}R_{\tp}R_i(n)),
 \end{align*}
 which is the same as $R_kN_g(u)$.
\end{proof}

\begin{corollary}\label{cor-box-tambara}
 Our definition of $N_h$ makes the Mackey functor $M\btm N$ into a
 Tambara functor.
\end{corollary}
\begin{proof}
 It will suffice to check that the relations in
 Proposition~\ref{prop-UG-pres} are satisfied.  Many of these involve
 only $T$ and $R$ so they are automatically satisfied, because $M\btm
 N$ is a Mackey functor.  Of the remaining relations, part~(a) is
 covered by Lemma~\ref{lem-box-N-composite} (together with the easy
 fact that $N_1=1$).  Part~(b) is visibly built in to our definition
 of $N_h$.  Part~(c) is Lemma~\ref{lem-box-N-mackey}.
\end{proof}

\begin{lemma}\label{lem-tambara-coproduct}
 There is a morphism $i\:M\to M\btm N$ of Tambara functors
 given by $i(m)=m\ot 1$ for all $m\in M(X)$ (where $1$ denotes the
 multiplicative identity element for the semiring structure on $N(X)$
 discussed in Proposition~\ref{prop-tambara-semiring}).  Similarly,
 there is a morphism $j\:N\to M\btm N$ given by $j(n)=1\ot n$.
 Moreover, the diagram $M\xra{i}M\btm N\xla{j}N$ is a coproduct.
\end{lemma}
\begin{proof}
 Consider maps $W\xra{f}X\xra{g}Y$ and an element $m\in M(X)$.  As
 $R_f\:N(X)\to N(W)$ is a semiring map we have
 \[ R_fi(m)=R_f(m\ot 1)=R_f(m)\ot R_f(1) = R_f(m)\ot 1=iR_f(m). \]
 Similarly, the map $N_g\:N(X)\to N(Y)$ preserves $1$ so we have 
 \[ N_gi(m)=N_g(m\ot 1)=N_g(m)\ot N_g(1)=N_g(m)\ot 1=iN_g(m). \]
 We can also use Frobenius reciprocity (Lemma~\ref{lem-frobenius}) to
 get 
 \[ T_gi(m)=T_g(m\ot 1)=T_g(m\ot R_g(1))=T_g(m)\ot 1=iT_g(m). \]
 This proves that $i$ is a morphism of Tambara functors, and by
 symmetry the same is true of $j$.

 Now suppose we have a Tambara functor $S$ and morphisms
 $M\xra{d}S\xla{e}N$.  We claim that for each $X$ there is a unique map
 $r\:(M\btm N)(X)\to S(X)$ satisfying
 \[ r(T_q(m\ot n)) = T_q(d(m)e(n)) \]
 for all $q\:U\to X$ and $(m,n)\in M(U)\tm N(U)$.  In view of
 Proposition~\ref{prop-box-equiv}, it will suffice to show that 
 \begin{itemize}
  \item[(a)] For all $U'\xra{r}U\xra{q}X$ and $(m',n)\in M(U')\tm N(U)$
   we have $T_{qr}(d(m')e(R_r(n)))=T_q(d(T_r(m'))e(n))$.
  \item[(b)] For all $U'\xra{r}U\xra{q}X$ and $(m,n')\in M(U)\tm N(U')$
   we have $T_{qr}(d(R_r(m))e(n'))=T_q(d(m)e(T_r(n')))$.
 \end{itemize}
 Part~(a) can be rewritten as
 $T_{qr}(d(m')R_r(e(n)))=T_q(T_r(d(m'))e(n))$ and in this form it is a
 special case of Proposition~\ref{prop-tambara-semiring}(b).  Part~(b)
 is similar, so we have maps $r\:(M\btm N)(X)\to S(X)$ as claimed.  It
 is clear by construction that $rT_f=T_fr$ for any $f\:W\to X$.  We
 claim that $rN_f=N_fr$ also holds.  To see this, consider an element
 $z=T_q(m\ot n)\in (M\btm N)(W)$, and write the distributor $\Dl(q,f)$
 as $(U\xla{p^*}A\xra{q^*}B\xra{r^*}X)$.  We then have
 \begin{align*}
  N_f(z) &= T_{r^*}(N_{q^*}R_{p^*}(m)\ot N_{q^*}R_{p^*}(n)) \\
  rN_f(z) &= T_{r^*}(d(N_{q^*}R_{p^*}(m))\;e(N_{q^*}R_{p^*}(n))) \\
    &= T_{r^*}(N_{q^*}R_{p^*}(d(m))\;N_{q^*}R_{p^*}(e(n))) \\
    &= T_{r^*}N_{q^*}R_{p^*}(d(m)\; e(n)) \\
    &= N_fT_q(d(m)\; e(n)) 
     = N_fr(z),
 \end{align*}
 as claimed.

 Now consider instead a map $g\:V\to W$; we claim that
 $rR_g=R_gr\:(M\btm N)(W)\to S(V)$.  To see this, we construct a
 pullback square as follows.
 \[ \xymatrix{
  \tU \ar[r]^{\tg} \ar[d]_{\tq} & U \ar[d]^q \\
  V \ar[r]_g & W
 } \] 
 For $z=T_q(m\ot n)$ as before, we then have 
 \begin{align*}
  R_g(z)  &= R_gT_q(m\ot n) 
           = T_{\tq}R_{\tg}(m\ot n) 
           = T_{\tq}(R_{\tg}(m)\ot R_{\tg}(n)) \\
  rR_g(z) &= T_{\tq}(d(R_{\tg}(m))\;e(R_{\tg}(n)))
           = T_{\tq}(R_{\tg}(d(m))\;R_{\tg}(e(n))) \\
          &= T_{\tq}R_{\tg}(d(m)\;e(n)) 
           = R_gT_q(d(m)\;e(n)) 
           = R_gr(z)
 \end{align*}
 as required.  This proves that $r$ is a morphism of Tambara
 functors.  It is easily seen to be the unique one with $ri=d$ and
 $rj=e$, so the diagram $M\xra{i}M\btm N\xla{j}N$ is a coproduct, as
 claimed. 
\end{proof}

\section{Limits and colimits}
\label{sec-colimits}

In this section we will analyse limits and colimits in the categories
$\Mackey_G$ and $\Tambara_G$.  We could do this by quoting general
results about coloured Lawvere theories, but it is not hard to argue
more directly.

\begin{proposition}\label{prop-limits}
 Let $S:\CJ\to\Tambara_G$ be a diagram of Tambara functors, and put
 $S^*(X)=\invlim_{j\in\CJ}S(j)(X)$.  Then there is a canonical way to
 make $S^*$ into a Tambara functor, and it becomes the limit of the
 diagram in $\Tambara_G$.  Moreover, the same applies for diagrams of
 Mackey functors.
\end{proposition}
\begin{proof}
 It is standard that $S^*$ is the limit of the diagram in the category
 of all functors from $\bCU_G$ to sets.  Moreover, as limits commute
 with products, we see that the map
 $S^*(X\amalg Y)\to S^*(X)\tm S^*(Y)$ is a bijection for all $X$ and
 $Y$ in $\bCU_G$, so $S^*$ is a Tambara functor.  The rest is clear.
\end{proof}

Filtered colimits are equally easy:
\begin{proposition}\label{prop-filtered-colimits}
 Let $S:\CJ\to\Tambara_G$ be a diagram of Tambara functors, where
 $\CJ$ is a filtered category.  Put $S^*(X)=\colim_{j\in\CJ}S(j)(X)$.
 Then there is a canonical way to make $S^*$ into a Tambara functor,
 and it becomes the colimit of the diagram in $\Tambara_G$.  Moreover,
 the same applies for diagrams of Mackey functors.
\end{proposition}
\begin{proof}
 It is standard that $S^*$ is the colimit of the diagram in the
 category of all functors from $\bCU_G$ to sets.  Moreover, as
 filtered colimits in the category of sets commute with finite
 products, we see that the map $S^*(X\amalg Y)\to S^*(X)\tm S^*(Y)$ is
 a bijection for all $X$ and $Y$ in $\bCU_G$, so $S^*$ is a Tambara
 functor.  The rest is clear.
\end{proof}

\begin{corollary}\label{cor-coproducts}
 The categories $\Mackey_G$ and $\Tambara_G$ have all set-indexed
 coproducts. 
\end{corollary}
\begin{proof}
 Finite coproducts exist in both categories: in $\Mackey_G$ they are
 the same as finite products, and in $\Tambara_G$ they are tensor
 products (by Proposition~\ref{prop-tambara-coproduct}).  Infinite
 coproducts can be expressed as filtered colimits of finite
 coproducts.  
\end{proof}

We can now generalise Definition~\ref{defn-semigroup-congruence}:
\begin{definition}\label{defn-congruence}\ \\
 A \emph{congruence} on an algebraic structure $P$ is a substructure
 of $P\tm P$ that is also an equivalence relation.  In more detail:
 \begin{itemize}
  \item[(a)] A congruence on a semigroup $M_1$ is a subsemigroup of
   $M_1\tm M_1$ that is also an equivalence relation.
  \item[(b)] A congruence on a semiring $S_1$ is a subsemiring of
   $S_1\tm S_1$ that is also an equivalence relation.
  \item[(c)] A congruence on a Mackey functor $M$ is a sub-Mackey
   functor $E\leq M\tm M$ such that the subset $E(X)\sse M(X)\tm M(X)$
   is an equivalence relation on $M(X)$ for all $X$.
  \item[(d)] A congruence on a Tambara functor $S$ is a sub-Tambara
   functor $E\leq S\tm S$ such that the subset $E(X)\sse S(X)\tm S(X)$
   is an equivalence relation on $S(X)$ for all $X$.
 \end{itemize}
\end{definition}

\begin{proposition}\label{prop-cong}\ \\
 \begin{itemize}
  \item[(a)] For any morphism $\phi\:S\to S'$ of Tambara functors there
   is a congruence $\eqker(\phi)$ on $S$ given by 
   \[ \eqker(\phi)(X) =
       \{(a,b)\in S(X)\tm S(X) \st \phi(a)=\phi(b)\in S'(X)\}.
   \]
  \item[(b)] For any congruence $E$ on $S$ there is a unique way to
   make the quotient sets $(S/E)(X)=S(X)/E(X)$ into a Tambara functor
   such that the projection $\pi\:S\to S/E$ is a Tambara morphism.
  \item[(c)] If $E\leq\eqker(\phi)$ then there is a unique Tambara
   morphism $\ov{\phi}\:S/E\to S'$ with $\phi=\ov{\phi}\pi$, but if
   $E\not\leq\eqker(\phi)$ then there is no such morphism.
 \end{itemize}
 Moreover, the corresponding statements also hold for Mackey
 functors. 
\end{proposition}
\begin{proof}
 Most of this is clear but we offer a few pointers.
 \begin{itemize}
  \item[(a)] If $\om$ is a bispan from $X$ to $Y$ and
   $(a,b)\in\eqker(\phi)(X)$ then 
   \[ \phi(f_\om(a)) = f_\om(\phi(a)) =
        f_\om(\phi(b)) = \phi(f_\om(b))
   \]
   so $f_\om(a,b)\in\eqker(\phi)(Y)$.  Moreover, it is clear from the
   diagram 
   \[ \xymatrix{
    S(X\amalg Y) \ar[r]^{\phi} \ar[d]_{(R_i,R_j)}^\simeq &
    S'(X\amalg Y) \ar[d]^{(R_i,R_j)}_\simeq \\
    S(X)\tm S(Y) \ar[r]_{\phi\tm\phi} & 
    S'(X)\tm S'(Y)
   } \]
   that $\eqker(\phi)(X\amalg Y)=\eqker(\phi)(X)\tm\eqker(\phi)(Y)$.  This
   proves that $\eqker(\phi)$ is a sub-Tambara functor of $S\tm S$.  It
   is clear that it is also an equivalence relation.
  \item[(b)] Consider a bispan $\om$ from $X$ to $Y$.  Any element
   $a\in(S/E)(X)$ can be written as $\pi(a_0)$ for some
   $a_0\in S(X)$.  If $a=\pi(a_0)=\pi(a_1)$ then $(a_0,a_1)\in E(X)$
   but $E$ is a sub-Tambara functor of $S\tm S$ so
   $(f_\om(a_0),f_\om(a_1))\in E(Y)$ so
   $\pi(f_\om(a_0))=\pi(f_\om(a_1))$.  We thus have a well-defined
   operation $f_\om\:(S/E)(X)\to (S/E)(Y)$ given by
   $\pi(a_0)\mapsto\pi(f_\om(a_0))$.  Moreover, for any finite
   $G$-sets $X$ and $Y$ we have $S(X\amalg Y)=S(X)\tm S(Y)$ and
   $E(X\amalg Y)=E(X)\tm E(Y)$; it follows directly from this that
   $(S/E)(X\amalg Y)=(S/E)(X)\tm(S/E)(Y)$.  The rest is now clear.
  \item[(c)] Left to the reader.
 \end{itemize}
\end{proof}

\begin{proposition}\label{prop-tambara-coeq}
 The category of Tambara functors has coequalisers, as does the
 category of Mackey functors.  
\end{proposition}
\begin{proof}
 Consider a pair of Tambara morphisms $\phi,\psi\:S'\to S$.  Let $\CE$
 be the collection of all congruences $E\leq S\tm S$ such that for all
 finite $G$-sets $X$ and all elements $a'\in S'(X)$ we have
 $(\phi(x),\psi(x))\in E(X)$.  (Note that any congruence is determined
 by the subsets $E(G/H)\sse S(G/H)\tm S(G/H)$ for all $H\leq G$, so
 $\CE$ is a set rather than a proper class.)  Now put 
 \[ E_1(X) =
    \{(a,b)\in S(X)\tm S(X) \st
      (a,b)\in E(X) \text{ for all } E\in\CE\}.
 \]
 One can check that this is itself a congruence with
 $(\phi(a'),\psi(a'))\in E_1(X)$ for all $a'\in S'(X)$, so $E_1$ is
 the smallest element of $\CE$.  Using Proposition~\ref{prop-cong} we
 see that the projection $\pi\:S\to S/E_1$ is a coequaliser for $\phi$
 and $\psi$.  The proof for Mackey functors is the same.
\end{proof}

\begin{corollary}\label{cor-tambara-cocomplete}
 The category of Tambara functors has colimits for all set-indexed
 diagrams, as does the category of Mackey functors.
\end{corollary}
\begin{proof}
 For any diagram $S\:\CJ\to\Tambara_G$, there is a well-known way to
 construct the colimit as the coequaliser of a pair of maps
 \[ \xymatrix{ \coprod_{u\:i\to j} S(i)
      \ar@/^1ex/[rr] \ar@/_1ex/[rr] && \coprod_k S(k).
 } \]
 The proof for Mackey functors is the same.
\end{proof}

\begin{proposition}\label{prop-reflexive-coeq}
 Let $\phi,\psi\:S'\to S$ be morphisms of Tambara functors, with
 coequaliser $\pi\:S\to Q$.  Suppose that there is a Tambara morphism
 $\sg\:S\to S'$ with $\phi\sg=\psi\sg=1$ (or in other words, that we
 have a reflexive coequaliser diagram).  Then for each finite $G$-set
 $X$, the map $\pi\:S(X)\to Q(X)$ is a coequaliser (in the category of
 sets) for the maps $\phi,\psi\:S'(X)\to S(X)$.
\end{proposition}
\begin{proof}
 For each finite $G$-set $X$ we put 
 \begin{align*}
  F(X) &= \img((\phi,\psi)\:S'(X)\to S(X)\tm S(X)) \\
  E(X) &= \text{ the smallest equivalence relation on $S(X)$
   containing $F(X)$ }.
 \end{align*}
 We claim that $E$ is actually a sub-Tambara functor of $S\tm S$.
 Assuming this, we see that it is actually a congruence on $S$, and
 thus that it is the smallest congruence containing $F$, so $Q=S/E$.
 It is clear by construction that $S(X)/E(X)$ is the coequaliser of
 the maps $\phi,\psi\:S'(X)\to S(X)$, so the proposition will follow.  

 To see that $E$ is a sub-Tambara functor, we first consider a bispan
 $\om$ from $X$ to $Y$.  Put
 \[ E^*(X) =
      \{(a,b)\in S(X)\tm S(X)\st (f_\om(a),f_\om(b))\in E(Y)\}.
 \]
 As $E(Y)$ is an equivalence relation, it follows easily that $E^*(X)$
 is an equivalence relation.  Moreover, as $F$ is a sub-Tambara
 functor we have $(f_\om\tm f_\om)(F(X))\sse F(Y)\sse E(Y)$, so
 $F(X)\sse E^*(X)$.  Using the minimality property of $E(X)$ we deduce
 that $E(X)\sse E^*(X)$, which means that
 $(f_\om\tm f_\om)(E(X))\sse E(Y)$.  This proves that $E$ is a
 subfunctor of $S\tm S$.  

 Now let $X$ and $Y$ be any two finite $G$-sets.  By applying the
 above to the bispans $R_i\in\bCU_G(X\amalg Y,X)$ and
 $R_j\in\bCU_G(X\amalg Y,Y)$ we see that
 $E(X\amalg Y)\sse E(X)\tm E(Y)$.  We would like to prove that in fact
 $E(X\amalg Y)=E(X)\tm E(Y)$.  To see this, consider an element 
 $a\in S(X)$ and put 
 \[ E[a](Y) =
     \{(b_0,b_1)\in S(Y)\tm S(Y)\st
         ((a,b_0),(a,b_1))\in E(X\amalg Y)\}.
 \]
 Note that for any $a'\in S'(Y)$ we have a point
 $(\sg(a),a')\in S'(X\amalg Y)$ with $\phi(\sg(a),a')=(a,\phi(a'))$
 and $\psi(\sg(a),a')=(a,\psi(a'))$ so 
 \[ ((a,\phi(a')),\;(a,\psi(a')))\in F(X\amalg Y). \]
 This shows that $F(Y)\sse E[a](Y)$, and $E[a](Y)$ is easily seen to
 be an equivalence relation, so $E(Y)\sse E[a](Y)$.  In other words,
 whenever $a\in S(X)$ and $(b_0,b_1)\in E(Y)$ we have
 $((a,b_0),(a,b_1))\in E(X\amalg Y)$.  By a similar argument, whenever
 $(a_0,a_1)\in E(X)$ and $b\in S(Y)$ we have
 $((a_0,b),(a_1,b))\in E(X\amalg Y)$.  As $E(X\amalg Y)$ is a
 transitive relation, we deduce that for all $(a_0,a_1)\in E(X)$ and
 all $(b_0,b_1)\in E(Y)$ we have
 $((a_0,b_0),(a_1,b_1))\in E(X\amalg Y)$.  Thus, we have
 $E(X)\tm E(Y)\sse E(X\amalg Y)$ as required.  This completes the
 proof that $E$ is a sub-Tambara functor of $S\tm S$, and the
 proposition follows from this as we explained earlier.
\end{proof}

Congruences in an additively complete Tambara functor biject with
Tambara ideals, which have been studied in detail by
Nakaoka~\cite{na:itf}.  We recall the main definition:
\begin{definition}\label{defn-tambara-ideal}
 Let $S$ be an additively complete Tambara functor.  An \emph{ideal}
 in $S$ is a collection of ideals $I(X)\leq S(X)$ such that for every
 $f\:X\to Y$ we have
 \begin{align*}
  R_f(I(Y)) & \sse I(X) \\
  T_f(I(X)) & \sse I(Y) \\
  N_f(I(X)) & \sse I(Y) + N_f(0).
 \end{align*}
\end{definition}
Recall here that $N_f(0)$ is $(0,1)$ with respect to the splitting 
\[ S(Y) = S(f(X)) \tm S(Y\sm f(X)), \]
as was proved in Lemma~\ref{lem-norm-zero}.

\begin{proposition}\label{prop-congruence-ideals}
 Let $S$ be an additively complete Tambara functor.  Then for any
 congruence $E$ on $S$ we can define a Tambara ideal $I_E$ by
 \[ I_E(X)=\{a\in S(X)\st (a,0)\in E(X)\}. \]
 Moreover, the construction $E\mapsto I_E$ gives a bijection between
 congruences and ideals.
\end{proposition}
This is more or less clear from the results of~\cite{na:itf}, but we
will spell it out.
\begin{proof}
 Let $\pi$ denote the usual quotient morphism from $S$ to $S/E$.  As
 $S(X)$ has additive inverses, the same is true of the quotient
 $(S/E)(X)$, so $(S/E)(X)$ is also a ring rather than just a
 semiring.  Moreover, $I_E(X)$ is the kernel of the map
 $\pi\:S(X)\to(S/E)(X)$.  As $R_f$ and $T_f$ are additive
 homomorphisms that commute with $\pi$, it is clear that
 $R_f(I_E(Y))\sse I_E(X)$ and $T_f(I_E(X))\sse I_E(Y)$.  Now suppose
 we have an element $a\in I_E(X)$.  Then $(a,0)\in E(X)$ but $E$ is a
 sub-Tambara functor of $S\tm S$ so $(N_f(a),N_f(0))\in E(Y)$.  Also
 $E(Y)$ is an equivalence relation, so it contains
 $(-N_f(0),-N_f(0))$, and it is closed under addition, so
 $(N_f(a)-N_f(0),0)\in E(Y)$.  This means that $N_f(a)-N_f(0)\in
 I_E(Y)$, so $N_f(I_E(X))\sse I_E(Y)+N_f(0)$ as required.

 Now let $J$ be an arbitrary Tambara ideal, and put 
 \[ F_J(X) = \{(a,b)\in S(X)\st a-b\in J(X)\}.  \]
 It is clear that this is a Mackey functor congruence; we claim that
 it is also closed under norm maps.  Suppose we have
 $(a,b)\in F_J(X)$, so $a=b+c$ for some $c\in J(X)$.  The pair $(b,c)$
 gives an element of the semiring 
 \[ S(X\tm\{0,1\})\simeq S(X\amalg X)\simeq S(X)\tm S(X). \]
 The distributor for
 \[ X\tm\{0,1\} \xra{\text{proj}} X \xra{f} Y \]
 is easily identified with the diagram
 \[ X\tm\{0,1\} \xla{(p_0,p_1)} A_0\amalg A_1
     \xra{(q_0,q_1)} B \xra{r} Y,
 \]
 where
 \begin{align*}
  A_0 &= \{(x,C)\st x\in X,\; C\sse f^{-1}\{f(x)\},\; x\not\in C\} \\
  A_1 &= \{(x,C)\st x\in X,\; C\sse f^{-1}\{f(x)\},\; x\in C\} \\
  B &= \{(y,C) \st y\in Y,\; C\sse f^{-1}\{y\}\} \\
  p_i(x,C) &= (x,i) \\
  q_i(x,C) &= (f(x),C) \\
  r(y,C) &= y.
 \end{align*}
 It follows that if we put 
 \begin{align*}
  b' &= R_{p_0}(b)\in S(A_0) & b'' &= N_{q_0}(b') \in S(B) \\
  c' &= R_{p_1}(c)\in J(A_1) & c'' &= N_{q_1}(c') \in J(B)+N_{q_1}(0)
 \end{align*}
 then 
 \[ N_f(a)=N_f(b+c)=N_fT_{\text{proj}}(b,c)=T_r(b''\;c''). \]
 Now split $B$ as $B_0\amalg B_1$, where 
 \begin{align*}
  B_0 &=\{(y,\emptyset)\st y\in Y\} \simeq Y\} \\
  B_1 &=\{(y,C)\in B \st C\neq\emptyset\} = q_1(A_1),
 \end{align*}
 and put $r_i=r|_{B_i}$.   We write $b''=(b''_0,b''_1)$ and
 $c''=(c''_0,c''_1)$ with respect to the splitting $S(B)=S(B_0)\tm
 S(B_1)$.  As $c''\in J(B)+N_{q_1(0)}$ and $B_1=q_1(A_1)$ we have
 $c''_0\in 1+J(B_0)$ and $c''_1\in J(B_1)$, so 
 \[ N_f(a) = T_{r_0}(b''_0c''_0) + T_{r_1}(b''_1c''_1) \in 
     J(Y) + T_{r_0}(b''_0).
 \] 
 Now $r_0$ is an isomorphism, and the preimage of $B_0$ in
 $A_0\amalg A_1$ is $\{(x,\emptyset)\st x\in X\}\sse A_0$, and this is
 mapped isomorphically by $p_0$ to $X$.  Using this it is not hard to
 see that $T_{r_0}(b''_0)=N_f(b)$.  We thus have
 $N_f(a)\in N_f(b)+J(Y)$, so $N_f(a,b)\in F_J(Y)$ as required.  

 It is clear that the constructions $E\mapsto I_E$ and $J\mapsto F_J$
 are inverse to each other, so we have a bijection between congruences
 and ideals.
\end{proof}

\section{Semigroup semirings}
\label{sec-semigroup-semirings}

For any set $M$, we have a free semigroup $\N[M]$ with one basis
element (denoted $[m]$) for each element $m\in M$.  If $M$ is itself a
semigroup, we can introduce a product on the set $\N[M]$ by the usual
rule 
\[ (\sum_ia_i[m_i]) (\sum_jb_j[n_j]) = \sum_{i,j} a_ib_j[m_i+n_j]
\]
(so $[m][n]=[m+n]$).  This makes $\N[M]$ into a semiring.  Conversely,
if $S$ is a semiring, we write $US$ for $S$ considered as a semigroup
under multiplication.  It is easy to see that there is an adjunction
\[ \Semirings(\N[M],S) \simeq \Semigroups(M,US). \]

Our object in this section is to set up an analogous theory for Mackey
functors and Tambara functors, which reduces to the above in the case
where the group of equivariance is trivial.  This construction was
also given by Nakaoka~\cite{na:tmf}.

\begin{definition}\label{defn-US}
 We define $U\:\Tambara_G\to\Mackey_G$ as follows.  For a Tambara
 functor $S$, the corresponding Mackey functor $US$ is given on
 objects by $US(X)=S(X)$.  For any map $f\:X\to Y$, the operator
 $R^{US}_f\:US(Y)\to US(X)$ is just the same as
 $R^S_f\:S(Y)\to S(X)$, and the operator $T^{US}_f\:US(X)\to US(Y)$ is
 the same as $N^S_f\:S(X)\to S(Y)$.  We call $US$ the \emph{underlying
 multiplicative Mackey functor} of $S$.
\end{definition}

Our problem is to understand the left adjoint to $U$.

\begin{definition}\label{defn-AMX}
 Let $M$ be a Mackey functor.  For any finite $G$-set $X$, we define a
 groupoid $\CA[M](X)$ as follows.  The objects are triples $(U,u,m)$,
 where $U$ is a finite $G$-set and $u\:U\to X$ is an equivariant map
 and $m\in M(U)$.  The morphisms from $(U,u,m)$ to $(U',u',m')$ are
 the equivariant bijections $p\:U\to U'$ for which $u'p=u$ and
 $R_p(u')=u$ (or equivalently $u'=T_p(u)$).  We write $A[M](X)$ for
 the set of isomorphism classes in this groupoid.  We also write
 $[U,u,m]$ for the isomorphism class of $(U,u,m)$.
\end{definition}

\begin{definition}
 Suppose we have an object $(U,u,m)\in\CA[M](X)$.
 \begin{itemize}
  \item[(a)] For any map $g\:X\to Y$, we define
   $T_g(U,u,m)=(U,gu,m)\in\CA[M](Y)$.  There is an evident way to
   define an action on morphisms so that this becomes a functor
   $T_f\:\CA[M](X)\to\CA[M](Y)$.  It therefore induces an operation
   $T_f\:A[M](X)\to A[M](Y)$.
  \item[(b)] We also define $N_g(U,u,m)\in\CA[M](Y)$ as follows: we
   form the distributor $\Dl(u,g)=(U\xla{p}A\xra{q}B\xra{r}Y)$, then
   put $N_g(U,u,m)=(B,r,T_qR_p(m))$.  This again induces an operation
   $N_f\:A[M](X)\to A[M](Y)$.
  \item[(c)] Suppose instead we have a map $f\:W\to X$.  We form the
   pullback square 
   \[ \xymatrix{
    \tU \ar[r]^{\tu} \ar[d]_{\tf} & W \ar[d]^f \\
    U \ar[r]_u & X
   } \]
   and then put $R_f(U,u,m)=(\tU,\tu,R_{\tf}(m))$.  This gives an
   operation $R_f\:A[M](X)\to A[M](W)$.
 \end{itemize}
\end{definition}

\begin{proposition}
 The above definitions make $A[M]$ into a Tambara functor.  Moreover,
 this construction gives a functor $A[-]\:\Mackey_G\to\Tambara_G$,
 which is left adjoint to $U$.
\end{proposition}

The rest of this section will constitute the proof.

\begin{lemma}\label{lem-AM-RR-TT}
 For any $X\xra{g}Y\xra{h}Z$ we have
 \begin{align*}
  R_{hg} &= R_gR_h\:A[M](Z)\to A[M](X) \\
  T_{hg} &= T_hT_g\:A[M](X)\to A[M](Z).
 \end{align*}
\end{lemma}
\begin{proof}
 Straightforward.
\end{proof}

\begin{lemma}\label{lem-AM-NN}
 For any $X\xra{g}Y\xra{h}Z$ we have
 $N_{hg}=N_hN_g\:A[M](X)\to A[M](Z)$.
\end{lemma}
\begin{proof}
 Consider an object $(W,f,m)\in\CA[M](X)$.  We have a chain of maps
 $W\xra{f}X\xra{g}Y\xra{h}Z$ which we use to build a diagram 
 \[ \xymatrix{
  \tA \ar[r]^\al \ar[d]_i &
  A \ar[r]^p \ar[d]^q &
  W \ar[r]^f &
  X \ar[d]^g \\
  \tB \ar[r]_\bt \ar[d]_j &
  B \ar[rr]_r & &
  Y \ar[d]^h \\
  \tC \ar[rrr]_k &&& 
  Z
 } \]
 as in Proposition~\ref{prop-NNT}.  Now
 $\Dl(f,g)=(W\xla{p}A\xra{q}B\xra{r}Y)$, so
 $N_g[W,f,m]=[B,r,T_qR_p(m)]$.  To apply $N_h$ to this, we use the
 distributor $\Dl(r,h)=(B\xla{\bt}\tB\xra{j}\tC\xra{k}Z)$, giving
 $N_hN_g[W,f,m]=[\tC,k,T_jR_\bt T_qR_p(m)]$.  As the top left square
 is cartesian we have $R_\bt T_q=T_iR_\al$, so 
 $N_hN_g[W,f,m]=[\tC,k,T_{ji}R_{p\al}(m)]$.  We also know that
 $\Dl(f,hg)=(W\xla{p\al}\tA\xra{ji}\tC\xra{k}Z)$, so this is the same
 as $N_{hg}[W,f,m]$, as required.
\end{proof}

\begin{lemma}\label{lem-AM-NT}
 Suppose we have maps $X\xra{g}Y\xra{h}Z$ with distributor
 $\Dl(g,h)=(X\xla{p}A\xra{q}B\xra{r}Z)$.  Then
 $N_hT_g=T_rN_qR_p\:A[M](X)\to A[M](Z)$. 
\end{lemma}
\begin{proof}
 Consider an object $(W,f,m)\in\CA[M](X)$.  We use the diagram 
 \[ \xymatrix{
  W \ar@{=}[d] &
  \tA \ar[l]_{\tp} \ar[d]^i \ar[r]^{\tq} &
  \tB \ar[r]^{\tr} \ar[dd]^k &
  Z \ar@{=}[dd] \\
  W \ar[d]_f &
  A^* \ar[l]_{p^*} \ar[d]^j \\
  X &
  A \ar[l]^p \ar[r]_q &
  B \ar[r]_r &
  Z
 } \]
 constructed in Proposition~\ref{prop-NTT}.  We have
 $T_g(W,f,m)=(W,gf,m)$ and
 $\Dl(gf,h)=(W\xla{\tp}\tA\xra{\tq}\tB\xra{\tr}Z)$ so
 $N_hT_g(W,f,m)=(\tB,\tr,N_{\tq}R_{\tp}(m))$.  On the other hand, as
 the bottom left square is cartesian, we have
 $R_p(W,f,m)=(A^*,j,R_{p^*}(m))$.  As
 $\Dl(j,q)=(A^*\xla{i}\tA\xra{\tq}\tB\xra{k}B)$, this gives 
 \[ N_qR_p(W,f,m)=(\tB,k,N_{\tq}R_iR_{p^*}(m)) 
     = (\tB,k,N_{\tq}R_{\tp}(m)). 
 \]
 We now apply $T_r$, noting that $rk=\tr$, to get 
 \[ T_rN_qR_p(W,f,m)
     = (\tB,\tr,N_{\tq}R_{\tp}(m)) = N_hT_g(W,f,m)
 \]
 as required.
\end{proof}

\begin{lemma}\label{lem-AM-NRT}
 Suppose we have a pullback square
 \[ \xymatrix{
  \tX \ar[d]_j \ar[r]^{\tg} & \tY \ar[d]^k \\
  X \ar[r]_g & Y.
 } \]
 Then $R_kT_g=T_{\tg}R_j$ and $R_kN_g=N_{\tg}R_j$ as operators
 $A[M](X)\to A[M](\tY)$.
\end{lemma}
\begin{proof}
 Consider an object $z=(W,f,m)\in A[M](X)$.  First note that
 $T_g(z)=(W,gf,m)$.  Let $\tW$ be the pullback of $\tW$ along $f$, so
 we have a diagram  
 \[ \xymatrix{
  \tW \ar[d]_i \ar[r]^{\tf} & \tX \ar[d]_i \ar[r]^{\tg} & \tY \ar[d]^j \\
  W \ar[r]_f & X \ar[r]_g & Y }
 \]
 in which both squares are cartesian.  It follows that the outer
 rectangle is also cartesian, so
 $R_kT_g(z)=R_k(W,gf,m)=(\tW,\tg\tf,R_i(m))$.  On the other hand, we
 have $R_j(z)=(\tW,\tf,R_i(m)))$, so
 $T_{\tg}R_j(z)=(\tW,\tg\tf,R_i(m))$, which is the same as
 $R_kT_g(z)$, as required.

 Now consider the diagram 
 \[ \xymatrix{
  \tW \ar[d]_i &
  \tA \ar[l]_{\tp} \ar[r]^{\tq} \ar[d]_\al &
  \tB \ar[r]^{\tr} \ar[d]^\bt &
  \tY \ar[d]^k \\
  W &
  A \ar[l]^p \ar[r]_q &
  B \ar[r]_r &
  Y
 } \]
 constructed in Proposition~\ref{prop-NRT}.  The bottom row is
 $\Dl(f,g)$, so we have $N_g(z)=(B,r,T_qR_p(m))$.  As the right hand
 square is cartesian, this gives
 $R_kN_g(z)=(\tB,\tr,R_\bt T_qR_p(m))$.  As the middle square is also
 cartesian we have $R_\bt T_q=T_{\tq}R_\al\:M(A)\to M(\tB)$.  As
 $p\al=i\tp$ we also have $R_\al R_p=R_{\tp}R_i\:M(W)\to M(\tA)$.
 Putting this together, we get
 $R_kN_g(z)=(\tB,\tr,T_{\tq}R_{\tp}R_i(m))$.  On the other hand, we
 have $R_j(z)=(\tW,\tf,R_i(m))$ and $\Dl(\tf,\tg)$ is the top row of
 the above diagram so we also have
 $N_{\tg}R_j(z)=(\tB,\tr,T_{\tq}R_{\tp}R_i(m))$ as required.
\end{proof}

\begin{corollary}\label{cor-AM-tambara}
 The above definitions make $A[M]$ into a Tambara functor.
\end{corollary}
\begin{proof}
 It will suffice to check the conditions in
 Proposition~\ref{prop-UG-pres}.  Part~(a) is covered by
 Lemmas~\ref{lem-AM-RR-TT} and~\ref{lem-AM-NN}.  Part~(b) is
 Lemma~\ref{lem-AM-NT}, and part~(c) is Lemma~\ref{lem-AM-NRT}.
\end{proof}

\begin{definition}
 For any Mackey functor $M$, we define $\eta\:M(X)\to UA[M](X)$ by
 $\eta(m)=[X,1,m]$.  For any Tambara functor $S$, we define
 $\ep\:A[US](X)\to S(X)$ by $\ep[W,f,m]=T_f(m)$.
\end{definition}

\begin{proposition}\label{prop-A-U-adjoint}
 The above definitions give natural maps that satisfy the triangular
 identities and so display $A[-]\:\Mackey_G\to\Tambara_G$ as left
 adjoint to $U\:\Tambara_G\to\Mackey_G$.
\end{proposition}
\begin{proof}
 We first claim that $\eta$ gives a morphism of Mackey functors.  For
 any map $g\:X\to Y$, we must show that the diagrams
 \[ \xymatrix{
   M(X) \ar[d]_\eta \ar[r]^{T^M_g} &
   M(Y) \ar[d]^\eta  &
   M(Y) \ar[d]_\eta \ar[r]^{R^M_g} &
   M(X) \ar[d]^\eta \\
   UA[M](X) \ar[r]_{T^{UA[M]}_g} &
   UA[M](Y) &
   UA[M](Y) \ar[r]_{R^{UA[M]}_g} &
   UA[M](X)
 } \]
 commute.  By definition, these are the same as
 \[ \xymatrix{
   M(X) \ar[d]_\eta \ar[r]^{T^M_g} &
   M(Y) \ar[d]^\eta  &
   M(Y) \ar[d]_\eta \ar[r]^{R^M_g} &
   M(X) \ar[d]^\eta \\
   A[M](X) \ar[r]_{N^{A[M]}_g} &
   A[M](Y) &
   A[M](Y) \ar[r]_{R^{A[M]}_g} &
   A[M](X)
 } \]
 Consider an element $m\in M(X)$.  The distributor $\Dl(1,g)$ is just
 $(X\xla{1}X\xra{g}Y\xra{1}Y)$, so 
 \[ N_g\eta(m)=N_g[X,1,m]=[Y,N_gR_1(m)]=\eta(N_g(m)), \]
 which proves that the left square commutes.  Consider instead an
 element $n\in M(Y)$.  As the square
 \[ \xymatrix{
  X \ar[r]^1 \ar[d]_g & X \ar[d]^g \\
  Y \ar[r]_1 & Y
 } \] 
 is cartesian, we have 
 \[ R_g\eta(n)=R_g[Y,1,n]=[X,1,R_g(n)]=\eta R_g(n) \]
 as required.  It is also clear that $\eta$ is natural with respect to
 morphisms $M\to M'$ of Mackey functors.

 We next claim that $\ep$ gives a morphism of Tambara functors.  For
 any map $g\:X\to Y$, we must show that the diagrams 
 \[ \xymatrix{
  A[US](X) \ar[d]_\ep \ar[r]^{T^{A[US]}_g} &
  A[US](Y) \ar[d]^\ep &
  A[US](X) \ar[d]_\ep \ar[r]^{N^{A[US]}_g} &
  A[US](Y) \ar[d]^\ep &
  A[US](Y) \ar[d]_\ep \ar[r]^{R^{A[US]}_g} &
  A[US](X) \ar[d]^\ep  \\
  S(X) \ar[r]_{T^S_g} &
  S(Y) &
  S(X) \ar[r]_{N^S_g} &
  S(Y) &
  S(Y) \ar[r]_{R^S_g} &
  S(X) &
 } \]
 commute.  Consider an element $a=[W,f,m]\in A[US](X)$.  We then have 
 \[ \ep(T_g(a)) = \ep[W,gf,m] = T_{gf}(m) = T_g(T_f(m)) = T_g\ep(a),
 \]
 which proves that the first diagram commutes.  Now consider the
 distributor $\Dl(f,g)=(W\xla{p}A\xra{q}B\xra{r}Y)$.  We have
 \[ \ep(N_g(a)) = \ep[B,r,N_qR_p(m)] = T_rN_qR_p(m) = N_gT_f(m) =
     N_g\ep(a),
 \] 
 which proves that the middle diagram commutes.

 Now consider instead an element $b=[V,e,n]\in A[US](Y)$.  We form a
 pullback square
 \[ \xymatrix{
  W \ar[r]^j \ar[d]_f & V \ar[d]^e \\
  X \ar[r]_g & Y
 } \]
 so that $R_g(b)=[W,f,R_j(n)]$.  This gives 
 \[ \ep R_g(b)=T_fR_j(n)=R_gT_e(n)= R_g\ep(b), \]
 which proves that the right hand diagram commutes.  It is also clear
 that $\ep$ is natural with respect to morphisms $S\to S'$ of Tambara
 functors.  

 This just leaves the triangular identities.  The first of these says
 that the composite 
 \[ US(X) \xra{\eta} UA[US](X) \xra{U\ep} US(X) \]
 is the identity.  For any $s\in US(X)=S(X)$ we have
 $\ep(\eta(s))=\ep[X,1,s]=T^{US}_1(s)=s$ as required.  The second one
 says that the composite
 \[ A[M](X) \xra{A[\eta]} A[UA[M]](X) \xra{\ep} A[M](X) \]
 is also the identity.  Consider an element $a=[W,f,m]\in A[M](X)$.  We
 then have $A[\eta](a)=[W,f,\eta(m)]=[W,f,[W,1,m]]$, so 
 \[ \ep(A[\eta](a))=T^{A[M]}_f [W,1,m]=[W,f,m] = a, \]
 as required.
\end{proof}

\section{Green semirings}
\label{sec-green}

\begin{definition}\label{defn-green}
 A \emph{Green semiring} is a Mackey functor $S$ equipped with a
 morphism $\mu\:S\btm S\to S$ that is commutative, associative and
 unital in the obvious sense.
\end{definition}

We will not make much use of Green semirings.  Our purpose in this
section is simply to explain some nonobvious ways to understand their
relationship with Tambara functors.

\begin{proposition}\label{prop-green-explicit}
 Let $S$ be a Mackey functor.  To make $S$ into a Green semiring is
 the same as to give each set $S(X)$ a semiring structure (extending
 the usual semigroup structure) in such a way that
 \begin{itemize}
  \item[(a)] For every map $f\:X\to Y$ of finite $G$-sets, the
   resulting map $R_f\:S(Y)\to S(X)$ is a homomorphism of semirings.
  \item[(b)] Moreover, for any $a\in S(X)$ and $b\in S(Y)$ we have
   $T_f(a\;R_f(b))=T_f(a)\;b$ in $S(Y)$.
 \end{itemize}
\end{proposition}
\begin{proof}
 This follows easily from Proposition~\ref{prop-box-equiv}. 
\end{proof}

\begin{definition}\label{defn-green-bispan}
 We say that an equivariant map $f\:U\to V$ is \emph{green} if for all
 $u\in U$, the image $f(u)\in V$ has the same isotropy group as $u$.
 We say that a bispan $\om=(X\xla{p}A\xra{q}B\xra{r}Y)$ is
 \emph{green} if the map $q$ is green.
\end{definition}

Here of course it is automatic that the isotropy group of $f(u)$ is at
least as large as that of $u$; the force of the condition is that it
can be no larger.  

\begin{proposition}\label{prop-green-composite}
 The composite of any two composable green bispans is again green, and
 all identity bispans are green, so the green bispans give a
 subcategory $\bCG_G\sse\bCU_G$ (containing all the objects but not
 all the morphisms).  Moreover, this subcategory is closed under products.
\end{proposition}
\begin{proof}
 Consider a pair of green bispans 
 \begin{align*}
  \om_0 &= (X_0\xla{p_0}A_0\xra{q_0}B_0\xra{r_0}X_1) \in \CU(X_0,X_1) \\
  \om_1 &= (X_1\xla{p_1}A_1\xra{q_1}B_1\xra{r_1}X_2) \in \CU(X_1,X_2).
 \end{align*}
 The composite is the bispan 
 \[ \om = (X_0\xla{p}A\xra{q}B\xra{r}X_2) \in \CU(X_0,X_2)
 \] 
 given by
 \begin{align*}
  A &= \{(a_0,a_1,s) \st 
         s\:q_1^{-1}\{q_1(a_1)\}\to B_0,\;r_0s=p_1,\;
         a_0\in q_0^{-1}\{s(a_1)\}\} \\
  B &= \{(b_1,s) \st 
          s\:q_1^{-1}\{b_1\}\to B_0,\;\;r_0s=p_1\} \\
  p(a_0,a_1,s) &= p_0(a_0) \\
  q(a_0,a_1,s) &= (q_1(a_1),s) \\
  r(b_1,s)     &= r_1(b_1).
 \end{align*}
 Consider a point $a=(a_0,a_1,s)\in A$ and an element $g\in G$ that
 satisfies $g.q(a)=q(a)$; we must show that $g.a=a$.  The condition
 $g.q(a)=q(a)$ means that $g.q_1(a_1)=q_1(a_1)$ (so $g$ preserves the
 fibre $F=q_1^{-1}\{q_1(a_1)\}$) and that the map $s\:F\to B_0$
 satisfies $g\circ s=s\circ g$.  As $\om_1$ is green and
 $g.q_1(a_1)=q_1(a_1)$ we must have $g.a_1=a_1$.  As $s$ is
 $g$-equivariant this gives $g.s(a_1)=s(a_1)$.  As $a\in A$ we also
 have $q_0(a_0)=s(a_1)$, so $g.q_0(a_0)=q_0(a_0)$.  As $\om_0$ is
 green we can conclude that $g.a_0=a_0$.  As $a_0$ and $a_1$ are fixed
 by $g$ and $s$ is $g$-equivariant we see that the triple
 $a=(a_0,a_1,s)$ has $g.a=a$, as required.  This shows that composites
 of green bispans are green, and the corresponding fact for identities
 is trivial.

 Now consider a pair of objects $Y,Z\in\bCU_G$.  We have inclusions
 $Y\xra{i}Y\amalg Z\xla{j}Z$, and we have seen that the resulting
 diagram $Y\xla{R_i}Y\amalg Z\xra{R_j}Z$ is a product diagram in
 $\bCU_G$.  Using the proof of this we see that a bispan
 $\al\in\CU_G(U,Y\amalg Z)$ is green iff the composites $R_i\circ\al$
 and $R_j\circ\al$ are green.  The bispans $R_i$ and $R_j$ themselves
 are also clearly green.  It follows that our diagram is still a
 product diagram in $\bCG_G$.
\end{proof}

\begin{definition}\label{defn-green-functor}
 A \emph{Green functor} is a product-preserving functor
 $\bCG_G\to\Sets$.  Any Green functor has operators $T_f$, $N_f$ and
 $R_f$ just as for Tambara functors, except that $N_f$ is only defined
 when $f$ is green.  We write $\Green_G$ for the category of Green
 functors, and $\Green'_G$ for the category of Green semirings.  
\end{definition}

The rest of this section is devoted to constructing an equivalence
$\Green_G\simeq\Green'_G$.  

\begin{lemma}\label{lem-green-proj}
 Any projection map $\pi\:I\tm V\to V$ (where $G$ acts trivially on
 $I$) is green.  Conversely, if $p\:U\to V$ is green and $V$ is a
 $G$-orbit then $p$ can be written as a composite
 $U\xra{m}I\tm V\xra{\pi}V$, where $I$ is again $G$-fixed and $m$ is
 an equivariant bijection. 
\end{lemma}
\begin{proof}
 The first claim is clear.  For the second claim, choose a point
 $v_1\in V$, and let $H$ be the isotropy group of $v_1$.  Put
 $I=p^{-1}\{v_1\}\sse U$.  For each $u\in U$ we claim that the orbit
 $Gu$ meets $I$ in a single point.  Indeed, as $V$ is an orbit, we
 must have $v_1=g.p(u)$ for some $g$, which means that $g.u\in I$.  If
 we also have $g'.u\in I$ then
 $g'g^{-1}.v_1=g'g^{-1}.p(gu)=p(g'u)=v_1$, so $g'g^{-1}\in H$.
 However, the green condition means that $H$ also stabilises every
 point in $I$, so $g'g^{-1}.gu=gu$, or in other words $g'u=gu$ as
 required.  We can thus define $m_1(u)$ to be the unique point in
 $Gu\cap I$; this gives a map $m_1\:U\to I$ which is equivariant if we
 give $I$ the trivial action.  We now put
 $m(u)=(m_1(u),p(u))\in I\tm V$ (so $\pi m=p$).  Suppose that
 $m(u)=m(u')$.  As $m_1(u)=m_1(u')$ we see that $u$ and $u'$ lie in
 the same orbit, say $u'=g.u$.  As $p(u')=p(u)$ we see that $g$
 stabilises $p(u)$ and thus (by the green condition) stabilises $u$,
 so $u'=u$.  This shows that $m$ is injective.  Now consider an
 arbitrary point $(u,v)\in I\tm V$.  As $V$ is an orbit there exists
 $g$ with $g.v_1=v$, and it follows that $m(g.u)=(u,v)$, proving that
 $m$ is surjective.
\end{proof}

\begin{lemma}\label{lem-green-norm}
 Let $S$ be a Green semiring, and let $f\:X\to Y$ be a green map.
 Then there is a unique function $N_f\:S(X)\to S(Y)$ with the
 following property: for every transitive $G$-set $U$ and every
 $G$-map $j\:U\to Y$ and every $m\in S(X)$, we have
 \[ R_jN_f(m)=\prod_{i\:U\to X,\;fi=j} R_i(m) \in S(U). \]
\end{lemma}
\begin{proof}
 We can write $Y=U_1\amalg\dotsb\amalg U_r$, where the sets $U_k$ are
 the orbits in $Y$, and let $j_k\:U_k\to Y$ be the inclusion.  As $S$
 is product-preserving we find that the maps $R_{j_k}\:S(Y)\to S(U_k)$
 give a bijection $\lm\:S(Y)\to\prod_kS(U_k)$.  Define
 $\nu_k\:S(X)\to S(U_k)$ by $\nu_k(m)=\prod_iR_i(m)$, where $i$ runs
 over equivariant maps $U_k\to X$ with $fi=j_k$.  These maps combine
 to give a map $\nu\:S(X)\to\prod_kS(U_k)$, and we put
 $N_f=\lm^{-1}\nu\:S(X)\to S(Y)$.  It is clear that this is the only
 map that can possibly have the stated property.

 Now let $U$ be an arbitrary transitive $G$-set, and let $j\:U\to Y$
 be a $G$-map. Then we must have $j=j_kp$ for some $k$
 and some surjective $G$-map $p\:U\to U_k$.  Let $I_k$ be the set of
 $G$-maps $i'\:U_k\to X$ with $fi'=j_k$, and let $I$ be the set of 
 $G$-maps $i\:U\to X$ with $fi=j$.  We have a map $p^*\:I_k\to I$
 given by $p^*(i')=i'p$, and using the fact that $f$ is green we see
 that this is bijective.  It follows that 
 \[ R_jN_f(m) =
     R_pR_{j_k}N_f(m) = 
     R_p\prod_{i'\in I_k}R_{i'}(m) =
     \prod_{i'\in I_k}R_{i'p}(m) = 
     \prod_{i\in I}R_i(m)
 \]
 as required.
\end{proof}

\begin{lemma}\label{lem-N-proj}
 For a projection map $\pi\:I\tm Y\to Y$ we have
 $N_\pi(m)=\prod_{i\in I}R_{\sg_i}(m)$, where $\sg_i\:Y\to I\tm Y$  is
 given by $\sg_i(y)=(i,y)$.
\end{lemma}
\begin{proof}
 Straightforward.
\end{proof}

\begin{lemma}\label{lem-green-functor}
 The maps $N_f$ as in Lemma~\ref{lem-green-norm} make $S$ into a Green
 functor. 
\end{lemma}
\begin{proof}
 By the evident analog of Proposition~\ref{prop-UG-pres}, it will
 suffice to check the following:
 \begin{itemize}
  \item[(a)] For all finite $G$-sets $X$, $Y$ and $Z$, and all
   $G$-maps $X\xla{f}Y\xra{g}Z$, we have $R_{gf}=R_fR_g\:S(Z)\to S(X)$
   and $T_{gf}=T_gT_f\:S(X)\to S(Z)$.  Moreover, if $f$ and $g$ are
   green, we also have $N_{gf}=N_gN_f\:S(X)\to S(Z)$.  
  \item[(b)] If we have $G$-maps $X\xra{f}Y\xra{g}Z$ with distributor
   $\Dl(f,g)=(X\xla{p}A\xra{q}B\xra{r}Z)$ (where $g$ is green, and $q$
   is therefore also green) then $N_gT_f=T_rN_qR_p\:S(X)\to S(Z)$.
  \item[(c)] For any cartesian square
   \[ \xymatrix{
    W \ar[r]^f \ar[d]_g & X \ar[d]^h \\
    Y \ar[r]_k & Z
   } \]
   (of finite $G$-sets and equivariant maps) we have
   $T_fR_g=R_hT_k\:S(X)\to S(Y)$.  If $f$ and $h$ are green, we also
   have $N_fR_g=R_hN_k\:S(X)\to S(Y)$.
 \end{itemize}
 In part~(a), the equations $R_{gf}=R_fR_g$ and $T_{gf}=T_gT_f$ hold
 because $S$ is assumed to be a Mackey functor.  It is also clear that
 the composite $N_gN_f$ has the property that defines $N_{gf}$.

 Now consider a square as in~(c).  The equation $T_fR_g=R_hT_k$ holds
 because $S$ is assumed to be a Mackey functor.  Suppose that $f$ (and
 therefore also $h$) is green, and consider an element $m\in S(Y)$.
 We need to show that $N_fR_g(m)=R_hN_k(m)\in S(X)$, and it will
 suffice to show that $R_jN_fR_g(m)=R_jR_hN_k(m)$ for all $G$-orbits
 $U$ and all injective $G$-maps $j\:U\to X$.  On the left hand side, we
 have $R_jN_fR_g(m)=\prod_{i\in I}R_iR_g(m)=\prod_{i\in I}R_{gi}(m)$,
 where $I$ is the set of $G$-maps $i\:U\to W$ with $fi=j$.  On the
 right hand side, we have 
 \[ R_jR_hN_k(m) = R_{hj}N_k(m) = \prod_{i'\in I'} R_{i'}(m), \]
 where $I'$ is the set of maps $i'\:U\to Y$ that satisfy $ki'=hj$.  As
 the square is cartesian, we see that the map $i\mapsto gi$ gives a
 bijection $I\to I'$, so $R_jN_fR_g(m)=R_jR_hN_k(m)$ as claimed.

 We now turn to~(b).  It is clear that everything works independently
 over the different orbits of $Z$, so we may assume that $Z$ itself is
 a single orbit.  Using Lemma~\ref{lem-green-proj} we can then reduce
 to the case where $Y=I\tm Z=\coprod_{i\in I}Z$ and $g\:I\tm Z\to Z$
 is just the projection.  We can then split the diagram
 $(X\xra{f}I\tm Y)$ as a disjoint union of diagrams
 $(X_i\xra{f_i}Z)$.  An element $m\in S(X)$ corresponds to a system of
 elements $m_i\in S(X_i)$, and using Lemma~\ref{lem-N-proj} we find
 that $N_gT_f(m)=\prod_iT_{f_i}(m_i)$.  Next, by inspecting the
 definitions we find that the distributor $\Dl(f,g)$ has the form
 $(X\xla{p}I\tm B\xra{\pi}B\xra{r}Z)$, where
 \begin{align*}
  B &= \{(z,u)\st z\in Z,\; u\in\prod_if_i^{-1}\{z\}\sse\prod_iX_i\}  \\
  p(i,z,u) &= u_i \\
  r(z,u) &= z.
 \end{align*}
 Using this together with Lemma~\ref{lem-N-proj} again we get
 \[ \textstyle T_rN_\pi R_p(m)=T_r(\prod_iR_{p_i}(m_i)), \]
 where $p_i(z,u)=p(i,z,u)=u_i$. 

 It will now be convenient to identify $I$ with $\{1,\dotsc,n\}$ say.  Put 
 \begin{align*}
  B' &= \{(z,u)\st z\in Z,\; u\in\prod_{i<n}f_i^{-1}\{z\}\sse\prod_iX_i\}  \\
  p'_i(z,u) &= u_i \hspace{6em} (\text{ for } i<n) \\
  r'(z,u) &= z
 \end{align*}
 and let $t\:B\to B'$ be the evident truncation map.  We then have a
 cartesian square as follows:
 \[ \xymatrix{ 
  B \ar[r]^{p_n} \ar[d]_t \ar[dr]^r & X_n \ar[d]^{f_n} \\
  B' \ar[r]_{r'} & Z
 } \]
 Using Frobenius reciprocity twice, and the fact that $p_i=p'_it$ for
 $i<n$, and the cartesian property of the square, we obtain
 \begin{align*}\textstyle
  T_r\left(\prod_iR_{p_i}(m_i)\right)
   &= T_{r'}T_t\left(R_{p_n}(m_n)\;R_t\prod_{i<n}R_{p'_i}(m_i)\right) \\
   &= T_{r'}\left(T_tR_{p_n}(m_n)\;\prod_{i<n}R_{p'_i}(m_i)\right) \\
   &= T_{r'}\left(R_{r'}T_{f_n}(m_n)\;\prod_{i<n}R_{p'_i}(m_i)\right)  \\
   &= T_{f_n}(m_n)\; T_{r'}\left(\prod_{i<n}R_{p'_i}(m_i)\right).
 \end{align*}
 By an evident inductive extension of this, we get 
 \[ T_r\left(\prod_iR_{p_i}(m_i)\right) = \prod_iT_{f_i}(m_i) \]
 or in other words $T_rN_\pi R_p(m)=N_gT_f(m)$ as claimed.
\end{proof}

\begin{lemma}\label{lem-green-forget}
 Let $S$ be a Green functor.  Then the underlying Mackey functor has a
 canonical structure as a Green semiring.
\end{lemma}
\begin{proof}
 We first claim that each set $S(X)$ has a canonical structure as a
 semiring.  Moreover, for every map $f\:X\to Y$:
 \begin{itemize}
  \item[(a)] The map $R_f\:S(Y)\to S(X)$ is a semiring homomorphism
   (which we use to regard $S(X)$ as an $S(Y)$-module).
  \item[(b)] The map $T_f\:S(X)\to S(Y)$ is a homomorphism of
   $S(Y)$-modules (and in particular, respects addition).
  \item[(c)] If $f$ is green then the map $N_f\:S(X)\to S(Y)$ sends
   $1$ to $1$ and respects multiplication. 
 \end{itemize}
 This is proved in the same way as
 Proposition~\ref{prop-tambara-semiring}; we just need to observe that
 the bispans used to define the semiring structure are all green.

 In view of Proposition~\ref{prop-green-explicit}, this gives the
 required Green semiring structure.
\end{proof}

\begin{proposition}\label{prop-green-equiv}
 The categories $\Green_G$ and $\Green'_G$ are equivalent.
\end{proposition}
\begin{proof}
 If $S$ is a Green semiring, we write $\Phi(S)$ for the same Mackey
 functor equipped with norm maps $N_f$ as given by
 Lemma~\ref{lem-green-norm}.  Using Lemma~\ref{lem-green-functor} we
 see that this gives a functor $\Phi\:\Green'_G\to\Green_G$.
 Similarly, Lemma~\ref{lem-green-forget} defines a functor
 $\Psi\:\Green_G\to\Green'_G$.  It is easy to see that these are
 inverse to each other. 
\end{proof}

\section{Additive completion}
\label{sec-completion}

Given a commutative semigroup $A$, there is a universal example of a
commutative group $A^+$ equipped with a homomorphism $A\to A^+$; we
call this the \emph{additive completion} of $A$.  In this section we
explain and extend Tambara's proof~\cite{ta:mt}*{Theorem 6.1} of the
following result:

\begin{theorem}\label{thm-tambara-plus}
 If $M$ is a Mackey functor then there is another Mackey functor $M^+$
 with $M^+(X)=M(X)^+$ for all $X$.  If $S$ is a Tambara functor then
 there is another Tambara functor $S^+$ with $S^+(X)=S(X)^+$ for all
 $X$.
\end{theorem}

The Mackey functor case is straightforward, because Mackey functors
can be regarded as additive functors from $\bCA_G$ to $\Semigroups$,
and additive completion gives an additive functor
$\Semigroups\to\Ab$.  The real problem is to handle the nonlinearity
of the norm operations for a Tambara functor.  The rest of this
section will constitute the proof.

We will use the following construction for $A^+$:
\begin{definition}\label{defn-plus}
 Let $A$ be a semigroup.  We put 
 \[ \hA = A^2 = \{(a_+,a_-)\st a_+,a_-\in A\}, \]
 and regard this as a semigroup using the obvious coordinatewise
 addition.  We then put 
 \[ E_A = \{((a_+,a_-),(b_+,b_-))\in \hA\tm\hA \st 
            a_++b_-+x=a_-+b_++x \text{ for some } x\in A \}
 \] 
 and check that this is a congruence on $\hA$.  It follows that the
 quotient set $\hA/E_A$ inherits a semigroup structure; we write $A^+$
 for this semigroup. 
\end{definition}

\begin{proposition}\label{prop-hS-semiring}
 Let $S$ be a semiring.  Then the rule
 \[ (a_+,a_-).(b_+,b_-) = (a_+b_++a_-b_-,\;a_+b_-+a_-b_+) \]
 makes $\hS$ into a semiring (with identity element $(1,0)$).
 Moreover, $E_S$ is a subsemiring of $\hS\tm\hS$.
\end{proposition}
\begin{proof}
 The semiring axioms for $\hS$ are straightforward and are left to the
 reader.  (Essentially, we have defined $\hS$ to be the group semiring
 $S[C_2]$.)  It is also clear that $E_S$ contains the additive and
 multiplicative identity elements and is closed under addition.  The
 real issue is to prove that $E_S$ is closed under multiplication.
 Consider elements $s_i=((a_{i+},a_{i-}),(b_{i+},b_{i-}))\in E_S$ for
 $i=0,1$.  This means that there are elements $x_i\in S$ such that
 $P_i=Q_i$, where 
 \begin{align*}
  P_i &= a_{i+}+b_{i-}+x_i \\
  Q_i &= a_{i-}+b_{i+}+x_i.
 \end{align*}
 We now have $s_0s_1=((a_+,a_-),(b_+,b_-))$, where 
 \begin{align*}
  a_+ &= a_{0+}a_{1+} + a_{0-}a_{1-} &
  a_- &= a_{0+}a_{1-} + a_{0-}a_{1+} \\
  b_+ &= b_{0+}b_{1+} + b_{0-}b_{1-} &
  b_- &= b_{0+}b_{1-} + b_{0-}b_{1+}.
 \end{align*}
 Put 
 \begin{align*}
   x &= a_{1+}x_0+a_{0+}x_1+x_0b_{1-}+x_1b_{0-}+2x_0x_1+
          a_{1+}b_{0-}+a_{0+}b_{1-}+b_{0-}b_{1-} \\
   u_0 &= a_{0-}a_{1-} + x_0x_1 + b_{0+}b_{1-} + b_{1+}b_{0-} \\
   u_1 &= b_{0+}b_{1-} + b_{1+}b_{0-} + b_{0+}b_{1+} + 
            b_{0+}x_1 + b_{1+}x_0 + 2x_0x_1 \\
   u_2 &= b_{0+}b_{1-} + b_{1+}b_{0-} + b_{0+}b_{1+} + 
            a_{0+}a_{1-} + a_{1+}a_{0-} + a_{0-}b_{1-} + a_{1-}b_{0-} \\
   u_3 &= b_{0+}b_{1+} + a_{0+}a_{1-} + a_{1+}a_{0-} + 
            a_{0+}x_1 + a_{1+}x_0 + 2x_0x_1.
 \end{align*}
 By direct expansion one can check that
 \begin{align*}
  a_++b_-+x  &= P_0P_1 + u_0 \\ 
  Q_0Q_1+u_0 &= a_{0-}Q_1 + a_{1-}Q_0 + u_1 \\
  a_{0-}P_1 + a_{1-}P_0 + u_1 &= x_0Q_1 + x_1Q_0 + u_2 \\
  x_0P_1 + x_1P_0 + u_2 &= b_{0-}Q_1 + b_{1-}Q_0 + u_3 \\
  b_{0-}P_1 + b_{1-}P_0 + u_3 &= a_-+b_++x.
 \end{align*}
 We are also given that $P_i=Q_i$, so it follows that
 $a_++b_-+x=a_-+b_++x$, which proves that $s_0s_1\in E_S$ as required.
\end{proof}

\begin{corollary}\label{cor-S-plus-semiring}
 The set $S^+=\hS/E_S$ has a unique semiring structure for which the
 quotient map $\hS\to S^+$ is a semiring homomorphism.
\end{corollary}
\begin{proof}
 This follows formally from the fact that $E_S$ is both an equivalence
 relation and a subsemiring in $\hS\tm\hS$.
\end{proof}

\begin{proposition}\label{prop-green-plus}
 If $S$ is a Green functor, then $\hS$ and $S^+$ also have canonical
 structures as Green functors.
\end{proposition}
\begin{proof}
 We will actually work with Green semiring structures, which are
 equivalent to Green functor structures as explained in
 Section~\ref{sec-green}.  It is clear that
 Proposition~\ref{prop-hS-semiring} gives the required semiring
 structures on $\hS(X)$ and $S^+(X)$ and that the maps $R_f$ are
 semiring morphisms.  The key issue is to check the Frobenius
 reciprocity formula $T_f(a)b=T_f(a\,R_f(b))$ for $f\:X\to Y$ and
 $a\in\hS(X)$ and $b\in\hS(Y)$.  Here we have $a=(a_+,a_-)$ for some
 $a_+,a_-\in S(X)$ and similarly $b=(b_+.b_-)$ for some $b_+,b_-\in
 S(Y)$.  We thus have
 \begin{align*}
  T_f(a)\,b
   &= (T_f(a_+),T_f(a_-))\;.\;(b_+,b_-) \\
   &= (T_f(a_+)b_+ + T_f(a_-)b_-\;,\;T_f(a_+)b_- + T_f(a_-)b_+) \\
   &= (T_f(a_+R_f(b_+)+a_-R_f(b_-))\;,\;T_f(a_+R_f(b_-)+a_-R_f(b_+))) \\
   &= T_f(a\; R_f(b))
 \end{align*}
 as required.  The quotient map $\hS(X)\to S^+(X)$ is a surjective
 semiring morphism that commutes with maps of the form $R_f$ and
 $T_f$.  The reciprocity formula for $S^+$ therefore follows from the
 formula for $\hS$.
\end{proof}

We next want to define norm maps for $S^+$ (assuming that $S$ itself
is a Tambara functor).  A key problem is to define $N_f(-1)\in S^+(Y)$
for every $f\:X\to Y$.  If $|f^{-1}\{y\}|=d$ for all $y$ then it is
easy to think that $N_f(-1)$ should be $(-1)^d$, and if we assume this
then it is not hard to determine how $N_f$ should behave in general.
Unfortunately, this approach does not lead to well-defined operations
satisfying the Tambara axioms.  To see this, consider the case where
$|G|=2$ and $f$ is the unique map $G\to 1$ (so $d=2$).  As in
Section~\ref{sec-two} we need to have 
\[ 0 = N_f(0) = N_f(1+(-1)) = N_f(1) + T_f(-1) + N_f(-1) 
     = 1 - T_f(1) + N_f(-1),
\]
so $N_f(-1)=T_f(1)-1\in S^+(1)$.  If $S$ is the Burnside semiring
Tambara functor and $t=T_f(1)\in S(1)$ then $S(1)=\N\{1,t\}$ (with
$t^2=2t$) and $S^+(1)=\Z\{1,t\}$ so the element $N_f(-1)=t=1$ is
definitely different from $(-1)^d=1$.  Because of this it is hard to
give an explicit formula for norm maps on $S^+$.  Instead, Tambara
used an indirect approach which we now explain.

Consider a map $f\:X\to Y$ of finite $G$-sets, and put 
\begin{align*}
 V(f)   &= \{(y,C)\st y\in Y,\; C\sse f^{-1}\{y\}\} \\
 V_n(f) &= \{(y,C)\in V(f)\st \; |C|=n\}. 
\end{align*}
Given a Green functor $S$, we will define a ``convolution product''
(different from the standard product) on $S(V(f))$.  We start with the
nonequivariant case for motivation.

Suppose we have a semiring $S_1$ and we put $S(X)=\Map(X,S_1)$ for all
$X$.  We can define the convolution product on $S(V(f))$ by 
\[ (a_1\vee a_2)(y,C) = \sum_{C=C_1\amalg C_2} a_1(y,C_1)\,a_2(y,C_2). \]
This is easily seen to make $S(V(f))$ into a commutative semiring,
with identity element given by 
\[ e(y,C) = \begin{cases} 
    1 & \text{ if } C = \emptyset \\
    0 & \text{ otherwise. }
   \end{cases}
\]
Moreover, $S(V(f))$ splits as the product of pieces $S(V_n(f))$ (which
are zero for $n$ sufficiently large) and 
$S(V_n(f))\vee S(V_m(f))\sse S(V_{n+m}(f))$ so in fact we have a
graded semiring.  We can define $\chi\:S(X)\to S(V(f))$ by 
\[ \chi(a)(y,C) = \prod_{x\in C} a(x) \]
and it is easy to see that $\chi(0)=e$ and
$\chi(a+b)=\chi(a)\chi(b)$.  There is an inclusion $j\:Y\to V(f)$
given by $j(y)=(y,f^{-1}\{y\})$ and we have 
\[ R_j\chi(a)(y) = \chi(a)(y,f^{-1}\{y\})
    = \prod_{f(x)=y} a(x) = (N_fa)(y).
\]
This shows that $\chi$ determines $N_f$.

We now give the corresponding definitions for Green functors and
Tambara functors.

\begin{definition}\label{defn-convolution}
 For $n,m\geq 0$ we put 
 \begin{align*}
  V^m(f) &= \{(y,C_1,\dotsc,C_m)\st y\in Y,\; C_i\sse f^{-1}\{y\},\;
                C_i\cap C_j = \emptyset \text{ for } i\neq j\} \\
  V^m_n(f) &= \{(y,C_1,\dotsc,C_m)\in V^m(f) \st \sum_i|C_i|=n\}.
 \end{align*}
 We define maps $p_1,p_2,p_{12}\:V^2(f)\to V(f)$ by 
 \begin{align*}
  p_1(y,C_1,C_2) &= (y,C_1) \\
  p_2(y,C_1,C_2) &= (y,C_2) \\
  p_{12}(y,C_1,C_2) &= (y,C_1\cup C_2).
 \end{align*}
 Now let $S$ be a Green functor.  For $a_1,a_2\in S(V(f))$ we define 
 \[ a_1\vee a_2 = T_{p_{12}}(R_{p_1}(a_1)\,R_{p_2}(a_2)) \in
     S(V(f)).
 \]
 We also write $e=T_k(1)$, where $k\:Y\to V(f)$ is given by
 $k(y)=(y,\emptyset)$ (so $k\:Y\simeq V_0(f)$).
\end{definition}

\begin{proposition}\label{prop-convolution-semiring}
 The above product makes $S(V(f))$ into a commutative graded semiring.
\end{proposition}
\begin{proof}
 We will prove associativity and leave the rest to the reader.  Just
 by expanding the definitions, we have
 \begin{align*}
  (a_1\vee a_2)\vee a_3 
   &= T_{p_{12}}(R_{p_1}(a_1\vee a_2)\,R_{p_2}(a_3)) \\
   &= T_{p_{12}}(R_{p_1}T_{p_{12}}(R_{p_1}(a_1)\,R_{p_2}(a_2))\;
                 R_{p_2}(a_3)).
 \end{align*}
 Next, we define maps as follows:
 \begin{align*}
  q_i      &\: V^3 \to V   & q_i(y,C_1,C_2,C_3) &= (y,C_i) \\
  q_{12}   &\: V^3 \to V   & q_{1,2}(y,C_1,C_2,C_3) &= (y,C_1\cup C_2) \\
  q_{123}  &\: V^3 \to V   & q_{123}(y,C_1,C_2,C_3) &= (y,C_1\cup C_2\cup C_3) \\
  q_{1,2}  &\: V^3 \to V^2 & q_{1,2}(y,C_1,C_2,C_3) &= (y,C_1,C_2) \\
  q_{12,3} &\: V^3 \to V^2 & q_{1,2}(y,C_1,C_2,C_3) &= (y,C_1\cup C_2,C_3).
 \end{align*}
 These satisfy 
 \begin{align*}
  p_1q_{1,2}  &= q_1    & p_2q_{1,2}  &= q_2 & p_{12}q_{1,2}  &= q_{12} \\
  p_1q_{12,3} &= q_{12} & p_2q_{12,3} &= q_3 & p_{12}q_{12,3} &= q_{123},
 \end{align*}
 and one can check that the square 
 \[ \xymatrix{
     V^3 \ar[d]_{q_{12,3}} \ar[r]^{q_{1,2}} & 
     V^2 \ar[d]^{p_{12}} \\
     V^2 \ar[r]_{p_1} &
     V
 } \]
 is cartesian.  This means that
 $R_{p_1}T_{p_{12}}=T_{q_{12,3}}R_{q_{1,2}}$. 
 Using this together with Frobenius reciprocity for $q_{12,3}$ we get
 \begin{align*}
  (a_1\vee a_2)\vee a_3 
   &= T_{p_{12}}(T_{q_{12,3}}R_{q_{1,2}}(R_{p_1}(a_1)\,R_{p_2}(a_2))\;
                 R_{p_2}(a_3)) \\
   &= T_{p_{12}q_{12,3}}(
       R_{q_{1,2}}R_{p_1}(a_1)\;
       R_{q_{1,2}}R_{p_2}(a_2)\;
       R_{q_{12,3}}R_{p_2}(a_3)
      ) \\
   &= T_{q_{123}}(R_{q_1}(a_1)\;R_{q_2}(a_2)\;R_{q_3}(a_3)).
 \end{align*}
 A similar argument gives the same description for
 $a_1\vee(a_2\vee a_3)$.  Note that we have only used ordinary
 products and not norms, so we only need a Green functor, not a
 Tambara functor.
\end{proof}

\begin{definition}\label{defn-chi-SV}
 We put 
 \[ U(f) = \{(x,C)\st x\in X,\; x\in C\sse f^{-1}\{f(x)\}\}, \]
 and we define maps $X\xla{r}U(f)\xra{t}V(f)$ by $r(x,C)=x$ and
 $t(x,C)=(f(x),C)$.  We also define $j\:Y\to V(f)$ by
 $j(y)=(y,f^{-1}\{y\})$.  For any Tambara functor $S$, we then define
 $\chi=N_tR_r\:S(X)\to S(V(f))$.
\end{definition}

\begin{proposition}\label{prop-chi-hom}
 The map $\chi$ satisfies $\chi(0)=e$ and
 $\chi(a_1+a_2)=\chi(a_1)\vee\chi(a_2)$.
\end{proposition}
\begin{proof}
 First, we have $\chi(0)=N_t(0)$.  We observe that $\img(t)^c=V_0(f)$,
 and using Lemma~\ref{lem-norm-zero} it follows that $\chi(0)=e$. 

 Now suppose we have elements $a_1,a_2\in S(X)$ and we put
 $b_i=R_r(a_i)$ and $c_i=N_t(b_i)=\chi(a_i)$.  We must show that
 $\chi(a_1+a_2)=c_1\vee c_2$.  The elements $b_i$ combine to give an
 element $b\in S(U(f)\tm\{1,2\})$ and the projection
 $s\:U(f)\tm\{1,2\}\to U(f)$ has $T_s(b)=b_1+b_2$.  We therefore need
 to understand the distributor $\Dl(s,t)$.  For any point
 $(y,C)\in V(f)$ we have $t^{-1}\{(y,C)\}\simeq C$.  Moreover, to give
 a map $m\:t^{-1}\{(y,C)\}\to U(f)\tm\{1,2\}$ with $sm=1$ is the same
 as to give an arbitrary map $C\to\{1,2\}$, or equivalently a
 splitting of $C$ as a disjoint union of subsets $C_1$ and $C_2$.
 Using this we can identify $\Dl(s,t)$ with the diagram 
 \[ U(f)\tm\{1,2\} \xla{u}
    U^2(f) \xra{t_2} 
    V^2(f) \xra{p_{12}}
    V(f)
 \]
 where 
 \begin{align*}
  U^2(f) &= \{(x,C_1,C_2) \st x\in C_1\amalg C_2\sse f^{-1}\{f(x)\}\} \\
  u(x,C_1,C_2) &= 
   \begin{cases}
    (x,C_1\cup C_2,1) & \text{ if } x\in C_1 \\
    (x,C_1\cup C_2,2) & \text{ if } x\in C_2.
   \end{cases} \\
  t_2(x,C_1,C_2) &= (f(x),C_1,C_2).
 \end{align*}
 We conclude that 
 \[ \chi(a_1+a_2) = N_tT_s(b) = T_{p_{12}}N_{t_2}R_u(b) 
      = T_{p_{12}}N_{t_2}R_{(r\tm 1)\circ u}(a).
 \]

 On the other hand, we have 
 \[ \chi(a_1)\vee\chi(a_2) = 
     T_{p_{12}}(R_{p_1}(\chi(a_1))R_{p_2}(\chi(a_2))).
 \]
 Here $R_{p_1}(\chi(a_1))=R_{p_1}N_tR_r(a_1)$.  We have a commutative
 diagram 
 \[ \xymatrix{
     U(f) \ar[d]_r & 
     U^{2,1}(f) \ar[l]_{u_1} \ar[d]_{v_1} \ar[r]^{t_{2,1}} &
     V^2(f) \ar[d]^{p_1} \\
     X &
     U(f) \ar[l]^r \ar[r]_t &
     V(f),
    }
 \] 
 where
 \begin{align*}
  U^{2,1}(f) &= \{(x,C_1,C_2)\in U^2\st x\in C_1\} \\
  v_1(x,C_1,C_2) &= (x,C_1) \\
  t_{2,1} &= t_2 |_{U^{2,1}(f)} \\
  u_1 &= u |_{U^{2,1}(f)}.
 \end{align*}
 The right hand square is cartesian, giving 
 \[ R_{p_1}(\chi(a_1)) = R_{p_1}N_tR_r(a_1) =
     N_{t_{2,1}}R_{v_1}R_r(a_1) = N_{t_{2,1}}R_{ru_1}(a_2).
 \]
 After obtaining a similar expression for $R_{p_2}(\chi(a_2)$ and
 noting that $U^2(f)=U^{2,1}(f)\amalg U^{2,2}(f)$ we find that 
 $R_{p_1}(\chi(a_1))R_{p_2}(\chi(a_2))=N_{t_2}R_{(r\tm 1)\circ u}(a)$ 
 so $\chi(a_1)\vee\chi(a_2)=\chi(a_1+a_2)$ as claimed.
\end{proof}

\begin{proposition}\label{prop-Rj-chi}
 The map $\chi$ also satisfies $R_j\chi(a)=N_f(a)$.
\end{proposition}
\begin{proof}
 We have a commutative diagram
 \[ \xymatrix{
  & X \ar[r]^f \ar[d]^i \ar[dl]_1 &
  Y \ar[d]^j \\
  X &
  U(f) \ar[l]^r \ar[r]_t &
  V(f)
 } \]
 where $i(x)=(x,f^{-1}\{f(x)\})$.  One can check that the square is
 cartesian, so $R_j\chi=R_jN_tR_r=N_fR_iR_r=N_f$.
\end{proof}

\begin{definition}\label{defn-GSVf}
 As before we define $k\:Y\to V(f)$ by $k(y)=(y,\emptyset)$, and we
 put 
 \[ GS(V(f))=\{a\in S(V(f))\st R_k(a)=1\}. \]
\end{definition}

\begin{proposition}\label{prop-Rk}
 For any Green functor $S$, the map $R_k\:S(V(f))\to S(Y)$ satisfies
 $R_k(e)=1$ and $R_k(b_1\vee b_2)=R_k(b_1)\,R_k(b_2)$.  If $S$ is a
 Tambara functor then we also have $R_k(\chi(a))=1$.
\end{proposition}
\begin{proof}
 Define $k^2\:Y\to  V^2(f)$ by $k^2(y)=(y,\emptyset,\emptyset)$.  We
 then have commutative diagrams as follows, in which the squares are
 cartesian: 
 \[ \xymatrix{
     Y \ar[r]^1 \ar[d]_1 & Y \ar[d]^k \\ 
     Y \ar[r]_k & V(f)
    } \hspace{4em}
    \xymatrix{
     & Y \ar[dl]_k \ar[d]^{k^2} \ar[r]^1 & Y \ar[d]^k \\
     V(f) & V^2(f) \ar[l]^{p_i} \ar[r]_{p_{12}} & V(f)
    } \hspace{4em}
    \xymatrix{
     & \emptyset \ar[dl] \ar[d] \ar[r] & Y \ar[d]^k \\
     X & U(f) \ar[l]^r \ar[r]_t & V(f) 
    }
 \]
 The first diagram gives $R_k(e)=R_kT_k(1)=1$.  
 The second diagram gives
 \[ R_k(b_1\vee b_2) = 
    R_kT_{p_{12}}(R_{p_1}(b_1)\,R_{p_2}(b_2)) = 
    R_{k^2}(R_{p_1}(b_1)\,R_{p_2}(b_2)) = 
    R_k(b_1)\,R_k(b_2).
 \]
 The third diagram gives 
 \[ R_k\chi(a) = R_kN_tR_r(a) =
    N_{\emptyset\to Y} R_{\emptyset\to X} (a) = 
    N_{\emptyset\to Y}(1) = 1.
 \]
\end{proof}

\begin{proposition}\label{prop-GSVf-group}
 The set $GS(V(f))$ is a submonoid of $S(V(f))$ (under the convolution
 product).  If $S$ is additively complete, then $GS(V(f))$ is an
 abelian group under convolution.
\end{proposition}
\begin{proof}
 From the facts that $R_k(e)=1$ and
 $R_k(b_1\vee b_2)=R_k(b_1)\,R_k(b_2)$ it is clear that $GS(V(f))$ is
 a sumonoid.  Now suppose that $S$ is additively complete and
 $a\in GS(V(f))$.  Put $N=\max\{|f^{-1}\{y\}|\;\st y\in Y\}$, so the
 standard grading on $S(V(f))$ is zero in degrees larger than $N$.  We
 can thus write $a=\sum_{i=0}^Na_i$ with $a_i$ in degree $i$.  As
 $R_k(a)=1$ and $k$ gives a bijection $Y\to V_0(f)$ we see that
 $a_0=e$.  Put $b=\sum_{i=1}^Na_i$ so $a=e+b$.  Using the grading we
 see that $b^{N+1}=0$, so the element $c=\sum_{i=0}^n(-b)^i$ gives an
 inverse for $a$.
\end{proof}

\begin{definition}\label{defn-Nf-plus}
 Let $S$ be a Tambara functor, and let $\eta\:S\to S^+$ be the evident
 morphism of Green functors.  As $\eta\chi\:S(X)\to GS^+(V(f))$ is a
 homomorphism from a semigroup to a group, the universal property of
 $S^+(X)$ gives us a homomorphism $S^+(X)\to GS^+(V(f))$ making the
 left square below commute.
 \[ \xymatrix{
  S(X) \ar[d]_\eta \ar[r]^\chi &
  GS(V(f)) \ar[d]_\eta \ar[r]^{\text{inc}} &
  S(V(f)) \ar[d]^\eta \ar[r]^{R_j} &
  S(Y) \ar[d]^\eta \\
  S^+(X) \ar[r]_{\chi^+} &
  GS^+(V(f)) \ar[r]_{\text{inc}} &
  S^+(V(f)) \ar[r]_{R_j} &
  S^+(Y)
 } \]
 The other two squares commute automatically.  We also recall that the
 top composite $S(X)\to S(Y)$ is $N_f$, and we define
 $N_f\:S^+(X)\to S^+(Y)$ to be the bottom composite.
\end{definition}

In order to establish the properties of these norm maps, we need to
recall the theory of polynomial maps between abelian
groups~\cites{dr:orr,drsi:brp}.

\begin{definition}\label{defn-poly}
 Let $A$ and $B$ be abelian groups.  We will write $M$ for the set of
 all functions (not necessarily homomorphisms) from $A$ to $B$.
 \begin{itemize}
  \item[(a)] We define $\ep\:M\to B$ by $\ep(f)=f(0)$.
  \item[(b)] For $a\in A$ we define $\dl[a]\:M\to M$ by
   $(\dl[a]f)(x)=f(a+x)-f(x)$.  It is easy to see that
   $\dl[a]\dl[a']=\dl[a']\dl[a]$ for all $a,a'\in A$. 
  \item[(c)] Given finite sets $J\sse I$ and a map $a\:I\to A$ we also
   write $\dl[J,a]=\prod_{j\in J}\dl[a(j)]\:M\to M$.
  \item[(d)] We say that $f\in M$ is \emph{polynomial of degee at most $n$}
   if $\dl[I,a]f=0$ whenever $|I|>n$.  We say that $f\in M$ is
   \emph{polynomial} if this condition is satisfied for some $n$.
 \end{itemize}
\end{definition}

\begin{remark}
 It is easy to see that $f$ is polynomial of degree at most $0$ iff it
 is constant, and polynomial of degree at most $1$ iff it is a constant
 plus a homomorphism.
\end{remark}

\begin{remark}
 If we put $\sg(J,a)=\sum_{j\in J}a(j)$ we find that
 \[ (\dl[I,a]f)(x) =
      \sum_{J\sse I} (-1)^{|I\sm J|} f(\sg(J,a)+x).
 \]
\end{remark}

\begin{remark}\label{rem-delta-zero}
 If $a(i)=0$ for some $i\in I$, it is clear that $\dl[I,a]f=0$ for all
 $f$. 
\end{remark}

\begin{lemma}\label{lem-mobius}
 For any $f\in M$ and $a\:I\to A$ we have 
 \[ f(a+\sg(I,a)) = \sum_{J\sse I}(\dl[J,a]f)(x). \] 
\end{lemma}
\begin{proof}
 By definition we have
 \begin{align*}
  \sum_{J\sse I}(\dl[J,a]f)(x) 
   &= \sum_{K\sse J\sse I} (-1)^{|J\sm K|}f(\sg(K,a)+x) \\
   &= \sum_{K\sse I} \left(f(\sg(K,a)+x)
        \sum_{K\sse J\sse I}(-1)^{|J\sm K|}\right).
 \end{align*}
 It is straightforward to check that the inner sum is $1$ when $K=I$
 and $0$ otherwise, and the claim follows.
\end{proof}

\begin{proposition}\label{prop-poly-composite}
 Suppose we have maps $A\xra{f}B\xra{g}C$ where $f$ is polynomial of
 degree at most $n$, and $g$ is polynomial of degree at most $m$.
 Then $gf$ is polynomial of degree at most $nm$.
\end{proposition}
\begin{proof}
 Consider a map $a\:I\to A$, and an element $x\in A$.  Let $PI$ be the
 set of subsets of $I$, and for $J\in PI$ put
 \[ b(J) = (\dl[J,a]f)(x). \]
 Now 
 \[ (\dl[I,a]gf)(x)=\sum_{J\sse I}(-1)^{|I\sm J|}g(f(\sg(J,a)+x)), \]
 but Lemma~\ref{lem-mobius} gives
 \[ f(\sg(J,a)+x) = 
    \sum_{K\sse J}(\dl[K,a]f)(x) = \sum_{K\in PJ} b(K).
 \]
 This in turn gives 
 \[ g(f(\sg(J,a)+x) = g(\sg(PJ,b)+0) =
     \sum_{T\sse PJ} (\dl[T,b]g)(0),
 \] 
 so 
 \[ (\dl[I,a]gf)(x) =
      \sum_{J\sse I} (-1)^{|I\sm J|} \sum_{T\sse PJ} (\dl[T,b]g)(0) =
      \sum_{T\sse PI} (\dl[T,b]g)(0)
        \sum_{\bigcup T\sse J\sse I}(-1)^{|I\sm J|}.
 \]
 Here it is easy to see that the inner sum is $1$ if $\bigcup T=I$,
 and $0$ otherwise.  We conclude that
 \[ (\dl[I,a]gf)(x) =
      \sum_{T\sse PI,\;\bigcup T=I} (\dl[T,b]g)(0).
 \]
 For the term $(\dl[T,b]g)(0)$ to be nonzero we must have $|T|\leq m$.
 Moreover, by Remark~\ref{rem-delta-zero} we must also have
 $b(K)\neq 0$ for all $K\in T$, which forces $|K|\leq n$.  Together
 these imply that $|\bigcup T|\leq nm$, but $\bigcup T=I$ so
 $|I|\leq nm$.  It follows that $gf$ is polynomial of degree at most
 $nm$, as claimed.
\end{proof}

\begin{proposition}\label{prop-poly-equal}
 Let $f,g\:A\to B$ be polynomial maps.  Let $A_0$ be a subsemigroup of
 $A$ that generates $A$ as a group, and suppose that $f|_{A_0}=g|_{A_0}$.
 Then $f=g$.
\end{proposition}
\begin{proof}
 The difference $h=f-g$ has $h(A_0)=0$, and it is clearly polynomial
 of some degree $d$ say.  If $d=0$ then $h$ is constant but $h(A_0)=0$
 so $h=0$.  Now suppose that $d>0$.  Any element $a\in A$ can be
 written as $a_+-a_-$ with $a_+,a_-\in A_0$.  Put $k=\dl[a_-]h$, and
 note that this is polynomial of degree at most $d-1$.  For $x\in A_0$
 we have $a_-+x\in A_0$ and $k(x)=h(a_-+x)-h(x)=0-0=0$.  By induction
 it follows that $k=0$, so $h(x)=h(a_-+x)$ for all $x$.  Now take
 $x=a$ to get $h(a)=h(a_+)\in h(A_0)=0$.  We conclude that $h=0$ as
 required. 
\end{proof}

\begin{lemma}\label{lem-norm-poly}
 Let $f\:X\to Y$ be a map of finite $G$-sets, and let $d$ be the
 maximum of the numbers $|f^{-1}\{y\}|$ for $y\in Y$.  Then
 the map $N_f\:S^+(X)\to S^+(Y)$ is polynomial of degree at most $d$.  
\end{lemma}
\begin{proof}
 We have $N_f=R_j\chi^+$, where $R_j$ is a homomorphism and so is
 polynomial of degree at most one.  In view of
 Proposition~\ref{prop-poly-composite}, it will suffice to prove that
 $\chi^+$ is polynomial of degree at most $d$.  For this proof we will
 just write $uv$ for the convolution product $u\vee v$.  We also note
 that when $k>d$ we have $V_k(f)=\emptyset$ and so $S^+(V(f))_k=0$.
 It follows that the ideal $K=\sum_{k>0}S^+(V(f))_k$ satisfies
 $K^{d+1}=0$.  We have seen that the image of $\chi^+$ is contained in
 $GS^+(V(f))$ so we can write $\chi^+(a)=1+\chi_0(a)$ for some map
 $\chi_0\:S(X)\to K$.  We can rearrange the relation
 $\chi^+(a+x)=\chi^+(a)\chi^+(x)$ to get
 $\dl[a]\chi^+=\chi_0(a)\chi^+$.  It follows inductively that for any
 map $a\:I\to S^+(X)$ we have 
 \[ \dl[I,a]\chi^+=\left(\prod_{i\in I}\chi_0(a(i))\right) \chi^+. \]
 Using $K^{d+1}=0$ we deduce that $\dl[I,a]\chi^+=0$ when $|I|>d$, as
 required. 
\end{proof}

\begin{proposition}\label{prop-tambara-plus}
 The maps $N_f\:S^+(X)\to S^+(Y)$ make $S^+$ into a Tambara functor.
 Moreover, if $T$ is any additively complete Tambara functor and
 $\phi\:S\to T$ is a morphism of Tambara functors then there is a
 unique morphism $\phi^+\:S^+\to T$ of Tambara functors with
 $\phi^+\eta=\phi$. 
\end{proposition}
\begin{proof}
 Suppose we have maps $X\xra{f}Y\xra{g}Z$.  We must show that
 $N_{gf}=N_gN_f\:S^+(X)\to S^+(Z)$.  By construction, the following
 diagrams commute:
 \[ \xymatrix{
  S(X) \ar[r]^{N_f} \ar[d]_\eta &
  S(Y) \ar[r]^{N_g} \ar[d]_\eta &
  S(Z)              \ar[d]_\eta &&
  S(X) \ar[r]^{N_{gf}} \ar[d]^\eta &
  S(Z)                 \ar[d]^\eta \\
  S^+(X) \ar[r]_{N_f} &
  S^+(Y) \ar[r]_{N_g} &
  S^+(Z) &&
  S^+(X) \ar[r]_{N_{gf}} &
  S^+(Z).
 } \]
 We also know that $N_{gf}=N_gN_f$ on $S(X)$, and it follows that
 $N_{gf}=N_gN_f$ on the image of $\eta\:S(X)\to S^+(X)$.
 Moreover, both $N_{gf}$ and $N_gN_f$ are algebraic, so it follows
 from Proposition~\ref{prop-poly-equal} that $N_{gf}=N_gN_f$ on
 $S^+(X)$.  All the other Tambara functor axioms can be verified in
 the same way.

 Now suppose we have an additively complete Tambara functor $T$ and
 a morphism $\phi\:S\to T$ of Tambara functors.  It is clear that
 there is a unique morphism $\phi^+\:S^+\to T$ of Mackey functors with
 $\phi^+\eta=\phi$, and we just need to check that this is compatible
 with norm maps.  In more detail, given $f\:X\to Y$ we must show that
 for any $f\:X\to Y$ we have $N_f\phi^+=\phi^+N_f\:S^+(X)\to T(Y)$.
 Here both $N_f\phi^+$ and $\phi^+N_f$ are algebraic, so it will
 suffice to check that they agree on $\eta(S(X))$, but that is clear
 by naturality.
\end{proof}

\section{Modules over Tambara functors}
\label{sec-modules}

In this section we study two different possible definitions for
modules over a Tambara functor.

\begin{definition}\label{defn-naive-module}
 Let $S$ be a Tambara functor.  As in the Section~\ref{sec-green} we
 see that there are canonical maps $A\xra{\eta}S\xla{\mu}S\btm S$
 making $S$ into a Green ring.  By a \emph{naive $S$-module} we mean a
 Mackey functor $M$ equipped with a map $\nu\:S\btm M\to M$ making the
 obvious diagram commute:
 \[ \xymatrix{
  S\btm S\btm M \ar[r]^{\mu\btm 1} \ar[d]_{1\btm\nu} &
  S\btm M \ar[d]_\nu &
  M \ar[l]_(0.4){\eta\btm 1} \ar[dl]^1 \\
  S\btm M \ar[r]_\nu & M
 } \]

 We write $\NMod_S$ for the category of naive $S$-modules.
\end{definition}

\begin{remark}\label{rem-NMod}
 In view of Proposition~\ref{prop-box-equiv}, to give a naive
 $S$-module structure on $M$ is the same as to make each $M(X)$ into
 an $S(X)$-module in such a way that 
 \begin{itemize}
  \item[(a)] For all $f\:X\to Y$ and $b\in S(Y)$ and $n\in M(Y)$ we
   have $R_f(bn)=R_f(b)\;R_f(n)$.
  \item[(b)] For all $f\:X\to Y$ and $b\in S(Y)$ and $m\in M(X)$ we
   have $bT_f(m)=T_f(R_f(b)\;m)$.
 \end{itemize}
 In particular, if $S$ is additively complete, then multiplication by
 $-1\in S(X)$ provides additive inverses in $M(X)$, so $M(X)$ is also
 additively complete.  Given this, one can check that $\NMod_S$ is an
 abelian category.
\end{remark}

As the category $\NMod_S$ depends only on the underlying Green ring, it
is natural to ask whether there is a more refined category that
somehow takes account of the norm maps.  We can define such a category
following an idea of Waldhausen in the context of stable homotopy
theory, which takes a foundational result in the Andr\'e-Quillen
homology theory for commutative rings and turns it into a
definition.  This will of course mean that our definition is a good
basis for an Andr\'e-Quillen homology theory for Tambara functors.
One can thus expect it to be useful when studying topological
Andr\'e-Quillen homology for equivariant ring spectra, which is
important for a number of applications.

\begin{definition}\label{defn-Mod}
 Let $S$ be a Tambara functor.  We write $\AugAlg_S$ for the category
 of augmented $S$-algebras, or in other words Tambara functors $T$
 equipped with morphisms $S\xra{\eta}T\xra{\ep}S$ such that
 $\ep\eta=1$.  This category has finite products, given by
 \[ (T\tm_ST')(X) =
      \{(a,a')\in T(X)\tm T'(X)\st \ep(a)=\ep(a')\in S(X)\},
 \]
 so it is meaningful to talk about semigroup objects in $\AugAlg_S$.
 (As always, our semigroups are assumed to be commutative.)  We call
 these \emph{$S$-modules}, and we write $\Mod_S$ for the category of
 such objects.
\end{definition}

\begin{remark}\label{rem-Mod-semiadditive}
 In any category $\CC$ with finite products, one can check that the
 category $\Semigroups(\CC)$ of semigroup objects in $\CC$ is
 semiadditive, with finite products (and therefore finite coproducts)
 given by products in the underlying category $\CC$.  Thus, $\Mod_S$
 is always a semiadditive category.
\end{remark}

\begin{definition}\label{defn-Lambda}
 For $T\in\Mod_S$, we put 
 \[ (\Lm T)(X) = \{ a\in T(X) \st \ep(a) = 0 \in S(X)\}. \]
 It is clear this is a sub-Mackey functor of $T$ and that the product
 map $T\btm T\to T$ restricts to give a product $S\btm\Lm T\to\Lm T$
 making $\Lm T$ a naive $S$-module.  We thus have a functor
 $\Lm\:\Mod_S\to\NMod_S$. 
\end{definition}

\begin{proposition}\label{prop-Mod-abelian}
 Let $S$ be an additively complete Tambara functor.
 \begin{itemize}
  \item[(a)] Any $T\in\Mod_S$ has a natural splitting $T=S\oplus\Lm T$
   as Mackey functors.  
  \item[(b)] With respect to the splitting in~(a), the semigroup
   structure map $\sg\:T(X)\tm_{S(X)}T(X)\to T(X)$ is given by
   $\sg(a,u,v)=(a,u+v)$. 
  \item[(c)] The map $\chi(a,u)=(a,-u)$ is a morphism of augmented
   $S$-algebras, so $T$ is actually a group object (not just a
   semigroup object) in $\AugAlg_S$.
  \item[(d)] For any morphism $\phi\:T\to T'$ in $\Mod_S$, the Mackey
   functors $S\oplus\ker(\Lm\phi)$, $S\oplus\img(\Lm\phi)$ and
   $S\oplus\cok(\Lm\phi)$ have unique structures as $S$-modules such that
   the evident Mackey morphisms
   \[ S\oplus\ker(\Lm\phi) \to T \to S\oplus\img(\Lm\phi)
        \to T' \to S\oplus\cok(\Lm\phi)
   \]
   are $S$-module morphisms.
  \item[(e)] $\Mod_S$ is an abelian category.
 \end{itemize}
\end{proposition}
\begin{proof}
 \begin{itemize}
  \item[(a)] Any $a\in T(X)$ can be written as
   $\eta(\ep(a))+(a-\eta(\ep(a)))$ with $a-\eta(\ep(a))\in\Lm T(X)$.  The
   claim is clear from this.
  \item[(b)] Because $\sg$ is supposed to give a semigroup structure,
   there must be a morphism $\zt\:S\to T$ in $\AugAlg_S$ (to provide
   the zero) such that the diagram
   \[ \xymatrix{ 
     T \ar[r]^(0.4){(\zt,1)} \ar[dr]_1 &
     T\tm_ST \ar[d]^\sg &
     T \ar[l]_(0.4){(1,\zt)} \ar[dl]^1 \\
     & T
   }\]
   commutes.  Now $\eta$ is the only morphism from $S$ to $T$ in
   $\AugAlg_S$, so we must have $\zt=\eta$.  Given this, and using the
   splitting from~(a), commutativity of the diagram means that
   $\sg(a,0,v)=(a,v)$ and $\sg(a,u,0)=(a,u)$.  In particular, we have
   $\sg(a,0,0)=(a,0)$ and $\sg(0,u,0)=(0,u)$ and $\sg(0,0,v)=(0,v)$.
   Moreover, $\sg$ arises from a morphism $T\tm_ST\to T$ of Mackey
   functors, so it preserves addition.  Claim~(b) clearly follows.
  \item[(c)] We are given that $\sg$ is a Tambara morphism, and it
   follows that we can define a Tambara morphism
   $\tht\:T\tm_ST\to T\tm_ST$ by $\tht(s,t)=(\sg(s,t),t)$.  Now~(b)
   tells us that $\tht(a,u,v)=(a,u+v,v)$, so the component 
   \[ \tht_X\:(T\tm_ST)(X)\to(T\tm_ST)(X) \]
   is bijective for all $X$, with inverse $(a,u,v)\mapsto(a,u-v,v)$.
   It is standard that if all components of a natural transformation
   are isomorphisms, then their inverses form a natural transformation
   inverse to the original one.  Thus $\tht$ is invertible, with
   $\tht^{-1}(a,u,v)=(a,u-v,v)$.  The map $\chi$ is
   $\tht^{-1}\circ(\eta,1)$, so it is a morphism of Tambara functors.
   It is clearly also compatible with augmentation and satisfies
   $\sg\circ(1,\chi)=\eta\ep$, so it provides inverses for the
   semigroup structure on $T$.  
  \item[(d)] It is clear that $S\oplus\img(\Lm\phi)$ is the image of
   the Tambara morphism $\phi$ and so has a unique Tambara structure
   compatible with the evident morphisms
   \[ T \xra{} S\oplus \img(\Lm\phi) \xra{} T'. \]
   It is also clear that this is compatible with the augmentations.
   Moreover, we see from~(b) that this is also compatible with the
   relevant semigroup structures.  Next, $S\oplus\ker(\Lm\phi)$ can be
   expressed as the equaliser of $\phi$ and the composite
   $T\xra{\ep}S\xra{\eta'}T'$, both of which are Tambara morphisms.
   All claims about $S\oplus\ker(\Lm\phi)$ follow from this
   description.  Next, we claim that the functor
   $J=\img(\Lm\phi)\leq\Lm T'\leq T'$ is a Tambara ideal, as in
   Definition~\ref{defn-tambara-ideal}.  Indeed, it is clear that $J$
   is a sub-Mackey functor, and that $J(X)$ is an ideal in $T(X)$ for
   all $X$.  Now suppose we have a map $f\:X\to Y$ of finite $G$-sets
   and an element $b\in\img(\Lm\phi)(X)$, so $b=\phi(a)$ for some
   $a\in T(X)$ with $\ep(a)=0$.  As $\phi$ is natural we have
   $N^{T'}_f(b)=\phi(N^T_f(a))$ and $N^{T'}_f(0)=\phi(N^T_f(0))$, so   
   \[ N^{T'}_f(b)= \phi(N^T_f(a)-N^T_f(0)) + N^{T'}_f(0).
   \]
   Because $\ep(a)=0$ we also have 
   \[ \ep(N^T_f(a)-N^T_f(0)) = N^S_f(\ep(a)) - N^S_f(\ep(0)) = 0, \]
   so $N^T_f(a)-N^T_f(0)\in\Lm(T)(X)$, so
   $N^{T'}_f(b)\in J(Y)+N^{T'}_f(0)$ as required.  It follows that the
   quotient $T'/\img(\Lm\phi)=S\oplus\cok(\Lm\phi)$ has a natural
   structure as a Tambara functor.  Using~(a) and~(b) we again see
   that everything is compatible with augmentations and semigroup
   structures. 
  \item[(e)] This is clear from~(a) to~(d).
 \end{itemize}
\end{proof}

Although the structure of genuine modules is more subtle than that of
naive modules, it turns out that every naive module can be given a
canonical genuine module structure.  However, not every genuine module
arises from this construction.  We can state this more precisely as
follows: 

\begin{proposition}\label{prop-aug-alg}
 There is a functor $\Pi\:\NMod_S\to\Mod_S$ and a natural isomorphism
 $\Lm\Pi(M)\simeq M$ in $\NMod_S$.  There is also a natural morphism
 $\phi\:\Pi\Lm(T)\to T$ of Mackey functors for all $T\in\Mod_S$,
 which is an isomorphism if $S$ is additively complete.  However,
 $\phi$ need not be a morphism of Tambara functors.
\end{proposition}

The rest of this section will constitute the proof.

We first explain how the theorem generalises a straightforward and
well-known fact.  Let $S$ be a commutative semiring, and let $M$ be an
$S$-module.  We can then define a semiring structure on the group 
$\Pi M=S\oplus M$ by $(s,m).(s',m')=(ss',sm'+s'm)$.  We have ring
homomorphisms $S\xra{\eta}\Pi M\xra{\ep}S$ given by $\eta(s)=(s,0)$
and $\ep(s,m)=s$, so $\Pi M$ is an augmented $S$-algebra.  The copy of
$M$ inside $\Pi M$ is an ideal satisfying $M^2=0$.  The categorical
product of augmented $S$-algebras has the property that
\[ (S\oplus I)\tm_S(S\oplus J) = S\oplus I\oplus J. \] Using this, we
find that the map $(s,m,m')\mapsto(s,m+m')$ can be regarded as an
augmented algebra homomorphism $\Pi M\tm_S\Pi M\to\Pi M$, making $\Pi
M$ a commutative semigroup object in $\AugAlg_S$.  We thus have a
functor $\Pi\:\NMod_S\to\Mod_S$, and it is clear that $\Lm\Pi M=M$.
In the opposite direction, suppose we have an object $T\in\Mod_S$.  We
can then define $\phi\:\Pi\Lm T=S\tm\Lm T\to T$ by $\phi(s,m)=s+m$.
This clearly preserves addition.  Moreover, $T$ is assumed to be a
semigroup object, with addition given by a semiring homomorphism
$\sg\:T\tm_ST\to T$ say.  Given $m_0,m_1\in\Lm T$ we have elements
$n_0=(m_0,0)$ and $n_1=(0,m_1)$ in $T\tm_ST$ with $n_0n_1=0$.  It is
formal that $\sg(n_i)=m_i$ and $\sg$ preserves products so $m_0m_1=0$.
Using this we see that $\phi$ also preserves multiplication.  If $S$
(and therefore $T$) is additively complete, we have an inverse given
by $\phi^{-1}(t)=(\ep(t),t-\ep(t))$.

To see what can go wrong when $S$ is not additively complete, consider
the sets
\begin{align*}
 T &= \N[x]/x^2=\{n+mx\st n,m\in\N\} \\
 T' &= \{n+mx\in T\st n=m=0 \text{ or } n>0\} \\
 T'' &= T/((n+mx)\sim n \text{ whenever } n>0). 
\end{align*}
We can regard $T$ as an $\N$-algebra with augmentation $\ep(n+mx)=n$.
In fact, it is a semigroup object in $\AugAlg_{\N}$ with addition map
$\sg\:T\tm_{\N}T\to T$ given by $\sg(n+mx,n+kx)=n+(m+k)x$.  One can
check that $T'$ is a subobject of $T$ and $T''$ is a quotient object.
We have $\Lm T'=0$ so the map $\Pi\Lm T'\to T'$ is not surjective.  On
the other hand, we have $\Lm T''=\N x$ and $\Pi\Lm T''=T$ so the map
$\Pi\Lm T''\to T''$ is not injective.

We now start to define a functor $\Pi$ in the Tambara context.

\begin{definition}\label{defn-Pi}
 Fix a Tambara functor $S$.  For any naive $S$-module $M$ we put
 $\Pi M=S\oplus M$, which is a Mackey functor in an obvious way.  We
 define maps $S\xra{\eta}\Pi M\xra{\ep}S$ by $\eta(a)=(a,0)$ and
 $\ep(a,m)=a$.  We also define $\sg\:\Pi M\tm_S\Pi M\to\Pi M$ by
 $\sg(a,m,m')=(a,m+m')$. 

 Now consider a map $f\:X\to Y$ of finite $G$-sets.  We put 
 \[ F(f) = \{(x,x')\in X^2\st x\neq x',\; f(x)=f(x')\}, \]
 and we let $\pi,\pi'\:F(f)\to X$ be the obvious projections.  We then
 define $N_f\:(\Pi M)(X)\to(\Pi M)(Y)$ by 
 \[ N_f(a,m) = (N_f(a),\; T_f((N_\pi R_{\pi'}(a)).m)) \]
\end{definition}

To see that this definition is reasonable, consider the case where $G$
is the trivial group, so we have a semiring $S_1$ and an $S_1$-module
$M_1$ such that $S(X)=\Map(X,S_1)$ and $M(X)=\Map(X,M_1)$ for all
$X$.  From these data we can construct the augmented $S_1$-algebra 
$\Pi M_1$ and then the Tambara functor $T(X)=\Map(X,\Pi M_1)$.
Alternatively, we can define a Mackey functor $\Pi M$ and construct norm
maps as in Definition~\ref{defn-Pi}.  There is an obvious way to
identify $T$ with $\Pi M$ as Mackey functors, and
Definition~\ref{defn-Pi} is designed to ensure that this
identification is compatible with norm maps.  To see this, consider a
map $f\:X\to Y$ and maps $a\:X\to S_1$ and $m\:X\to M_1$, and put
$(b,n)=N^T_f(s,m)\in T(Y)$.  It will be notationally convenient to
think of $S_1$ and $M_1$ as subsets of $\Pi M_1$ so $(a,m)$ can be
written as $a+m$.  By definition we have
\[ b(y)+n(y)=\prod_{f(x)=y}(a(x)+m(x)). \]
Recall that the product in $\Pi M_1$ of any two elements of $M_1$ is
zero.  Thus, in expanding out the above product, we need only consider
the monomial $\prod_{f(x)=y}a(x)$ (which is $N^S_f(a)(y)$) and the
monomials that involve a single factor $m(x)$.  This observation gives 
\begin{align*}
 n(y) &= \sum_{f(x)=y} m(x) \prod_{f(x')=y,\;x'\neq x} a(x') \\
  &= \sum_{f(x)=y} m(x) \prod_{x'\st (x,x')\in F(f)} (R_{\pi'}a)(x,x') 
   = (T_f(m\;N_\pi R_{\pi'}(a)))(y)
\end{align*}
as required.

\begin{proposition}\label{prop-Pi-tambara}
 The above definition makes $\Pi M$ into an object of $\Mod_S$.
\end{proposition}

The proof will be divided into several lemmas.

\begin{lemma}\label{lem-Pi-NN}
 For any maps $X\xra{f}Y\xra{g}Z$ we have
 $N_{gf}=N_gN_f\:(\Pi M)(X)\to(\Pi M)(Y)$.
\end{lemma}
\begin{proof}
 For any map $u\:A\to B$ we will write $\pi_u$ and $\pi'_u$ for the
 two projections $F(u)\to A$ that were previously called $\pi$ and
 $\pi'$.  We also note that $F(f)\sse F(gf)\sse X^2$, and we introduce
 the sets
 \begin{align*}
  P &= \{(y,x')\in Y\tm X\st y\neq f(x'),\;g(y)=gf(x')\} \\
  Q &= \{(x,y')\in X\tm Y\st f(x)\neq y',\;gf(x)=g(y')\} \\
  F^*(gf) &= F(gf)\sm F(f) \\
    &= \{(x,x')\in X\tm X\st f(x)\neq f(x'),\;gf(x)=gf(x')\}.
 \end{align*}
 We let $\tht$ and $\tht'$ denote the two projections $F^*(gf)\to X$,
 so we have a commutative diagram
 \[ \xymatrix{
  F^*(gf) \ar[r]_{1\tm f} \ar[d]^{f\tm 1}
           \ar@/^2ex/[rr]^\tht \ar@/_2ex/[dd]_{\tht'} &
  Q \ar[r] \ar[d]^{f\tm 1} &
  X \ar[d]^f \\
  P \ar[r]_{1\tm f} \ar[d] &
  F(g) \ar[r]_{\pi_g} \ar[d]^{\pi'_g} &
  Y \\
  X \ar[r]_f &
  Y
 } \]
 in which the three squares are cartesian.

 Consider an element $(a,m)\in(\Pi M)(X)$. Put 
 \begin{align*}
  u &= N_{\pi_f}R_{\pi'_f}(a) \in S(X) \\
  v &= N_{\pi_g}R_{\pi'_g}N_f(a) \in S(Y), 
 \end{align*}
 so 
 \begin{align*}
  N_f(a,m) &= \left(N_f(a),T_f(m\,u)\right) \\
  N_gN_f(a,m) &= \left(N_{gf}(a),T_g(T_f(mu)\,v)\right).
 \end{align*}
 Recall that $T_f(mu)v=T_f(muR_f(v))$.  Here
 $R_f(v)=R_fN_{\pi_g}R_{\pi'_g}N_f(a)$, and using the cartesian
 properties of our diagram we see that this is the same as
 $N_{\tht}R_{\tht'}(a)$.  We now have 
 \[ u R_f(v) = N_{\pi_f}R_{\pi'_f}(a) \; N_{\tht}R_{\tht'}(a). \]
 Using the splitting $F(gf)=F(f)\amalg F^*(gf)$ we see that this is
 the same as $N_{\pi_{gf}}R_{\pi'_{gf}}(a)$.  Putting this together,
 we obtain 
 \[ N_gN_f(a,m) =
     \left(N_{gf}(a),\;T_{gf}(m\;N_{\pi_{gf}}R_{\pi'_{gf}}(a))\right),
 \]
 which is by definition $N_{gf}(a,m)$.
\end{proof}

\begin{lemma}\label{lem-Pi-RN}
 For any cartesian square
 \[ \xymatrix{
  W \ar[r]^f \ar[d]_g & X \ar[d]^h \\ Y \ar[r]_k & Z
 } \]
 we have $R_kN_h=N_gR_f\:(\Pi M)(Y)\to(\Pi M)(X)$.
\end{lemma}
\begin{proof}
 By definition we have 
 \begin{align*}
  R_kN_h(a,m) &=
   (R_kN_h(a),R_kT_h(m\,N_{\pi_h}R_{\pi'_h}(a))) \\
   &= (N_gR_f(a),T_gR_f(m\,N_{\pi_h}R_{\pi'_h}(a))) \\
   &= (N_gR_f(a),T_g(R_f(m)\,R_fN_{\pi_h}R_{\pi'_h}(a))) \\
  N_gR_f(a,m) &=
   (N_gR_f(a),T_g(R_f(m)\,N_{\pi_g}R_{\pi'_g}R_f(a))).
 \end{align*}
 It will therefore suffice to prove that
 \[ R_f N_{\pi_h} R_{\pi'_h} (a) = N_{\pi_g}R_{\pi'_g}R_f(a). \]
 Next, consider a point $(w,w')\in F(g)$, so $w\neq w'$ but
 $g(w)=g(w')$.  Using the cartesian property of the given square, we
 see that $f(w)\neq f(w')$ but $hf(w)=hf(w')$, so
 $(f(w),f(w'))\in F(h)$.  We can thus define $f_2\:F(g)\to F(h)$ by
 $f_2(w,w')=(f(w),f(w'))$, and this gives a commutative diagram as
 follows:
 \[ \xymatrix{
  W \ar[d]_f &
  F(g) \ar[l]_{\pi'_g} \ar[r]^{\pi_g} \ar[d]^{f_2} &
  W \ar[d]^f \\
  X &
  F(h) \ar[l]^{\pi'_h} \ar[r]_{\pi_h} &
  X
 } \]
 It is straightforward to check that the squares are in fact
 cartesian.  Using the cartesian property of the right square and the
 commutativity of the left square we get
 \[ R_f N_{\pi_h} R_{\pi'_h} = 
    N_{\pi_g} R_{f_2} R_{\pi'_h} = 
    N_{\pi_g} R_{\pi'_g} R_f
 \]
 as required.
\end{proof}

\begin{lemma}\label{lem-Pi-NT}
 Suppose we have maps $X\xra{f}Y\xra{g}Z$ with distributor 
 \[ \Dl(f,g) = (X\xla{p}A\xra{q}B\xra{r}Z). \]
 Then $N_gT_f=T_rN_qR_p\:(\Pi M)(X)\to(\Pi M)(Z)$.
\end{lemma}
\begin{proof}
 Consider a point $(a,m)\in(\Pi M)(X)$ and put
 $b=N_gT_f(a)=T_rN_qR_p(a)\in S(Z)$.  We have $N_gT_f(a,m)=(b,n)$ and
 $T_rN_qR_p(a,m)=(b,n')$ for certain elements $n,n'\in M(Z)$ and we
 need to show that $n=n'$.  Note that
 \begin{align*}
  T_rN_qR_p(a,m)
   &= T_rN_q(R_p(a),\;R_p(m)) 
    = T_r(N_qR_p(a),\;T_q(R_p(m)\,N_{\pi_q}R_{\pi'_q}R_p(a))) \\
   &= (b,\;T_{rq}(R_p(m)\,N_{\pi_q}R_{p\pi'_q}(a))),
 \end{align*}
 so $n'=T_{rq}(R_p(m)\,N_{\pi_q}R_{p\pi'_q}(a))$.  On the other hand,
 one can see from the definitions that $rq=gfp\:A\to Z$.  Using this
 together with Frobenius reciprocity for $p$ we get
 \[ n' = T_{gf}T_p(R_p(m)\,N_{\pi_q}R_{p\pi'_q}(a))
       = T_{gf}(m\,T_pN_{\pi_q}R_{p\pi'_q}(a)).
 \]
 We write $c'=T_pN_{\pi_q}R_{p\pi'_q}(a)\in S(X)$ so that
 $n'=T_{gf}(m\,c')$. 

 Next, using the definitions and Frobenius reciprocity for $f$ we have
 \begin{align*}
  N_gT_f(a,m)
   &= N_g(T_f(a),T_f(m)) 
    = (b,\;T_g(T_f(m)\,N_{\pi_g}R_{\pi'_g}T_f(a))) \\
   &= (b,T_{gf}(m\,R_fN_{\pi_g}R_{\pi'_g}T_f(a))).
 \end{align*}
 Thus, if we put $c=R_fN_{\pi_g}R_{\pi'_g}T_f(a)\in S(X)$ we have
 $n=T_{gf}(m\,c)$.  It will thus suffice to prove that $c=c'$.

 We now define a cartesian square
 \[ \xymatrix{
  P \ar[d]_{\psi} \ar[r]^\phi & 
  X \ar[d]^f \\
  F(g) \ar[r]_{\pi'_g} & 
  Y
 } \]
 by 
 \begin{align*}
  P &= \{(y,x')\in Y\tm X\st y\neq f(x'),\; g(y)=gf(x')\} \\
  \phi(y,x') &= x' \\
  \psi(y,x') &= (y,f(x')).
 \end{align*}
 This gives $R_{\pi'_g}T_f=T_\psi R_\phi$ so
 $c=R_fN_{\pi_g}T_\psi R_\phi(a)$.  

 We next want to construct the distributor $\Dl(\psi,\pi_g)$.  To
 describe this, we put $G_y=g^{-1}\{g(y)\}\sm\{y\}$.  The distributor
 involves the fibres $\pi_g^{-1}\{y\}\sse F(g)$, but $\pi'_g$
 gives a natural bijection $\pi_g^{-1}\{y\}\to G_y$ and it will be
 convenient to use $G_y$ instead.  We can now identify
 $\Dl(\psi,\pi_g)$ with the diagram 
 \[ P \xla{p^*} A^* \xra{q^*} B^* \xra{r^*} Y, \]
 where
 \begin{align*}
  B^* &= \{(y,s)\st y\in Y,\; s\:G_y\to X,\; fs=1\} \\
  A^* &= \{(y,y',s)\st y\in Y,\; s\:G_y\to X,\; fs=1,\; y'\in G_y\} \\
  p^*(y,y',s) &= (y,s(y')) \\
  q^*(y,y',s) &= (y,s) \\
  r^*(y,s) &= y. 
 \end{align*}
 This gives 
 \[ c = R_f T_{r^*} N_{q^*} R_{\phi p^*}(a). \]

 We now construct another diagram
 \[ \xymatrix{ &
     F(q) \ar[dl]_{p\pi'_q} \ar[d]^\sg \ar[r]^{\pi_q} &
     A \ar[d]^\rho \ar[r]^p &
     X \ar[d]^f \\
     X &
     A^* \ar[l]^{\phi p^*} \ar[r]_{q^*} &
     B^* \ar[r]_{r^*} &
     Y
 } \]
 Recall that 
 \[ A = \{(y,t)\st y\in Y,\; t\:g^{-1}\{g(y)\}\to X,\; ft=1\} \]
 and $p(y,t)=t(y)$.  We define $\rho(y,t)=(y,t|_{G_y})$; this gives a
 map $\rho\:A\to B^*$ making the right square commute.  Next, recall
 that 
 \[ B = \{(z,t)\st z\in Z,\; t\:g^{-1}\{z\}\to X,\; ft=1\} \]
 and $q(y,t)=(g(y),t)$.  Using this we get
 \[ F(q) = \{(y,y',t) \st (y,y')\in F(g),\;
               t\:g^{-1}\{g(y)\}=g^{-1}\{g(y')\}\to X,\;
                ft=1\},
 \]
 so we can define $\sg\:F(q)\to A^*$ by
 $\sg(y,y',t)=(y,y',t|_{G_y})$.  This makes the middle square and the
 left hand triangle commute.  One can check that the two squares are
 cartesian, and it follows that 
 \[ R_fT_{r^*}N_{q^*}R_{\phi p^*} = 
    T_pR_\rho N_{q^*}R_{\phi p^*} = 
    T_p N_{\pi_q} R_\sg R_{\phi p^*} = 
    T_p N_{\pi_q} R_{p\pi'_q}.
 \]
 This implies that $c=c'$, as required.
\end{proof}

\begin{proof}[Proof of Proposition~\ref{prop-Pi-tambara}]
 We first show that $\Pi S$ is a Tambara functor.  It will suffice to
 check the conditions in Proposition~\ref{prop-UG-pres}.  The
 nonobvious part of condition~(a) is covered by Lemma~\ref{lem-Pi-NN}.
 Condition~(b) is Lemma~\ref{lem-Pi-NT}, and the nonobvious part
 of~(c) is Lemma~\ref{lem-Pi-RN}.

 Next, it is clear by inspection that the maps
 $S\xra{\eta}\Pi M\xra{\ep} S$ and 
 $\Pi M\tm_S\Pi M\xra{\sg}\Pi M$ preserve norm maps as well as being
 Mackey morphisms, so they make $\Pi M$ into a semigroup object in
 $\AugAlg_S$, or in other words, an object of $\Mod_S$.
\end{proof}

\begin{proof}[Proof of Proposition~\ref{prop-aug-alg}]
 We now have a functor $\Pi\:\NMod_S\to\Mod_S$, and it is clear that
 $\Lm\Pi=1$.  We can define $\phi\:\Pi\Lm T\to T$ by $\phi(a,m)=a+m$,
 and it is clear that this is a morphism of Mackey functors.  In fact,
 the same argument that we used in the semiring case shows that $\phi$
 preserves multiplication, so it is a morphism of Green functors.  In
 the additively complete case we again have an inverse map
 $\phi^{-1}(t)=(\ep(t),t-\ep(t))$.  Example~\ref{eg-Pi-bad} will
 exhibit a case where $\phi$ does not preserve more general norms.
\end{proof}

\begin{example}\label{eg-Pi-bad}
 We will work in the category $\TP$ of Tambara pairs as described in
 Section~\ref{sec-two}.  We proved there that $\TP$ is equivalent to
 the category of Tambara functors for the group of order $2$.  

 Let $S$ and $T$ be the Tambara pairs described in
 Example~\ref{eg-TP}.  Put $T_2=T\tm_ST$, which consists of the
 semirings 
 \begin{align*} 
  A_2 &= A\tm_\Z A = \Z\{1,\al_0,\al_1\} \\
  B_2 &= B\tm_\Z B = \Z\{1,\bt_0,\bt_1\}\oplus(\Z/2)\{\gm_0,\gm_1\}.
 \end{align*}
 We define $\sg\:S_2\to S$ by
 \begin{align*}
  \sg(i+j_0\al_0+j_1\al_1) &= i+(j_0+j_1)\al \\
  \sg(i+j_0\bt_0+j_1\bt_1+k_0\gm_0+k_1\gm_1) &=
   i+(j_0+j_1)\bt+(k_0+k_1)\gm.
 \end{align*}
 One can check that this is a morphism of Tambara pairs.  For example,
 we will show that $\sg$ commutes with the norm map.  Note that $T_2$
 is defined as a sub-Tambara-pair of $T\tm T$, which gives the rule
 \[ \nrm(i+j_0\al_0+j_1\al_1) =
     i^2+ij_0\bt_0+ij_1\bt_1+j_0^2\gm_0+j_1^2\gm_1,
 \]
 so 
 \[ \sg(\nrm(i+j_0\al_0+j_1\al_1))=
      i^2 + (ij_0+ij_1)\bt + (j_0^2+j_1^2)\gm.
 \]
 On the other hand, we have 
 \[ \nrm(\sg(i+j_0\al_0+j_1\al_1)) 
     = \nrm(i+(j_0+j_1)\al)
     = i^2 + (ij_0+ij_1)\bt + (j_0^2+j_1^2+2j_0j_1)\gm.
 \]
 These are the same because $2\gm=0$.  It is clear that $\sg$ is
 commutative, associative and unital, so we have given $T$ the
 structure of a commutative group object in $\AugAlg_S$, so
 $T\in\Mod_S$.  We can use $\phi$ to identify $\Pi\Lm T$ with $T$ as
 Mackey pairs.  However, we claim that this does not respect norm
 maps.  Indeed, in $T$ we have $\nrm(\al)=\gm$ by definition.  Let
 $\ep$ be the map $G\to 1$, so the value of $\nrm(\al)$ in $\Pi\Lm T$
 is found by calculating $N_\ep(0,\al)$ as in
 Definition~\ref{defn-Pi}.  This gives 
 \[ N_\ep(0,\al)
     = (N_\ep(0),T_\ep(\al\, N_{\pi_\ep}R_{\pi'_\ep}(0)))
     = (N_\ep(0),T_\ep(\al\, N_{\pi_\ep}(0))).
 \]
 Here the maps $\ep\:G\to 1$ and $\pi_\ep\:F(\ep)\to G$ are both
 surjective, so Lemma~\ref{lem-norm-zero} tells us us that
 $N_\ep(0)=0$ and $N_{\pi_\ep}(0)=0$.  It follows that $\nrm(\al)=0$
 in $\Pi\Lm T$, so $\phi$ is not a Tambara morphism.
\end{example}

\section{Coefficient systems}
\label{sec-csys}

In this section we introduce the category $\CSys_G$ of coefficient
systems, and the related category $\MCSys_G$ of multiplicative
coefficient systems.  These will be linked by adjunctions to
$\Mackey_G$ and $\Tambara_G$.  Later we will define $\Q$-linear
analogues denoted by $\Q\CSys_G$ and so on, and show that our
adjunctions restrict to give equivalences $\Q\Mackey_G\simeq\Q\CSys_G$
and $\Q\Tambara_G\simeq\Q\MCSys_G$.

\begin{definition}
 Let $\Orbt_G$ denote the category of transitive $G$-sets and
 equivariant isomorphisms between them.  A \emph{coefficient system}
 will mean a functor from $\Orbt_G$ to the category of semigroups.  We
 write $\CSys_G$ for the category of coefficient systems. 
\end{definition}
\begin{remark}
 As all morphisms in $\Orbt_G$ are invertible, covariant functors can
 be converted to contravariant functors and vice-versa.  We prefer to
 use covariant functors here to maximise compatibility with our later
 discussion of rational Tambara functors.  Given a coefficient system
 $N$ and an isomorphism $f\:U\to V$ of transitive $G$-sets, we will
 write $f_*$ for the resulting map $N(U)\to N(V)$.
\end{remark}

\begin{definition}\label{defn-rho-cosets}
 For any subgroup $H\leq G$ and any $g\in G$ we define
 $\rho(g)\:G/H\to G/gHg^{-1}$ by 
 \[ \rho(g)(xH)=xHg^{-1} = xg^{-1}\;gHg^{-1}. \]

 Now let $N$ and $N'$ be coefficient systems.  For any map
 $u\:N(G/H)\to N'(G/H)$ we write $g.u$ for the composite
 \[ N(G/gHg^{-1}) \xra{\rho(g^{-1})_*}
    N(G/H) \xra{u}
    N'(G/H) \xra{\rho(g)_*}
    N'(G/gHg^{-1}).
 \]
 This defines a map 
 \[ g \: \Map(N(G/H),N'(G/H)) \to
         \Map(N(G/gHg^{-1}),N'(G/gHg^{-1})).
 \]
 We give the set $\prod_{K\leq G}\Map(N(G/K),N'(G/K))$ the unique
 $G$-action such that the diagrams
 \[ \xymatrix{
  \prod_{K\leq G}\Map(N(G/H),N'(G/H))
  \ar[d]_{\pi_H}
  \ar[rrr]^g &&&
  \prod_{K\leq G}\Map(N(G/H),N'(G/H))
  \ar[d]^{\pi_{gHg^{-1}}} \\
  \Map(N(G/H),N'(G/H)) \ar[rrr]_g &&&
  \Map(N(G/gHg^{-1}),N'(G/gHg^{-1}))
 } \]
 commute.
\end{definition}

\begin{proposition}\label{prop-CSys-maps}
 Let $N$ and $N'$ be coefficient systems.  Then there is a natural
 isomorphism 
 \[ \CSys_G(N,N') =
     \left[\prod_{K\leq G}\Map(N(G/K),N'(G/K))\right]^G
 \]
\end{proposition}
\begin{proof}
 Consider a morphism $u\:N\to N'$ of coefficient systems.  For each
 subgroup $H\leq G$ we have a component $u_H\in\Map(N(G/H),N'(G/H))$.
 Together these give an element
 $\tht(u)\in\prod_H\Map(N(G/H),N'(G/H))$.  This determines $u$,
 because every orbit is isomorphic to $G/H$ for some $H$.  Every
 morphism $G/H\to G/K$ has the form $\rho(g)$ for some $g$ with
 $gHg^{-1}=K$.  Naturality with respect to such a morphism is
 equivalent to the identity $u_{gHg^{-1}}=g.u_H$, and this holds for
 all $g$ and $H$ iff $\tht(u)$ is a $G$-fixed point in
 $\prod_H\Map(N(G/H),N'(G/H))$.  The claim is clear from this.
\end{proof}

At the expense of some arbitrary choices, we can cut this description
down further, as follows.
\begin{definition}\label{defn-subgroup-system}
 A \emph{subgroup system} for $G$ is a list $(H_1,\dotsc,H_r)$ of
 subgroups such that
 \begin{itemize}
  \item[(a)] Each subgroup of $G$ is conjugate to $H_i$ for precisely
   one value of $i$.
  \item[(b)] We have $|H_i|\leq|H_j|$ whenever $i\leq j$. 
 \end{itemize}
 It is clear that we must have $H_1=1$ and $H_r=G$.  We let $N_i$
 denote the normaliser in $G$ of $H_i$, and put $W_i=N_i/H_i$.
\end{definition}

Given a subgroup system as above, we note that any transitive $G$-set
is isomorphic to $G/H_i$ for a unique value of $i$, and that the
automorphism group of $G/H_i$ is $W_i$.  It follows that if we regard
$W_i$ as a groupoid with one object, then $\Orbt_G$ is equivalent to
the coproduct $\coprod_iW_i$.  Using this we get an equivalence 
\[ \CSys_G\simeq\prod_{i=1}^r\Mod_{\N[W_i]}. \]

We will exhibit an adjunction between functors
$q\:\Mackey_G\to\CSys_G$ and $r\:\CSys_G\to\Mackey_G$, and show that
this restricts to an equivalence between the corresponding rational
categories.

\begin{definition}\label{defn-q}
 Given $M\in\Mackey_G$ and $U\in\Orbt_G$ we let $(qM)(U)$ be the
 quotient of $M(U)$ by the sum of all subsemigroups $T_uM(U')$, where
 $u\:U'\to U$ is a map of transitive $G$-sets that is not an
 isomorphism.  This is clearly functorial for isomorphisms
 $U_0\to U_1$, and for arbitrary morphisms $M_0\to M_1$ of Mackey
 functors.  It therefore gives a functor $q\:\Mackey_G\to\CSys_G$.
\end{definition}
\begin{remark}
 In slightly different language, $(qM)(G/H)$ is $M(G/H)$ modulo
 transfers from $M(G/K)$ for proper subgroups $K<H$.
\end{remark}

\begin{definition}\label{defn-Phi}
 We define a functor $\Phi\:(\bCA_G)^{\opp}\to\CSys_G$ as follows.  On
 objects, we define $\Phi(X)(T)=\N\{\Map_G(T,X)\}$ (so
 $\Phi(X)(G/H)=\Z\{X^H\}$).  We write $[u]$ for the basis element
 corresponding to a $G$-map $u\:T\to X$.

 A morphism $X\to Y$ in $\CA_G$ is represented by a diagram
 $X\xla{p}A\xra{q}Y$ of finite $G$-sets.  The corresponding map 
 \[ \al\:\N\{\Map_G(T,Y)\} \to \N\{\Map_G(T,X)\} \]
 sends $[T\xra{v}Y]$ to the sum of the elements $[T\xra{a}A\xra{p}X]$
 for all $G$-maps $T\xra{a}A$ such that $qa=v$.  

 Alternatively, the functor $\Map_G(T,-)$ (from finite $G$-sets to
 finite sets) preserves finite limits and so induces a functor
 $\bCA_G\to\bCA_1$.  However, there is an evident isomorphism
 $\bCA_G\simeq(\bCA_G)^{\opp}$, and $\bCA_1$ is canonically equivalent
 to the category of finitely generated free abelian semigroups, so we get a
 functor $(\CA_G)^{\opp}\to\Semigroups$, which we denote by $\phi^T$.
 We then have $\Phi(X)(T)=\phi^T(X)$.
\end{definition}

\begin{definition}\label{defn-r}
 We now define $r\:\CSys_G\to\Mackey_G$ by 
 \[ r(N)(X) =
     \CSys_G(\Phi(X),N) =
     \left[\prod_{H\leq G}\Map(X^H,N(G/H))\right]^G.
 \]
\end{definition}
\begin{remark}\label{rem-free-orbit}
 One checks that $\Phi(G)(G)=\N[G]$ but $\Phi(G)(T)=0$ if
 $T\not\simeq G$.  It follows that $r(N)(G)=N(G)$.
\end{remark}

We can rewrite the definition in a useful way when $X$ is an orbit.

\begin{proposition}\label{prop-rNGH}
 There is a natural isomorphism
 \[ r(N)(G/H)=\left[\prod_{K\leq H}N(G/K)\right]^H. \]
 If $K_1,\dotsc,K_r$ is a subgroup system for $H$, this can be written
 as 
 \[ r(N)(G/H) = \prod_{i=1}^r N(G/K_i)^{W_HK_i}. \]
\end{proposition}
\begin{proof}
 From the definitions, we have
 \[ r(N)(G/H) =
     \left[\prod_{K\leq G}\Map((G/H)^K,N(G/K))\right]^G.
 \]
 Suppose we have $n\in r(N)(G/H)$.  For each $K\leq H$, the basepoint
 $H\in G/H$ is fixed by $K$, so we have an element
 $n_K(H)\in N(G/K)$.  We can thus define a map 
 \[ \phi \: r(N)(G/H) \to \prod_{K\leq H}N(G/K) \]
 by $\phi(n)_K=n_K(H)$.  As $n$ is $G$-invariant, it follows that
 $\phi(n)$ is $H$-invariant.

 Now suppose instead we have an element
 $m\in\left[\prod_{K\leq H}N(G/K)\right]^H$.  Consider a subgroup
 $L\leq G$, and an element $xH\in(G/H)^L$.  This means that $LxH=xH$
 and so $x^{-1}Lx\leq H$.  We thus have an element
 $m_{x^{-1}Lx}\in N(G/x^{-1}Lx)$ and a map
 $\rho(x)\:G/x^{-1}Lx\to G/L$ giving an element
 $\rho(x)_*(m_{x^{-1}Lx})\in N(G/L)$.  Using the fact that $m$ is
 $H$-invariant, we see that this element does not depend on the choice
 of element $x$ representing the coset $xH$.  We can thus define
 $\psi(m)\in\prod_L\Map((G/H)^L,N(G/L))$ by
 $\psi(m)_L(xH)=\rho(x)_*(m_{x^{-1}Lx})$.  One can check that this is
 invariant under the action of $G$, so we have defined a map 
 \[ \psi\:\left[\prod_{K\leq H}N(G/K)\right]^H \to r(N)(G/H). \] 
 We leave it to the reader to show that this is inverse to $\phi$. 
\end{proof}

\begin{remark}\label{rem-rN-limit}
 As yet another way to formulate the definition, we introduce the
 category $(\Orbt_G\downarrow X)$.  The objects are diagrams
 $(U\xra{x}X)$, where $U$ is a $G$-orbit and $x$ is a $G$-map.  The
 morphisms from $(U\xra{x}X)$ to $(U'\xra{x'}X)$ are the
 $G$-isomorphisms $p\:U\to U'$ with $x'p=x$.  Any coefficient
 system $N$ gives a functor $(U\xra{x}X)\mapsto N(U)$ from this
 category to $\Semigroups$, and $rN(X)$ is easily seen to be the
 inverse limit of this functor.  We will use this interpretation
 without further comment when convenient.

 One can check that in this picture the operations associated to a map
 $f\:X\to Y$ are
 \begin{align*}
  (R_fn)(U\xra{x}X) &= n(U\xra{fx}Y) \\
  (T_fm)(U\xra{y}Y) &= \sum_{x\in\Map_G(U,X),\;fx=y} m(U\xra{x}X).
 \end{align*}
\end{remark}

\begin{proposition}\label{prop-rq-unit}
 There is a natural map $\eta\:M\to rqM$ of Mackey functors given by 
 \[ (\eta m)(U\xra{x}X) = \pi R_x(u)\in (qM)(U) \]
 (where $\pi$ is the quotient map $M(U)\to qM(U)$).
\end{proposition}

The proof will be given after some preliminaries.

\begin{definition}\label{defn-sec}
 For any equivariant map $f\:X\to Y$, we put 
 \[ \Sec(f) = \{s\in\Map_G(Y,X)\st fs=1_Y\} =
     \{\text{equivariant sections of } f\}.
 \]
 If $Y$ is a single $G$-orbit, we also put
 \[ \Sec'(f) = \{\text{ orbits } U\sse X \st f\:U\to Y 
                  \text{ is an isomorphism }. \}
 \]
\end{definition}

\begin{lemma}\label{lem-sections}
 \begin{itemize}
  \item[(a)] Let $f\:U\to V$ be a $G$-map between transitive
   $G$-sets.  Then $f$ is automatically surjective (and so
   $|U|\geq|V|$).  Moreover, $f$ is an isomorphism iff it is injective
   iff it is a split epimorphism iff $|U|=|V|$. 
  \item[(b)] More generally, let $f\:X\to V$ be a map of finite
   $G$-sets where $V$ is transitive.  We then have
   \[ \Sec'(f) = \{\text{ orbits } U\sse X \st |U|=|V|\}. \]
   Moreover, there is a bijection $\Sec(f)\simeq\Sec'(f)$ given by
   $s\mapsto s(V)$ and $U\mapsto(f|_U)^{-1}$.
  \item[(c)] Suppose we have maps $V\xra{y}Y\xla{f}X$ of finite
   $G$-sets, where $V$ is transitive.  Form a pullback square 
   \[ \xymatrix{ 
    U \ar[r]^{x'} \ar[d]_e & X \ar[d]^f \\ V \ar[r]_y & Y,
   } \]
   and let $U_1,\dotsc,U_n$ be the $G$-orbits in $V$.  Let
   $e_i\:U_i\to V$ and $x'_i\:U_i\to X$ be the restrictions of $e$ and
   $x'$.  We may assume that the $U_i$ are numbered so that
   $e_1,\dotsc,e_p$ are isomorphisms (for some $p$, possibly $p=0$) and
   $e_{p+1},\dotsc,e_n$ are not isomorphisms.  For $i\leq p$ put
   $x_i=x'_ie_i^{-1}\:V\to X$.  Then the maps $x_1,\dotsc,x_p$ are all
   different, and this is a complete list of all $G$-maps $x\:V\to X$
   with $fx=y$.
 \end{itemize}
\end{lemma}
\begin{proof}
 \begin{itemize}
  \item[(a)] It is clear that $f(U)$ is a nonempty $G$-invariant
   subset of the transitive $G$-set $V$, so $f(U)=V$.  If $f$ is a
   split epimorphism of $G$-sets then the splitting map $g\:V\to U$
   must also be surjective by the same logic, and it follows easily
   that $g$ is an inverse for $f$.  The rest is clear.
  \item[(b)] If $s\in\Sec(f)$ is a section then $s(V)\sse X$ is an
   orbit and $f\:s(V)\to V$ is a split epimorphism of $G$-sets, hence
   an isomorphism by~(a).  Everything else is clear from this.
  \item[(c)] The $G$-maps $x\:V\to X$ lifting $y$ biject naturally
   with the sections of the map $e\:U\to V$.  Given this, claim~(c)
   follows from~(b).
 \end{itemize}
\end{proof}

\begin{proof}[Proof of Proposition~\ref{prop-rq-unit}]
 Given maps $W\xra{g}X\xra{f}Y$ of finite $G$-sets, we must show that
 the diagram 
 \[ \xymatrix{
   M(W) \ar[d]_\eta & 
   M(X) \ar[l]_{R_g} \ar[d]^\eta \ar[r]^{T_f} &
   M(Y) \ar[d]^\eta \\
   rqM(W) &
   rqM(X) \ar[l]^{R_g} \ar[r]_{T_f} &
   rqM(Y)
 } \]
 commutes.  For the left square, we have
 \begin{align*}
  (R_g\eta m)(U\xra{w}W) 
   &= (\eta m)(U\xra{gw}X) 
    = \pi R_{gw}(m) 
    = \pi R_w(R_gm) \\
   &= (\eta R_g m)(U\xra{w}W)
 \end{align*}
 as required.

 For the right hand side, we have
 $(\eta T_fm)(U\xra{y}Y)=\pi R_yT_f(m)$.  If we form a pullback as in
 Lemma~\ref{lem-sections} we find that
 \[ \pi R_yT_f = \pi T_eR_v = \sum_{i=1}^n \pi T_{e_i}R_{v_i}. \]
 If $i>p$ then $e_i\:V_i\to U$ is not an isomorphism so
 $\pi T_{e_i}=0$.  If $i\leq p$ then $e_i$ is an isomorphism so
 $T_{e_i}=R_{e_i^{-1}}$ so
 $\pi T_{e_i}R_{v_i}=\pi R_{v_ie_i^{-1}}=\pi R_{x_i}$.  It follows
 that 
 \[ (\eta T_fm)(U\xra{y}Y)=\pi R_yT_f(m)=
     \sum_{i=1}^p \pi R_{x_i}(m) = 
      \sum_{i=1}^p (\eta m)(U\xra{x_i}X) = 
       (T_f\eta m)(U\xra{y}Y)
 \] 
 as required.  This proves that $\eta\:M\to rqM$ is a morphism of
 Mackey functors.  Naturality in $M$ is clear.
\end{proof}

\begin{proposition}\label{prop-qr-unit}
 There is a natural map $\ep\:qrN\to N$ of coefficient systems given
 by $\ep\pi(n)=n(U\xra{1}U)\in N(U)$ (for all transitive $G$-sets $U$
 and elements $n\in (rN)(U)$).
\end{proposition}
\begin{proof}
 Suppose we have an element $n\in (rN)(U)$, so we have
 $n(U'\xra{t}U)\in N(U')$ for all transitive $U'$ and all $G$-maps
 $t\:U'\to U$.  In particular, we can put
 $\ep'n=n(U\xra{1}U)\in{}N(U)$.  This defines a map
 $\ep'\:(rN)(U)\to N(U)$.  Suppose that $n=T_f(p)$ for some map
 $f\:V\to U$ (where $V$ is transitive and $f$ is not invertible) and
 $p\in(rN)(V)$.  Then the element $\ep'n=(T_fp)(U\xra{1}U)$ is by
 definition a sum of terms indexed the equivariant sections of $f$,
 but there are no sections, so $\ep'n=0$.  It follows that
 $\ep'\:rN(U)\to N(U)$ induces a map $\ep\:qrN(U)\to N(U)$.  This is
 easily seen to be natural for isomorphisms $U_0\to U_1$ of transitive
 $G$-sets, so we have a morphism $\ep\:qrN\to N$ of coefficient
 systems.  Naturality in $N$ is also clear.
\end{proof}

\begin{proposition}\label{prop-qr-adjoint}
 The maps $\eta$ and $\ep$ are the unit and counit of an adjunction
 between $q$ and $r$.
\end{proposition}
\begin{proof}
 The map $\eta$ gives rise to a natural map 
 \[ \lm\:\CSys_G(qM,N) \to \Mackey_G(M,rN) \]
 by $\lm(\al)=(M\xra{\eta}rqM\xra{r(\al)}rN)$.  Suppose we have a map
 $x\:U\to X$ of finite $G$-sets with $U$ transitive, and an element
 $m\in M(X)$.  We then have $\lm(\al)(m)\in (rN)(X)$ and thus
 $\lm(\al)(m)(U\xra{x}X)\in N(U)$.  On the other hand, we have
 $R_x(m)\in M(U)$ and so $\pi(R_x(m))\in qM(U)$ and
 $\al(\pi(R_x(m)))\in N(U)$.  By inspecting the definitions, we obtain
 the formula 
 \[ \lm(\al)(m)(U\xra{x}X) = \al(\pi(R_x(m))). \]

 Similarly, the map
 $\ep$ gives rise to a natural map 
 \[ \rho\:\Mackey_G(M,rN) \to \CSys_G(qM,N) \]
 by $\rho(\bt)=(qM\xra{q(\bt)}qrN\xra{\ep}N)$.  For any transitive
 $G$-set $U$ and $m\in M(U)$ we have $\bt(m)\in(rN)(U)$ and so
 $\bt(m)(U\xra{1}U)\in N(U)$.    By inspecting the
 definitions, we obtain the formula
 \[ \rho(\bt)(\pi(m)) = \bt(m)(U\xra{1}U). \]

 We need to show that $\lm$ and $\rho$ are inverse to each other.  If
 we start with $\al\:qM\to N$ we have
 \[
  \rho(\lm(\al))(\pi(m))
   = \lm(\al)(m)(U\xra{1}U) 
   = \al(\pi(R_1(m))) = \al(\pi(m)),
 \]
 so $\rho(\lm(\al))=\al$.  If instead we start with $\bt\:M\to rN$ we
 have 
 \[ \lm(\rho(\bt))(m)(U\xra{x}X) = 
     \rho(\bt)(\pi(R_x(m))) = 
      \bt(R_x(m))(U\xra{1}U).
 \]
 Now, $\bt$ is assumed to be a morphism $M\to rN$ of Mackey functors
 so $\bt R_x=R_x\bt$.  The map $R_x\:(rN)(X)\to(rN)(U)$ is defined by 
 \[ (R_xn)(T\xra{t}U) = n(T\xra{xt}X). \]
 In particular, we have $(R_xn)(U\xra{1}U)=n(U\xra{x}X)$.  Putting
 this together, we get
 \[ \lm(\rho(\bt))(m)(U\xra{x}X) = 
      R_x(\bt(m))(U\xra{1}U) = \bt(m)(U\xra{x}X).
 \]
 This holds for all $X$, $U$, $x$ and $m$, so $\lm(\rho(\bt))=\bt$ as
 required. 
\end{proof}

\begin{proposition}\label{prop-ep-split}
 There is a natural map $\sg\:N\to qrN$ (for $N\in\CSys_G$) such that
 $\ep\sg=1$.  
\end{proposition}
\begin{proof}
 Consider an object $T\in\Orbt_G$ and an element $u\in N(T)$.  Given
 an orbit $U$ and a $G$-map $p\:U\to T$ we define
 $\sg_0(u)(U,p)=p^{-1}_*(u)$ if $p$ is an isomorphism, and
 $\sg_0(u)(U,p)=0$ otherwise.  This gives a system of maps
 $\sg_0(u)\:\Map_G(U,T)\to N(U)$ that are natural for isomorphisms of
 $U$, or in other words an element $\sg_0(u)\in(rN)(T)$.  We let
 $\sg(u)$ denote the image of $\sg_0(u)$ in $(qrN)(T)$.  This defines a
 natural  map $\sg\:N\to qrN$, and it is clear from the definitions
 that $\ep\sg=1$.
\end{proof}

Recall that $\Mackey_G$ is a symmetric monoidal category under the box
product operation $\btm$, with the Burnside semiring Mackey functor
$A$ acting as the unit.  There is also a simpler symmetric monoidal
structure on $\CSys_G$ given by $(N\ot N')(U)=N(U)\ot N'(U)$.  For
this, the unit is the constant functor $c\N$ given by $c\N(U)=\N$ for
all $U$.

\begin{proposition}\label{prop-q-symmon}
 There are natural isomorphisms $qA\to c\N$ and
 $q(M\btm M')\to q(M)\ot q(N')$ making $q$ a symmetric monoidal
 functor. 
\end{proposition}
\begin{proof}
 Let $T$ be a transitive $G$-set.  Recall that $A(T)$ is the set of
 isomorphism classes of finite $G$-sets equipped with a map to $U$.
 We define $\al'\:A(T)\to\N$ by 
 \[ \al'[W\xra{t}T] = |\Sec(t)| = |\Sec'(t)|. \]
 It is easy to see that this is a homomorphism.  It has
 $\al'[T\xra{1}T]=1$, so it is surjective.  

 Recall also that for $f\:U\to T$ and $[W\xra{u}U]\in A(U)$ we have
 $T_f[W\xra{u}U]=[W\xra{fu}T]$, and note that if $u$ has no sections
 then $fu$ has no sections.  From this it is easy to see that $\al'$
 induces an isomorphism $qA(T)\to\N=c\N(T)$.  This is clearly natural
 and so gives an isomorphism $qA\to c\N$.

 Now suppose we have two Mackey functors $M$ and $M'$.  Given a map
 $u\:X\to U$ and elements $m\in M(X)$ and $m'\in M'(X)$ we define 
 \[ \bt'(u,m,m') = \sum_{s\in\Sec(u)} \pi(R_s(m))\ot\pi(R_s(m')) \in 
     (qM)(U)\ot (qM')(U) = (qM\ot qM')(U).
 \]
 Now suppose we have a map $p\:W\to X$ and $m\in M(X)$ and
 $n'\in M'(W)$.  We claim that
 $\bt'(up,R_p(m),n')=\bt'(u,m,T_p(n'))$.  Indeed, we have
 \begin{align*}
  \bt'(up,R_p(m),n')
   &= \sum_{t\in\Sec(up)} \pi R_{pt}(m)\ot \pi R_t(n') 
    = \sum_{s\in\Sec(u)} \pi R_s(m)\ot
       \left(\sum_{\begin{array}{c}t\:T\to W\\ pt=s\end{array}}
        \pi R_t(n')
       \right) \\
  \bt'(u,m,T_p(n')) 
   &= \sum_{s\in\Sec(u)} \pi R_s(m)\ot\pi R_sT_p(n').
 \end{align*}
 For $s\in\Sec(u)$ we can form a pullback square 
 \[ \xymatrix{
  \widetilde{T} \ar[r]^{\widetilde{s}} \ar[d]_{\widetilde{p}} &
  W \ar[d]^p \\
  T \ar[r]_s &
  X.
 } \]
 We then have $\pi R_sT_p=\pi T_{\widetilde{p}}R_{\widetilde{s}}$.  We
 can write this as a sum over the orbits $V\sse\widetilde{T}$.  If the
 map $\widetilde{p}|_V\:V\to T$ is not an isomorphism, then the
 corresponding contribution is killed by $\pi$.
 Lemma~\ref{lem-sections}(c) tells us that the remaining orbits biject
 with $G$-maps $t\:T\to W$ satisfying $pt=s$.  The identity
 $\bt'(up,R_p(m),n')=\bt'(u,m,T_p(n'))$ now follows easily.  By a
 similar argument, for any $n\in M(W)$ and $m'\in M'(X)$ we have
 $\bt'(up,n,R_p(m'))=\bt'(u,T_p(n),m')$.  It now follows that there is
 a unique map $\bt''\:(M\btm M')(U)\to(qM\ot qM')(U)$ satisfying
 $\bt''(T_u(m\ot m'))=\bt'(u,m,m')$ for all $(u,m,m')$ as above.  It
 is clear by construction that $\bt''T_p=0$ for all non-split maps
 $V\xra{p}U$ of orbits, so there is an induced map 
 $\bt\:q(M\btm M')(U)\to(qM\ot qM')(U)$.

 In the opposite direction, suppose we have $m\in M(U)$ and
 $m'\in M'(U)$, giving $\pi(m\ot m')\in q(M\btm M')(U)$.  If $m$ has
 the form $T_p(n)$ for some non-split map $V\xra{p}U$ of orbits, then
 we have $\pi(m\ot m')=\pi T_p(n\ot R_p(m'))=0$.  Using this and its
 symmetrical counterpart, we see that there is a well-defined map
 $\gm\:qM(U)\ot qM'(U)\to q(M\btm M')(U)$ such that
 $\gm(\pi(m)\ot\pi(m'))=\pi(m\ot m')$.  We leave the reader to check
 that this is inverse to $\bt$.
\end{proof}

\begin{definition}\label{defn-MCSys}
 A \emph{multiplicative coefficient system} is a covariant functor
 $\Orb_G\to\Rings$, where $\Orb_G$ is the category of transitive
 $G$-sets and equivariant maps.  We write $\MCSys_G$ for the category
 of such objects.  Given $R\in\MCSys_G$ and a morphism $f\:U\to T$ in
 $\Orbt_G$ we write $N_f$ for the induced map $R(U)\to R(T)$.  If $f$
 is an isomorphism we may use the alternative notation $f_*$ for
 $N_f$.  By this rule, we can regard $R$ as a covariant functor
 $\Orbt_G\to\Ab$, or in other words, a coefficient system.
\end{definition}

\begin{proposition}\label{prop-qS-MCS}
 If $S$ is a $G$-Tambara functor, then $qS$ is naturally a
 multiplicative coefficient system.
\end{proposition}

The proof will be given after the following lemma.
\begin{lemma}\label{lem-qS-MCS}
 Let $X\xra{g}T\xra{f}U$ be maps of finite $G$-sets, where $T$ and $U$
 are orbits.  Let $S$ be a Tambara functor, and let
 $\pi\:S(U)\to qS(U)$ be the quotient map.   Then 
 \[ \pi N_f T_g = \sum_{k\in\Sec(g)} \pi N_f R_k \: 
     S(X) \to qS(U).
 \]
\end{lemma}
\begin{proof}
 Let $X\xla{p}A\xra{q}B\xra{r}U$ be the distributor $\Dl(g,f)$, so
 $\pi N_fT_g=\pi T_rN_qR_p\:S(X)\to qS(U)$.  Now write $B$ as a
 disjoint union of orbits, say $B=\coprod_iB_i$, and let $r_i$ be the
 restriction of $r$ to $B_i$.  If $r_i$ is not an isomorphism then
 $\pi T_{r_i}=0$ by the definition of $q$.  If $r_i$ is an isomorphism
 then the map $m_i=(U\xra{r_i^{-1}} B_i \to B)$ is an equivariant
 section of $r$, with $T_{r_i}=R_{m_i}$.  Using this we see that
 $\pi T_r=\sum_{m\in\Sec(r)}\pi R_m$, and so 
 $\pi T_rN_qR_p=\sum_{m\in\Sec(r)}\pi R_mN_qR_p$.

 Now consider a section $k\in\Sec(g)$.  Recall that 
 \[ B = \{(u,s)\st u\in U,\; s\:f^{-1}\{u\}\to X,\; gs=1\}. \]
 For any $u\in B$ we put $k'(u)=(u,k|_{f^{-1}\{u\}})\in B$, so
 $rk'(u)=u$.  It is not hard to check that the construction
 $k\mapsto k'$ gives a bijection $\Sec(g)\to\Sec(r)$,so 
 $\pi T_rN_qR_p=\sum_{k\in\Sec(g)}\pi R_{k'}N_qR_p$.

 Next, define $k''\:T\to A$ by $k''(t)=(f(t),k|_{f^{-1}\{f(t)\}})$.
 We then have a commutative diagram as follows, in which the square is
 cartesian:
 \[ \xymatrix{
  & T \ar[dl]_k \ar[d]^{k''} \ar[r]^f & U \ar[d]^{k'} \\
  X & A \ar[l]^p \ar[r]_q & B.
 } \] 
 This gives
 \[ R_{k'}N_qR_p = N_fR_{k''}R_p = N_f R_{pk''} = N_fR_k \]
 as required.
\end{proof}

\begin{proof}[Proof of Proposition~\ref{prop-qS-MCS}]
 First, for any non-isomorphic map $p\:V\to U$ of $G$-orbits, the
 Frobenius reciprocity formula $T_p(a)b=T_p(a\,R_p(b))$ shows that the
 image of $T_p$ is an ideal in $S(U)$.  It follows easily that $qS$
 has a unique semiring structure for which the projection
 $\pi\:S(U)\to qS(U)$ is a homomorphism.

 Next, consider a map $f\:T\to U$ of $G$-orbits.  We first apply
 Lemma~\ref{lem-qS-MCS} to the fold map $T\amalg T\to T$, noting that
 this has only the two obvious sections; the conclusion is that
 $\pi N_f(a_0+a_1)=\pi N_f(a_0)+\pi N_f(a_1)$ for all
 $a_0,a_1\in S(T)$, so $\pi N_f\:S(T)\to qS(U)$ is an additive
 homomorphism.  Now apply the lemma instead to a map $g\:X\to T$,
 where $X$ is an orbit and $g$ is not an isomorphism.  We recall from
 Lemma~\ref{lem-sections} that $\Sec(g)=\emptyset$, so $\pi N_fT_g=0$.
 It follows that there is a unique additive homomorphism
 $N_f\:qS(T)\to qS(U)$ such that the square
 \[ \xymatrix{
  S(X) \ar[d]_\pi \ar[r]^{N_f} & S(Y) \ar[d]^\pi \\
  qS(X) \ar[r]_{N_f} & qS(Y)
 } \]
 commutes.  As $N_f\:S(X)\to S(Y)$ preserves products, the same is
 true of the induced map $qS(X)\to qS(Y)$.
\end{proof}

\begin{proposition}\label{prop-rR-tambara}
 If $R$ is a multiplicative coefficient system, then $rR$ is naturally
 a Tambara functor.
\end{proposition}
\begin{proof}
 Consider a map $f\:X\to Y$; we need to define
 $N_f\:(rR)(X)\to(rR)(Y)$, and check that it is compatible with all
 other structure.  Consider $m\in(rR)(X)$, so we have an element
 $m(T\xra{x}X)\in R(T)$ for each transitive $G$-set $T$ and each
 $G$-map $x\:T\to X$.  Consider instead a map $y\:T\to Y$, and form
 the pullback square 
 \[ \xymatrix{
      \tT \ar[r]^u \ar[d]_v & X \ar[d]^f \\
      T \ar[r]_y & Y
    }
 \]
 Each orbit $U\sse\tT$ comes equipped with maps $T\xla{v}U\xra{u}X$
 so we have $m(U\xra{u}X)\in R(U)$ and $N_vm(U\xra{u}X)\in R(T)$.  Let
 $p\in R(T)$ be the product of all these terms, as $U$ runs over
 the orbits in $\tT$.  This is completely natural, so we can define
 $(N_fm)(T\xra{y}Y)=p$ and this defines a map $N_f\:(rR)(X)\to(rR)(Y)$
 as required.  It is clear that this depends functorially on $f$.

 We next check the Mackey property.  Consider a pullback square 
 \[ \xymatrix{
  X \ar[r]^j \ar[d]_{f} & X' \ar[d]^{f'} \\
  Y \ar[r]_k & Y',
 } \]
 and an element $m\in (rR)(X')$.  We must show that
 \[ (N_fR_j(m))(T\xra{y}Y)=(R_kN_{f'}(m))(T\xra{y}Y) \]
 for all $G$-orbits $T$ and all $G$-maps $T\xra{y}Y$.  We define $\tT$
 as before, giving a diagram 
 \[ \xymatrix{
  \tT \ar[r]^u \ar[d]_v & X \ar[r]^j \ar[d]_{f} & X' \ar[d]^{f'} \\
  T \ar[r]_y & Y \ar[r]_k & Y'
 } \]
 in which both squares are cartesian.  This gives
 \[ (N_fR_j(m))(T\xra{y}Y) = 
     \prod_{U\in\orb(\tT)} N_{U\to T}(R_jm)(U\xra{u}X) = 
     \prod_{U\in\orb(\tT)} N_{U\to T}m(U\xra{ju}X'). 
 \]
 On the other hand, we have 
 \[ (R_kN_{f'}(m))(T\xra{y}Y)=(N_{f'}m)(T\xra{ky}Y'). \]
 As both squares are cartesian we see that the full
 rectangle is also cartesian and so can be used to compute $N_{f'}$,
 giving the same answer as before.

 Finally, consider a pair of maps $X\xra{f}Y\xra{g}Z$ with distributor 
 \[ \Dl(f,g) = (X \xla{p} A \xra{q} B \xra{r} Z). \]
 Consider an element $m\in (rR)(X)$, a $G$-orbit $T$ and a $G$-map
 $T\xra{z}Z$.  We must show that 
 \[ (N_gT_fm)(T\xra{z}Z) = (T_rN_qR_pm)(T\xra{z}Z). \]
 For the left hand side, we form the pullback square
 \[ \xymatrix{
      \tT \ar[r]^u \ar[d]_v & Y \ar[d]^g \\
      T \ar[r]_z & Z.
    }
 \]
 For each orbit $U\sse\tT$ we let $L_U$ denote the the set of $G$-maps 
 $s\:U\to X$ with $fs=u|_U$.  We then put $L=\prod_UL_U$, which can be
 identified with the set of $G$-maps $t\:\tT\to X$ with $ft=U$.  From the
 definitions we have 
 \[ (N_gT_fm)(T\xra{z}Z) =
     \prod_U N_{v|_U}((T_fm)(U\xra{u}Y)) = 
     \prod_U N_{v|_U}\left(\sum_{s\in L_U}m(U\xra{s}X)\right).
 \]
 We can next use the splitting $L=\prod_UL_U$ to expand out the
 product, and recall that $N_{v|_U}\:R(U)\to R(T)$ is a ring map, to get 
 \[ (N_gT_fm)(T\xra{z}Z) =
      \sum_{t\in L} \prod_U N_{v|_U}(m(U\xra{t|_U}X)).
 \]
 
 Now consider instead the right hand side.  Put $m'=R_pm$ and
 $m''=N_qm'$ so $T_rN_qR_pm=T_rm''$.  By definition,
 $(T_rm'')(T\xra{z}Z)$ can be written as a sum over $G$-maps $T\to B$
 lifting $z$.  For any $t\in L$ we can define $t'\:T\to B$ by 
 \[ t'(a) = \left(z(a),\; 
              g^{-1}\{z(a)\} \xra{u^{-1}} v^{-1}\{a\} \xra{t} X
            \right).
 \]
 One can check that $rt'=z$, and that any $G$-map with $rt'=z$ arises
 in this way.  We therefore have 
 \[ (T_rm'')(T\xra{z}Z) = \sum_{t\in L} m''(T\xra{t'}B)
      = \sum_{t\in L} (N_qm')(T\xra{t'}B).
 \]
 Next, we have a cartesian square 
 \[ \xymatrix{
  \tT \ar[r]^{t''} \ar[d]_{v} & A \ar[d]^q \\
  T \ar[r]_{t'} & B
 } \]
 where 
 \[ t''(a) = \left(u(a),\; 
              g^{-1}\{g(u(a))\} = g^{-1}\{z(v(a))\}
                \xra{u^{-1}} v^{-1}\{v(a)\} \xra{t} X
            \right).
 \]
 This gives 
 \[ (N_qm')(T\xra{t'}B) = \prod_U N_{v|_U}(m'(U \xra{t''}A)). \]
 Moreover, we have $m'=R_pm$ so $m'(U \xra{t''}A)=m(U\xra{pt''}X)$.
 It is clear from the definitions that $pt''=t$.  After unwinding all
 this we get 
 \[ (T_rN_qR_pm)(T\xra{z}Z) = \sum_{t\in L}\prod_U
     N_{v|_U}(m(U\xra{t}X)),
 \]
 which is the same as $(N_gT_fm)(T\xra{z}Z)$, as required.
\end{proof}

\begin{remark}\label{rem-rR-product}
 The Tambara structure on $rR$ makes $(rR)(X)$ into a semiring, with
 multiplication on $(rR)(X)$ given by $N_s$, where $s\:X\amalg X\to X$
 is the fold map.  It is not hard to check that the rule is just the
 obvious one: 
 \[ (mm')(T\xra{x}X) = m(T\xra{x}X)\; m'(T\xra{x}X) \in R(T). \] 
\end{remark}

\begin{proposition}\label{prop-qr-tambara}
 The functor $q\:\Tambara_G\to\MCSys_G$ is left adjoint to
 $r\:\MCSys_G\to\Tambara_G$.  
\end{proposition}
\begin{proof}
 First, consider a Tambara functor $S$, and a map $f\:X\to Y$ of
 finite $G$-sets.  We claim that the square
 \[ \xymatrix{
     S(X) \ar[r]^{N_f} \ar[d]_{\eta} & S(Y) \ar[d]^\eta \\
     (rqS)(X) \ar[r]_{N_f} & (rqS)(Y)
    }
 \]
 commutes.  To see this, consider an element $m\in S(X)$, a $G$-orbit
 $T$ and a $G$-map $y\:T\to Y$.  We must show that
 \[ (\eta N_f m)(T\xra{y}Y) = (N_f\eta m)(T\xra{y}Y) \in (qS)(T). \]
 The left hand side is by definition $\pi(R_yN_fm)$.  We form the
 pullback square 
 \[ \xymatrix{
      \tT \ar[r]^u \ar[d]_v & X \ar[d]^f \\
      T \ar[r]_y & Y
    }
 \]
 so the Mackey property gives $R_yN_f=N_vR_u$.  We split $\tT$ as a
 disjoint union of orbits $U$ and note that 
 \[ N_vR_u(m) = \prod_U N_{v|_U}R_{v|_U}m. \]
 The norm maps for $qS$ satisfy $\pi N=N\pi$ by construction, 
 so the left hand side of our equation is
 $\prod_UN_{v|_U}\pi(R_{v|_U}m)$.  The right hand side is 
 \[ (N_f\eta m)(T\xra{y}Y) =
     \prod_U N_{v|_U}((\eta m)(U\xra{u|_U}X) = 
     \prod_U N_{v|_U}\pi(R_{v|_U}m),
 \]
 which is the same as the left hand side.  We have seen previously
 that $\eta$ is a morphism of Mackey functors, so it commutes with the
 $T$ and $R$ operators as well as the norms, so it is a Tambara
 morphism. 

 Now consider a multiplicative coefficient system $R$, and a map
 $f\:T\to U$ of $G$-orbits.  We claim that the square
 \[ \xymatrix{
  (qrR)(T) \ar[d]_\ep \ar[r]^{N_f} & (qrR)(U) \ar[d]^\ep \\
  R(T) \ar[r]_{N_f} & R(U) 
 } \] 
 commutes.  To see this, consider an element $m\in(rR)(T)$.  We then
 have $\ep N_f(\pi(m))=(N_f(\pi(m)))(U\xra{1}U)$.  This is by
 definition a product of factors indexed by the orbits in the pullback
 of $U\xra{1}U$ along $f\:T\to U$.  Of course this pullback is just
 $T$, so there is only one orbit, and the corresponding factor is
 $N_fm$ as required.  This means that $\ep$ is a morphism of
 multiplicative coefficient systems.

 We previously proved triangular identities for $\eta$ and $\ep$.  We
 can now reinterpret these as identities in $\Tambara_G$ and
 $\MCSys_G$ rather than $\Mackey_G$ and $\CSys_G$, and they show that
 we still have an adjunction in this richer context.
\end{proof}

Recall that in Section~\ref{sec-mackey} we defined a functor
$\om\:\Mackey_G\to\Semigroups_G$ by $\om(M)=M(G)$, and studied its
left and right adjoints $d,c\:\Semigroups_G\to\Mackey_G$.  In
Example~\ref{eg-MapGXA-tambara} we remarked that we can define
$\om\:\Tambara_G\to\Semirings_G$ and its right adjoint
$c\:\Tambara_G\to\Semirings_G$ in the same way.  However, in this
context it is much harder to understand the left adjoint to $\om$.
This will be the subject of Section~\ref{sec-witt}, where (following
Brun~\cite{br:wvt}) we build a connection with Witt rings in the sense
of Dress and Siebeneicher~\cite{drsi:brp}.  For the moment we just
discuss a parallel adjunction in terms of coefficient systems, which
is easier.

\begin{definition}\label{defn-omega-prime}
  We define $\om'\:\MCSys\to\Semirings_G$ by $\om'(R)=R(G)$.
\end{definition}

\begin{remark}\label{rem-om-q-r}
 Recall that for any Tambara functor $S$ we have $(qS)(G)=S(G)$.  On
 the other hand, using Remark~\ref{rem-free-orbit} we see that
 $(rR)(G)=R(G)$ for any multiplicative coefficient system $R$.  This
 means that the diagram
 \[ \xymatrix{
  \Tambara_G \ar[d]_\om \ar[r]^q &
  \MCSys_G \ar[d]^{\om'} \ar[r]^r &
  \Tambara_G \ar[d]^\om \\
  \Semirings_G \ar@{=}[r] &
  \Semirings_G \ar@{=}[r] &
  \Semirings_G
 } \]
 commutes up to natural isomorphism.
\end{remark}

\begin{definition}\label{defn-L-prime}
 Consider an object $A\in\Semirings_G$.  For any $G$-orbit $T$, let
 $L'A(T)$ be the commutative semiring generated by symbols $i_t(a)$
 (for $t\in T$ and $a\in A$) modulo relations
 \begin{align*}
  i_t(0) &= 0 &
  i_t(1) &= 1 \\
  i_t(a+b) &= i_t(a) + i_t(b) &
  i_t(ab) &= i_t(a) i_t(b) \\
  i_{gt}(a) &= i_t(g^{-1}a).
 \end{align*}
 Next, for any morphism $u\:T\to U$ of orbits, there is an induced
 semiring map $N_u\:L'A(T)\to L'A(U)$ given by $N_ui_t(a)=i_{u(t)}(a)$.
 This makes $L'A$ into a multiplicative coefficient system.
\end{definition}

\begin{remark}\label{rem-L-prime-rep}
 It is straightforward to check the universal property
 \[ \Semirings(L'A(T),B) =
     \Map_G(T,\Semirings(A,B)) =
      \Semirings_G(A,\Map(T,B))
 \]
 for all semirings $B$.  Note that neither $L'A(T)$ nor $B$ has a
 $G$-action here, so the $G$-actions on $\Semirings(A,B)$ and
 $\Map(T,B)$ are purely determined by the $G$-actions on $A$ and $T$.
\end{remark}

We can give an alternative description of $L'$ as follows.
\begin{definition}\label{defn-coinvariant-semiring}
 Let $A$ be a semiring with an action of $G$.  We let $A_{[G]}$ denote
 the quotient of $A$ by the smallest semiring congruence containing
 $(a,ga)$ for all $a\in A$ and $g\in G$.  We call this the
 \emph{coinvariant semiring} for $A$.  
\end{definition}

\begin{remark}
 $A_{[G]}$ is in general a proper quotient of the coinvariant
 semigroup $A_G$ as in Definition~\ref{defn-coinvariant-semigroup}.
 For example, if $G$ is the group of order two acting by conjugation
 on $\C$, then $\C_G=\R$ but $\C_{[G]}=0$.
\end{remark}

\begin{remark}\label{rem-r-L-prime}
 It is now not hard to see that $(L'A)(G/K)=A_{[K]}$.  It follows from
 Proposition~\ref{prop-rNGH} that 
 \[ (rL'A)(G/H) = \left[\prod_{K\leq H} A_{[K]}\right]^H. \]
\end{remark}

\begin{proposition}\label{prop-L-prime-om-prime}
 The functor $L'\:\Semirings_G\to\MCSys_G$ is left adjoint to
 $\om'\:\MCSys_G\to\Semirings_G$.  
\end{proposition}
\begin{proof}
 For any $A\in\Semirings_G$ we can define 
 \[ \eta\:A\to\om'L'A=L'A(G) \]
 by $\eta(a)=i_1(a)$.  This is a semiring map and satisfies
 \[ g\,\eta(a) = N_{\rho(g)}(i_1(a)) = i_{\rho(g)(1)}(a) =
     i_{g^{-1}}(a) = i_1(ga) = \eta(ga),
 \]
 so it is $G$-equivariant.  

 Now consider a multiplicative coefficient system $R$.  Let $T$ be a
 $G$-orbit.  For any $t\in T$ we have a $G$-map $\hat{t}\:G\to T$
 given by $\hat{t}(g)=gt$.  We let $\ep\:L'(R(G))(T)\to R(T)$ be the
 unique ring map such that $\ep(i_t(a))=N_{\hat{t}}(a)$ for all
 $a\in R(G)$.  This defines a natural map $\ep\:L'\om'\to 1$.

 We leave it to the reader to check the triangle identities, giving an
 adjunction as claimed.
\end{proof}

\section{Filtrations}
\label{sec-filt}

\begin{definition}\label{defn-k-free}
 \begin{itemize}
  \item[(a)] We say that a $G$-set $X$ is \emph{$k$-free} if all
   stabiliser groups of points in $X$ have order at most $k$.
  \item[(b)] For any $G$-orbit $U$, the \emph{coorder} of $U$ is
   $|G|/|U|$ (which is the order of the stabiliser group of any point
   in $U$).
  \item[(c)] We say that a coefficient system $N$ is \emph{$k$-pure}
   if $N(U)=0$ whenever $\coord(U)\neq k$.
  \item[(d)] We say that a Mackey functor $M$ is \emph{$k$-pure} if
   the coefficient system $qM$ is $k$-pure.
  \item[(e)] For any coefficient system $N$, we define subsystems
   $F_kN$ by 
   \[ F_kN(U) = \begin{cases}
       N(U) & \text{ if } \coord(U) \leq k \\
       0 & \text{ otherwise. }
      \end{cases}
   \]
  \item[(f)] For any Mackey functor $M$ and any $G$-set $X$, we let
   $F_kM(X)$ denote the set of elements $m\in M(X)$ that can be
   written in the form $m=T_pn$ for some $k$-free $G$-set $W$, some
   $G$-map $p\:W\to X$, and some element $n\in M(W)$.
 \end{itemize}
\end{definition}

\begin{lemma}\label{lem-k-free}
 If $Y$ is $k$-free and there exists a $G$-map $f\:X\to Y$ then $X$ is
 also $k$-free.
\end{lemma}
\begin{proof}
 This is clear because $\stab_G(x)\leq\stab_G(f(x))$.
\end{proof}

\begin{proposition}\label{prop-FkM-mackey}
 $F_kM$ is a sub-Mackey functor of $M$, which satisfies $F_kM(X)=M(X)$
 whenever $X$ is $k$-free.  Moreover, we have $qF_kM=F_kqM$.
\end{proposition}
\begin{proof}
 It is clear from the definitions that
 $(F_kM)(X_0\amalg X_1)=(F_kM)(X_0)\tm(F_kM)(X_1)$.  It is also clear
 by construction that for any map $f\:X\to Y$, we have
 $T_fF_kM(X)\sse F_kM(Y)$.  Now consider an element $m\in F_kM(Y)$.
 By definition there exists a map $p\:W\to Y$ and an element
 $n\in M(W)$ with $m=T_pn$.  Now form a pullback square
 \[ \xymatrix{
  V \ar[r]^{\tilde{f}} \ar[d]_q & W \ar[d]^p \\
  X \ar[r]_f & Y.
 } \]
 As $V$ admits a map to $W$, it must be $k$-free.  We also have
 $R_pm=T_qR_{\tilde{f}}n$, so $R_fm\in F_kM(X)$.  This shows that
 $F_kM$ is indeed a sub-Mackey functor.  If $X$ is $k$-free then we
 can write every element $m\in M(X)$ as $T_1m$, showing that
 $F_kM(X)=M(X)$. 

 Now let $U$ be a $G$-orbit.  If $\coord(U)\leq k$ then it is clear that
 $F_kM(U)=M(U)$, and more generally $F_kM(V)=M(V)$ for all $G$-orbits
 that admit a map to $U$.  Thus, all ingredients in the definition of
 $qF_kM(U)$ are the same as the corresponding ingredients for
 $qM(U)$, so $qF_kM(U)=qM(U)$.  Suppose instead that $\coord(U)>k$.
 Every element $m\in F_kM(U)$ can be written as $T_pn$ for some
 $k$-free $G$-set $Y$ and some $n\in M(Y)$.  Note that $M(Y)=F_kM(Y)$,
 and that no orbit in $Y$ can map isomorphically to $U$.  It follows
 that $m$ becomes zero in $qF_kM(U)$, but $m$ was arbitrary so
 $qF_kM(U)=0$.  This shows that $qF_kM=F_kqM$.
\end{proof}

\begin{proposition}\label{prop-qM-M}
 Let $M$ be a $k$-pure Mackey functor.  
 \begin{itemize}
  \item[(a)] Then for orbits $T$ with $\coord(T)<k$ we have
   $M(T)=qM(T)=0$.
  \item[(b)] For orbits $T$ with $\coord(T)=k$ we have
   $qM(T)=M(T)$.
  \item[(c)] For $(k-1)$-free $G$-sets $X$ we have $M(X)=0$.
  \item[(d)] For $k$-free $G$-sets $X$, the map
   $\eta\:M(X)\to rqM(X)$ is an isomorphism.
  \item[(e)] If $qM(T)=0$ whenever $\coord(T)=k$, then $M=0$.
 \end{itemize}
\end{proposition}
\begin{proof}
 First, $k$-purity of $M$ means by definition that $qM(T)=0$ whenever
 $\coord(T)\neq k$.

 Next, recall that $qM(T)$ is the quotient of $M(T)$ by the sum of the
 images of transfers from orbits with strictly larger coorder.  Thus,
 if $M(U)=0$ such orbits (which is vacuously true if $T=G$), then
 $qM(T)=M(T)$.  In particular, if $qM(T)=0$, then $M(T)=0$.
 Claims~(a) and~(b) follow easily by induction on the coorder.  The
 same argument also proves~(e).

 Claim~(c) follows from~(a) by decomposing $X$ into orbits.
 Similarly, it will be enough to prove~(d) when $X=G/H$ with $|H|=k$.
 Then Proposition~\ref{prop-rNGH} gives
 \[ rqM(X) = \left[\prod_{K\leq H} qM(G/K)\right]^H. \]
 If $K$ is a proper subgroup of $H$ then $qM(G/K)=0$, but in the case
 $K=H$ we have $qM(G/H)=M(G/H)$.  Thus, the above product is just
 $M(G/H)$, and $H$ acts trivially on this, so $rqM(G/H)=M(G/H)$ as
 required.  (We leave the reader to check that the isomorphism
 $M(G/H)\to rqM(G/H)$ implicit in this argument is just the same as
 $\eta$.) 
\end{proof}

\begin{proposition}\label{prop-pure-filt}
 Let $M$ be a $k$-pure, additively complete Mackey functor.  Then
 $F_jM=0$ for $j<k$, and $F_jM=M$ for $j\geq k$.
\end{proposition}
\begin{proof}
 First suppose that $j<k$.  We then have $M(Y)=0$ for all $j$-free
 sets $Y$, and it follows immediately that $F_jM=0$.  Next, as $q$ has
 a right adjoint, it preserves colimits, so $q(M/F_kM)=q(M)/q(F_kM)$.
 Now $q(F_kM)=F_kq(M)$ by Proposition~\ref{prop-FkM-mackey}, and
 purity implies that $F_kq(M)=q(M)$, so we have $q(M/F_kM)=0$.
 Proposition~\ref{prop-qM-M} therefore tells us that $M/F_kM=0$.  As
 $M$ is additively complete, pathologies like
 Remark~\ref{rem-funny-quotient} cannot occur, and we deduce that
 $F_kM=M$ as claimed.   If $k\leq j\leq |G|$ then it is clear that
 $F_kM\leq F_jM$, and so $F_jM=M$ as well.
\end{proof}

\begin{lemma}\label{lem-complete-quotient}
 Let $M$ be a semigroup, and let $N$ be a subsemigroup that is
 additively complete.  Then the set
 \[ E_N =
     \{(m_0,m_1)\in M^2\st m_1=m_0+n \text{ for some } n\in N\}
 \]
 is the smallest congruence on $M$ containing $0\tm N$, so
 $M/N=M/E_N$.  It follows that $M/N=0$ iff $M=N$, and that $M/N$ is
 additively complete iff $M$ is additively complete.
\end{lemma}
\begin{proof}
 Straightforward.
\end{proof}

\begin{proposition}\label{prop-qM-complete}
 Let $M$ be a Mackey functor such that the coefficient system $qM$ is
 additively complete.  Then $M$ is also additively complete, as are
 the subobjects $F_kM$ and the quotient objects $F_kM/F_{k-1}M$.
\end{proposition}
\begin{proof}
 Consider an orbit $G/H$.  We may assume by induction that all the
 semigroups $M(G/K)$ with $|K|<|H|$ are additively complete.  Thus, if
 we let $N$ denote the sum of the images of the maps
 $T_p\:M(G/K)\to M(G/H)$, then $N$ is additively complete.  By
 assumption the quotient $qM(G/H)=M(G/H)/N$ is additively complete,
 and it follows from the Lemma that $M(G/H)$ is additively complete.
 Any finite $G$-set $X$ can be written as a disjoint union of orbits,
 so $M(X)$ is additively complete as claimed.

 Next, recall that $qF_kM=F_kqM$.  From this it is clear that $qF_kM$
 is also additively complete, so $F_kM$ is additively complete.  It is
 clear that any quotient of additively complete Mackey functors is
 again additively complete.  In particular, this applies to
 $F_kM/F_{k-1}M$. 
\end{proof}

\section{Rational Mackey functors and Tambara functors}
\label{sec-rational}

\begin{definition}\label{defn-rational}
 We say that a semigroup $A$ is \emph{rational} if every element has
 an additive inverse, and the map $n.1_A\:A\to A$ is an isomorphism
 for all positive integers $n$.  (This means that $A$ has a unique
 structure as a vector space over $\Q$ extending the given addition
 law.)  We say that a Mackey functor $M$ is rational if $M(X)$ is
 rational for all $X$, and similarly for Green functors, Tambara
 functors and (multiplicative) coefficient systems.  We write
 $\Q\Mackey_G$ for the category of rational Mackey functors, and
 similarly for $\Q\Green_G$, $\Q\Tambara_G$, $\Q\CSys_G$ and
 $\Q\MCSys_G$. 
\end{definition}

\begin{theorem}\label{thm-mackey-rational}
 The functors $q$ and $r$ restrict to give an equivalence
 $\Q\Mackey_G\simeq\Q\CSys_G$, and also an equivalence
 $\Q\Tambara_G\simeq\Q\MCSys_G$.
\end{theorem}

The proof will follow after some preparatory results.

\begin{definition}
 Let $k$ be a divisor of $|G|$.  We put 
 \[ C_k = \{A\sse G\st A=Hx \text{ for some } x\in G 
            \text{ and } H\leq G \text{ with } |H|=k \}.
 \]
 After noting that $gHx=(gHg^{-1})gx$, we see that $C_k$ has a natural 
 $G$-action by multiplication on the left. 
\end{definition}

\begin{lemma}\label{lem-Ck-orbits}
 All orbits in $C_k$ have coorder $k$.  Moreover, for every orbit $U$
 of coorder $k$ we have $|\Map_G(U,C_k)|=|G|/k=|U|$.
\end{lemma}
\begin{proof}
 It is clear that $\stab_G(Hx)=H$.  Thus, for $K\leq G$ with $|K|=k$
 we have 
 \[ |\Map_G(G/K,C_k)|=|(C_k)^K|=|\{Kx\st x\in G\}|=|G|/k. \]
\end{proof}

\begin{proposition}\label{prop-qr-rational}
 Let $N$ be a rational coefficient system.  Then the map
 $\ep\:qrN\to N$ is an isomorphism.
\end{proposition}
\begin{proof}
 It is clear that $N$ is a direct sum of pure coefficient systems, so
 we may assume that $N$ itself is $k$-pure for some $k$ dividing
 $|G|$.  Given this, Proposition~\ref{prop-rNGH} can be rewritten as 
 \[ rN(G/H) = \left[\prod_{K\leq H,|K|=k} N(G/K)\right]^H. \]
 If $|H|<k$ then the product has no terms, so $rN(G/H)=0$, so
 certainly $qrN(G/H)=0=N(G/H)$.  If $|H|=k$ then the product
 is just $N(G/H)$ with $H$ acting trivially, so $rN(G/H)=N(G/H)$.
 Moreover, the previous case shows that there is nothing to kill to
 form $qrN(G/H)$, so $qrN(G/H)=N(G/H)$ as well.  We now see that the
 map $\ep\:qrN(U)\to N(U)$ is an isomorphism whenever
 $\coord(U)\leq k$.

 Now suppose instead that $\coord(U)>k$.  We then have $N(U)=0$, so we
 must show that $qrN(U)=0$.  Let $p\:C_k\tm U\to U$ be the
 projection.   Consider an element $m\in rN(U)$, given by a natural
 system of elements $m(V\xra{u}U)\in N(V)$ for all $G$-orbits $V$
 (wlog with $\coord(V)=k$) and all $G$-maps $u\:V\to U$.  From the
 definitions we have 
 \begin{align*}
  (R_pm)(V\xra{(c,u)}C_k\tm U) &= m(V\xra{u}U) \\
  (T_pR_pm)(V\xra{u}U)
    &= \sum_{c\in\Map_G(U,C_k)} (R_pm)(V\xra{(c,u)}C_k\tm U) \\
    &= |\Map_G(U,C_k)| m(V\xra{u}U) \\
    &= (|G|/k) m(V\xra{u}U), 
 \end{align*}
 so $T_pR_pm=(|G|/k)m$.  Now let $\pi\:rN(U)\to qrN(U)$ be the
 projection as usual.  Note that $C_k$ is $k$-free so the same is true
 of $C_k\tm U$, but $\coord(U)>k$; it follows that $\pi T_p=0$.  We
 therefore have $(|G|/k)\pi(m)=0\in qrN(U)$, but $N$ is rational so
 $\pi(m)=0$ as required.
\end{proof}

\begin{corollary}\label{cor-eta-epi}
 Let $M$ be a rational Mackey functor.  Then the map $\eta\:M\to rqM$
 is surjective.  
\end{corollary}
\begin{proof}
 The triangular identity for the $(q,r)$ adjunction says that the
 composite
 \[ qM \xra{q\eta_M} qrqM \xra{\ep_{qM}} qM \]
 is the identity.  Using the proposition we see that $\ep_{qM}$ is an
 isomorphism, so $q\eta_M$ is also an isomorphism.  Next, we note that
 $q$ has a right adjoint, so it preserves colimits, so
 $q(\cok(\eta_M))=\cok(q\eta_M)=0$.  Proposition~\ref{prop-qM-M}(e)
 now tells us that $\cok(\eta_M)=0$.  As everything is rational and
 therefore additively complete, we can deduce that $\eta_M$ is
 surjective. 
\end{proof}

\begin{proposition}\label{prop-eta-iso-pure}
 Let $k$ be a divisor of $|G|$, and let $M$ be a $k$-pure rational
 Mackey functor.  Then the map $\eta\:M\to rqM$ is an isomorphism.
\end{proposition}
\begin{proof}
 As in Remark~\ref{rem-rN-limit}, we can think of $rqM(X)$ as the
 inverse limit of the functor $(\Orbt_G\downarrow X)\to\Semigroups$
 given by $(U\xra{x}X)\mapsto qM(U)$.  As $qM$ is $k$-pure we have
 $qM(U)=0$ unless $U$ has coorder $k$, and in that case
 Proposition~\ref{prop-qM-M} tells us that $qM(U)=M(U)$.  We can thus
 let $\CP$ be the full subcategory of $(\Orbt_G\downarrow X)$
 containing the diagrams $(U\xra{x}X)$ where $U$ has coorder $k$, and
 we find that $rqM(X)$ is the limit of the functor $\CP\to\Semigroups$
 given by $(U\xra{x}X)\mapsto M(U)$.  For any element $m$ of this
 inverse limit, we define $\psi(m)\in M(X)$ by 
 \[ \psi(m) =
     \sum_{U\in\orb(C_k)}\sum_{x\in\Map_G(U,X)}
      T_x(m(U\xra{x}X)).
 \]
 We claim that this defines a morphism $\psi\:rqM\to M$ of Mackey
 functors.  To see this, consider a map $f\:X\to Y$ of finite
 $G$-sets, and the resulting diagram 
 \[ \xymatrix{
  rqM(X) \ar[r]^{T_f} \ar[d]_\psi &
  rqM(Y) \ar[r]^{R_f} \ar[d]_\psi &
  rqM(X)              \ar[d]_\psi \\
  M(X)   \ar[r]_{T_f}             &
  M(Y)   \ar[r]_{R_f}             &
  M(X).
 } \]
 For the left square, we recall that transfers in $rqM$ are defined by 
 \[ (T_fm)(U\xra{y}Y) = \sum_{x\in\Map_G(U,X),fx=y} \; m(U\xra{x}X), \]
 so
 \begin{align*}
  \psi T_fm &= \sum_U\sum_{y\in\Map_G(U,Y)}T_y((T_fm)(U\xra{y}Y)) 
             = \sum_U\sum_{y\in\Map_G(U,Y)}\sum_{x\in\Map_G(U,X),fx=y}
                 T_y(m(U\xra{x}X)) \\
            &= \sum_U\sum_{x\in\Map_G(U,X)} T_{fx}(m(U\xra{x}X)) 
             = T_f\left(\sum_U\sum_{x\in\Map_G(U,X)}T_x(m(U\xra{x}X))\right) \\
            &= T_f\psi(m).
 \end{align*}
 For the right square, consider an element $n\in rqM(Y)$.  We then
 have 
 \begin{align*}
  R_f\psi(n)
   &= \sum_U\sum_{y\in\Map_G(U,Y)} (R_fT_yn)(U\xra{y}Y)
 \end{align*}
 To analyse this, we form a pullback square
 \[ \xymatrix{
  \tU \ar[r]^i \ar[d]_j & U \ar[d]^y \\ X \ar[r]_f & Y.
 } \]
 This gives $R_fT_y=T_iR_j$, and this can be written as a sum over
 orbits $V\sse\tU$.  As $V$ admits a map to $U$, it must have coorder
 at most $k$.  If the coorder is strictly less than $k$ then $M(V)=0$
 and so $V$ does not contribute to $T_iR_jn$.  The orbits of coorder
 $k$ map isomorphically to $U$ and so biject with $G$-maps $x\:U\to X$
 satisfying $fx=y$.  We deduce that
 \begin{align*}
  R_f\psi(n)
   &= \sum_U\sum_{y\in\Map_G(U,Y)}\sum_{x\in\Map_G(U,X),fx=y}
       T_x(m(U\xra{y}Y)) \\
   &= \sum_U\sum_{x\in\Map_G(U,X)}T_x(m(U\xra{fx}))
    = \psi R_f(m).
 \end{align*}
 Next, we claim that $\psi\eta(m)=(|G|/k)m$ for all $G$-sets $X$ and
 elements $m\in M(X)$.  Indeed, Proposition~\ref{prop-pure-filt} tells
 us that $M=F_kM$, so every element of $M(X)$ has the form $T_pn$ for
 some $k$-free set $W$ and some $G$-map $p\:W\to X$.  Using this we
 can reduce to the case where $X$ is a $k$-free orbit.  If the coorder
 of $X$ is strictly less than $k$, then $M(X)=0$ and there is nothing
 to prove.  We therefore assume that $\coord(X)=k$.  We now have
 $\eta(m)(U\xra{x}X)=\pi(R_xm)\in qM(U)$, and the $\pi$ makes no
 difference if $\coord(U)=k$.  We therefore have
 \[ \psi\eta(m) =
     \sum_{U\in\orb(C_k)}\sum_{x\in\Map_G(U,X)}T_xR_xm.
 \]
 Here both $U$ and $X$ are orbits of coorder $k$, so any $G$-map
 $x\:U\to X$ is an isomorphism, so $T_xR_x$ is the identity.  We also
 see that the inversion map $x\mapsto x^{-1}$ gives a bijection
 $\Map_G(U,X)\to\Map_G(X,U)$, so 
 \[ \sum_{U\in\orb(C_k)}|\Map_G(U,X)| = 
    \left|\coprod_{U\in\orb(C_k)}\Map_G(X,U)\right| = 
    |\Map_G(X,C_k)| = |G|/k.
 \]
 This gives $\psi\eta(m)=(|G|/k)m$ as claimed.  As $M$ is rational, we
 conclude that $\psi\eta$ is an isomorphism, so $\eta$ is injective
 and $\psi$ is surjective.  We saw in Corollary~\ref{cor-eta-epi} that
 $\eta$ is also surjective, so it is an isomorphism.
\end{proof}

\begin{proof}[Proof of Theorem~\ref{thm-mackey-rational}]
 We need to show that the unit map $\eta\:M\to rqM$ and the counit map
 $\ep\:qrN\to N$ are isomorphisms when $N$ and $M$ are rational.
 (This will prove both the Mackey functor statement and the Tambara
 functor statement.)

 The counit is covered by Proposition~\ref{prop-qr-rational}.  For the
 unit, it will be enough to prove by induction that the maps
 $\eta\:F_kM\to rqF_kM$ are isomorphisms for all $k$.  Put
 $Q_kM=F_kM/F_{k-1}M$, and consider the commutative diagram
 \[ \xymatrix{
  F_{k-1}M \ar@{ >->}[r] \ar[d]_{\eta} &
  F_kM \ar@{->>}[r] \ar[d]_\eta &
  Q_kM \ar[d]^\eta \\
  rqF_{k-1}M \ar@{ >->}[r] &
  rqF_kM \ar@{->>}[r] &
  rqQ_kM.
 } \]
 The top row is short exact by definition.
 Proposition~\ref{prop-FkM-mackey} tells us that $qF_j=F_jq$, and $q$
 preserves colimits so
 $qQ_kM=(qF_kM)/(qF_{k-1}M)=(F_kqM)/(F_{k-1}qM)$.  It follows that
 $Q_kM$ is $k$-pure, and that the sequence
 $qF_{k-1}M\to qF_kM\to qQ_kM$ is a split short exact sequence of
 coefficient systems.  As any additive functor preserves split short
 exact sequences, we see that the bottom row of the above diagram is
 again short exact.  The right hand vertical map is an isomorphism by 
 Proposition~\ref{prop-eta-iso-pure}, and the left hand map can be
 assumed to be an isomorphism by induction, so the middle map is an
 isomorphism by the Five Lemma.
\end{proof}

\begin{example}\label{eg-rational-burnside}
 Consider the rational Burnside ring Tambara functor 
 \[ \Q\ot A(X)=\Q\ot\bCA_G(1,X)\simeq\Q\ot\bCU_G(\emptyset,X). \]
 We saw in Proposition~\ref{prop-q-symmon} that $qA=c\N$.  In the same
 way, we can check that $q(\Q\ot A)=c\Q$, where $c\Q$ denotes the
 constant functor $c\Q\:\Orbt_G\to\Rings$ with value $\Q$.  It follows
 that $\Q\ot A\simeq rc\Q$, and so
 \[ \Q\ot A(X) = \left[\prod_{H\leq G} \Map(X^H,\Q)\right]^G. \]
\end{example}

We next want to discuss the rational representation ring.  For this,
we first need to investigate naturality properties of the construction
sending $G/H$ to the centre of $H$.

\begin{construction}\label{cons-centre}
 For any $G$-orbit $T$ we have a translation category $\Trans(G,T)$ and
 a functor 
 \[ F_T\:\Trans(G,T)\to\Groups \]
 given by
 \[ F_T(x)=\stab_G(x)=\{g\in G\st gx=x\}. \]
 We write $Z_T$ for the inverse limit of $F_T$.  One checks that
 $Z_{G/H}$ is the centre of $H$.  In particular, if $H$ is abelian
 then $Z_{G/H}=H$.

 Now suppose we have a map $q\:U\to T$ and thus
 $q_*\:\Trans(G,U)\to\Trans(G,T)$.  There is a natural monomorphism
 $F_U\to F_T\circ q_*$, and by general nonsense this gives maps
 \[ Z_U \xra{} \invlim (F_T\circ q_*) \xla{} Z_T. \]
 We claim that if $T$ (and thus $U$) has abelian isotropy then the
 first of these maps is injective and the second is an isomorphism, so
 we have a natural inclusion $Z_U\to Z_T$.  Indeed, if $q$ is just the
 projection $G/H\to G/K$ (with $H\leq K$), then the above maps are
 easily identified with the maps
 \[ ZH \xra{} Z_KH  \xla{} ZK, \]
 and the claim about the abelian case follows immediately.

 One can now check that this construction gives a functor
 \[ \{G-\text{orbits with abelian isotropy}\} \to
    \{\text{ abelian groups }\}.
 \]
\end{construction}

\begin{example}\label{eg-cyclotomic}
 Let $R$ be the representation semiring Tambara functor, as in
 Example~\ref{eg-representation-tambara}.  Recall that $R(G/H)$ is
 just the semiring of isomorphism classes of complex representations
 of $H$.  Consider the transfer maps
 $\Q\ot R(G/C)\to\Q\ot R(G/H)$ for cyclic subgroups $C\leq H$.
 Artin's induction theorem says that the sum of the images of these
 maps is all of $\Q\ot R(G/H)$.  It follows that the corresponding
 multiplicative coefficient system has $q(\Q\ot R)(T)=0$ unless $T$
 has cyclic isotropy groups.  If $C$ is cyclic of order $n$ one can
 check that $q(\Q\ot R)(G/C)$ is isomorphic to the cyclotomic
 field $\Q(e^{2\pi i/n})$.

 We can make this more functorial as follows.  First, put
 $C^*=\Hom(C,S^1)$, which is again cyclic of order $n$.  One can check
 that $\Q[C^*]$ has a unique maximal ideal $\mxi_C$ such that the
 natural map from $C^*$ to the field $K(C)=\Q[C^*]/\mxi$ is injective.
 Next, suppose we have an injective homomorphism $i\:C\to D$ of
 abelian groups, with $|D|/|C|=m$ say.  Given a character $\al\in C^*$
 we can choose $\bt\in D^*$ with $i^*(\bt)=\al$, and it is not hard to
 see that $\bt^m$ is independent of the choice of $\bt$, so we can
 define $i_\bullet(\al)=\bt^m$.  This is an injective homomorphism of
 cyclic groups, and it follows that it induces a homomorphism
 $i_\bullet\:K(C)\to K(D)$.  One can check that for any $G$-orbit $T$
 with cyclic isotropy, there is a natural isomorphism
 $q(\Q\ot R)(T)\simeq K(Z_T)$.
\end{example}

\begin{example}\label{eg-free-orbits}
 Put $R(T)=\Map(T,\Q)$ when $T$ is a free orbit, and $R(T)=0$ if $T$
 is not free.  For any map $f\:U\to T$, either $T$ is free (in which
 case $f$ is an isomorphism and we put 
 $R_f=(f^{-1})^*\:\Map(U,\Q)\to\Map(T,\Q)$) or $T$ is not free (in
 which case $R_f=0$).  This gives a functor $\Orbt_G\to\Alg_\Q$, and the
 associated Tambara functor is just $S(X)=\Map(X,\Q)$.
\end{example}

\begin{proposition}\label{prop-L-om-rational}
 The restricted functor $\om\:\Q\Tambara_G\to\Q\Semirings_G$ is right
 adjoint to the functor $rL'$ (where $L'$ is as in
 Definition~\ref{defn-L-prime}). 
\end{proposition}
\begin{proof}
 This is clear from Propositions~\ref{prop-qr-tambara}
 and~\ref{prop-L-prime-om-prime}.
\end{proof}

\section{Change of groups}
\label{sec-change-groups}

In this section we will investigate various functors relating the
categories $\bCU_G$, $\Tambara_G$ and $\MCSys_G$ for different groups
$G$.

Let $H$ be a subgroup of $G$.  There is then an evident forgetful
functor $\res\:\bCU_G\to\bCU_H$.  We can also define a functor from
finite $H$-sets to finite $G$-sets by $\ind(Y)=G\tm_HY$.  If $T$ is a
set of coset representatives then the underlying set of $\ind(Y)$ is
just $T\tm Y$.  Using this, we see that $\ind$ preserves all the
constructions used to define composition in $\bCU_H$, so we get a
functor $\ind\:\bCU_H\to\bCU_G$.

\begin{proposition}\label{prop-ind-res}
 The functor $\res\:\bCU_G\to\bCU_H$ is left adjoint to
 $\ind\:\bCU_H\to\bCU_G$.  Moreover, both these functors preserve
 categorical products.  
\end{proposition}
\begin{proof}
 First, consider a $G$-set $W$ equipped with a map
 $f\:W\to\ind(Y)=G\tm_HY$.  Note that $\ind(Y)$ contains $H\tm_HY=Y$,
 and put $W_0=f^{-1}(W)$, which is an $H$-set.  There is a unique
 $G$-map $G\tm_HW_0\to W$ extending the identity on $W_0$, and one
 checks that this is an isomorphism.  It follows that the category of
 $G$-sets over $\ind(Y)$ is equivalent to the category of $H$-sets
 over $Y$.

 Any morphism in $\bCU_H(\res(X),Y)$ is represented by a diagram 
 \[ P_0 = (X\xla{f_0}A_0\xra{g_0}B_0\xra{h_0}Y) \]
 of finite $H$-sets.  Let $f$ be the unique $G$-equivariant extension
 of $f_0$ over $\ind(A_0)$, so we have a diagram
 \[ P = (X\xla{f}\ind(A_0)\xra{\ind(g_0)}\ind(B_0)
         \xra{\ind(h_0)}\ind(Y)),
 \]
 representing a morphism in $\bCU_G(X,\ind(Y))$.  Using our first
 paragraph, we see that this procedure gives a bijection
 $\bCU_H(\res(X),Y)=\bCU_G(X,\ind(Y))$. 

 Products in $\bCU_G$ and $\bCU_H$ are given by disjoint union, and it
 is clear that both $\ind$ and $\res$ preserve disjoint unions.
\end{proof}

\begin{definition}
 We define $\coind\:\Tambara_H\to\Tambara_G$ and
 $\res\:\Tambara_G\to\Tambara_H$ by $\coind(S)=S\circ\res$ and
 $\res(S')=S'\circ\ind$.
\end{definition}

\begin{proposition}
 The functor $\coind\:\Tambara_H\to\Tambara_G$ is right adjoint to
 $\res\:\Tambara_G\to\Tambara_H$. 
\end{proposition}
\begin{proof}
 This follows formally from Proposition~\ref{prop-ind-res}, using the
 (co)unit maps and triangular identities.
\end{proof}

\section{Witt vectors}
\label{sec-witt}

We now discuss the relationship between Tambara functors and the
generalised Witt rings of Dress and Siebeneicher~\cite{drsi:brp}.
Similar results have been obtained by Brun~\cite{br:wvt},
Elliott~\cite{el:bri} and Nakaoka~\cite{na:tmf}.

\begin{proposition}\label{prop-L-exists}
 The functor $\om\:\Tambara_G\to\Semirings_G$ (given by $\om(S)=S(G)$)
 has a left adjoint $L\:\Semirings_G\to\Tambara_G$.
\end{proposition}
\begin{proof}
 Put $P(X)=\bCU_G(G,X)$, so $P$ is a Tambara functor and the
 Yoneda lemma gives $\Tambara_G(P,S)\simeq S(G)$.  For any
 $A\in\Semirings_G$, we let $L_0A$ denote the coproduct of copies of
 $P$ indexed by $A$, so $\Tambara_G(L_0A,S)=\Map(A,S(G))$.  Similarly,
 we let $L_1A$ denote the coproduct of copies of $P$ indexed by the
 set 
 \[ Q = 1 \amalg 1 \amalg (A\tm A) \amalg (A\tm A) \amalg (G\tm A), \]
 so 
 \begin{align*}
  \Tambara_G(L_1A,S) &= \Map(Q,S(G)) \\
   &= S(G)\tm S(G) \tm \Map(A\tm A,S(G)) \tm \Map(A\tm A,S(G))
      \tm\Map(G\tm A,S(G)).
 \end{align*}
 Given a map $f\in\Map(A,S(G))$ we define $\phi^*(f)\in\Map(Q,S(G))$ by 
 \[ \phi^*(f)=
     (f(0),f(1),(a,b)\mapsto f(a+b),(a,b)\mapsto f(ab),
      (g,a)\mapsto f(ga))).
 \]
 We also define $\psi^*(f)\in\Map(Q,S(G))$ by 
 \[ \psi^*(f)=(0,1,(a,b)\mapsto f(a)+f(b),
      (a,b)\mapsto f(a)f(b),(g,a)\mapsto g\,f(a)).
 \]
 It is clear that the equaliser of the maps 
 \[ \phi^*,\psi^*\:\Map(A,S(G)) \to \Map(Q,S(G)) \]
 is $\Semirings_G(A,S(G))=\Semirings_G(A,\om(S))$.  Moreover, the
 Yoneda lemma implies that $\phi^*$ and $\psi^*$ arise from morphisms
 $\phi,\psi\:L_1A\to L_0A$ in $\Tambara_G$.  If we let $LA$ denote the
 coequaliser of these maps (which exists by
 Proposition~\ref{prop-tambara-coeq}), we obtain an isomorphism
 $\Tambara_G(LA,S)\simeq\Semirings_G(A,\om(S))$, which is natural in
 $S$.  It is now standard that there is a canonical way to define $L$
 on morphisms, giving a functor $\Semirings_G\to\Tambara_G$ that is
 left adjoint to $\om$.
\end{proof}

\begin{definition}\label{defn-witt}
 For any semiring $A$, we define a semiring $W_GA$ as follows: we give
 $A$ the trivial $G$-action, then apply the functor
 $L\:\Semirings_G\to\Tambara_G$, then evaluate on the singleton
 $G$-set to get $W_GA=(LA)(1)$.  We call this the semiring of $G$-Witt
 vectors for $A$. 
\end{definition}

When $A$ is additively complete, we will show that $W_G(A)$ is the
same as the ring defined by Dress and Siebeneicher~\cite{drsi:brp}.
In particular, if $G$ is cyclic of prime-power order, we recover the
usual $p$-typical Witt vectors.  

One could think about generalising the definition to cover the case
where $A$ has nontrivial $G$-action.  Some results in this context
have been stated by Brun, but they appear to contain some
inaccuracies.  We will therefore ignore this possible generalisation:
for the rest of this section we take $A$ to be a semiring equipped
with the trivial $G$-action.

We suspect that the proper context for our results is really a
theory of global Tambara functors, similar to Webb's theory of global
Mackey functors~\cite{we:tcs}.  However, we will leave that for future
work.

We now formulate our comparison with the theory of Dress and
Siebeneicher in more detail.
\begin{definition}\label{defn-ghost}
 We write $\Sub(G)$ for the set of subgroups of $G$, and $\sub(G)$ for
 the quotient set of conjugacy classes of subgroups.  We define a map 
 \[ \gm\:\Map(\sub(G),A) \to \Map(\sub(G),A) \]
 (called the \emph{ghost map}) by 
 \[ \gm(a)([H]) = \sum_{[K]} |(G/H)^K| a([K])^{|H|/|K|}. \]
 (Note here that if $(G/H)^K$ is nonempty then $K$ is conjugate to a
 subgroup of $H$ so $|H|/|K|$ is a positive integer.)
\end{definition}

\begin{theorem}\label{thm-witt}
 Let $A$ be an additively complete semiring.  Then $W_GA$ is also
 additively complete.  Moreover, there is a natural bijection
 $\tau\:\Map(\sub(G),A)\to W_GA$ and a natural ring map
 $\phi\:W_GA\to\Map(\sub(G),A)$ such that $\phi\tau=\gm$. 
\end{theorem}
This is essentially equivalent to the main result of Brun
in~\cite{br:wvt}, but with our richer theory of Tambara functors in
hand, we can give a somewhat different perspective on the proof.

\begin{proof}
 Combine Propositions~\ref{prop-phi-tau} and~\ref{prop-tau-iso}
 below. 
\end{proof}

Dress and Siebeneicher proved that thare is a unique natural ring
structure on $\Map(\sub(G),A)$ such that $\gm$ is a ring
homomorphism.  We can use $\tau$ to transport the ring structure on
$W_GA$ to $\Map(\sub(G),A)$, and by the uniqueness clause, this must
give the ring structure discussed by Dress and Siebeneicher.  They
define $W_GA$ to be $\Map(\sub(G),A)$ with this structure, so $\tau$
gives an isomorphism from their $W_GA$ to our $W_GA$.

We now start work on the proof of Theorem~\ref{thm-witt}.

First, it turns out that the functor $L$ can be recovered from the
functors $W_H$ for subgroups $H$ of $G$, as explained in
Corollary~\ref{cor-witt-res} below.

\begin{lemma}\label{lem-L-res}
 The following diagram commutes up to natural isomorphism:
 \[ \xymatrix{
     \Semirings_G \ar[d]_\res \ar[r]^{L} & \Tambara_G \ar[d]^\res \\
     \Semirings_H \ar[r]_{L} & \Tambara_H
    }
 \]
\end{lemma}
\begin{proof}
 It will suffice to show that the followings diagram of right adjoints
 is commutative:
 \[ \xymatrix{
     \Semirings_G & \Tambara_G \ar[l]_\om \\
     \Semirings_H \ar[u]^\coind & \Tambara_H \ar[l]^\om \ar[u]_\coind
    }
 \]
 Here $\coind\:\Semirings_H\to\Semirings_G$ is given by
 $\coind(A)=\Map_H(G,A)$.  

 If $S\in\Tambara_H$, then $\om\coind(S)=\coind(S)(G)=S(\res(G))$, whereas
 $\coind\om(S)=\Map_H(G,S(H))$.  Given $g\in G$ we define
 $i_g\:H\to G$ by $i_g(h)=hg$.  This is an $H$-map and so gives
 $i_g^*\:S(G)\to S(H)$.  Putting these maps together for all $g$, we
 get a map $\xi\:S(G)\to\Map_H(G,S(H))$.  If we choose a set $T$ of
 coset representatives, we can identify $\xi$ with the standard map
 $S(\coprod_{t\in T}H)\to\prod_{t\in T}S(H)$, which is an
 isomorphism, as required.
\end{proof}
\begin{corollary}\label{cor-witt-res}
 For any semiring $A$ and any $H\leq G$ we have $(LA)(G/H)=W_H(A)$. 
\end{corollary}
\begin{proof}
 We will write $L_G$ and $L_H$ for the functors $L$ defined using $G$
 and $H$ respectively.  The lemma gives 
 \[ (L_GA)(G/H)=(L_GA)(\ind(1))=(\res L_G A)(1)=(L_HA)(1)=W_H(A). \]
\end{proof}
\begin{remark}\label{rem-om-L}
 One can check directly that for the trivial group $1$ we have
 $W_1A=A$.  Thus, for general $G$ we have $\om LA=(LA)(G/1)=W_1A=A$.
 More precisely, after chasing through the definitions we see that the
 unit map $A\to\om LA=(LA)(G)$ for the $(L,\om)$ adjunction is an
 isomorphism. 
\end{remark}

\begin{proposition}\label{prop-defn-nu}
 There are natural maps $\nu\:A\to LA(U)$ for all orbits $U$, with the
 following properties:
 \begin{itemize}
  \item[(a)] $\nu\:A\to LA(G)=\om LA$ is the unit map for the
   $(L,\om)$ adjunction.
  \item[(b)] For any map $f\:U\to V$ of orbits, we have
   $N_f(\nu_U(a))=\nu_V(a)$.
  \item[(c)] $\nu(1)=1$, and $\nu(ab)=\nu(a)\nu(b)$.
  \item[(d)] $\nu(0)=0$.
 \end{itemize}
 (However, $\nu$ does not respect addition.)
\end{proposition}
\begin{proof}
 We define $\nu_G\:A\to LA(G)$ to be the unit so that~(a) holds.  Given
 any orbit $U$ we choose a point $u\in U$.  This gives a $G$-map
 $\hat{u}\:G\to U$ by $\hat{u}(g)=gu$, and this gives a map
 $N_{\hat{u}}\:LA(G)\to LA(U)$.  If we use a different point $v\in U$
 then $v=xu$.  We then have $\hat{u}=\hat{v}\circ\rho(x)$ (where
 $\rho(x)(g)=gx^{-1}$), but the action of $G$ on $LA(G)\simeq A$ via
 $\rho$ is trivial by assumption, so $N_{\hat{u}}=N_{\hat{v}}$.  We
 can thus put $\nu_U=T_{\hat{u}}\circ\nu_G$ (for any choice of $u$).
 All the claimed properties are clear.
\end{proof}

\begin{proposition}\label{prop-L-prime-constant}
 The multiplicative coefficient system $L'A$ is just the constant
 functor $\Orb_G\to\Semirings$ with value $A$.  
\end{proposition}
\begin{proof}
 Recall from Definition~\ref{defn-L-prime} that 
 $L'A(U)$ is generated by symbols $i_u(a)$
 (for $u\in U$ and $a\in A$) modulo relations
 \begin{align*}
  i_t(0) &= 0 &
  i_t(1) &= 1 \\
  i_t(a+b) &= i_t(a) + i_t(b) &
  i_t(ab) &= i_t(a) i_t(b) \\
  i_{gt}(a) &= i_t(g^{-1}a).
 \end{align*}
 As $G$ acts transitively on $U$ and trivially on $A$, the last
 equation tells us that all the maps $i_u$ are the same, so we can
 just call them $i$; and the map $i\:A\to L'A(U)$ is an isomorphism.
 The norm maps satisfy $N_fi_u=i_{fu}$ by definition, and this means
 that we can identify $N_f$ with $1_A$.
\end{proof}
\begin{corollary}\label{cor-r-L-prime}
 There is a natural isomorphism 
 \[ rL'A(X)=\Map(\pi_0(\Orbt_G\downarrow X),A). \]
\end{corollary}
\begin{proof}
 See Remark~\ref{rem-rN-limit}.
\end{proof}

\begin{proposition}\label{prop-qLA}
 Let $A$ be a semiring with trivial $G$-action.  Then for all orbits
 $U$, the composite $\ov{\nu}=(A\xra{\nu}LA(U)\xra{\pi}qLA(U))$ is an
 isomorphism.  Moreover, for any map $f\:U\to V$ of orbits, we have
 $N_f\ov{\nu}_U=\ov{\nu}_V$.  These maps therefore assemble to give an
 isomorphism $\ov{\nu}\:L'A\to qLA$ of multiplicative coefficient
 systems. 
\end{proposition}
\begin{proof}
 As $q$ is left adjoint to $r$ and $L$ is left adjoint to $\om$, we
 see that $qL$ is left adjoint to $\om r$, which is the same as $\om'$
 by Remark~\ref{rem-om-q-r}.  As $\om'$ is right adjoint to $L'$, we
 have $qL=L'$.  By unwinding the definitions we see that this map is
 the same as $\ov{\nu}$.
\end{proof}
\begin{corollary}\label{cor-LA-complete}
 If $A$ is additively complete, then so is $LA$.  Moreover, the
 filtration layers $F_kLA$ and their quotients are also additively
 complete. 
\end{corollary}
\begin{proof}
 It is clear from the Proposition that $qLA$ is additively complete,
 and the rest follows from Proposition~\ref{prop-qM-complete}.
\end{proof}

\begin{definition}\label{defn-ghost-phi}
 We define a Tambara morphism $\phi\:LA\to rL'A$ by
 \[ \phi = (LA \xra{\eta} rqLA \xra{r(\ov{\nu}^{-1})} rL'A). \]
 Equivalently, $\phi$ corresponds to $\ov{\nu}^{-1}$ with respect to
 the $(q,r)$ adjunction.  By evaluating at $G/G$ and using
 Proposition~\ref{prop-rNGH} to analyse $rL'$ we get a map 
 \[ \phi \: W_GA = LA(G/G) \to rL'A(G/G) = \Map(\sub(G),A). \]
\end{definition}

\begin{remark}\label{rem-ghost-iso}
 If $A$ is a $\Q$-algebra, we see from
 Proposition~\ref{prop-L-om-rational} and Remark~\ref{rem-r-L-prime}
 that $\phi\:LA\to rL'A$ is an isomorphism.
\end{remark}

\begin{lemma}\label{lem-T-nu}
 For any diagram $(U\xra{x}X)$ in $(\Orbt_G\downarrow X)$, the
 function $T_x\nu_U\:A\to LA(X)$ depends only on the isomorphism class
 of $(U\xra{x}X)$ in $(\Orbt_G\downarrow X)$.  
\end{lemma}
\begin{proof}
 If $(V\xra{y}X)$ is isomorphic to $(U\xra{x}X)$ then there is an
 isomorphism $f\:U\to V$ with $yf=x$.  This gives
 $T_x\nu_U=T_yT_f\nu_U$.  However, as $f$ is an isomorphism we have
 $T_f=R_f^{-1}=N_f$, and we have seen that $N_f\nu_U=\nu_V$, so
 $T_x\nu_U=T_y\nu_V$ as required.
\end{proof}

\begin{definition}\label{defn-tau}
 We define maps $\tau\:rL'A(X)\to LA(X)$ as follows.  Choose a list of
 diagrams $(U_i\xra{x_i}X)$ (for $1\leq i\leq m$ say) containing
 precisely one representative of each isomorphism class in
 $(\Orbt_G\downarrow X)$ and put 
 \[ \tau(a) = \sum_{i=1}^m T_{x_i}\nu_{U_i}(a). \]
 It follows from Lemma~\ref{lem-T-nu} that this does not depend on the
 choice of representatives.
\end{definition}
\begin{remark}\label{rem-tau-bad}
 The map $\tau$ does not respect addition or multiplication, so it
 does not give a morphism of Mackey functors or Tambara
 functors.  One can check that it is compatible with $R_f$ and $T_f$
 when $f$ is injective, but not otherwise.   
\end{remark}

\begin{proposition}\label{prop-phi-tau}
 The composite 
 \[ \Map(\sub(G),A) = rL'A(1) \xra{\tau} LA(1)=W_G(A) 
     \xra{\phi} rL'A(1) = \Map(\sub(G),A)
 \]
 is the ghost map $\gm$.
\end{proposition}
\begin{proof}
 First, we have already seen how to identify $LA(G)$ and $rL'A(G)$
 with $A$, and we leave the reader to check that under these
 identifications the maps $\tau\:rL'A(G)\to LA(G)$ and
 $\phi\:LA(G)\to rL'A(G)$ are just the identity. 

 Now consider a subgroup $H\leq G$, and let $p$ and $q$ denote the
 obvious maps $G\xra{p}G/H\xra{q}1$.  Consider an element
 $a\in LA(G)=A$.  As $\phi$ is a Tambara morphism we have 
 \[ \phi(T_q\nu(a))=\phi(T_qN_p(a))=T_qN_p(\phi(a))=T_qN_p(a). \]
 The last expression involves the operators $T_q$ and $N_p$ on $rL'A$,
 which are determined by Remark~\ref{rem-rN-limit}
 and Proposition~\ref{prop-rR-tambara}.  First, for any morphism $G/K\xra{t}G/H$,
 the element $(N_pa)(G/K\xra{t}G/H)$ is by definition a product
 indexed by the orbits $U\sse G\tm_{G/H}(G/K)$.  Each such orbit
 admits a projection map to $G$ and so must be isomorphic to $G$.  It
 follows that the term corresponding to each orbit is just $a$.  By
 counting the number of points in $G\tm_{G/H}(G/K)$ we also see that
 the number of orbits is $|H|/|K|$.  We conclude that
 $(N_pa)(G/K\xra{t}G/H)=a^{|H|/|K|}$.  Moreover, $(T_qN_pa)(G/K\to 1)$
 is by definition the sum of the above elements for all $G$-maps
 $t\:G/K\to G/H$.  We now have
 \[ \phi(T_q\nu(a))(G/K\to 1) = |(G/H)^K| a^{|H|/|K|}. \]

 Next, we note that there is a bijection $\sub(G)\to\pi_0(\Orbt_G)$
 given by $H\mapsto G/H$.  We choose a list of subgroups
 $H_1,\dotsc,H_m$ containing precisely one representative of each
 conjugacy class, and we make some slight notational changes
 corresponding to the above bijection, giving
 \[ \phi(\tau(a))(H_j) =
      \sum_i |(G/H_i)^{H_j}| a(H_i)^{|H_i|/|H_j|}.
 \]
 This is the same as the ghost map, as claimed.
\end{proof}

\begin{corollary}\label{cor-tau-mono}
 Suppose that $A$ is additively complete and torsion-free.  Then the
 maps $\tau:rL'A(X)\to LA(X)$ are injective for all $X$.
\end{corollary}
\begin{proof}
 All the relevant functors convert disjoint unions to products, so we
 can reduce to the case $X=G/H$.  Here $LA(G/H)=W_HA$ and similarly
 $rL'A(G/H)=\Map(\sub(H),A)$.  One can check that these
 identifications are compatible with $\tau$, so we can reduce further
 to the case $X=G/G=1$ which appears in the Proposition.  There we
 have $\phi\tau=\gm$, so it will suffice to check that $\gm$ is
 injective.  Choose a subgroup system $H_1,\dotsc,H_r$ (as in
 Definition~\ref{defn-subgroup-system}).  Suppose we have elements
 $a,b\in\Map(\sub(G),A)$ with $\gm(a)=\gm(b)$.  Fix $k\leq |G|$, and
 assume inductively that $a_i=b_i$ for all $i>k$.  After cancelling
 the terms corresponding to $H_i$ for $i>k$, we get 
 \[ \sum_{i\leq k} |(G/H_i)^{H_j}| a(H_i)^{|H_i|/|H_j|} =
    \sum_{i\leq k} |(G/H_i)^{H_j}| b(H_i)^{|H_i|/|H_j|}
 \] 
 for all $j$.  Take $j=k$, and note that the definition of a subgroup
 system implies that $(G/H_i)^{H_k}=\emptyset$ for $i<k$.  Thus, we
 only have the term for $i=k$, giving
 $|(G/H_k)^{H_k}|a(H_k)=|(G/H_k)^{H_k}|b(H_k)$.  As
 $|(G/H_k)^{H_k}|>0$ and $A$ is torsion-free we deduce that
 $a(H_k)=b(H_k)$.  The claim now follows by decreasing induction on
 $k$. 
\end{proof}

\begin{proposition}\label{prop-tau-epi}
 Suppose that $A$ is additively complete.  Then the maps
 $\tau:rL'A(X)\to LA(X)$ are surjective for all $X$. 
\end{proposition}
\begin{proof}
 Firstly, all semigroups that we consider will be additively complete
 by Corollary~\ref{cor-LA-complete}, so we can use standard methods
 with kernels and cokernels and so on.

 We filter $LA$ by sub-Mackey functors $F_kLA$ as in
 Section~\ref{sec-filt}, and put $Q_kLA=F_kLA/F_{k-1}LA$.  Next, let
 $\CO_k(X)$ be the category of orbits of coorder $k$ over $X$, so
 $(\Orbt_G\downarrow X)=\coprod_k\CO_k(X)$.  Put
 $P_k(X)=\Map(\pi_0(\CO_k(X)),A)$, so $rL'A(X)=\prod_kP_k(X)$.  Let
 $\tau_k$ be the restriction of $\tau$ to $P_k(X)$.  It will be enough
 to prove that $F_kLA(X)=F_{k-1}LA(X)+\img(\tau_k)$, or equivalently
 that the composite
 \[ \sg_k = (P_k(X) \xra{\tau_k} F_kLA(X) \xra{} Q_kLA(X)) \]
 is surjective.  Now choose a system of representatives
 $(U_i\xra{x_i}X)$ for the isomorphism classes in $\CO_k(X)$.  It is
 straightforward to check that for any Mackey functor $M$ we have
 $F_kM(U_i)=M(U_i)$ and $Q_kM(U_i)=qM(U_i)$.  Moreover, we have
 $F_kM(X)=F_{k-1}M(X)+\sum_iT_{x_i}M(U_i)$, so the maps $T_{x_i}$
 induce an epimorphism $\sg'_k\:\prod_iqM(U_i)\to Q_kM(X)$.  In the
 case $M=LA$ we have $qM(U_i)=A$.  After recalling that the composite
 $\ov{\nu}=\pi\nu$ respects addition, we find that $\sg_k=\sg'_k$, so
 $\sg_k$ is surjective as required.
\end{proof}

\begin{proposition}\label{prop-tau-iso}
 If $A$ is additively complete, then $\tau\:rL'A(X)\to LA(X)$ is
 bijective for all $X$.
\end{proposition}
\begin{proof}
 This is clear from Corollary~\ref{cor-tau-mono} and
 Proposition~\ref{prop-tau-epi} if $A$ is torsion-free.  

 Now suppose we have a coequaliser diagram
 \[ \xymatrix{ A_2
     \ar@/^1ex/[r]^\al \ar@/_1ex/[r]_\bt &
     A_1 \ar[r]^\gm & A_0
 } \]
 in the category of rings.  Suppose also that this is a reflexive
 coequaliser, which means that there is a ring map $\sg\:A_1\to A_2$
 with $\al\sg=\bt\sg=1$.  Suppose that the result holds for $A_1$ and
 $A_2$; we claim that it also holds for $A_0$.  To see this, consider
 the diagram
 \[ \xymatrix{
  rL'A_2(X) \ar@/^1ex/[r] \ar@/_1ex/[r] \ar[d]_\tau &
  rL'A_1(X) \ar[r] \ar[d]_\tau &
  rL'A_0(X) \ar[d]^\tau \\
  LA_2(X) \ar@/^1ex/[r] \ar@/_1ex/[r] &
  LA_1(X) \ar[r] &
  LA_0(X).
 } \]
 As $L$ has a right adjoint, it preserves reflexive coequalisers, and
 Proposition~\ref{prop-reflexive-coeq} tells us that reflexive
 coequalisers can be computed separately for each $X$, so the bottom
 row above is a reflexive coequaliser in the category of sets.  From
 the explicit description in Corollary~\ref{cor-r-L-prime} we see
 that the top row is also a reflexive coequaliser.  The first two
 vertical maps are bijective by assumption, and it follows that the
 same is true of the third.

 Now let $A_0$ be any ring.  Let $A_1$ be a polynomial algebra over
 $\Z$ with one generator for each element of $A_0$; then $A_1$ is
 torsion-free, and there is an evident surjective homomorphism
 $\gm\:A_1\to A_0$.  Put 
 \[ A_2 = \{(a,a')\in A_1\tm A_1\st \gm(a)=\gm(a')\}. \]
 This is clearly also torsion-free.  Let $\al,\bt\:A_2\to A_1$ be the
 two projections, and define $\sg\:A_1\to A_2$ by $\sg(a)=(a,a)$.
 This gives a reflexive coequaliser diagram, and we deduce that
 $\tau\:rL'A(X)\to LA(X)$ is bijective as claimed.
\end{proof}

We next discuss how our approach relates to that of
Elliott~\cite{el:cwb} (but we will gloss over certain issues related
to additive completion).  For any semigroup $M$ we have a semiring
$\N[M]$ (and a ring $\Z[M]$) to which we can apply the above
discussion to define a semiring $W_G(\N[M])$ and a ring
$W_G(\Z[M])$.  Elliott mostly considers this case, and deduces results
for more general rings by noting that they can be expressed as
quotients of polynomial rings.  With hindsight we can see that he is
really working with Tambara functors, and most of the numerous maps
that he constructs have natural interpretations in that theory.  One
key point is as follows: we can construct a ``coconstant'' Mackey
functor $dM$ as in Definition~\ref{defn-coconstant}, and then do the
Tambara analogue of the semigroup-semiring construction to get a
Tambara functor $A[dM]$ as in Section~\ref{sec-semigroup-semirings}.
We then have
\begin{align*}
 \Tambara_G(A[dM],S)
  &= \Mackey_G(dM,US) 
   = \Semigroups_G(M,US(G)) 
   = \Semirings_G(\N[M],S(G)) \\
  &= \Tambara_G(L\N[M],S),
\end{align*}
so $L\N[M]\simeq A[dM]$.  In this case we can define maps
\[ \nu'_U\:\N[M]\to L\N[M](U) = A[dM](U) \]
(for $G$-orbits $U$) as follows.  First, recall that
$dM(U)=\Map(U,M)_G$, and as $G$ acts trivially on $M$, this is easily
identified with $M$.  Next, every element of $A[M](U)$ is represented
by a pair $(X\xra{t}U,m)$, where $m\in dM(X)$.  Thus, for any $m\in M$
the pair $(U\xra{1}U,m)$ represents an element $\lm(m)\in A[dM](U)$.
From the definitions we see that $\lm(m_0+m_1)=\lm(m_0)\lm(m_1)$.  We
can thus define a semiring map $\nu'_U\:\N[M]\to A[dM](U)$ by
$\nu'_U(\sum_in_i[m_i])=\sum_in_i\lm(m_i)$.  

These maps $\nu'_U$ have similar properties to the maps $\nu_U$
considered earlier, and we can use them in the same way to define maps
$\tau'\:rL'\N[M](X)\to L\N[M](X)$ analogous to $\tau$.  In more
detail, we choose a list of diagrams $(U_i\xra{x_i}X)$ (for $1\leq
i\leq m$ say) containing precisely one representative of each
isomorphism class in $(\Orbt_G\downarrow X)$ and put  
\[ \tau'(a) = \sum_{i=1}^m T_{x_i}\nu'_{U_i}(a). \]

This map $\tau'$ is easier to use than $\tau$ because it is additive,
and it is also a bijection.  This can be proved along the same lines
used for $\tau$, or one can just exhibit an inverse as follows.

\begin{definition}\label{defn-beta-prime}
 For any orbit $U$ and any element $[W\xra{u}U,m]\in A[dM](U)$ (so
 $m\in dM(W)$) we define 
 \[ \bt_0[W\xra{u}U,M]= \sum_{w\in\Sec(u)} R_w(m)
     \in \N[dM(U)] = \N[M].
 \]
 For any $G$-set $X$ and any element $a\in A[dM](X)$ we define
 \[ \bt'(a)\in rL'\N[M](X) = \Map(\pi_0(\Orbt_G\downarrow X),\N[M]) \]
 by $\bt'(a)(U\xra{x}X)=\bt_0(R_x(a))$.
\end{definition}

It is not hard to check directly that $\bt'$ is inverse to $\tau'$.
We also have a map $\gm'=\phi\tau'$ analogous to the ghost map $\gm$.
To describe this, we first define maps $a\mapsto a^{\ip{k}}$ on
$\N[M]$ (for $k\in\N$) by 
\[ \left(\sum_i n_i[m_i]\right)^{\ip{k}} =
    \sum_{i} n_i [km_i].
\]
This is similar to the usual power map $a\mapsto a^k$, because
$a^k=a^{\ip{k}}$ whenever $a=[m]$ for some $m$, and also
$a^p=a^{\ip{p}}\pmod{p}$ whenever $p$ is prime.  One can check that 
\[ \gm'(a)([H]) = \sum_{[K]} |(G/H)^K| a([K])^{\ip{|H|/|K|}}. \]
A key part of Elliott's argument is to show that $\phi$ and $\phi'$
have the same image.

\appendix
\renewcommand{\thesection}{\Alph{section}}

\section{Semiadditive categories}
\label{apx-semiadditive}

In this appendix we set up a theory of semiadditive categories.  This
is largely the same as the more familiar theory of additive
categories.  However, there is one technical point that creates a need
for a detailed account.  Given a small semiadditive category $\CA$, we
write $\Mod(\CA)$ for the category of product-preserving functors 
$M\:\CA\to\Sets$.  We will show that every such $M$ can also be
regarded as an additive functor $\CA\to\Semigroups$.  Now suppose that
$\CA$ has a biadditive tensor product making it a symmetric monoidal
category.  We would like to use the well-known construction of
Day~\cite{da:ccf} to obtain from this a biadditive tensor product on
$\Mod(\CA)$.  However, there are various subtleties about whether we
require additivity at various stages in the construction, and whether
we use $\Sets$ or $\Semigroups$ as the ambient category.  For our
applications to Tambara functors, the key point is that we can work
with apparently non-additive constructions and still get a biadditive
functor $\btm\:\Mod(\CA)\tm\Mod(\CA)\to\Mod(\CA)$.  This will be
explained in more detail below.

\subsection{Basics}

\begin{definition}\label{defn-semiadditive}
 Let $\CA$ be a category.  We say that $\CA$ is \emph{semiadditive} if
 \begin{itemize}
  \item[(a)] It has finite products (and in particular, a terminal
   object, denoted by $0$).
  \item[(b)] The terminal object is also initial (so for each $a$ and
   $b$ we have unique maps $a\xra{}0\xra{}b$, whose composite we call
   the zero map).
  \item[(c)] Let $a_1\xla{p_1}a\xra{p_2}a_2$ be a product
   diagram, and let $a_1\xra{i_1}a\xla{i_2}a_2$ be the unique
   maps such that
   \[ p_1i_1 = 1 \hspace{3em}
      p_2i_1 = 0 \hspace{3em}
      p_1i_2 = 0 \hspace{3em}
      p_2i_2 = 1.
   \]
   Then the diagram $a_1\xra{i_1}a\xla{i_2}a_2$ is a coproduct diagram.
 \end{itemize}
\end{definition}

\begin{remark}\label{rem-finite-coproducts}
 By an evident inductive extension of~(c), any product of finitely many
 factors is also a coproduct.  (The case of no factors is~(b), and the
 case of one factor is trivial.)  In more detail, given objects
 $a_1,\dotsc,a_n$ there is an object $a=\bigoplus_ta_t$ and
 morphisms
 \[ a_t \xra{i_t} a \xra{p_u} a_u \]
 such that $p_ti_t=1$, and $p_ui_t=0$ for all $u\neq t$, and the maps
 $i_t$ give a coproduct diagram, and the maps $p_u$ give a product
 diagram.  For any object $x$ and system of maps $f_u\:x\to a_u$, we
 write $[f_1,\dotsc,f_n]$ for the unique map $x\to a$ such that 
 $p_u\circ[f_1,\dotsc,f_n]=f_u$ for all $u$.  Dually, given maps
 $g_t\:a_t\to y$, we write $\ip{g_1,\dotsc,g_n}$  the unique
 map $a\to y$ with $\ip{g_1,\dotsc,g_n}\circ j_t=g_t$ for all $t$.
 Note that in the case $n=2$ we have
 \begin{align*}
  i_1 &= [1,0] \: a_1\to a_1\oplus a_2 & 
  i_2 &= [0,1] \: a_2\to a_1\oplus a_2 \\
  p_1 &= \ip{1,0} \: a_1\oplus a_2\to a_1 &
  p_2 &= \ip{0,1} \: a_1\oplus a_2\to a_2.
 \end{align*}
\end{remark}

We now resolve an apparent asymmetry in the definition.
\begin{lemma}\label{lem-semiadditive}
 Let $\CA$ be a semiadditive category.  Let $a_1\xra{i_1}a\xla{i_2}a_2$
 be a coproduct diagram, and let $a_1\xla{p_1}a\xra{p_2}a_2$ be the
 unique maps such that 
 \[ p_1i_1 = 1 \hspace{3em}
    p_2i_1 = 0 \hspace{3em}
    p_1i_2 = 0 \hspace{3em}
    p_2i_2 = 1.
 \]
 Then the diagram $a_1\xla{p_1}a\xra{p_2}a_2$ is a product diagram.
\end{lemma}
\begin{proof}
 By axiom~(a), there exists a product diagram
 $a_1\xla{p'_1}a'\xra{p'_2}a_2$.  From this we can build a coproduct
 diagram $a_1\xra{i'_1}a'\xla{i'_2}a_2$ as in axiom~(c).  By the
 standard uniqueness property of coproducts, there is a unique map
 $f\:a\to a'$ with $i'_tf=i_t$ for $t=1,2$, and this map $f$ is an
 isomorphism.  It follows easily that the diagram
 $a_1\xla{p'_1f}a\xra{p'_2f}a_2$ is a product diagram, so it will
 suffice to prove that $p'_1f=p_1\:a\to a_1$ and $p'_2f=p_2$.  Note
 that $p'_1fi_1=p'_1i'_1=1=p_1i_1$ and $p'_1fi_2=p'_1i'_2=0=p_1i_2$.
 As $a_1\xra{i_1}a\xla{i_2}a_2$ is a coproduct diagram, it follows
 that $p'_1f=p_1$, and a symmetrical argument shows that $p'_2f=p_2$,
 as required.
\end{proof}

\begin{definition}\label{defn-preadditive}
 A \emph{preadditive category} is a category $\CB$ with a given
 semigroup structure on each morphism set $\CB(b_0,b_1)$, with the
 property that each composition map
 \[ \CB(b_1,b_2) \tm \CB(b_0,b_1) \to \CB(b_0,b_2) \]
 is bilinear.  Now let $F\:\CB\to\CC$ be a functor between preadditive
 categories.  We say that $F$ is a \emph{preadditive functor} if each
 map  
 \[ F \: \CB(b_0,b_1) \to \CC(Fb_0,Fb_1) \]
 is a semigroup homomorphism.
\end{definition}

\begin{proposition}\label{prop-semiadditive-hom}
 Let $\CA$ be a semiadditive category.  Then $\CA$ has a canonical
 structure as a preadditive category.
\end{proposition}
\begin{proof}
 Given maps $g_1,g_2\:b\to c$ we define
 \[ g_1+g_2 = (b \xra{[1,1]} b\oplus b \xra{\ip{g_1,g_2}} c). \]
 In the case $g_2=0$ we observe that $g_1p_1\:b\oplus b\to c$ has the
 defining property of $\ip{g_1,0}$; it follows easily that
 $g_1+0=g_1$.  Next, put $\sg=[p_2,p_1]\:b\oplus b\to b\oplus b$.  One
 can check that this is the same as $\ip{i_2,i_1}$, and that
 $\sg\circ[1,1]=[1,1]$ and $\ip{g_1,g_2}\circ\sg=\ip{g_2,g_1}$.  Using
 this we see that $g_2+g_1=g_1+g_2$.  Finally, if $g_3$ is a third
 morphism from $b$ to $c$ then it is not hard to check that
 $g_1+(g_2+g_3)$ and $(g_1+g_2)+g_3$ are both equal to the composite
 \[ b \xra{[1,1,1]} b\oplus b\oplus b \xra{\ip{g_1,g_2,g_3}} c. \]
 It follows that $\CA(b,c)$ is a semigroup.

  Now consider a map $h\:c\to d$.  It is clear from the definitions
  that $h\circ\ip{g_1,g_2}=\ip{hg_1,hg_2}$ and thus that
  $h\circ(g_1+g_2)=(hg_1)+(hg_2)$.  Consider instead a map
  $f\:a\to b$.  This of course induces a map 
  \[ f\oplus f = [fp_1,fp_2]\:a\oplus a \to b\oplus b. \]
  By considering the diagram 
  \[ \xymatrix{
   a \ar[r]^{[1,1]} \ar[d]_f &
   a\oplus a \ar[r]^{\ip{g_1f,g_2f}} \ar[d]^{f\oplus f} &
   c \ar@{=}[d] \\
   b \ar[r]_{[1,1]} &
   b\oplus b \ar[r]_{\ip{g_1,g_2}} &
   c
  } \]
  we see that $(g_1+g_2)f=(g_1f)+(g_2f)$.  This proves that
  composition is bilinear.
\end{proof}

\begin{proposition}\label{prop-sum}
 Let $\CA$ be a semiadditive category.  Then in the context of axiom~(c)
 or Lemma~\ref{lem-semiadditive} we always have
 \[ i_1p_1+i_2p_2 = 1 \: a \to a. \]
\end{proposition}
\begin{proof}
 Suppose that the maps $i_t$ and $p_t$ are as in axiom~(c) or
 Lemma~\ref{lem-semiadditive}.  We then have 
 \[ p_1\circ (i_1p_1+i_2p_2) = p_1i_1p_1+p_1i_2p_2 
     = 1\circ p_1 + 0\circ p_2 = p_1 = p_1\circ 1,
 \]
 and similarly $p_2\circ(i_1p_1+i_2p_2)=p_2\circ 1$.  As
 $a_1\xla{p_1}a\xra{p_2}a_2$ is a product diagram, this means that
 $i_1p_1+i_2p_2=1$.  
\end{proof}

\begin{proposition}\label{prop-recognise}
 Let $\CA$ be a preadditive category.
 \begin{itemize}
  \item[(a)] An object $a\in\CA$ is initial iff it is terminal iff the
   semigroup $\CA(a,a)$ is trivial.
  \item[(b)]
   Suppose we have a diagram 
   \[ \xymatrix{
     a_1 \ar@/^1ex/[r]^{i_1} &
     x \ar@/^1ex/[l]^{p_1} \ar@/_1ex/[r]_{p_2} &
     a_2 \ar@/_1ex/[l]_{i_2}
   } \]
   with 
   \[ p_1i_1 = 1 \hspace{3em}
      p_2i_1 = 0 \hspace{3em}
      p_1i_2 = 0 \hspace{3em}
      p_2i_2 = 1 \hspace{3em}
      i_1p_1+i_2p_2 = 1.
   \]
   Then the diagram $a_1\xra{i_1}x\xla{i_2}a_2$ is a coproduct, and the
   diagram $a_1\xla{p_1}x\xra{p_2}a_2$ is a product.
 \end{itemize}
\end{proposition}
\begin{proof}
 \begin{itemize}
  \item[(a)] If $a$ is either initial or terminal then $\CA(a,a)$ is a
   singleton and so is trivial as a semigroup.  Conversely, if
   $\CA(a,a)$ is trivial then in particular the identity morphism is
   the same as the zero element.  Now for any $f\:x\to a$ we have
   $f=1_a\circ f=0\circ f$, and because composition is biadditive this
   is the zero element of $\CA(x,a)$.  This proves that $a$ is
   terminal, and a dual argument shows that it is also initial.
  \item[(b)]
   Suppose we have maps $a_1\xra{f_1}u\xla{f_2}a_2$.  We put
   $f=f_1p_1+f_2p_2\:x\to u$.  This has 
   \[ fi_1=f_1p_1i_1+f_2p_2i_1=f_1\circ 1+f_2\circ 0 = f_1, \]
   and similarly $fi_2=f_2$.  Let $f^*\:x\to u$ be any map with
   $f^*i_t=f_t$ for $t=0,1$.  We then have
   \[ f^* = f^*\circ 1_x = f^*\circ(i_1p_1+i_2p_2) =
       f^*i_1p_1+f^*i_2p_2 = f_1p_1+f_2p_2 = f.
   \]
   This proves that the diagram $a_1\xra{i_1}x\xla{i_2}a_2$ is a
   coproduct.  Dually, given maps $a_1\xla{g_1}v\xra{g_2}a_2$ we find
   that the map $g=i_1g_1+i_2g_2\:v\to x$ is the unique one satisfying
   $p_tg=g_t$ for $t=0,1$ so the diagram $a_1\xla{p_1}x\xra{p_2}a_2$ is
   a product.
 \end{itemize}
\end{proof}

\begin{corollary}
 Let $\CA$ be a preadditive category.  Suppose that there is an object
 satisfying part~(a) of the Proposition, and that for all objects
 $a_1$ and $a_2$ there is a diagram as in part~(b) of the Proposition.
 Then $\CA$ is semiadditive.
\end{corollary}
\begin{proof}
 Follows directly from the Proposition.
\end{proof}

\begin{proposition}
 Let $F\:\CA\to\CB$ be a functor between semiadditive categories.  Then
 the following are equivalent:
 \begin{itemize}
  \item[(a)] $F$ sends finite coproduct diagrams to finite coproduct
   diagrams.
  \item[(b)] $F$ sends finite product diagrams to finite product
   diagrams.
  \item[(c)] $F$ is a preadditive functor.
 \end{itemize}
 (We say that $F$ is \emph{semiadditive} if these conditions hold.)
\end{proposition}
\begin{proof}
 We first claim that any of the three conditions implies that $F$
 sends the zero object to the zero object (and thus sends zero
 morphisms to zero morphisms).  In case~(a) this holds because $0$ is
 the coproduct of the empty family, and in case~(b) because $0$ is the
 product of the empty family.  In case~(c) we note that $1=0$ as
 endomorphisms of $0$.  Any functor preserves identity maps, and $F$
 preserves zero maps by hypothesis, so $1=0$ as endomorphisms of $F0$.
 This means that the unique maps $0\to F0\to 0$ are inverse to each
 other, so $F0\simeq 0$ as claimed.

 Now suppose that~(a) holds.  Consider a product diagram
 $a_1\xla{p_1}a\xra{p_2}a_2$.  Let $a_1\xra{i_1}a\xla{i_2}a_2$ be the
 unique maps such that
 \[ p_1i_1 = 1 \hspace{3em}
    p_2i_1 = 0 \hspace{3em}
    p_1i_2 = 0 \hspace{3em}
    p_2i_2 = 1,
 \]
 so the diagram $a_1\xra{i_1}a\xla{i_2}a_2$ is a coproduct diagram by
 axiom~(c) in Definition~\ref{defn-semiadditive}.  By assumption, the
 diagram $Fa_1\xra{Fi_1}Fa\xla{Fi_2}Fa_2$ is a coproduct diagram in
 $\CB$.  We can apply Lemma~\ref{lem-semiadditive} to this diagram to see
 that $Fa_1\xla{Fp_1}Fa\xra{Fp_2}Fa_2$ is a product diagram.  This
 shows that~(a) implies~(b), and a dual argument shows that~(b)
 implies~(a).  If these conditions hold, then $F$ respects all the
 structure used to define addition, so we see that~(c) also holds.
 Conversely, suppose that~(c) holds.  Any binary product or coproduct
 diagram in $\CA$ gives rise to a diagram 
 \[ \xymatrix{
   a_1 \ar@/^1ex/[r]^{i_1} &
   x \ar@/^1ex/[l]^{p_1} \ar@/_1ex/[r]_{p_2} &
   a_2 \ar@/_1ex/[l]_{i_2}
 } \]
 with 
 \[ p_1i_1 = 1 \hspace{3em}
    p_2i_1 = 0 \hspace{3em}
    p_1i_2 = 0 \hspace{3em}
    p_2i_2 = 1 \hspace{3em}
    i_1p_1+i_2p_2 = 1
 \]
 as in Proposition~\ref{prop-recognise}.  Applying $F$ gives a diagram
 in $\CB$ with the same properties, which is therefore both a product
 and a coproduct.  This shows that~(c) implies~(a) and~(b).
\end{proof}

\subsection{Modules}

\begin{definition}
 Let $\CA$ be a small semiadditive category.  An \emph{$\CA$-module} is a
 product-preserving functor $M\:\CA\to\Sets$.  We write $\CM(\CA)$ for
 the category of $\CA$-modules.  If $f\:a\to b$ is a morphism in $\CA$
 and $M$ is an $\CA$-module, we write $f_*$ for the induced map
 $M(a)\to M(b)$.
\end{definition}

We next show that if $M$ is a $\CA$-module, then the sets $M(a)$
have canonical structures as semigroups.  The cleanest way to
formulate this is as follows.

\begin{definition}
 Let $\CA$ be a small preadditive category.  We write $\CM'(\CA)$ for
 the category of preadditive functors from $M\:\CA\to\Semigroups$.
\end{definition}

\begin{proposition}
 The forgetful functor $U\:\Semigroups\to\Sets$ induces an isomorphism
 of categories $U_*\:\CM'(\CA)\to\CM(\CA)$.
\end{proposition}
\begin{proof}
 First, recall that products of semigroups are just constructed
 by giving the product set a suitable semigroup structure, which means
 that $U$ presevres products, so $U_*$ gives a functor from
 $\CM'(\CA)$ to $\CM(\CA)$ as indicated.  

 Now suppose we have $M\in\CM(\CA)$.  Consider an object $a\in\CA$.
 As $M$ preserves products we see that $M(0)$ is a singleton.  We
 also have a unique map $0\:0\to a$, giving a map $0_*\:M(0)\to M(a)$;
 we again write $0$ for the unique element in the image of this map.

 Now construct a biproduct diagram 
 \[ \xymatrix{
   a \ar@/^1ex/[r]^{i_1} &
   a\oplus a \ar@/^1ex/[l]^{p_1} \ar@/_1ex/[r]_{p_2} &
   a \ar@/_1ex/[l]_{i_2}
 } \]
 as before.  Put $s=p_1+p_2=\ip{1,1}\:a\oplus a\to a$.  By hypothesis,
 the map $((p_1)_*,(p_2)_*)\:M(a\oplus a)\to M(a)^2$ is a bijection.
 Given $m_1,m_2\in M(a)$ we let $m\in M(a\oplus a)$ be the unique
 element with $(p_t)_*(m)=m_t$ for $t=1,2$, then we define
 $m_1+m_2=s_*(m)\in M(a)$.   By an argument similar to that for
 Proposition~\ref{prop-semiadditive-hom}, this gives an semigroup
 structure on $M(a)$.  Now let $f\:a\to b$ be a morphism in $\CA$.  By
 considering the diagram
 \[ \xymatrix{
  M(a)\tm M(a) \ar[d]_{f_*\tm f_*} &&
  M(a\oplus a)
    \ar[ll]^\simeq_{((p_1)_*,(p_2)_*)}
    \ar[rr]^{s_*}
    \ar[d]^{(f\oplus f)_*} &&
  M(a) \ar[d]^{f_*} \\
  M(b)\tm M(b) &&
  M(b\oplus b)
    \ar[ll]_\simeq^{((p_1)_*,(p_2)_*)}
    \ar[rr]_{s_*} &&
  M(b)
 } \]
 we see that $f_*(m_1+m_2)=f_*(m_1)+f_*(m_2)$, so $f_*$ is a
 homomorphism.  (As we are dealing with semigroups rather than groups we
 need to check separately that $f_*(0)=0$, but this is easy.)  We can
 thus use these structures to regard $M$ as a functor $\CA\to\Semigroups$,
 which preserves products and so is semiadditive.  In other words, we have
 an object in $\CM'(\CA)$, which we call $FM$.

 Now suppose we have a natural transformation $\al\:M\to N$ between
 product-preserving functors $\CA\to\Sets$.  By considering the
 diagram 
 \[ \xymatrix{
  M(a)\tm M(a) \ar[d]_{\al_a\tm\al_a} &&
  M(a\oplus a)
    \ar[ll]^\simeq_{((p_1)_*,(p_2)_*)}
    \ar[rr]^{s_*}
    \ar[d]^{\al_{a\oplus a}} &&
  M(a) \ar[d]^{\al_a} \\
  N(a)\tm N(a) &&
  N(a\oplus a)
    \ar[ll]_\simeq^{((p_1)_*,(p_2)_*)}
    \ar[rr]_{s_*} &&
  N(a)
 } \]
 we see that $\al_a(m_1+m_2)=\al_a(m_1)+\al_a(m_2)$.  We also have
 $\al_a(0)=0$, so $\al_a$ is a homomorphism.  Using this, we see that
 our construction actually gives a functor $F\:\CM(\CA)\to\CM'(\CA)$.
 It is clear by construction that the composite
 $U_*\circ F\:\CM(\CA)\to\CM(\CA)$ is equal (not just isomorphic) to
 the identity.

 In the opposite direction, suppose we have an object
 $M'\in\CM'(\CA)$.  Consider an object $a\in\CA$ and elements
 $m_1,m_2\in M'(a)$.  As $M'(a\oplus a)$ has a specified semigroup
 structure, it is meaningful to define
 $m=(i_1)_*(m_1)+(i_2)_*(m_2)\in M'(a\oplus a)$.  As $M'$ is a functor
 with values in $\Semigroups$ we see that $(p_t)_*\:M'(a\oplus a)\to M'(a)$
 is a homomorphism, and using this it follows that $(p_t)_*(m)=m_t$,
 so we have the same element $m$ as discussed previously.  Similarly,
 the map $s_*\:M'(a\oplus a)\to M'(a)$ is a homomorphism, and
 $si_1=si_2=1$.  It follows that $s_*(m)=m_1+m_2$.  This means that
 the originally given semigroup structure on $M'(a)$ is the same as the
 one constructed by the functor $F$, so $M'=FU_*(M')$.  Thus, the
 functors $F$ and $U_*$ are mutually inverse.
\end{proof}

\begin{proposition}\label{prop-res-equiv}
 Let $\CA$ be a small semiadditive category, and let $\CA_0$ be full
 subcategory such that every object in $\CA$ can be expressed as a
 coproduct of some finite family of objects in $\CA_0$.  Then the
 restriction functor $\res\:\CM'(\CA)\to\CM'(\CA_0)$ is an equivalence of
 categories. 
\end{proposition}
\begin{proof}
 We first claim that $\res$ is faithful.  Consider a pair of morphisms
 $\al,\bt\:M\to N$ in $\CM'(\CA)$ with $\res(\al)=\res(\bt)$, and an
 object $a\in\CA$; we must show that $\al_a=\bt_a\:M(a)\to N(a)$.  By
 hypothesis we can find objects $a_1,\dotsc,a_n\in\CA_0$ and morphisms
 $a_t\xra{i_t}a\xra{p_u}a_u$ as in Remark~\ref{rem-finite-coproducts},
 so $p_ui_u=1$ and $p_ui_t=0$ for $u\neq t$ and $\sum_ti_tp_t=1$.  We
 have seen that $N$ preserves products, and the maps $p_t$ give a
 product diagram, so it will suffice to show that
 $N(p_t)\circ\al_a=N(p_t)\circ\bt_a$ for all $t$.  By naturality we
 have $N(p_t)\circ\al_a=\al_{a_t}\circ M(p_t)$ and similarly for
 $\bt$.  However, as $a_t\in\CA_0$ and $\res(\al)=\res(\bt)$ we have
 $\al_{a_t}=\bt_{a_t}$, and the claim follows.

 We next claim that $\res$ is full.  To see this, consider a morphism
 $\al_0\:\res(M)\to\res(N)$.  Fix an object $a\in\CA$.  We claim that
 there is a unique map $\al_a\:M(a)\to N(a)$ such that for all
 $a'\in\CA_0$ and $f\:a\to a'$, the diagram 
 \[ \xymatrix{
     M(a) \ar[r]^{\al_a} \ar[d]_{M(f)} & 
     N(a) \ar[d]^{N(f)} \\
     M(a') \ar[r]_{(\al_0)_{a'}} & N(a')
    }
 \]
 commutes.  To see this, we choose objects $a_t\in\CA_0$ and morphisms
 $i_t$ and $p_t$ as before.  As the maps $N(p_t)$ give a product
 diagram, we see that there is a unique map $\al_a\:M(a)\to N(a)$
 with $N(p_t)\circ\al_a=(\al_0)_{a_t}\circ M(p_t)$ for all $t$.  Now
 consider an arbitrary object $a'\in\CA_0$ and a morphism
 $f\:a\to a'$.  Recall that $1_a=\sum_ti_tp_t$, so
 $f=\sum_t(fi_t)p_t$.  Now $fi_t\:a_t\to a'$ is a morphism in $\CA_0$,
 so $(\al_0)_{a'}M(f_{it})=N(fi_t)(\al_0)_{a_t}$.  Using this we get 
 \begin{align*}
  (\al_0)_{a'}M(f)
   &= \sum_t (\al_0)_{a'}\circ M(fi_t)\circ M(p_t) 
    = \sum_t N(fi_t)\circ (\al_0)_{a_t} \circ M(p_t) \\
   &= \sum_t N(fi_t)\circ N(p_t)\circ\al_a 
    = \sum_t N(fi_tp_t)\circ\al_a 
    = N(f)\circ\al_a,
 \end{align*}
 so $\al_a$ has the claimed property.  Uniqueness is clear, and
 naturality follows from uniqueness.  We thus have a morphism
 $\al\:M\to N$ with $\res(\al)=\al_0$ as required.

 Now consider an object $M_0\in\CM'(\CA_0)$.  For any $a\in\CA$ we
 have a preadditive functor $K_a\:\CA_0\to\Semigroups$ given by
 $K_a(a')=\CA(a,a')$.  This gives an object $K_a\in\CM'(\CA_0)$, and
 we define
 \[ M(a) = \CM'(\CA_0)(K_a,N_0), \]
 and note that this has a semigroup structure under pointwise
 addition.  Now $K_a$ is a preadditive contravariant functor of $a$,
 so $M$ is a preadditive covariant functor, or in other words
 $M\in\CM'(\CA)$.  It follows from the Yoneda Lemma that
 $\res(M)\simeq M_0$, so $\res$ is also essentially surjective.
\end{proof}

\subsection{Tensor products}

Let $\CA$ be a small semiadditive category, and suppose that $\CA$ has a
symmetric monoidal structure given by a functor $\ot\:\CA\tm\CA\to\CA$
with a unit object $1\in\CA$.  For typographical convenience, we will
also use the symbol $\mu$ for the functor $\ot\:\CA\tm\CA\to\CA$.  We
will assume that this functor is biadditive, which means that for
$f_1,f_2\:a\to b$ and $g_1,g_2\:c\to d$ we have 
\[ (f_1+f_2)\ot(g_1+g_2) = 
    f_1\ot g_1 + f_1\ot g_2 + f_2\ot g_1 + f_2\ot g_2 \:
     a\ot c\to b\ot d,
\]
and also $f\ot 0=0$ and $0\ot g=0$.  

\begin{definition}
 Let $M$ and $N$ be $\CA$-modules.  We define $M*N\:\CA\tm\CA\to\Sets$
 by $(M*N)(a,b)=M(a)\tm N(b)$.  We then put $M\btm N=\colim_\mu(M*N)$
 (the left Kan extension of $M*N$ along $\mu$), so
 $M\btm N\:\CA\to\Sets$.
\end{definition}

We can make this more explicit by recalling the standard construction
of Kan extensions as colimits over comma categories.  Write 
$Q=M\btm N$ for brevity.  Consider a morphism $f\:u\ot v\to a$ in
$\CA$ together with elements $m\in M(u)$ and $n\in N(v)$.  These data
give an element $\tht(f,m,n)\in Q(a)$, and all elements of $Q(a)$
arise in this way.  Now suppose we have maps $r\:u'\to u$ and
$s\:v'\to v$, and elements $m'\in M(u')$ and $n'\in N(v')$, such that
$r_*(m')=m$ and $s_*(n')=n$.  In this context we have
\[ \tht(f,m,n) = \tht(f,r_*(m'),s_*(n')) = 
    \tht(f\circ(r\ot s),m',n') \in Q(a).
\]
Moreover, all identities between elements of $Q(a)$ are consequences
of identities of this type.  For any map $g\:a\to b$, the induced map
$g_*\:Q(a)\to Q(b)$ is just $g_*(\tht(f,m,n))=\tht(gf,m,n)$.  

\begin{proposition}
 The functor $M\btm N\:\CA\to\Sets$ preserves products (and so is an
 $\CA$-module). 
\end{proposition}
\begin{proof}
 Consider a finite family of objects $(a_t)_{t\in T}$.  Put
 $a=\bigoplus_ta_t$ and let $a_t\xra{i_t}a\xra{p_t}a_t$ be the usual
 maps.  Let $\phi\:Q(a)\to\prod_tQ(a_t)$ be the map whose $t$'th
 component is $(p_t)_*\:Q(a)\to Q(a_t)$.  The claim is that $\phi$ is
 a bijection.  

 Suppose we have maps $f_t\:u_t\ot v_t\to a_t$ and elements
 $m_t\in M(u_t)$ and $n_t\in N(v_t)$ giving elements
 $q_t=\tht(f_t,m_t,n_t)\in Q(a_t)$.  Put $u=\bigoplus_tu_t$ and
 $v=\bigoplus_tv_t$.  By the product property of $a$, there is a
 unique map $f\:u\ot v\to a$ such that the diagram
 \[ \xymatrix{
  u\ot v \ar[r]^{f} \ar[d]_{p_t\ot p_t} & a \ar[d]^{p_t} \\
  u_t\ot v_t \ar[r]_{f_t} & a_t 
 } \]
 commutes for all $t$.  As the functor $M$ preserves products, there
 is a unique element $m\in M(u)$ with $(p_t)_*(m)=m_t\in M(u_t)$ for all
 $t$.  Similarly, there is a unique element $n\in N(v)$ with
 $(p_t)_*(n)=n_t\in N(v_t)$ for all $t$.  Now put
 $q=\tht(f,m,n)\in Q(a)$.  We have
 \[ (p_t)_*(q)=\tht(p_tf,m,n)=
     \tht(f_t\circ(p_t\ot p_t),m,n) = 
     \tht(f_t,(p_t)_*(m),(p_t)_*(n)) = 
     \tht(f_t,m_t,n_t) = q_t,
 \]
 so $\phi(q)=(q_t)_{t\in T}$.  This proves that $\phi$ is surjective.

 The element $q$ here depends on the maps $f_t$ and the elements $m_t$
 and $n_t$, so we can define 
 \[ \psi_0((f_t)_{t\in T},(m_t)_{t\in T},(n_t)_{t\in T}) = q. \]
 We next claim that this map $\psi_0$ is compatible with the defining
 relations for $Q$, so it induces a map $\psi\:\prod_tQ(a_t)\to Q(a)$
 with $\psi((q_t)_{t\in T})=q$.  To see this, suppose we have maps
 $r_t\:u'_t\to u_t$ and $s_t\:v'_t\to v_t$, together with elements
 $m'_t\in M(u'_t)$ and $n'_t\in N(v'_t)$ satisfying
 $(r_t)_*(m'_t)=m_t$ and $(s_t)_*(n'_t)=n_t$.  This means 
 that $q_t=\tht(f_t\circ(r_t\ot s_t),m'_t,n'_t)$.  Put
 \[ q' =
     \psi_0((f_t\circ(r_t\ot s_t))_{t\in T},
            (m'_t)_{t\in T},(n'_t)_{t\in T});
 \]
 we need to show that $q'=q$.  For this, we put $u'=\bigoplus_tu'_t$
 and $v'=\bigoplus_tv'_t$, and let $m'$ and $n'$ be the evident
 elements of $M(u')$ and $N(v')$.  By definition, we have
 $q'=\tht(g,m',n')$, where $g\:u'\ot v'\to a$ is the unique map such
 that $p_t\circ g=f_t\circ(r_t\ot s_t)\circ(p_t\ot p_t)$ for all $t$.
 By inspecting the diagram
 \[ \xymatrix{
  u'\ot v' \ar[r]^{r\ot s} \ar[d]_{p_t\ot p_t} &
  u\ot v \ar[r]^f \ar[d]_{p_t\ot p_t} &
  a \ar[d]^{p_t} \\
  u'_t\ot v'_t \ar[r]_{r_t\ot s_t} &
  u_t\ot v_t \ar[r]_{f_t} &
  a_t
 } \]
 we see that $g=f\circ(r\ot s)$, so $q'=\tht(f\circ(r\ot s),m',n')$.
 By the defining relations for $Q(a)$, this is the same as
 $\tht(f,r_*(m'),s_*(n'))$.  It is straightforward to check that
 $r_*(m')=m$ and $s_*(n')=n$, so $q'=q$ as claimed.  It follows easily
 that there is a well-defined map $\psi\:\prod_tQ(a_t)\to Q(a)$ as
 described previously, and that $\phi\psi=1\:Q(a)\to Q(a)$.

 In the opposite direction, consider an arbitrary element
 $q\in Q(a)$, where $a=\bigoplus_ta_t$ as before.  We can represent
 $q$ as $\tht(f,m,n)$ for some morphism $f\:u\ot v\to a$ and elements
 $m\in M(u)$ and $n\in N(v)$.  This means that
 $\phi(q)=(\tht(\pi_tf,m,n))_{t\in T}$.  Put $u'=\bigoplus_{t\in T}u$
 and $v'=\bigoplus_{t\in T}v$.  Let $m'\in M(u')$ be the unique
 element such that $(p_t)_*(u')=u$ for all $t$, and similarly for
 $n'\in N(v')$.  By the product property of $a$, there is a unique map
 $g\:u'\ot v'\to a$ such that the following diagram commutes for all
 $t$: 
 \[ \xymatrix{
  u'\ot v' \ar[rr]^g \ar[d]_{p_t\ot p_t} & & 
  a \ar[d]^{p_t} \\
  u\ot v \ar[r]_f &
  a \ar[r]_{p_t} &
  a_t.
 } \]
 From the definitions, we see that $\psi\phi(q)=\tht(g,m',n')$.  Now
 let $d\:u\to u'$ be the diagonal map, so $p_td=1\:u\to u$ for all
 $t$.  Similarly, let $e\:v\to v'$ be the diagonal map.  One can then
 check that $m'=d_*(m)$ and $n'=e_*(n)$, and also that
 $g\circ(d\ot e)=f$.  From the defining relations fr $Q$ we therefore
 have
 \[ \psi\phi(q) = \tht(g,d_*(m),e_*(n)) =
     \tht(g\circ(d\ot e),m,n) = \tht(f,m,n) = q.
 \]
 Thus, $\psi$ is the required inverse for $\phi$.
\end{proof}

\begin{construction}
 Let $N$ be an $\CA$-module, and $u$ an object of $\CA$.  We write
 $T_uN$ for the composite functor
 \[ \CA \xra{u\ot(-)} \CA \xra{N} \Sets. \]
 As both $u\ot(-)$ and $N$ preserve products, the same is true of
 $T_uN$, so $T_uN$ is again an $\CA$-module.  It is formal to check
 that the assignment $u\mapsto T_uN$ gives an additive (and therefore
 product-preserving) functor $\CA\to\Mod(\CA)$.  Now let $M$ be another
 $\CA$-module.  We define $\uHom(M,N)\:\CA\to\Sets$ by 
 \[ \uHom(M,N)(u) = \Mod(\CA)(M,T_uN). \]
 It is again easy to see that this preserves products, so it is itself
 an $\CA$-module.
\end{construction}

\begin{proposition}\label{prop-box-adjoint}
 For all $\CA$-modules $L$, $M$ and $N$ there are natural isomorphisms 
 \[ \Mod(\CA)(L,\uHom(M,N)) \simeq \Mod(\CA)(L\btm M,N). \]
\end{proposition}
\begin{proof}
 A morphism $\al\:L\to\uHom(M,N)$ consists of maps
 $\al_u\:L(u)\to\Mod(\CA)(M,T_uN)$ for all $u$ that are natural in
 $u$.  Thus, for each $r\in L(u)$ we have $\al_u(r)\:M\to T_uN$,
 consisting of maps 
 \[ \al_u(r)_v\:M(v)\to T_uN(v) = N(u\ot v) \]
 that are natural in $v$.  Recall that we have functors
 $L*M,\mu^*N\:\CA\tm\CA\to\Sets$ given by $(L*M)(u,v)=L(u)\tm N(v)$
 and $(\mu^*N)(u,v)=N(u\ot v)$.  We define $\al^\#\:L*M\to\mu^*N$ by 
 \[ \al^\#_{u,v}(r,s) = \al_u(r)_v(s). \]
 It is straightforward to check that this construction is bijective.
 On the other hand, by the definition of Kan extensions, we see that
 natural maps $L*M\to\mu^*N$ biject with morphisms
 $L\btm M\to N$.
\end{proof}

\begin{proposition}\label{prop-box-symmetric}
 The functor $\btm$ gives a symmetric monoidal structure on
 $\Mod(\CA)$.  The unit object is the functor $I(a)=\CA(1,a)$, where
 $1$ is the unit object in $\CA$.
\end{proposition}
\begin{proof}
 First, let $\tau\:a\ot b\to b\ot a$ be the symmetry isomorphism for
 $\CA$.  For modules $M,N$ we then have an obvious isomorphism
 $(M*N)\circ\tau\simeq N*M$, and using standard properties of Kan
 extensions this gives an isomorphism
 $M\btm N\simeq N\btm M$.

 Next, the Yoneda lemma gives
 \[ \uHom(I,N)(u)=\Mod(\CA)(I,T_uN)=T_uN(1)=N(u\ot 1)=N(u) \]
 for all $u$, so $\uHom(I,N)=N$.  Now take $M=I$ in
 Proposition~\ref{prop-box-adjoint} to get
 $\Mod(\CA)(L,N)\simeq\Mod(\CA)(L\btm I,N)$.  This is natural in
 $L$ and $N$ and so gives $L\btm I\simeq L$.

 Now suppose we have three modules $L$, $M$ and $N$.  We would like to
 give a natural isomorphism $(L\btm M)\btm N\simeq L\btm(M\btm N)$.
 For any module $P$, we let $T(P)$ denote the set of natural maps 
 \[ L(u)\tm M(v) \tm N(w) \to P(u\ot v\ot w) \]
 (of functors $\CA^3\to\Sets$).  It will suffice to show that the
 functor $T$ is represented by both $(L\btm M)\btm N$ and
 $L\btm(M\btm N)$.  Proposition~\ref{prop-box-adjoint} gives 
 \[ \Mod(\CA)((L\btm M)\btm N,P) \simeq
     \Mod(\CA)(L\btm M,\uHom(N,P)) \simeq 
      [\CA^2,\Sets](L*M,\mu^*\uHom(N,P)).
 \]
 This is the set of maps $L(u)\tm M(v)\to\uHom(N,P)(u\ot v)$ that are
 natural in $u$ and $v$.  After filling in the definition of $\uHom$
 this becomes the set of maps $L(u)\tm M(v)\tm N(w)\to P(u\ot v\ot w)$
 that are natural in all variables, or in other words $T(P)$, as
 required.  A similar argument covers the case of $L\btm(M\btm N)$.
 We leave further quaestions of naturality and coherence to the
 reader. 
\end{proof}

\begin{proposition}\label{prop-box-representable}
 Let $H_a\:\CA\to\Sets$ be the functor represented by $a$.  Then $H_a$
 preserves products, so it can be regarded as an $\CA$-module.
 Moreover, there are natural isomorphisms
 $H_a\btm H_b\simeq H_{a\ot b}$ and
 $\uHom(H_a,M)\simeq T_aM$ for all $M$.
\end{proposition}
\begin{proof}
 It is tautological that $H_a$ preserves products.  There are natural
 identifications 
 \[ \uHom(H_a,M)(b) = \Mod(\CA)(H_a,T_bM) =
     (T_bM)(a) = M(a\ot b) = (T_aM)(b),
 \]
 which gives $\uHom(H_a,M)=T_aM$.  This in turn gives natural
 isomorphisms  
 \begin{align*}
  \Mod(\CA)(H_a\ot H_b,M)
   &= \Mod(\CA)(H_a,\uHom(H_b,M)) 
    = \Mod(\CA)(H_a,T_bM) \\
   &= (T_bM)(a) = M(a\ot b) = \Mod(\CA)(H_{a\ot b},M).
 \end{align*}
 By the Yoneda Lemma, this gives $H_a\ot H_b\simeq H_{a\ot b}$.
\end{proof}

\renewcommand{\refname}{Bibliography}


\begin{bibdiv}
\begin{biblist}
\bib{bo:hcaii}{book}{
  author={Borceux, Francis},
  title={Handbook of categorical algebra. 2},
  series={Encyclopedia of Mathematics and its Applications},
  volume={51},
  note={Categories and structures},
  publisher={Cambridge University Press},
  place={Cambridge},
  date={1994},
  pages={xviii+443},
  isbn={0-521-44179-X},
  review={\MR {1313497 (96g:18001b)}},
}

\bib{bovo:hia}{book}{
  author={Boardman, J. M.},
  author={Vogt, R. M.},
  title={Homotopy invariant algebraic structures on topological spaces},
  series={Lecture Notes in Mathematics, Vol. 347},
  publisher={Springer-Verlag},
  place={Berlin},
  date={1973},
  pages={x+257},
  review={\MR {0420609 (54 \#8623a)}},
}

\bib{br:wve}{article}{
  author={Brun, M.},
  title={Witt vectors and equivariant ring spectra applied to cobordism},
  journal={Proc. Lond. Math. Soc. (3)},
  volume={94},
  date={2007},
  number={2},
  pages={351--385},
  issn={0024-6115},
  review={\MR {2308231 (2008a:55008)}},
  doi={10.1112/plms/pdl010},
}

\bib{br:wvt}{article}{
  author={Brun, Morten},
  title={Witt vectors and Tambara functors},
  journal={Adv. Math.},
  volume={193},
  date={2005},
  number={2},
  pages={233--256},
  issn={0001-8708},
  review={\MR {2136887 (2005m:13034)}},
}

\bib{cr:ati}{article}{
  title={Algebraic theories and $(\infty ,1)$-categories},
  author={James Cranch},
  eprint={arXiv:1011.3243v1 [math.AT]},
}

\bib{da:ccf}{article}{
  author={Day, Brian},
  title={On closed categories of functors},
  conference={ title={Reports of the Midwest Category Seminar, IV}, },
  book={ series={Lecture Notes in Mathematics, Vol. 137}, publisher={Springer}, place={Berlin}, },
  date={1970},
  pages={1--38},
  review={\MR {0272852 (42 \#7733)}},
}

\bib{dr:orr}{article}{
  author={Dress, Andreas},
  title={Operations in representation rings},
  conference={ title={Representation theory of finite groups and related topics (Proc. Sympos. Pure Math., Vol. XXI, Univ. Wisconsin, Madison, Wis., 1970)}, },
  book={ publisher={Amer. Math. Soc.}, place={Providence, R.I.}, },
  date={1971},
  pages={39--45},
  review={\MR {0325740 (48 \#4086)}},
}

\bib{drsi:brp}{article}{
  author={Dress, Andreas W. M.},
  author={Siebeneicher, Christian},
  title={The Burnside ring of profinite groups and the Witt vector construction},
  journal={Adv. in Math.},
  volume={70},
  date={1988},
  number={1},
  pages={87--132},
  issn={0001-8708},
  review={\MR {947758 (89m:20025)}},
}

\bib{el:bri}{article}{
  author={Elliott, Jesse},
  title={Binomial rings, integer-valued polynomials, and $\lambda $-rings},
  journal={J. Pure Appl. Algebra},
  volume={207},
  date={2006},
  number={1},
  pages={165--185},
  issn={0022-4049},
  review={\MR {2244389 (2007f:13031)}},
}

\bib{el:cwb}{article}{
  author={Elliott, Jesse},
  title={Constructing Witt-Burnside rings},
  journal={Adv. Math.},
  volume={203},
  date={2006},
  number={1},
  pages={319--363},
  issn={0001-8708},
  review={\MR {2231048}},
}

\bib{jo:qck}{article}{
  author={Joyal, A.},
  title={Quasi-categories and Kan complexes},
  note={Special volume celebrating the 70th birthday of Professor Max Kelly},
  journal={J. Pure Appl. Algebra},
  volume={175},
  date={2002},
  number={1-3},
  pages={207--222},
  issn={0022-4049},
  review={\MR {1935979 (2003h:55026)}},
  doi={10.1016/S0022-4049(02)00135-4},
}

\bib{lu:htt}{book}{
  author={Lurie, Jacob},
  title={Higher topos theory},
  series={Annals of Mathematics Studies},
  volume={170},
  publisher={Princeton University Press},
  place={Princeton, NJ},
  date={2009},
  pages={xviii+925},
  isbn={978-0-691-14049-0},
  isbn={0-691-14049-9},
  review={\MR {2522659 (2010j:18001)}},
}

\bib{ma:cwm}{book}{
  author={MacLane, Saunders},
  title={Categories for the working mathematician},
  series={Graduate Texts in Mathematics},
  publisher={Springer--Verlag},
  date={1971},
  volume={5},
}

\bib{mama:eos}{article}{
  author={Mandell, Michael~A.},
  author={May, J.~Peter},
  title={Equivariant orthogonal spectra and $S$-modules},
  date={2002},
  journal={Mem. Amer. Math. Soc.},
  volume={159},
  pages={108},
}

\bib{na:btt}{article}{
  title={Biset transformations of Tambara functors},
  author={Nakaoka, Hiroyuki},
  eprint={arXiv:1105:0714},
}

\bib{na:fsm}{article}{
  title={On the fractions of semi-Mackey and Tambara functors},
  author={Nakaoka, Hiroyuki},
  eprint={arXiv:1103:3991},
}

\bib{na:gdc}{article}{
  title={A generalization of The Dress construction for a Tambara functor, and polynomial Tambara functors},
  author={Nakaoka, Hiroyuki},
  eprint={arXiv:1012:1911},
}

\bib{na:itf}{article}{
  title={Ideals of Tambara functors},
  author={Nakaoka, Hiroyuki},
  eprint={arXiv:1101:5982},
}

\bib{na:tfp}{article}{
  author={Nakaoka, Hiroyuki},
  title={Tambara functors on profinite groups and generalized Burnside functors},
  journal={Comm. Algebra},
  volume={37},
  date={2009},
  number={9},
  pages={3095--3151},
  issn={0092-7872},
  review={\MR {2554193 (2011b:19002)}},
  doi={10.1080/00927870902747605},
}

\bib{na:tmf}{article}{
  author={Nakaoka, Hiroyuki},
  title={Tambarization of a Mackey functor and its application to the Witt-Burnside construction},
  journal={Adv. Math.},
  volume={227},
  date={2011},
  number={5},
  pages={2107--2143},
  issn={0001-8708},
  review={\MR {2803797}},
  doi={10.1016/j.aim.2011.04.015},
}

\bib{st:cop}{article}{
  author={Steiner, R.},
  title={A canonical operad pair},
  date={1979},
  journal={Mathematical Proceedings of the Cambridge Philosophical Society},
  volume={86},
  pages={443\ndash 449},
}

\bib{ta:mt}{article}{
  author={Tambara, D.},
  title={On multiplicative transfer},
  journal={Comm. Algebra},
  volume={21},
  date={1993},
  number={4},
  pages={1393--1420},
  issn={0092-7872},
  review={\MR {1209937 (94a:19002)}},
}

\bib{we:tcs}{article}{
  author={Webb, Peter},
  title={Two classifications of simple Mackey functors with applications to group cohomology and the decomposition of classifying spaces},
  journal={J. Pure Appl. Algebra},
  volume={88},
  date={1993},
  number={1-3},
  pages={265--304},
  issn={0022-4049},
  review={\MR {1233331 (94f:20101)}},
  doi={10.1016/0022-4049(93)90030-W},
}

\bib{nlab:lt}{article}{
  author={nLab contributors},
  title={Lawvere theory},
  eprint={http://ncatlab.org/nlab/show/Lawvere+theory},
}

\bib{nlab:bc}{article}{
  author={nLab contributors},
  title={Bicategory},
  eprint={http://ncatlab.org/nlab/show/bicategory},
}

\end{biblist}
\end{bibdiv}
\end{document}